\documentclass[10pt,reqno]{amsart}

\usepackage{hyperref}
\usepackage{amsmath}
\usepackage{amssymb}
\usepackage{amsfonts}
\usepackage{amsthm}
\usepackage{latexsym}
\usepackage{esint}
\usepackage{mathtools}
\usepackage{mathrsfs}
\usepackage{graphicx}
\usepackage{bbm}
\usepackage{color}
\usepackage{bm}
\usepackage[top=1in, bottom=1.25in, left=1.25in, right=1.25in]{geometry}

\newtheorem{theorem}{Theorem}[section]
\newtheorem{lemma}[theorem]{Lemma}
\newtheorem{proposition}[theorem]{Proposition}
\newtheorem{corollary}[theorem]{Corollary}

\theoremstyle{definition}
\newtheorem{definition}[theorem]{Definition}

\theoremstyle{remark}
\newtheorem{remark}[theorem]{Remark}
\DeclareMathOperator*{\esssup}{ess\,sup}

\numberwithin{equation}{section}

\newcommand{\bu}{\mathbf{u}}

\begin{document}

\title[Martingale solutions of the stochastic Navier--Stokes--Poisson system]{Dissipative martingale solutions of the stochastically forced Navier--Stokes--Poisson system on domains without boundary}
%

\author{Donatella Donatelli}
\address{Department of Information Engineering, Computer Science and Mathematics \\
University of L’Aquila\\
67100 L’Aquila, Italy}
\email{donatella.donatelli@univaq.it}

\author{Pierangelo Marcati}
\address{GSSI - Gran Sasso Science Institute \\
Viale F. Crispi, 7 \\
67100 L’Aquila, Italy.}
\email{pierangelo.marcati@univaq.it, \, pierangelo.marcati@gssi.infn.it}

\author{Prince Romeo Mensah}
\address{GSSI - Gran Sasso Science Institute \\
Viale F. Crispi, 7 \\
67100 L’Aquila, Italy.}
\email{romeo.mensah@gssi.it}


\date{\today}



\begin{abstract}
We construct solutions to the randomly-forced Navier--Stokes--Poisson system in periodic three-dimensional domains  or in the whole three-dimensional Euclidean space. These solutions are weak in the sense of PDEs and also weak in the sense of probability. As such, they satisfy the system in the sense of distributions and the underlying probability space and the stochastic driving force are also unknowns of the problem. Additionally, these solutions dissipate energy, satisfies a relative energy inequality in the sense of \cite{dafermos1979second} and satisfy a renormalized form of the continuity equation in the sense of \cite{diperna1989ordinary}.
\end{abstract}

\maketitle

\section{Introduction}
\label{sec:intro}
Let $t\geq 0$ and $x\in \mathcal{O}$,  where $\mathcal{O}=\mathbb{T}^3$ is the three dimensional torus or $\mathcal{O}=\mathbb{R}^{3}$, $\vartheta, \nu^S>0$, $\nu^B\geq0$ be  constants and $f=f(x)$ be a given function. We will construct a class of solution to the following stochastically forced Navier--Stokes--Poisson system
\begin{align}
&\mathrm{d} \varrho  + \mathrm{div}(\varrho  \mathbf{u} ) \, \mathrm{d}t =0, 
\label{contEq}
\\
&\mathrm{d}(\varrho  \mathbf{u} ) +
\big[
\mathrm{div} (\varrho  \mathbf{u}  \otimes \mathbf{u} ) + \nabla  p( \varrho) \big]\, \mathrm{d}t 
=\big[ \nu^S \Delta \mathbf{u}  +(\nu^B + \nu^S)\nabla \mathrm{div} \mathbf{u} 
\nonumber\\
&\qquad \qquad \qquad   \qquad 
  + \vartheta \varrho  \nabla V  \big] \, \mathrm{d}t + \varrho \, \mathbf{G} (\varrho , f,  \varrho \mathbf{u} ) \, \mathrm{d}W(t) ,
\label{momEq}
\\
&\pm \Delta V  = \varrho  - f &
\label{elecField}
\end{align}
which are simultaneously weak  in  the sense of PDEs and in the sense of probability. The former notion of \textit{weak} means that each individual equation \eqref{contEq}, \eqref{momEq} and \eqref{elecField} is satisfied in the sense of distributions, and the latter notion of \textit{weak} means that the underlying probability space and the Wiener process $W$ are also unknowns of the problem. Furthermore,  equation \eqref{elecField}
is satisfied pointwise almost everywhere in its domain meaning that it indeed satisfied \textit{strongly} in the sense of PDEs.
\\
For simplicity, we assume that the system \eqref{contEq}--\eqref{elecField} is \textit{isentropic} so that the equation of state is
\begin{equation}
p(\varrho)=a \varrho^\gamma, \quad a>0, \quad \gamma >\frac{3}{2}.
\end{equation}
\begin{remark}
The result of this work also holds for generalized pressure laws. In particular, the result holds for any pressure satisfying
\begin{align}
\label{arbitraryPressure}
p\in C^1(\mathbb{R}_{\geq0}), \,\, p(0)=0, \,\, \text{and for all } \varrho\geq0, \,\, \frac{1}{a}\varrho^{\gamma-1} -b \leq p'(\varrho) \leq  a \varrho^{\gamma-1} +b
\end{align}
where $a>0$ and $b$ are constants. For some physical models covered by the relation \eqref{arbitraryPressure}, we refer the reader to \cite[Section 1.1]{ducomet2004global}.
\end{remark}
\begin{remark}
The result also holds in the lower one and two dimensions under an even stronger pressure law, i.e., when the adiabatic exponent satisfies $\gamma\geq 1$ and $\gamma>1$ respectively.
\end{remark}
The  system \eqref{contEq}--\eqref{elecField} under study has various applications including the study of semiconductor devices, nuclear fluids, stellar dynamics, amongst others. In particular, in  the momentum equation \eqref{momEq}, the force term $\vartheta \varrho  \nabla V$ may represent the  extraneous force acting on the fluid due  to the electric field or to self-gravitation. Accordingly, the Poisson equation represents the balance of the related potentials $V$ (electric or gravitational). The novelty of the above model is that the momentum equations incorporates also the possible influence of external forces (be it random or deterministic), possible noise and errors that the classical purely-deterministic system may fail to capture.

If we set $\mathbf{G}=0$ in \eqref{momEq}, then we obtain a deterministic system \eqref{contEq}--\eqref{elecField} for which existence of weak solution has been studied by \cite{donatelli2003local} for bounded domains and by \cite{ducomet2004global} on unbounded domains. The former  relies primarily on a fixed-point argument whereas the latter uses the now standard multi-layer approximation scheme for the construction of weak solutions to compressible systems. A relevant application of this multi-layer approximation scheme is in the construction of weak martingale solutions for the stochastic compressible Navier--Stokes system \eqref{contEq}--\eqref{momEq} with $V=0$. This was independently accomplished by \cite{Hof} (see also, the manuscript \cite{breit2018stoch}) for periodic boundaries and by \cite{smith2015random} for bounded domain. By building on \cite{breit2018stoch}, the author of \cite{mensah2016existence} extended the result to the whole space. Our result will therefore follow the manuscript \cite{breit2018stoch} to construct solutions to \eqref{contEq}--\eqref{elecField} on periodic domains and then follow the approach of \cite{mensah2016existence} to complete the picture for the whole space. Refer to Section \ref{sec:plan} for more details.

\subsection{Plan of paper}
\label{sec:plan}
In the next section, Section \ref{sec:prelim}, we present some notations, conventions and definitions that we will use throughout this paper. Since this paper discusses an existence result, we will in particular define the precise concept of a solution that we are interested in constructing. We will finally complete the section by stating the main results.

Our proof of existence of solution to \eqref{contEq}--\eqref{elecField} will rely on the energy method. Since this system is stochastic, we will present in Section \ref{sec:energyIneq}, the formal equivalent derivation of the energy inequality from which regularity of the solution variables are obtained. Unlike deterministic PDEs where such energy estimates are obtained by formally testing the system with the solution, the stochastic energy estimate will be obtained from a formal application of an infinitesimal It\^o's formula.

The main part of the proof of our first main result will commence in Section \ref{sec:firstLayer}. This is a preliminary approximation layer that uses Cauchy's collocation method to construct a stochastically strong solution to a finite-dimensional auxiliary system of \eqref{contEq}--\eqref{elecField} containing an additional artificial pressure term, an artificial viscosity terms and some cut-off of the velocity field in the mass and momentum balance equations \eqref{contEq}--\eqref{momEq}. In the deterministic sense, this solution will be strong in the auxiliary version of the continuity equation \eqref{contEq} and the Poisson equation \eqref{elecField} but weak in the auxiliary version of the momentum equation \eqref{momEq}. In order words, the former equations are satisfied pointwise almost everywhere in spacetime whereas the latter is satisfied in the sense of distributions. The construction of the solution to the fluid part \eqref{contEq}--\eqref{momEq} of the system, albeit the extra electric field term and the data $f$ input in the noise, follow the approach in \cite{breit2018stoch} by simply freezing these components in time. Given the information about the fluid system and further information on the data $f$, we are then able to construct a solution to the Poisson equation \eqref{elecField} by using standard regularity theorem for the Poisson equation.

The cut-off function of the velocity field mentioned in the previous paragraph is meant to deal with any anticipated vacuum region or potential blowup due to the noise. This cut-off is coupled with a
stopping time argument to establish the existence of a
unique solution to a finite-dimensional approximation derived from the previous layer. In particular, on each stopping time, we are able to construct a Faedo--Galerkin approximation with no potential blowup once these cut-offs are activated.
By passing to the limit in this cut-off parameter (which coincide with passing to the limit in the family of stopping times), we establish the existence of a unique Faedo--Galerkin approximation on the entire time interval in Section \ref{sec:secondLayer}. The notion of a solution being constructed in this layer is inherited from the previous layer with the obvious exemption of any cut-offs which essentially converge in various norms to the identity operator. Furthermore, this solution is shown to conserve energy so we obtain an energy \textit{equality} rather than an \textit{inequality}.

We finally transition from a finite-dimensional system as constructed in the previous layers to an infinitesimal system by passing to the limit in the finite-dimensional projection parameter in Section \ref{sec:thirdLayer}. The notion of a solution in the sense of PDEs is again inherited from the previous layers. However, from a strong solution in the sense of probabilities, we only obtain a weak solution in this infinitesimal approximation layer. Also, energy is dissipated in this layer rather than conserved. As a result of the solution being weak in the sense of probabilities, one is expected to prescribe an initial law and the underlying probability space and stochastic driving force becomes unknowns as well. The steps in the construction of this solution -- which will be replicated in subsequent approximation layers -- is a follows.
\begin{itemize}
\item We approximate a suitable law prescribed on  data defined in a infinitesimal space by a family of laws -- indexed by the discretization parameter -- prescribed on the data from the previous section. This constructed law satisfies the hypothesis of the main result, Theorem \ref{thm:mainThirdLayer} of this approximation layer.
\item By using the finiteness of the `limit' initial law construction above, we obtain a priori bounds for any family of solutions -- indexed by the discretization parameter -- uniformly in this parameter.
\item The laws on these approximate solutions are shown to be tight on the corresponding spaces in which they live and by using the stochastic compactness method by Jakubowski
 and Skorokhod, we obtain `limit' random variables for these solutions as a result of the finiteness of the `limit' initial law.
 \item We then show in that these `limit' random variables solves the required system without the  discretization parameter.
\end{itemize}
The bullet points above gives the general structure of the proof which is the same for the corresponding version of the strictly fluid system studied in \cite{breit2018stoch}. The extra detail which needs to be tackled is the establishment of compactness for the family of electric fields due to the additional Poisson equation. As is the case for stochastic fluid systems, this will follow from a tightness argument which requires a further degree of regularity for these fields besides the regularity obtained due to the conservation of energy. Fortunately, since the Possion equation is linear, this is quite straightforward and all we need is information on the data $f$ and the corresponding family of densities. We have the former by way of assumption and the information on the latter is derived from the tightness argument on the family of densities as shown in \cite{breit2018stoch}.

Section \ref{sec:fourthLayer} deals with the vanishing artificial viscosity limit. The main difference in the concept of a solution from the previous sections is that the solution to the continuity equation \ref{contEq} is now weak in the sense of PDEs. However, the construction of this solution will mimic the itemized steps in the previous paragraph. That is, we first construct a suitable law from a family of laws from the previous layer, show that any solution indexed by the current approximation parameter satisfies suitable bounds due to energy dissipation from the previous layer, show compactness of these approximate solutions and finally, identify the limit system. However, an intermediate step between obtaining the usual uniform estimates from the energy inequality and compactness is the requirement to improve the regularity of the density sequence. This problem already exists in the analyses of the deterministic Navier--Stokes(--Poisson) system and the stochastic Navier--Stokes system and they results from the fact that the continuity equation is only expected to be satisfied weakly in the sense of distributions after the passage to the limit in this artificial viscous term. This improvement in regularity of the density sequence will rely on the stochastic adaptation of the analysis of the \textit{effective viscous flux}. At the level of the stochastic Navier--Stokes system for fluids, \cite{breit2018stoch}, this aforementioned improved regularity is obtained from the application of It\^o's lemma which results in an equation corresponding to formally `testing' the momentum equation with the continuity equation. This crucial step is performed differently in our case. Instead of testing with the continuity equation, we use an equation derived from the combination of the continuity equation and the Poisson equation with the help of the inverse Laplace operator. This gives a much cleaner expression and more importantly, allow us to estimate the term containing the electric field.

The final step in the construction of a solution to \eqref{contEq}--\eqref{elecField} on the torus is the vanishing pressure limit which is presented in Section \ref{sec:fifthLayer}. The analyses is similar to Section \ref{sec:fourthLayer} but a further loss of regularity due to the loss of the artificial viscosity means an additional analyses of the \textit{oscillation defect measure} in order to obtain an improved regularity of the density sequence. Furthermore, unlike the preceding layer where the analyses of  the \textit{effective viscous flux} was performed by testing the momentum equation with a combined equation derived from the two other equations, we are unable to do that in this present layer due to the fact that a variational power of density--which no longer solves the continuity equation--is required. As a result, the analyses of both the \textit{effective viscous flux} and the \textit{oscillation defect measure} will have to follow the corresponding arguments in \cite{breit2018stoch} and thus requiring  a new treatment of electric field terms in both case. Fortunately, the conservation of mass and its boundedness together with the regularity of the electric field obtained from the energy estimate is enough to resolve this problem.

Section \ref{sec:existenceWholeSpace} studies a corresponding version of Theorem \ref{thm:one}
on the whole space. This is stated in Theorem \ref{thm:WS}. This involves approximating the problem on $\mathbb{R}^3$ by an increasing sequence of periodic problems, establish local-in-space uniform bounds for this family, show tightness of prescribed laws then proceed to apply Jakubowski-Skorokhod compactness to
obtain limit variables which are identified in the limit. The crucial term to be identified is  the pressure and again, requires on the analyses of the  \textit{effective viscous flux} and the  \textit{oscillation defect measure}. By combining the techniques in dealing with these two quantities from the preceding  layer with the result for the stochastic Navier--Stokes system on the whole space studied in \cite{mensah2016existence} and \cite{mensah2019theses}, we are able to show existence on the whole space.

A property of the solution that we will construct in this paper is that it satisfies a \textit{relative energy inequality}. This estimate dates back to work by Dafermos \cite{dafermos1979second} on the
connection between the second law of thermodynamics and the  stability property of
these processes. This inequality has several uses amongst which is the study of singular limits of the system  \eqref{contEq}--\eqref{elecField}. A formal derivation of this estimate is therefore presented in Section \ref{sec:relativeEnergy} for potential future application.
\section{Preliminaries}
\label{sec:prelim}
\subsection{Notations}
Beside the standard conventions and notations used in the literature, we will also make use of the following  notations. 
\begin{itemize}
\item For functions $F$ and $G$ and a variable $p$, we denote by $F \lesssim G$ and $F \lesssim_p G$, the existence of  a generic constant $c>0$ and another such constant $c(p)>0$ which now depends on $p$ such that $F \leq c\,G$ and $F \leq c(p) G$ respectively.
\item We denote by $\mathfrak{P}(X)$, the set of probability measures on $X$.
\item $\mathcal{M}_b(X)$ denotes the space of bounded Borel measures on $X$.
\item We use the following notations
\begin{align*}
\Vert \cdot \Vert_{L^p_x}:= \Vert \cdot \Vert_{L^p(\mathbb{T}^3)}, \quad \Vert \cdot \Vert_{L^p_t}:= \Vert \cdot \Vert_{L^p(0,T)}
\end{align*}
for the Lebesgue norm/ Lebesgue integrable functions. Analogous representation holds for  Sobolev functions, Bochner integrable functions, continuous functions and so on. We emphasis that this notation will only be used when working on the torus. The $L^p$-norm on any other geometry or set will be shown explicitly.
\item $\big[ L^p_x\big]^q=L^p_x \times \ldots \times L^p_x$ $q$-times for any $p\in[1,\infty]$.
\end{itemize}

\subsection{Assumptions on the stochastic force}
\label{sec:noiseAssumption}
We start our analysis in the periodic setting, namely $\mathcal{O}=\mathbb{T}^{3}$. Let $\big(\Omega, \mathscr{F}, (\mathscr{F}_t)_{t\geq 0}, \mathbb{P} \big)$ be a stochastic basis endowed with  a right-continuous filtration $(\mathscr{F}_t)_{t\geq 0}$ and let $W$ be an $(\mathscr{F}_t)$-cylindrical Wiener process, i.e.,
\begin{align}
\label{wiener}
W(t) = \sum_{k\in \mathbb{N}} \beta_k(t) e_k, \quad t\in [0,T]
\end{align}
where $(\beta_k)_{k\in \mathbb{N}}$ is a family of mutually independent real-valued Brownian motions and $(e_k)_{k\in \mathbb{N}}$ are orthonormal basis of a separable Hilbert space $\mathfrak{U}$. Since the formal sum \eqref{wiener} is  not expected to converge in $\mathfrak{U}$, we can construct a larger space $\mathfrak{U}_0 \supset  \mathfrak{U}$ as follows
\begin{align}
\label{auxiliarySpace}
\mathfrak{U}_0  =\left\{ v = \sum_{k\geq1}c_ke_k\,;\quad  \sum_{k\geq 1} \frac{c^2_k}{k^2}<\infty \right\}
\end{align}
and endow it with the norm
\begin{align*}
\Vert  v  \Vert^2_{\mathfrak{U}_0}  =  \sum_{k\in\mathbb{N}} \frac{c^2_k}{k^2}, \quad   v=\sum_{k\in\mathbb{N}}c_ke_k.
\end{align*}
One can now check that \eqref{wiener} converges in $\mathfrak{U}_0$. Furthermore, $W$ has $\mathbb{P}$-a.s. $C([0,T];\mathfrak{U}_0)$ sample paths and the  embedding $\mathfrak{U}\hookrightarrow \mathfrak{U}_0$ is Hilbert--Schmidt. See \cite{da2014stochastic}.

To ensure that the stochastic integral $\int_0^\cdot\varrho \,\mathbf{G}(\varrho, f, \varrho \mathbf{u})\mathrm{d}W$ is a well-defined $(\mathscr{F}_t)$-martingale taking value in  a suitable Hilbert space $W^{-l,2}({\mathbb{T}^3})$, $l \geq0$ say, we set $\mathbf{m}=\varrho \mathbf{u}$ and for $(x,\varrho,f, \mathbf{u}) \in {\mathbb{T}^3}\times \mathbb{R}_{\geq0}\times \mathbb{R}_{\geq0}\times{\mathbb{R}^3} $,  we  assume that there exists some functions $(\mathbf{g}_k)_{k\in \mathbb{N}}$  such that 
\begin{align}
\label{noiseSupportExi}
\mathbf{g}_k: {\mathbb{T}^3}\times \mathbb{R}_{\geq0}\times \mathbb{R}_{\geq0}\times{\mathbb{R}^3}  \rightarrow {\mathbb{R}^3}, \quad \quad \mathbf{g}_k\in C^1 \left( {\mathbb{T}^3}\times \mathbb{R}_{>0}\times \mathbb{R}_{\geq 0}\times{\mathbb{R}^3}  \right)
\end{align}
for any $k\in\mathbb{N}$  and in addition, $ \mathbf{g}_k$ satisfies the following growth conditions:
\begin{align}
&\vert  \mathbf{g}_k(x, \varrho, f, \mathbf{m})   \vert  \leq c_k \,\left( \varrho + f+ \vert \mathbf{m}\vert \right),
\label{stochCoeffBoundExi0}
\\
&\big\vert\nabla_{\varrho,\mathbf{m}} \, \mathbf{g}_k(x, \varrho, f, \mathbf{m}) \big) \big\vert   \leq c_k,
\label{stochCoeffBound1Exis0}
\\
&\sum_{k\in \mathbb{N}} c_k^2 \lesssim 1 \label{stochsummableConst}
\end{align}
for some constant $(c_k)_{k\in \mathbb{N}} \in [0,\infty)$.
Now if we define the map $\mathbf{G}(\varrho, f, \mathbf{m}):\mathfrak{U}\rightarrow    L^1(  \mathbb{T}^3  )$ by
$\mathbf{G} (\varrho, f, \mathbf{m}) e_k   =   \mathbf{g}_k(\cdot, \varrho(\cdot), f(\cdot), \mathbf{m} (\cdot))$,  then we can conclude that $ \varrho\,\mathbf{G}(\varrho, f, \mathbf{m})$
is uniformly bounded  in $L_2(\mathfrak{U};
W^{-l,2}({\mathbb{T}^3} ))$  provided that, $\varrho^\gamma$, $f^\frac{2\gamma}{\gamma-1}$ and  $\varrho \vert \mathbf{u} \vert^2$ are  integrable in ${\mathbb{T}^3}$. Indeed, if we let $(\varrho)_{\mathbb{T}^3}<\infty$ represent the mean density on the torus,  then it follows from Young's inequality that
\begin{equation}
\begin{aligned}
\label{rhof}
 \int_{\mathbb{T}^3}\varrho^{2} f^2  \, \mathrm{d}x 
 \lesssim (\varrho)_{\mathbb{T}^3}
 \int_{\mathbb{T}^3}\varrho \,f^2  \, \mathrm{d}x
\lesssim
  \int_{\mathbb{T}^3} \varrho^{\gamma}  \, \mathrm{d}x
  +
  \int_{\mathbb{T}^3} f^\frac{2\gamma}{\gamma-1}  \, \mathrm{d}x
\end{aligned}
\end{equation}
where $2 < \frac{2\gamma}{\gamma -1} <6$. Also, by using $\varrho \leq 1+\varrho^\gamma$, we also gain
\begin{equation}
\begin{aligned}
\label{rho2m}
 \int_{\mathbb{T}^3} \varrho^{2} ( \varrho^2 + \mathbf{m}^2)  \, \mathrm{d}x 
 \lesssim
  (\varrho)_{\mathbb{T}^3}^3
 \int_{\mathbb{T}^3}\varrho^{-1} ( \varrho^2 + \vert  \mathbf{m} \vert^2)  \, \mathrm{d}x
\lesssim
  \int_{\mathbb{T}^3} ( 1+\varrho^\gamma + \varrho\vert  \mathbf{u} \vert^2)   \, \mathrm{d}x.
\end{aligned}
\end{equation}
Finally, by using \eqref{stochCoeffBoundExi0}--\eqref{stochsummableConst} and \eqref{rhof}--\eqref{rho2m}, as well as Sobolev's embedding, it follows that
\begin{equation}
\begin{aligned}
\label{eq:gk1}
\big\Vert &\varrho\,\mathbf{G}(\varrho, f,  \mathbf{m})  \big\Vert^2_{L_2(\mathfrak{U};
W^{-l,2}_x)} 
=
 \sum_{k\in \mathbb{N}}\big\Vert \varrho\, \mathbf{g}_k(x, \varrho, f, \mathbf{m}) \big\Vert^2_{W^{-l,2}_x}
\lesssim
 \int_{\mathbb{T}^3} \sum_{k\in \mathbb{N}}\varrho^{2} \big\vert   \mathbf{g}_k(x, \varrho, f, \mathbf{m}) \big\vert^2 \, \mathrm{d}x
  \\&
\lesssim
 \int_{\mathbb{T}^3} \sum_{k\in \mathbb{N}} c_k\varrho^{2} ( \varrho^2 + \mathbf{m}^2)  \, \mathrm{d}x
 +
  \int_{\mathbb{T}^3} \sum_{k\in \mathbb{N}} c_k\varrho^{2} f^2  \, \mathrm{d}x
\lesssim
\int_{\mathbb{T}^3} \big( 1+ \varrho^\gamma + \varrho\vert \mathbf{u}\vert^2 + f^\frac{2\gamma}{\gamma-1} \big) \, \mathrm{d}x
\end{aligned}
\end{equation}
provided that $l \geq \frac{3}{2}$. Boundedness thus follow if $\varrho^\gamma$, $f^\frac{2\gamma}{\gamma-1}$ and  $\varrho \vert \mathbf{u} \vert^2$ are  integrable in ${\mathbb{T}^3}$.

\subsection{Concept of a solution}
\label{sec:solution}
 To continue, let us define the notions of  solution that we wish to construct in this paper. 
\begin{definition}
\label{def:martSolutionExis}
Let $\Lambda= \Lambda(\varrho, \mathbf{m},f)$ be a Borel probability measure on $\big[ L^1_x\big]^3 $. We say that
\begin{align}
\label{weakMartSol}
\big[(\Omega ,\mathscr{F} ,(\mathscr{F} _t)_{t\geq0},\mathbb{P} );\varrho , \mathbf{u} , V , W   \big]
\end{align}
is a \textit{dissipate martingale solution} of \eqref{contEq}--\eqref{elecField} with initial law $\Lambda$ provided
\begin{enumerate}
\item $(\Omega,\mathscr{F},(\mathscr{F}_t),\mathbb{P})$ is a stochastic basis with a complete right-continuous filtration;
\item $W$ is a $(\mathscr{F}_t)$-cylindrical Wiener process;
\item the density $\varrho \in C_w \big( [0,T]; L^\gamma_x\big) $ $\mathbb{P}$-a.s., it is $(\mathscr{F}_t)$-progressively measurable  and 
\begin{align*}
\mathbb{E}\, \bigg[ \sup_{t\in(0,T)}\Vert  \varrho  (t)   \Vert^{\overline{\gamma}p}_{L^{\overline{\gamma}}_x}\bigg]<\infty 
\end{align*}
for all $1\leq p<\infty$ where $\overline{\gamma}=\min\{2,\gamma\}$,
\item the velocity field $\mathbf{u} $ is an $(\mathscr{F}_t)$-adapted random distribution and 
\begin{align*}
\mathbb{E}\, \bigg[ \int_0^T\Vert  \mathbf{u}  \Vert^2_{W^{1,2}_x }\,\mathrm{d}t\bigg]^p<\infty 
\end{align*}
for all $1\leq p<\infty$,
\item the field $\Delta V \in C_w \big( [0,T]; L^\gamma_x\big) $ $\mathbb{P}$-a.s., it is $(\mathscr{F}_t)$-progressively measurable  and 
\begin{align*}
\mathbb{E}\, \bigg[ \sup_{t\in(0,T)}\Vert \Delta V  (t)   \Vert^{\overline{\gamma}p}_{L^{\overline{\gamma}}_x}\bigg]
+
\mathbb{E}\, \bigg[ \sup_{t\in(0,T)}\Vert \nabla V  (t)   \Vert^{2p}_{L^{2}_x}\bigg]
<\infty 
\end{align*}
for all $1\leq p<\infty$ where $\overline{\gamma}=\min\{2,\gamma\}$,
\item the momentum $\varrho  \mathbf{u}  \in C_w \big( [0,T]; L^{\frac{2\gamma}{\gamma+1}}_x\big) $ $\mathbb{P}$-a.s., it is $(\mathscr{F}_t)$-progressively measurable and 
\begin{align*}
\mathbb{E}\, \bigg[ \sup_{t\in(0,T)} \big\Vert  \varrho \vert \mathbf{u} \vert^2(t)  \big\Vert_{L^1_x}^p\bigg]
+
\mathbb{E}\, \bigg[ \sup_{t\in(0,T)} \big\Vert  \varrho  \mathbf{u}  (t)\big\Vert_{L^\frac{2\gamma}{\gamma+1}_x}^{\frac{2\gamma}{\gamma+1}p}\bigg]
<\infty 
\end{align*}
for all $1\leq p<\infty$,
\item there exists $\mathscr{F}_0$-measurable random variables $(\varrho_0, \varrho_0\mathbf{u}_0)= (\varrho(0), \varrho\mathbf{u}(0))$ such that $\Lambda = \mathbb{P}\circ (\varrho_0, \varrho_0 \mathbf{u}_0)^{-1}$;
\item for all $\psi \in C^\infty_c ([0,T))$ and $\phi \in C^\infty ({\mathbb{T}^3})$, the following
\begin{equation}
\begin{aligned}
\label{eq:distributionalSolMass}
&-\int_0^T \partial_t \psi \int_{{\mathbb{T}^3}} \varrho (t) \phi \, \mathrm{d}x  \mathrm{d}t  =  \psi(0) \int_{{\mathbb{T}^3}} \varrho _0 \phi \,\mathrm{d}x  +  \int_0^T \psi \int_{{\mathbb{T}^3}} \varrho  \mathbf{u} \cdot \nabla \phi \, \mathrm{d}x  \mathrm{d}t,
\end{aligned}
\end{equation}
hold $\mathbb{P}$-a.s.;
\item for all $\psi \in C^\infty_c ([0,T))$ and $\bm{\phi} \in C^\infty ({\mathbb{T}^3})$, the following
\begin{equation}
\begin{aligned}
\label{eq:distributionalSolMomen}
&-\int_0^T \partial_t \psi \int_{{\mathbb{T}^3}} \varrho  \mathbf{u} (t) \cdot \bm{\phi} \, \mathrm{d}x    \mathrm{d}t
=
\psi(0)  \int_{{\mathbb{T}^3}} \varrho _0 \mathbf{u} _0\cdot \bm{\phi} \, \mathrm{d}x  
+  \int_0^T \psi \int_{{\mathbb{T}^3}} \varrho  \mathbf{u} \otimes \mathbf{u} : \nabla \bm{\phi}\, \mathrm{d}x  \mathrm{d}t 
\\&- \nu^S\int_0^T \psi \int_{{\mathbb{T}^3}} \nabla \mathbf{u} : \nabla \bm{\phi} \, \mathrm{d}x  \mathrm{d}t 
-(\nu^B+ \nu^S )\int_0^T \psi \int_{{\mathbb{T}^3}} \mathrm{div}\,\mathbf{u} \, \mathrm{div}\, \bm{\phi}\, \mathrm{d}x  \mathrm{d}t    
\\&+   
\int_0^T \psi \int_{{\mathbb{T}^3}} p(\varrho)\, \mathrm{div} \, \bm{\phi} \, \mathrm{d}x  \mathrm{d}t
+\int_0^T \psi \int_{{\mathbb{T}^3}} \vartheta\varrho  \nabla V  \cdot \bm{\phi} \, \mathrm{d}x    \mathrm{d}t
\\&+
\int_0^T \psi \int_{{\mathbb{T}^3}}  \varrho\,\mathbf{G}(\varrho , f, \mathbf{m} ) \cdot \bm{\phi}\, \mathrm{d}x \mathrm{d}W   
\end{aligned}
\end{equation}
hold $\mathbb{P}$-a.s.;
\item equation \eqref{elecField} holds $\mathbb{P}$-a.s. for a.e.  $(t,x) \in(0,T)\times \mathbb{T}^3$; 
\item the  energy inequality 
\begin{equation}
\begin{aligned}
\label{augmentedEnergy}
-
\int_0^T& \partial_t \psi \int_{{\mathbb{T}^3}}\bigg[ \frac{1}{2} \varrho  \vert \mathbf{u}  \vert^2  
+ P(\varrho ) \pm
\vartheta \vert \nabla V \vert^2
\bigg]\,\mathrm{d}x \, \mathrm{d}t
+(\nu^B+\nu^S)
\int_0^T \psi \int_{{\mathbb{T}^3}}\vert  \mathrm{div}\, \mathbf{u}  \vert^2 \,\mathrm{d}x  \,\mathrm{d}t
\\&+\nu^S
\int_0^T\psi \int_{{\mathbb{T}^3}}\vert  \nabla \mathbf{u}  \vert^2 \,\mathrm{d}x  \,\mathrm{d}t
\leq
\psi(0)
 \int_{{\mathbb{T}^3}}\bigg[\frac{1}{2} \varrho _0\vert \mathbf{u} _0 \vert^2  
+ P(\varrho _0) \pm
\vartheta \vert \nabla V_0 \vert^2 \bigg]\,\mathrm{d}x
\\&
+\frac{1}{2}\int_0^T \psi
 \int_{{\mathbb{T}^3}}
 \varrho
 \sum_{k\in\mathbb{N}} \big\vert
 \mathbf{g}_k(x,\varrho ,f,\mathbf{m}  ) 
\big\vert^2\,\mathrm{d}x\,\mathrm{d}t
+\int_0^T \psi  \int_{{\mathbb{T}^3}}   \varrho\mathbf{u} \cdot\mathbf{G}(\varrho , f, \mathbf{m} )\,\mathrm{d}x\,\mathrm{d}W ;
\end{aligned}
\end{equation}
holds $\mathbb{P}$-a.s. for all $\psi \in C^\infty_c ([0,T))$, $\psi\geq0$ where
\begin{equation} 
\begin{aligned}
\label{PVarrho}
P(\varrho ) =\frac{1}{\gamma-1}p(\gamma)
\end{aligned}
\end{equation}
\item and  \eqref{contEq} holds in the renormalized sense, i.e., for any $\phi\in C^\infty_c([0,T)\times{\mathbb{T}^3})$ and $b \in C^1_b(\mathbb{R})$ such that $ b'(z)=0$ for all $z\geq M_b$,  we have that
\begin{equation}
\begin{aligned}
\label{renormalizedContExis}
-&\int_0^T\int_{{\mathbb{T}^3}}  b(\varrho )\,\partial_t\phi \,\mathrm{d}x \, \mathrm{d}t =
\int_{{\mathbb{T}^3}}  b\big(\varrho (0)\big)\,\phi(0) \,\mathrm{d}x 
\\&
+ \int_0^T\int_{{\mathbb{T}^3}} \big[b(\varrho ) \mathbf{u}  \big] \cdot \nabla\phi \, \mathrm{d}x \, \mathrm{d}t  
- \int_0^T\int_{{\mathbb{T}^3}} \big[ \big(  b'(\varrho ) \varrho   
-
 b(\varrho )\big)\mathrm{div}\mathbf{u} \big]\,  \phi\, \mathrm{d}x \, \mathrm{d}t
\end{aligned}
\end{equation} 
\end{enumerate}
holds $\mathbb{P}$-a.s.
\end{definition}

\subsection{Main result}
We now state the first main result of this work.
\begin{theorem}
\label{thm:one}
Let $\Lambda = \Lambda(\varrho, \mathbf{m},f)$ be a Borel probability measure on $\big[ L^1_x\big]^3$ such that
\begin{align*}
\Lambda\bigg\{\varrho\geq0, \quad M \leq  \varrho  \leq M^{-1}, \quad f \leq \overline{f}, \quad \mathbf{m} =0 \quad \text{ when } \varrho=0 \bigg\} =1, \quad
\end{align*}
holds for some deterministic constants $M, \overline{f}>0$. Also, 
 assume that
\begin{align*}
\int_{[L^1_x]^3 } \bigg\vert  
\int_{\mathbb{T}^3} \bigg[ \frac{\vert \mathbf{m} \vert^2}{2\varrho} + P(\varrho)  \pm
\vartheta \vert \nabla V \vert^2
\bigg] \, \mathrm{d}x
\bigg\vert^p \, \mathrm{d}\Lambda(\varrho, \mathbf{m},f) \lesssim 1
\end{align*}
holds for some $p\geq 1$ and that
\begin{align*}
f\in L^1_x \cap L^\infty_x \quad \mathbb{P}\text{-a.s.}
\end{align*} Finally assume that \eqref{noiseSupportExi}--\eqref{stochsummableConst} are satisfied. Then the exists a dissipative martingale solution of \eqref{contEq}--\eqref{elecField} in the sense of Definition \ref{def:martSolutionExis}.
\end{theorem}

\section{Formal derivation of the energy inequality}
\label{sec:energyIneq}
Let us first recall the following It\^o lemma. See \cite[Theorem A.4.1]{breit2018stoch} for the stronger version of this result.
\begin{theorem}
\label{thm:itoLemma}
Let $W$ be an $(\mathscr{F}_t)$-cylindrical Wiener process on the stochastic basis  $(\Omega, \mathscr{F}, (\mathscr{F}_t)_{t\geq0}, \mathbb{P})$. Let $(r,s)$ be a pair of stochastic processes on $(\Omega, \mathscr{F}, (\mathscr{F}_t)_{t\geq0}, \mathbb{P})$ satisfying 
\begin{equation}
\begin{aligned}
\mathrm{d}r &=[D r]\, \mathrm{d}t + [\mathbb{D}r]\, \mathrm{d}W,
\quad \quad
\mathrm{d}s &=[D s]\, \mathrm{d}t +[\mathbb{D}s]\, \mathrm{d}W
\end{aligned}
\end{equation}
on the cylinder $(0,T)\times \mathbb{T}^d$.
Now suppose that the following 
\begin{equation}
\begin{aligned}
r &\in C^\infty\big( [0,T]\times \mathbb{T}^d\big),
\quad
s \in C^\infty\big( [0,T]\times \mathbb{T}^d\big)
\end{aligned}
\end{equation}
holds $\mathbb{P}$-a.s. and that for all $1\leq q<\infty$
\begin{equation}
\begin{aligned}
\mathbb{E}\, \bigg[\sup_{t\in [0,T]} \Vert r  &\Vert^2_{W^{1,q}_x} \bigg]^q
+
\mathbb{E}\, \bigg[\sup_{t\in [0,T]} \Vert s \Vert^2_{W^{1,q}_x} \bigg]^q \lesssim_q 1.
\end{aligned}
\end{equation}
Furthermore, assume that $[Dr],\, [Ds],\,[\mathbb{D}r],\,[\mathbb{D}s]$ are progressively measurable and that
\begin{equation}
\begin{aligned}
[Dr], \, [Ds] &\in  L^q \big(\Omega;L^q( 0,T; W^{1,q}_x\big)
\\
[\mathbb{D}r], \, [\mathbb{D}s] &\in  L^2 \Big(\Omega;L^2 \big( 0,T;L_2(\mathfrak{U}; L^2_x\big)\Big)
\end{aligned}
\end{equation}
and 
\begin{equation}
\begin{aligned}
\bigg(\sum_{k\in \mathbb{N}} \big\vert [Dr](e_k) \big\vert^q \bigg)^\frac{1}{q},
\bigg(\sum_{k\in \mathbb{N}} \big\vert [Ds](e_k) \big\vert^q \bigg)^\frac{1}{q}
\in
L^q\big( \Omega\times(0,T) \times \mathbb{T}^d\big)
\end{aligned}
\end{equation}
holds. Finally, for some $\lambda\geq 0$, let $Q$ be $(\lambda+2)$-continuously differentiable function such that
\begin{equation}
\begin{aligned}
\mathbb{E}\, \sup_{t\in [0,T]}\, \Vert Q^j(r) \Vert^2_{W^{\lambda, q'}_x \cap C_x}
 <\infty, \quad j=0,1,2.
\end{aligned}
\end{equation}
Then
\begin{equation}
\begin{aligned}
\int_{\mathbb{T}^d}&\big(sQ(r) \big)(t)\, \mathrm{d}x 
=
\int_{\mathbb{T}^d}s_0Q(r_0) \, \mathrm{d}x
+
\int_0^t
\int_{\mathbb{T}^d}\bigg[ sQ'(r)\, [Dr] + \frac{1}{2}\sum_{k\in \mathbb{N}}s\,Q''(r)\big\vert [  \mathbb{D}r](e_k)\big\vert^2   \bigg] \, \mathrm{d}x\, \mathrm{d}t'
\\
&+
\int_0^t
\int_{\mathbb{T}^d} Q(r) [{D}s]  \, \mathrm{d}x\, \mathrm{d}t' 
+
\sum_{k\in \mathbb{N}}\int_0^t \int_{\mathbb{T}^d} [\mathbb{D} s](e_k) \, [\mathbb{D}r](e_k) \, \mathrm{d}x\, \mathrm{d}t'
\\
&+
\sum_{k\in \mathbb{N}}
\int_0^t\int_{\mathbb{T}^d}
\Big[sQ'(r) [\mathbb{D}r](e_k)  +Q(r) [\mathbb{D}s] (e_k) \Big]\, \mathrm{d}x\, \mathrm{d}\beta_k(t').
\end{aligned}
\end{equation}
\end{theorem}
In the following argument, we assume that all the unknowns are sufficiently regular and that no vacuum state exists. Subsequently, we rewrite our system as follows
\begin{align}
&\mathrm{d} \varrho  + \mathrm{div}(\varrho  \mathbf{u} ) \, \mathrm{d}t =0, 
\label{contEqFormal}
\\
&\mathrm{d}  \mathbf{u}  +
\big[
  \mathbf{u}  \cdot \nabla \mathbf{u}  + \varrho^{-1}\nabla  p( \varrho) \big]\, \mathrm{d}t 
=\big[ \nu^S \varrho^{-1}\Delta \mathbf{u}  +(\nu^B + \nu^S) \varrho^{-1} \nabla \mathrm{div} \mathbf{u} 
\nonumber\\
&\qquad   \qquad \qquad  \qquad \qquad + \vartheta   \nabla V  \big] \, \mathrm{d}t +  \sum_{k\in \mathbb{N}}  \mathbf{g}_k(x,\varrho , f,  \varrho \mathbf{u} ) \, \mathrm{d}\beta_k ,
\label{momEqFormal}
\\
&\pm \Delta V  = \varrho  - f .
\label{elecFieldFormal}
\end{align}
For $s=\varrho$ and $Q(r)=\frac{1}{2}\vert \mathbf{u} \vert^2$, applying It\^o's formula, Theorem \ref{thm:itoLemma} to the functional
\begin{align}
F(\varrho,  \mathbf{u})(t)=\int_{\mathbb{T}^3} \frac{1}{2} \varrho(t) \vert \mathbf{u}(t) \vert^2 \, \mathrm{d}x
\end{align}
yields
\begin{equation}
\begin{aligned}
F&(\varrho,  \mathbf{u})(t) = \int_{\mathbb{T}^3} \frac{1}{2}\varrho_0 \vert \mathbf{u}_0 \vert^2 \, \mathrm{d}x
-
\int_0^t\int_{\mathbb{T}^3}  \big[ \varrho\mathbf{u}\cdot \mathbf{u}\cdot \nabla \mathbf{u}+ \mathbf{u} \nabla p \big]\mathrm{d}x \, \mathrm{d}s
\\&
+
\int_0^t\int_{\mathbb{T}^3}  
\big[ \nu^S \mathbf{u}\cdot\Delta \mathbf{u}  +(\nu^B + \nu^S) \mathbf{u}\cdot\nabla \mathrm{div} \mathbf{u} 
+ \vartheta \varrho\mathbf{u} \cdot  \nabla V
  \big] \, \mathrm{d}x \, \mathrm{d}s 
\\
&-\int_0^t\int_{\mathbb{T}^3} \frac{1}{2}\vert \mathbf{u} \vert^2 \mathrm{div}(\varrho \mathbf{u}) \mathrm{d}x \, \mathrm{d}s +\sum_{k \in \mathbb{N}} \int_0^t\int_{\mathbb{T}^3} \frac{1}{2} \varrho \vert \mathbf{g}_k(x,\varrho , f, \varrho \mathbf{u} ) \vert^2 \mathrm{d}x\, \mathrm{d}s
\\
&+\sum_{k \in \mathbb{N}} \int_0^t\int_{\mathbb{T}^3} \varrho\mathbf{u} \cdot  \mathbf{g}_k(x,\varrho , f,  \varrho \mathbf{u}) \mathrm{d}x\, \mathrm{d}\beta_k.
\end{aligned}
\end{equation}
Now since $\varrho\mathbf{u}\cdot \mathbf{u}\cdot \nabla \mathbf{u}=\frac{1}{2}\varrho\mathbf{u}\cdot \nabla \vert \mathbf{u} \vert^2$, it cancels with the last non-noise term above after integrating by parts. 
For the pressure term, we use the following elementary identities  
\begin{align*}
\frac{a}{\varrho} \nabla \varrho^\gamma =a \gamma \varrho^{\gamma-2} \nabla \varrho =
\frac{a\gamma}{\gamma-1} \nabla \varrho^{\gamma-1}
\end{align*}
so that the use of the continuity equation yields
\begin{equation}
\begin{aligned}
\int_0^t\int_{\mathbb{T}^3}  
 \mathbf{u} \nabla p  \, \mathrm{d}x \, \mathrm{d}s 
& =
 -
 \int_0^t\int_{\mathbb{T}^3}  \mathrm{div}(\varrho
 \mathbf{u}) \frac{a\gamma}{\gamma -1} \varrho^{\gamma -1}  \, \mathrm{d}x \, \mathrm{d}s 
 =
 \int_0^t\int_{\mathbb{T}^3} \partial_s  P(\varrho)  \, \mathrm{d}x \, \mathrm{d}s.
\end{aligned}
\end{equation}
recall \eqref{PVarrho}. Finally, by integrating by parts in the $V$-term, we can substitute in the following continuity--Poisson equation
\begin{align}
\label{pmV}
- \mathrm{div}(\varrho  \mathbf{u} ) \, \mathrm{d}s=\mathrm{d} (\varrho - f)   =\mathrm{d} (\pm \Delta V)
\end{align}
which holds because $f$ is independent of time. By collecting the above arguments, we obtain
\begin{equation}
\begin{aligned}
\label{ito1}
F&(\varrho,  \mathbf{u})(t) 
+
\int_{\mathbb{T}^3}  
P(\varrho(t))
 \, \mathrm{d}x
\pm
\int_{\mathbb{T}^3}  
\vartheta \vert \nabla V(t) \vert^2
 \, \mathrm{d}x
= \int_{\mathbb{T}^3} \frac{1}{2} \varrho_0 \vert \mathbf{u}_0 \vert^2 \, \mathrm{d}x
+
\int_{\mathbb{T}^3}  
P(\varrho_0)
 \, \mathrm{d}x
\\&\pm
\int_{\mathbb{T}^3}  
\vartheta \vert \nabla V_0 \vert^2
 \, \mathrm{d}x
-
\int_0^t\int_{\mathbb{T}^3}  
\big[ \nu^S \vert \nabla \mathbf{u} \vert^2 +(\nu^B + \nu^S) \vert \mathrm{div} \mathbf{u}\vert^2 \big] \, \mathrm{d}x \, \mathrm{d}s 
\\&
 +\sum_{k \in \mathbb{N}} \int_0^t\int_{\mathbb{T}^3} \frac{1}{2} \varrho \vert \mathbf{g}_k(x,\varrho , f, \varrho \mathbf{u} ) \vert^2 \mathrm{d}x\, \mathrm{d}s
+\sum_{k \in \mathbb{N}} \int_0^t\int_{\mathbb{T}^3} \varrho\mathbf{u} \cdot  \mathbf{g}_k(x,\varrho , f,  \varrho \mathbf{u}) \mathrm{d}x\, \mathrm{d}\beta_k.
\end{aligned}
\end{equation}
Alternative to \eqref{pmV}, one can use the combination of integration by parts and the continuity equation so that
\begin{align*}
\int_0^t\int_{\mathbb{T}^3} \varrho \mathbf{u} \nabla V \, \mathrm{d}x \, \mathrm{d}s
= -
\int_0^t\int_{\mathbb{T}^3} \mathrm{div}(\varrho \mathbf{u}) V \, \mathrm{d}x \, \mathrm{d}s
= -
\int_0^t\int_{\mathbb{T}^3} \varrho \,\partial_s V \, \mathrm{d}x \, \mathrm{d}s
\end{align*}
holds.
The equation \eqref{ito1} is therefore equivalent to
\begin{equation}
\begin{aligned}
\label{ito2}
F&(\varrho,  \mathbf{u})(t) 
+
\int_{\mathbb{T}^3}  
P(\varrho(t))
 \, \mathrm{d}x
\mp
\int_{\mathbb{T}^3} \vartheta \varrho \partial_tG(t,x)\, \mathrm{d}x
= \int_{\mathbb{T}^3}\frac{1}{2} \varrho_0 \vert \mathbf{u}_0 \vert^2 \, \mathrm{d}x
+
\int_{\mathbb{T}^3}  
P(\varrho_0)
 \, \mathrm{d}x
 \\&\mp
\int_{\mathbb{T}^3} \varrho_0\partial_tG(t,x)\vert_{t=0} \, \mathrm{d}x
-
\int_0^t\int_{\mathbb{T}^3}  
\big[ \nu^S \vert \nabla \mathbf{u} \vert^2 +(\nu^B + \nu^S) \vert \mathrm{div} \mathbf{u}\vert^2  \big] \, \mathrm{d}x \, \mathrm{d}s 
\\&
 +\sum_{k \in \mathbb{N}} \int_0^t\int_{\mathbb{T}^3} \frac{1}{2} \varrho \vert \mathbf{g}_k(x,\varrho , f, \varrho \mathbf{u} ) \vert^2 \mathrm{d}x\, \mathrm{d}s
+\sum_{k \in \mathbb{N}} \int_0^t\int_{\mathbb{T}^3} \varrho\mathbf{u} \cdot  \mathbf{g}_k(x,\varrho , f,  \varrho \mathbf{u}) \mathrm{d}x\, \mathrm{d}\beta_k
\end{aligned}
\end{equation}
where $G(t,x)$ is the Green's function of the Poisson equation. For example, in $\mathbb{R}^3$, we can use the Green's function
\begin{align}
\label{greenFunction}
G(s,x) =\mp c \int_{\mathbb{R}^3} \frac{(\varrho - f )(s,y)}{\vert x-y \vert} \, \mathrm{d}y
\end{align}
where $c=c(\pi)$, which solves the Poisson equation so that
\begin{align*}
\varrho \partial_s G(s,x) &=\mp  \frac{c}{2} \int_{\mathbb{R}^3} \frac{\partial_s(\varrho^2 - 2f)(s,y)}{\vert x-y \vert} \, \mathrm{d}y
=\mp \partial_s \frac{c}{2}\varrho \int_{\mathbb{R}^3}  \frac{(\varrho - 2f)(s,y)}{\vert x-y \vert} \, \mathrm{d}y
=\mp \partial_s \frac{c}{2} \varrho\bigg( (\varrho - 2f)*_x \frac{1}{\vert x \vert} \bigg).
\end{align*}
It follows that
\begin{equation}
\begin{aligned}
\label{poissonPartAlt}
\int_0^t\int_{\mathbb{T}^3}\vartheta \varrho \mathbf{u} \nabla V \, \mathrm{d}x \, \mathrm{d}s
&= \pm
\int_{\mathbb{T}^3} \frac{c}{2} \vartheta \varrho \bigg( (\varrho - 2f)*_x \frac{1}{\vert x \vert} \bigg) (t)\, \mathrm{d}x
\\&
\mp
\int_{\mathbb{T}^3} \frac{c}{2} \vartheta \varrho_0\bigg( (\varrho_0 - 2f)*_x \frac{1}{\vert x \vert} \bigg) \, \mathrm{d}x.
\end{aligned}
\end{equation}
It may be useful to treat the pressure differently when one wants to study singular limits. To do so, we rewrite the continuity equation \eqref{contEqFormal} as
\begin{align}
\label{contEqFormalA}
\mathrm{d} (\varrho - \overline{\varrho})  + \mathrm{div}(\varrho  \mathbf{u} ) \, \mathrm{d}t =0
\end{align}
and also replace the pressure term in \eqref{momEqFormal} by
\begin{align*}
\varrho^{-1}\nabla [p(\varrho) - p(\overline{\varrho})]
\end{align*}
where $\overline{\varrho} \geq0$, so that if we define 
\begin{align*}
\overline{P}(\varrho) = \varrho \int_{\overline{\varrho}}^\varrho  \frac{p(z)}{z^2}\, \mathrm{d}z
\end{align*}
as the new pressure potential, we have that
\begin{align*}
\overline{P}(\overline{\varrho})=0, \quad p(\varrho)=\varrho\overline{P}'(\varrho) -\overline{P}(\varrho).
\end{align*}
Furthermore, by applying  Gauss's theorem to \eqref{contEqFormalA} in order to calculate the flux through $\mathbb{T}^3$, then we able to obtain  the mass conservation relation from which we gain
\begin{align*}
\frac{\mathrm{d}}{\mathrm{d}t}\int_{\mathbb{T}^3}\overline{P}\,'(\overline{\varrho})(\varrho -\overline{\varrho}) \, \mathrm{d}x=0.
\end{align*}
By collecting the above information, we obtain the following versions of the energy inequality
\begin{equation}
\begin{aligned}
\label{ito3}
F&(\varrho,  \mathbf{u})(t) 
+
\int_{\mathbb{T}^3}  
H\big(\varrho(t) , \overline{\varrho}\big)
 \, \mathrm{d}x
\pm
\int_{\mathbb{T}^3}  
\vartheta \vert \nabla V(t) \vert^2
 \, \mathrm{d}x
= \int_{\mathbb{T}^3} \frac{1}{2} \varrho_0 \vert \mathbf{u}_0 \vert^2 \, \mathrm{d}x
+
\int_{\mathbb{T}^3}  
H\big(\varrho_0 , \overline{\varrho} \big)
 \, \mathrm{d}x
\\&\pm
\int_{\mathbb{T}^3}  
\vartheta \vert \nabla V_0 \vert^2
 \, \mathrm{d}x
-
\int_0^t\int_{\mathbb{T}^3}  
\big[ \nu^S \vert \nabla \mathbf{u} \vert^2 +(\nu^B + \nu^S) \vert \mathrm{div} \mathbf{u}\vert^2 \big] \, \mathrm{d}x \, \mathrm{d}s 
\\&
 +\sum_{k \in \mathbb{N}} \int_0^t\int_{\mathbb{T}^3} \frac{1}{2} \varrho \vert \mathbf{g}_k(x,\varrho , f, \varrho \mathbf{u} ) \vert^2 \mathrm{d}x\, \mathrm{d}s
+\sum_{k \in \mathbb{N}} \int_0^t\int_{\mathbb{T}^3} \varrho\mathbf{u} \cdot  \mathbf{g}_k(x,\varrho , f,  \varrho \mathbf{u}) \mathrm{d}x\, \mathrm{d}\beta_k
\end{aligned}
\end{equation}
and 
\begin{equation}
\begin{aligned}
\label{ito2}
F&(\varrho,  \mathbf{u})(t) 
+
\int_{\mathbb{T}^3}  
H\big(\varrho(t) , \overline{\varrho} \big)
 \, \mathrm{d}x
\mp
\int_{\mathbb{T}^3}\varrho_0\partial_tG(t,x)\, \mathrm{d}x
= \int_{\mathbb{T}^3}\frac{1}{2} \varrho_0 \vert \mathbf{u}_0 \vert^2 \, \mathrm{d}x
+
\int_{\mathbb{T}^3}  
H\big(\varrho_0 , \overline{\varrho} \big)
 \, \mathrm{d}x
\\&\mp
\int_{\mathbb{T}^3}\varrho_0\partial_tG(t,x)\vert_{t=0}  \, \mathrm{d}x
-
\int_0^t\int_{\mathbb{T}^3}  
\big[ \nu^S \vert \nabla \mathbf{u} \vert^2 +(\nu^B + \nu^S) \vert \mathrm{div} \mathbf{u}\vert^2  \big] \, \mathrm{d}x \, \mathrm{d}s 
\\&
 +\sum_{k \in \mathbb{N}} \int_0^t\int_{\mathbb{T}^3} \frac{1}{2} \varrho \vert \mathbf{g}_k(x,\varrho , f, \varrho \mathbf{u} ) \vert^2 \mathrm{d}x\, \mathrm{d}s
+\sum_{k \in \mathbb{N}} \int_0^t\int_{\mathbb{T}^3} \varrho\mathbf{u} \cdot  \mathbf{g}_k(x,\varrho , f,  \varrho \mathbf{u}) \mathrm{d}x\, \mathrm{d}\beta_k
\end{aligned}
\end{equation}
where
\begin{align*}
H\big(\varrho , \overline{\varrho} \big)= \overline{P}(\varrho) - \overline{P}\,'(\overline{\varrho})(\varrho -\overline{\varrho})- \overline{P}(\overline{\varrho}).
\end{align*}

\section{The first approximation layer}
\label{sec:firstLayer}
Let $H_N$ be a finite dimensional space with an associated $L^2$-orthonormal projection
\begin{align*}
\Pi_N : L^2(\mathbb{T}^2) \rightarrow H_N.
\end{align*}
Then for any $k \in  \mathbb{N}_0$, $p\in (1,\infty)$ and $v \in W^{k,p}(\mathbb{T}^3)$, it follows that
\begin{equation}
\begin{aligned}
\label{projectionContinuity}
\Vert \Pi_N v\Vert_{W^{k,p}(\mathbb{T}^3)} \lesssim_{k,p} \Vert v \Vert_{W^{k,p}(\mathbb{T}^3)},
\, \quad
\Pi_N v \rightarrow v \quad \mathrm{in} \quad W^{k,p}(\mathbb{T}^3)
\end{aligned}
\end{equation}
as $N\rightarrow \infty$, c.f. \cite[Chapter 3]{grafakos2008classical}. We now consider the following cut--off
\begin{equation}
\begin{aligned}
\label{cuttOff}
\chi \in C^\infty(\mathbb{R}), \quad \chi(f) 
=
  \begin{cases}
    1     & \quad \text{if }v \leq 0\\
    \chi'(v)\leq 0 & \quad \text{if } 0<v <1 \\
    \chi(v)=0  & \quad \text{if }v \geq1
  \end{cases}
\end{aligned}
\end{equation}
and define the following
\begin{align}
&\chi_{\mathbf{u}_R}= \chi \big( \Vert \mathbf{u}\Vert_{H_N} -R \big), \label{chiUr}
\\
&\mathbf{g}_{k,\varepsilon}(x,\varrho,f, \varrho\mathbf{u})
=
\chi \bigg(\frac{\varepsilon}{\varrho} -1 \bigg)
\chi \bigg(\vert \mathbf{u}\vert -\frac{1}{\varepsilon} \bigg) \mathbf{g}_k(x, \varrho,f, \varrho\mathbf{u}).
\end{align}
From the definition of the cut--off function \eqref{cuttOff},  it follows from \eqref{stochCoeffBound1Exis0}--\eqref{stochsummableConst} that
\begin{align}
&\esssup_{(x, \varrho, f,  \mathbf{m}) }\Big(\vert  \mathbf{g}_{k,\varepsilon}  \vert 
+
\big\vert\nabla_{\varrho,  \mathbf{m}} \, \mathbf{g}_{k,\varepsilon} \big\vert \Big)   \leq c_{k,\varepsilon},
\label{stochCoeffBound1Exis0a}
\\
&\sum_{k\in \mathbb{N}} c_{k,\varepsilon}^2 \lesssim 1 \label{stochsummableConsta}
\end{align}
holds for some constants $(c_{k,\varepsilon})_{k\in \mathbb{N}} \subset [0,\infty)$.
The aim of this section is to now construct a solution to the following auxiliary problem
\begin{align}
&\mathrm{d} \varrho  + \mathrm{div}\big(\varrho (\chi_{\mathbf{u}_R}  \mathbf{u}) \big) \, \mathrm{d}t =\varepsilon \Delta \varrho\, \mathrm{d}t, 
\label{contEq1}
\\
&\mathrm{d}\Pi_N(\varrho  \mathbf{u} ) +
\Pi_N 
\big[
\mathrm{div} \big(\varrho  (\chi_{\mathbf{u}_R}  \mathbf{u})  \otimes \mathbf{u} \big) + \chi_{\mathbf{u}_R}  \nabla   p_{\delta}^{ \Gamma}(\varrho) \big]\, \mathrm{d}t =
\Pi_N
\big[\varepsilon\Delta(\varrho\mathbf{u})+ \nu^S \Delta \mathbf{u}   
\nonumber\\&
\qquad+(\nu^B + \nu^S)\nabla \mathrm{div} \mathbf{u} + \vartheta \varrho  \nabla V  \big] \, \mathrm{d}t   
+ \sum_{k\in \mathbb{N}}\Pi_N\big[ \varrho\Pi_N   \mathbf{g}_{k,\varepsilon}(\varrho , f, \varrho \mathbf{u} ) \big] \, \mathrm{d}\beta_k ,
\label{momEq1}
\\
&\pm \Delta V  = \varrho  - f, \label{elecField1}
\end{align}
where
\begin{align}
\label{higherPressure}
 p_{\delta}^{ \Gamma}(\varrho)
 =
 p(\varrho) +\delta(\varrho +\varrho^\Gamma)
\end{align}
for some $\delta>0$ and $\Gamma \geq \max\{6, \gamma\}$.
\begin{definition}
\label{def:pathSolFiniteDim}
Let $(\Omega, \mathscr{F}, (\mathscr{F}_t)_{t\geq0 }, \mathbb{P})$ be a stochastic basis where the filtration is complete and right-continuous. Let $W$ be an $(\mathscr{F}_t)$-cylindrical Wiener process and further let $(\varrho_0, \mathbf{u_0},f)$ be $\mathscr{F}_0$-measurable random variables belonging to $C^{2+\nu}(\mathbb{T}^3) \times H_N \times C^{\nu}(\mathbb{T}^3)$. Then 
$
[\varrho , \mathbf{u} , V ]$
is a \textit{pathwise solution} of \eqref{contEq1}--\eqref{elecField1} with data  $(\varrho_0, \mathbf{u_0},f)$ if
\begin{enumerate}
\item the density  $\varrho>0$, $\varrho \in C \big( [0,T]; C^{2+\nu}_x\big)$
 $\mathbb{P}$-a.s. and it is $(\mathscr{F}_t)$-adapted;
\item the velocity field $\mathbf{u} \in C \big( [0,T]; H_N\big) $ $\mathbb{P}$-a.s. and it is $(\mathscr{F}_t)$-adapted;
\item the electric field $V \in C \big( [0,T]; C^{2+\nu}_x\big)$ $\mathbb{P}$-a.s. and it is $(\mathscr{F}_t)$-adapted;
\item the following $(\varrho(0), \mathbf{u}(0)) = (\varrho_0, \mathbf{u}_0)$ holds $\mathbb{P}$-a.s;
\item equation \eqref{contEq1} holds $\mathbb{P}$-a.s. for a.e.  $(t,x) \in(0,T)\times \mathbb{T}^3$;
\item equation \eqref{elecField1} holds $\mathbb{P}$-a.s. for a.e.  $(t,x) \in(0,T)\times \mathbb{T}^3$;
\item for all $\psi \in C^\infty_c ([0,T))$ and $\bm{\phi} \in H_N$, the following
\begin{equation}
\begin{aligned}
&-\int_0^T \partial_t \psi \int_{{\mathbb{T}^3}} \varrho  \mathbf{u} (t) \cdot \bm{\phi} \, \mathrm{d}x    \mathrm{d}t
=
\psi(0)  \int_{{\mathbb{T}^3}} \varrho _0 \mathbf{u} _0 \cdot \bm{\phi} \, \mathrm{d}x  
+  \int_0^T \psi \int_{{\mathbb{T}^3}}  \varrho  (\chi_{\mathbf{u}_R}  \mathbf{u}) \otimes \mathbf{u} : \nabla \bm{\phi} \, \mathrm{d}x  \mathrm{d}t 
\\&-
 \nu^S\int_0^T \psi \int_{{\mathbb{T}^3}} \nabla \mathbf{u} : \nabla \bm{\phi} \, \mathrm{d}x  \mathrm{d}t 
-(\nu^B+ \nu^S )\int_0^T \psi \int_{{\mathbb{T}^3}} \mathrm{div}\,\mathbf{u} \, \mathrm{div}\, \bm{\phi} \, \mathrm{d}x  \mathrm{d}t    
\\&+   
\int_0^T \psi \int_{{\mathbb{T}^3}}  \chi_{\mathbf{u}_R}  p_{\delta}^{\Gamma}\, \mathrm{div} \, \bm{\phi} \, \mathrm{d}x  \mathrm{d}t
+\int_0^T \psi \int_{{\mathbb{T}^3}} \varepsilon\varrho\mathbf{u}    \Delta\bm{\phi} \, \mathrm{d}x   \mathrm{d}t
+\int_0^T \psi \int_{{\mathbb{T}^3}} \vartheta\varrho  \nabla V  \cdot \bm{\phi} \, \mathrm{d}x    \mathrm{d}t
\\&+
\int_0^T \psi \sum_{k\in \mathbb{N}} \int_{{\mathbb{T}^3}}  \Pi_N\big[ \varrho\Pi_N \mathbf{g}_{k,\varepsilon}(\varrho , f, \varrho \mathbf{u} ) \big] \cdot \bm{\phi}  \, \mathrm{d}x \mathrm{d}\beta_k  
\end{aligned}
\end{equation}
hold $\mathbb{P}$-a.s.
\end{enumerate}
\end{definition} 
We now state the main result of this section.
\begin{theorem}
\label{thm:mainFirstLayer}
Let $(\Omega, \mathscr{F}, (\mathscr{F}_t)_{t\geq0}, \mathbb{P})$ be a stochastic basis and let $\varrho_0 \in C^{2+\nu}(\mathbb{T}^3)$ for $\nu\in(0,1)$, $\mathbf{u}_0 \in H_N$, $f \in C^{\nu}(\mathbb{T}^3)$ be some $\mathscr{F}_0$-measurable random variables such that
\begin{align}
  &\varrho_0 \geq \underline{\varrho}>0, \quad \Vert \varrho_0\Vert_{C^{2+\nu}_x} \leq \overline{\varrho}, \quad  \Vert f \Vert_{C^{\nu}_x} \leq \overline{f} \quad  \mathbb{P}\text{-a.s.}, 
  \\
  &\mathbb{E} 
\Vert \mathbf{u} \Vert_{H_N}^p \leq \overline{\mathbf{u}}
\end{align}
for some $p\geq1$ and some deterministic constants $\underline{\varrho},\overline{\varrho}, \overline{f}, \overline{\mathbf{u}} >0$. Then there exists a unique pathwise solution $(\varrho, \mathbf{u}, V)$ of \eqref{contEq1}--\eqref{elecField1} in the sense of Definition \ref{def:pathSolFiniteDim} such that the estimates
\begin{align}
&\esssup_{t\in [0,T]} \bigg(\Vert \varrho(t) \Vert_{C^{2+\nu}_x}
+
\Vert \partial_t\varrho(t) \Vert_{C^{\nu}_x} 
+
\Vert \varrho^{-1}(t) \Vert_{C_x}
 \bigg)\lesssim 1, \quad \mathbb{P}\text{-a.s.}
\label{esssupDensity}\\
&\esssup_{t\in [0,T]} \bigg(\Vert V(t) \Vert_{C^{2+\nu}_x}
+
\Vert \partial_t V(t) \Vert_{C^{2+\nu}_x} 
 \bigg) \lesssim 1, \quad \mathbb{P}\text{-a.s.}
\label{esssupCarge}\\
&
\mathbb{E}
\esssup_{t\in [0,T]} \Vert \mathbf{u}(t) \Vert_{H_N}^p 
\lesssim 1 + 
\mathbb{E} \Vert \mathbf{u}_0 \Vert_{H_N}^p \label{esssupVelocity}
\end{align}
holds for some $p\geq1$.
\end{theorem}
The proof of Theorem \ref{thm:mainFirstLayer} requires two main ingredients: a stochastically weak solution and pathwise uniqueness. The precise definition of the former is given below.
\begin{definition}
\label{def:martSolFiniteDim}
If $\Lambda$ is a Borel probability measure on $C^{2+\nu}(\mathbb{T}^3) \times H_N \times C^{\nu}(\mathbb{T}^3)$. We say that 
\begin{align}
\big[(\Omega ,\mathscr{F} ,(\mathscr{F} _t)_{t\geq0},\mathbb{P} );\varrho , \mathbf{u} , V , W   \big]
\end{align}
is a \textit{martingale solution} of \eqref{contEq1}--\eqref{elecField1} with initial law $\Lambda$ provided
\begin{enumerate}
\item $(\Omega,\mathscr{F},(\mathscr{F}_t),\mathbb{P})$ is a stochastic basis with a complete right-continuous filtration;
\item $W$ is a $(\mathscr{F}_t)$-cylindrical Wiener process;
\item the density $\varrho \in C \big( [0,T]; C^{2+\nu}_x\big) $, it is $(\mathscr{F}_t)$-adapted  and 
$\varrho>0$ $\mathbb{P}$-a.s.;
\item the velocity field $\mathbf{u} \in C \big( [0,T]; H_N\big) $ $\mathbb{P}$-a.s. and it is $(\mathscr{F}_t)$-adapted;
\item the electric field $V \in C \big( [0,T]; C^{2+\nu}_x\big)$ $\mathbb{P}$-a.s. and it is $(\mathscr{F}_t)$-adapted;
\item there exists $\mathscr{F}_0$-measurable random variables $(\varrho_0, \mathbf{u}_0,f)= (\varrho(0), \varrho\mathbf{u}(0),f)$ such that $\Lambda = \mathbb{P}\circ (\varrho_0,  \mathbf{u}_0, f)^{-1}$;
\item we have that $\varrho(0)=\varrho_0$ $\mathbb{P}$-a.s. and \eqref{contEq1} holds $\mathbb{P}$-a.s. for a.e.  $(t,x) \in(0,T)\times \mathbb{T}^3$;
\item equation \eqref{elecField1} holds $\mathbb{P}$-a.s. for a.e.  $(t,x) \in(0,T)\times \mathbb{T}^3$;
\item for all $\psi \in C^\infty_c ([0,T))$ and $\bm{\phi} \in H_N$, the following
\begin{equation}
\begin{aligned}
&-\int_0^T \partial_t \psi \int_{{\mathbb{T}^3}} \varrho  \mathbf{u} (t) \cdot \bm{\phi} \, \mathrm{d}x    \mathrm{d}t
=
\psi(0)  \int_{{\mathbb{T}^3}} \varrho _0 \mathbf{u} _0 \cdot \bm{\phi} \, \mathrm{d}x  
+  \int_0^T \psi \int_{{\mathbb{T}^3}}  \varrho  (\chi_{\mathbf{u}_R}  \mathbf{u}) \otimes \mathbf{u} : \nabla \bm{\phi} \, \mathrm{d}x  \mathrm{d}t 
\\&-
 \nu^S\int_0^T \psi \int_{{\mathbb{T}^3}} \nabla \mathbf{u} : \nabla \bm{\phi} \, \mathrm{d}x  \mathrm{d}t 
-(\nu^B+ \nu^S )\int_0^T \psi \int_{{\mathbb{T}^3}} \mathrm{div}\,\mathbf{u} \, \mathrm{div}\, \bm{\phi} \, \mathrm{d}x  \mathrm{d}t    
\\&+   
\int_0^T \psi \int_{{\mathbb{T}^3}}  \chi_{\mathbf{u}_R}  p_{\delta}^{ \Gamma}\, \mathrm{div} \, \bm{\phi} \, \mathrm{d}x  \mathrm{d}t
+\int_0^T \psi \int_{{\mathbb{T}^3}} \varepsilon\varrho\mathbf{u}    \Delta\bm{\phi} \, \mathrm{d}x   \mathrm{d}t
+\int_0^T \psi \int_{{\mathbb{T}^3}} \vartheta\varrho  \nabla V  \cdot \bm{\phi} \, \mathrm{d}x    \mathrm{d}t
\\&+
\int_0^T \psi \sum_{k\in \mathbb{N}} \int_{{\mathbb{T}^3}}  \Pi_N\big[ \varrho\Pi_N \mathbf{g}_{k,\varepsilon}(\varrho , f, \varrho \mathbf{u} ) \big] \cdot \bm{\phi}  \, \mathrm{d}x \mathrm{d}\beta_k  
\end{aligned}
\end{equation}
hold $\mathbb{P}$-a.s.
\end{enumerate}
\end{definition} 
The following result is reminiscent of \cite[Theorem 4.1.3]{breit2018stoch} and it is establishes the existence of a stochastically weak solution in the sense of Definition \ref{def:martSolFiniteDim} with further uniform bounds on the observables.
\begin{proposition}
\label{prop:martSolFinite}
If $\Lambda$ is a Borel probability measure on $C^{2+\nu}(\mathbb{T}^3) \times H_N \times C^{\nu}(\mathbb{T}^3)$, $\nu \in (0,1)$ such that
\begin{align}
\label{firstLayerInitLawA}
\Lambda \Big\{  \varrho \geq \underline{\varrho}>0, \quad \Vert \varrho\Vert_{C^{2+\nu}_x} \leq \overline{\varrho}, \quad  \Vert f \Vert_{C^{\nu}_x} \leq \overline{f} \Big\} =1,
\end{align}
and
\begin{align}
\label{firstLayerInitLawB}
\int\Vert \mathbf{u} \Vert_{H_N}^p \, \mathrm{d}\Lambda \leq \overline{\mathbf{u}}
\end{align}
holds for some $p\geq1$ and some deterministic constants $\underline{\varrho},\overline{\varrho}, \overline{f}, \overline{\mathbf{u}} $. Then there exists a martingale solution of \eqref{contEq1}--\eqref{elecField1} in the sense of Definition \ref{def:martSolFiniteDim} such that the estimates
\begin{align}
&\esssup_{t\in [0,T]} \bigg(\Vert \varrho(t) \Vert_{C^{2+\nu}_x}
+
\Vert \partial_t\varrho(t) \Vert_{C^{\nu}_x} 
+
\Vert \varrho^{-1}(t) \Vert_{C_x}
 \bigg)\lesssim 1, \quad \mathbb{P}\text{-a.s.}
\label{esssupDensity1}\\
&\esssup_{t\in [0,T]} \bigg(\Vert V(t) \Vert_{C^{2+\nu}_x}
+
\Vert \partial_t V(t) \Vert_{C^{2+\nu}_x} 
 \bigg)\lesssim 1, \quad \mathbb{P}\text{-a.s.} \label{esssupCarge1}\\
&
\mathbb{E}
\esssup_{t\in [0,T]} \Vert \mathbf{u}(t) \Vert_{H_N}^p 
+
\mathbb{E}
 \Vert \mathbf{u}(t) \Vert_{C^\beta(0,T;H_N)}^p
\lesssim 1 + 
\mathbb{E} \Vert \mathbf{u}_0 \Vert_{H_N}^p \label{esssupVelocity1}
\end{align}
holds whenever $p>2$ and $\beta\in (0, \frac{1}{2}-\frac{1}{p})$.
\end{proposition}
In order to avoid repeating the results in \cite[Section 4.1.1--4.1.2.3]{breit2018stoch}, we only give the main ideas to solving Proposition \ref{prop:martSolFinite}. Just as was done in the cited book, we construct this weak solution via an iterative scheme exactly as was done in \cite[Section 4.1.1--4.1.2.3]{breit2018stoch}. So we consider the data
\begin{align}
\label{arbitraryData}
\varrho(t)=\varrho_0, \quad \mathbf{u}(t) = \mathbf{u}_0, \quad t\leq 0
\end{align}
the following invertible linear map
\begin{align}
\label{invertibleLinearMap}
\mathcal{M}[\varrho] :H_N \rightarrow H_N
\end{align}
satisfying
\begin{align*}
\mathcal{M}[\varrho](\mathbf{u}) =\Pi_N(\varrho\mathbf{u}), \quad
\mathcal{M}^{-1}[\varrho]\big(\Pi_N[\varrho\mathbf{u}] \big) =\mathbf{u}
\end{align*}
for any $\mathbf{u}\in H_N$ in the distributional sense, that is, the following
\begin{align*}
\int_{\mathbb{T}^3}\mathcal{M}[\varrho](\mathbf{u}) \cdot \bm{\psi} \, \mathrm{d}x =\int_{\mathbb{T}^3}(\varrho\mathbf{u})  \cdot \bm{\psi} \, \mathrm{d}x
\\
\int_{\mathbb{T}^3}\mathcal{M}^{-1}[\varrho](\varrho\mathbf{u}) \cdot \bm{\psi} \, \mathrm{d}x =\int_{\mathbb{T}^3}\mathbf{u}  \cdot \bm{\psi} \, \mathrm{d}x
\end{align*}
holds for all $\bm{\psi} \in H_N$.  See  \cite{feireisl2001existence} for further properties of this map but in particular,
\begin{align*}
\big\Vert  \mathcal{M}^{-1}[\varrho] \big\Vert_{\mathcal{L}(H_N,H_N^*)} \leq \Big[  \inf_{x\in \mathbb{T}^3} \varrho(x) \Big]^{-1}
\end{align*}
and
\begin{align}
\label{holderInvertibleMap}
\big\Vert  \mathcal{M}^{-1}[\varrho_1]
-
\mathcal{M}^{-1}[\varrho_2]
 \big\Vert_{\mathcal{L}(H_N,H_N^*)} \leq \Vert \varrho_1 -\varrho_2 \Vert_{L^1_x} 
\end{align}
holds for any $\varrho_1, \varrho_2 \in \mathbb{R}_{>0}$.
Then for integers $n \in \{0, \lfloor h^{-1}T \rfloor\}$, we examine the following system
\begin{align}
&\partial_t \varrho  + \mathrm{div}\big(\varrho (\chi_{\mathbf{u}_R(nh)}  \mathbf{u}(nh)) \big)  =\varepsilon \Delta \varrho(t), 
\label{contEq2}
\\
&\mathrm{d}\Pi_N(\varrho  \mathbf{u} ) +
\Pi_N 
\big[
\mathrm{div} \big(\varrho(t)  (\chi_{\mathbf{u}_R(nh)}  \mathbf{u}(nh))  \otimes \mathbf{u}(nh) \big) + \chi_{\mathbf{u}_R(nh)}  \nabla   p_{\delta}^{ \Gamma}(\varrho(t)) \big]\, \mathrm{d}t  \nonumber
\\&\qquad
=
\Pi_N
\big[\varepsilon\Delta(\varrho(t)\mathbf{u}(nh))+ \nu^S \Delta \mathbf{u}(nh) +(\nu^B + \nu^S)\nabla \mathrm{div} \mathbf{u}(nh)   \nonumber
\\&\qquad
+ \vartheta \varrho(t)  \nabla V(nh)  \big] \, \mathrm{d}t   
+ \sum_{k\in \mathbb{N}}\Pi_N\big[ \varrho(t)\Pi_N   \mathbf{g}_{k,\varepsilon}(\varrho(nh) , f(nh), (\varrho \mathbf{u})(nh) ) \big] \, \mathrm{d}\beta_k,
\label{momEq2}
\\
&\pm \Delta V  = \varrho(t)  - f(t), \label{elecField2}
\end{align}
for $t \in [ nh, (n+1)h)$ where
\begin{align*}
\varrho(nh) =\varrho(nh-) := \lim_{s\nearrow nh} \varrho(s), \quad
\mathbf{u}(nh) =\mathbf{u}(nh-) := \lim_{s\nearrow nh} \mathbf{u}(s).
\end{align*}
Since $\mathbf{u}(nh) \in H_N$ is frozen in time and smooth in space, by Proposition \ref{maximumPrin}, Equation \eqref{contEq2} has a unique classical solution
for any initial data $\varrho(nh)$ and by \eqref{maxPrinIneq}, this solution  is strictly positive so long as the initial data $\varrho(nh)$ is. This unique solution of \eqref{contEq2}, uniquely define a solution $V$ of \eqref{elecField2} since $f\in C^\nu_x$ is a given. 
\\
Now since the additional terms $V$ and $f$  in \eqref{momEq2} are frozen in time, the analysis in \cite[Section 4.1.1]{breit2018stoch} holds true in our case. In particular, by rewriting \eqref{momEq2} as follows
\begin{equation}
\begin{aligned}
\label{velEq}
&\mathbf{u}(t)
=
\mathcal{M}^{-1}[\varrho(t)](\varrho \mathbf{u})(nh)
-
\mathcal{M}^{-1}[\varrho(t)]
\int_{nh}^t \Pi_N\chi_{\mathbf{u}_R(nh)}
\nabla   p_{\delta}^{ \Gamma}(\varrho(s))\, \mathrm{d}s
\\&
-
\mathcal{M}^{-1}[\varrho(t)]
\int_{nh}^t
\Pi_N 
\Big[
\mathrm{div} \Big(\varrho(s)(  \chi_{\mathbf{u}_R(nh)} \mathbf{u}(nh))  \otimes \mathbf{u}(nh) \Big)   \Big]\, \mathrm{d}s \\&+
\mathcal{M}^{-1}[\varrho(t)]
\int_{nh}^t
\Pi_N
\Big[ \varepsilon\Delta \big(\varrho(s)\mathbf{u}(nh) \big)+ \nu^S \Delta \mathbf{u}(nh)  
+(\nu^B + \nu^S)\nabla \mathrm{div} \mathbf{u}(nh)   
\\
&+ \vartheta \varrho(s)  \nabla V(nh)  \Big] \, \mathrm{d}s + 
\mathcal{M}^{-1}[\varrho(t)]
\int_{nh}^t
\sum_{k\in \mathbb{N}}\Pi_N\big[\varrho(t) \Pi_N  \mathbf{g}_{k,\varepsilon} \big(\varrho(nh) , f(nh), (\varrho \mathbf{u})(nh) \big)\big] \, \mathrm{d}\beta_k , 
\end{aligned}
\end{equation}
then given that there exist a unique solutions to \eqref{contEq2} and \eqref{elecField2}, we obtain a unique solution for  \eqref{velEq} having the data \eqref{arbitraryData} and indeed any initial data $\varrho(nh)$. 
We can therefore deduce that \eqref{arbitraryData} and \eqref{contEq2}--\eqref{elecField2} uniquely generates  solutions $\varrho, \mathbf{u}$ and  $V$ which are $(\mathfrak{F}_t)$-progressively measurable and for any $n \in \mathbb{Z}_{>0}$, the following
\begin{align*}
\varrho>0, \quad\varrho \in C([0,T]; C^{2+\nu}_x), \quad \mathbf{u} \in C([0,T]; H_N), \quad  V\in C([0,T]; C^{2+\nu}_x),
\end{align*}
holds $\mathbb{P}$-a.s. It goes without saying but since $V$ solves an elliptic equation, there is natural gain of two spatial derivatives for $V$ given the regularity of $f$. Now using the equivalence of norms in the finite-dimensional space $H_N$ and the fact that $\varrho_0 \in C^{2+\nu}(\mathbb{T}^3)$, we deduce from \eqref{contEq2} that
\begin{align*}
\partial_t\varrho \in C([0,T]; C^{\nu}_x)
\end{align*}
holds $\mathbb{P}$-a.s. H\"older continuity in time of $\mathbf{u}$ follows the same argument as in \cite[Section 4.1.2.3]{breit2018stoch}.
Finally, having gained  H\"older regularity for the density and its time derivative and given that $f\in C^\nu(\mathbb{T}^3)$, standard regularity theorem for the Poisson equation, see for instant \cite[Chapters 2 and 4]{gilbarg2001elliptic},  yields \eqref{esssupCarge1}. In order words, given $f$, $V$ is uniquely determined by the density.
\subsection{Compactness}
In order to explore compactness, we denote by $(\varrho_h, \mathbf{u}_h, V_h, W)$,  the unique solution of \eqref{contEq2}--\eqref{elecField2} given in the summary above and define the following spaces 
\begin{align*}
\chi_\varrho = C^{\nu_0}\left([0,T];C^{2+\nu_0}_x\right), 
&\, \quad
\chi_{ \mathbf{u} } = C^{\beta_0}\left([0,T];H_N\right),
\\
\chi_V = C^{\nu_0}\left([0,T];C^{2+\nu_0}_x\right),
&\, \quad
\chi_W= C\left([0,T];\mathfrak{U}_0\right) ,
\end{align*}
where $\nu_0 \in (0,\nu)$ and $\beta_0 \in (0, \beta)$. Here $\nu$ and $\beta$ are the H\"older exponents in Proposition \ref{prop:martSolFinite}. We now let $\mu_{\varrho_h}$, $\mu_{\mathbf{u}_h}$,  $\mu_{V_h}$ and $\mu_{W}$ be the respective laws of  $\mathbf{u}_h$, $\varrho_h$, $V_h$ and $W$ on the respective spaces $\chi_\varrho$,  $\chi_{ \mathbf{u} }$, $\chi_V$ and $\chi_W$. Furthermore, we set $\mu_h$ as the joint law of $ \varrho_h, \mathbf{u}_h, V_h$ and $W$ on the space $\chi = \chi_{ \mathbf{u} } \times \chi_\varrho \times \chi_V \times \chi_W$.
\begin{lemma}
The set $\{\mu_h :  h\in (0,1) \}$  is tight on $\chi$.
\end{lemma}
\begin{proof}
To proof the above lemma, we first note the following:
\begin{itemize}
\item From \eqref{esssupDensity}, the set $\{\mu_{\varrho_h}:  h\in (0,1) \}$ is tight on $\chi_\varrho$.
\item From \eqref{esssupVelocity}, the set $\{\mu_{\mathbf{u}_h}:  h\in (0,1) \}$ is tight on $\chi_{\mathbf{u}}$.
\item The set $\{ \mu_W\}$ is tight on $\chi_W$ since its a Radon measure on a Polish space.
\end{itemize}
For further details, please refer to \cite[Proposition 4.1.5.]{breit2018stoch}.
\\
For $V$, we note from \eqref{esssupCarge1} that since $V\in W^{1,\infty}([0,T]; C^{2+\nu}_x)$, it has a Lipschitz continuous representation (not relabelled) so in particular, it follows from \eqref{esssupCarge1} that
\begin{align*}
V \in C^\nu\big([0,T]; C^{2+\nu}_x \big) \cap  W^{1,\infty}\big([0,T]; C^{2+\nu}_x \big)
\end{align*}
$\mathbb{P}$-a.s. It therefore follow from the compact embedding
\begin{align*}
C^\nu\big([0,T]; C^{2+\nu}_x \big) \cap  W^{1,\infty}\big(0,T; C^{2+\nu}_x\big)
\hookrightarrow
C^{\nu_0}\big([0,T]; C^{2+\nu_0}_x \big) 
\end{align*}
where $\nu_0\in(0,\nu)$ that the set
\begin{align*}
A_L &:= \Big\{ V\in C^\nu\big([0,T]; C^{2+\nu}_x \big) \cap  W^{1,\infty}\big([0,T]; C^{2+\nu}_x \big)
\\&
\, :\,
\Vert V(t) \Vert_{C([0,T]C^{2+\nu}_x)}
+
\Vert  V(t) \Vert_{W^{1,\infty}([0,T];C^{2+\nu}_x) }
\leq L
 \Big\}
\end{align*}
is relatively compact in $\chi_V$. Finally, by using Chebyshev's inequality, we deduce that the measure of the complement of the set above
\begin{align*}
\mu_{V_h}\Big((A_L)^C \Big) \leq \frac{1}{L} \mathbb{E}\Big(\Vert V(t) \Vert_{C([0,T]C^{2+\nu}_x)}
+
\Vert  V(t) \Vert_{W^{1,\infty}([0,T];C^{\nu}_x) } \Big) = 0
\end{align*}
as $L\rightarrow \infty$. It follows that $\{\mu_{V_h}:  h\in (0,1) \}$ is tight on $\chi_V$.
\end{proof}
Now since $\chi$ is a Polish space, we may use the classical Prokhorov’s and Skorokhod’s theorems to obtain the following result.
\begin{lemma}
\label{lem:Jakow}
The exists a subsequence (not relabelled) $\{\mu_h :  h\in (0,1) \}$, a complete probability space $(\tilde{\Omega}, \tilde{\mathscr{F}}, \tilde{\mathbb{P}})$ with $\chi$-valued random variables
\begin{align*}
(\tilde{\varrho}_h, \tilde{\mathbf{u}}_h,\tilde{V}_h, \tilde{W}_h) \text{ and } (\tilde{\varrho}, \tilde{\mathbf{u}},\tilde{V}, \tilde{W}), \quad h\in (0,1)
\end{align*}
such that 
\begin{itemize}
\item the law of $(\tilde{\varrho}_h, \tilde{\mathbf{u}}_h,\tilde{V}_h, \tilde{W}_h)$ on $\chi$ is $\mu_h$, $h\in(0,1)$,
\item the law of $(\tilde{\varrho}, \tilde{\mathbf{u}},\tilde{V}, \tilde{W})$ on $\chi$ is a Radon measure,
\item the following convergence
\begin{align*}
\tilde{\varrho}_h \rightarrow \tilde{\varrho} \quad \text{in}\quad \chi_\varrho, & \quad \qquad
\tilde{\mathbf{u}}_h \rightarrow \tilde{\mathbf{u}} \quad \text{in}\quad \chi_\mathbf{u}, \\
\tilde{V}_h \rightarrow \tilde{V} \quad  \text{in}\quad \chi_V, & \quad \qquad
\tilde{W}_h \rightarrow \tilde{W} \quad \text{in}\quad \chi_W
\end{align*}
holds $\tilde{
\mathbb{P}}$-a.s.
\end{itemize}
\end{lemma}
We have therefore constructed stochastic processes $(\tilde{\varrho}, \tilde{\mathbf{u}},\tilde{V}, \tilde{W})$ which are progressively measurable with respect to the following complete right-continuous canonical filtration
\begin{align*}
\tilde{\mathscr{F}}_t:= \sigma \bigg(\sigma_t[\tilde{\varrho}], \sigma_t[\tilde{\mathbf{u}}], \sigma_t[\tilde{V}], \bigcup_{k\in \mathbb{N}}\sigma_t[\tilde{\beta}_k]  \bigg), \quad t\in [0,T].
\end{align*}
Furthermore, by \cite[Lemma 2.1.35, Corollary 2.1.36]{breit2018stoch}, the stochastic process $\tilde{W}=\sum_{k\in \mathbb{N}}\tilde{\beta}_k(t)e_k$ is an $(\tilde{\mathscr{F}}_t)$-cylindrical Wiener process.
\subsection{Identification of the limit}
\label{sec:limitIdentificationFirstLayer}
We now show that the limit random variables derived from Lemma \ref{lem:Jakow} satisfies \eqref{contEq1}--\eqref{elecField1}.
\begin{lemma}
The equation \eqref{contEq1} is satisfied for $(\tilde{\varrho}, \tilde{\mathbf{u}})$ $\mathbb{P}$-a.s. for a.e.  $(t,x) \in(0,T)\times \mathbb{T}^3$. In addition, \eqref{momEq1} is satisfied for $(\tilde{\varrho}, \tilde{\mathbf{u}}, \tilde{V}, \tilde{W})$ for all $\psi\in C^\infty_c([0,T))$ and $\bm{\phi} \in H_N$ $\tilde{\mathbb{P}}$-a.s.
\end{lemma}
\begin{proof}
For the proof of the first part of the above lemma, see \cite[Lemma 4.1.7]{breit2018stoch}. For the second part, we follow the argument of \cite[Proposition 4.1.8]{breit2018stoch}. First of all, by using the equality of laws given by Lemma \ref{lem:Jakow}, we get that the estimates
\begin{align}
\label{uh}
\big\Vert  \tilde{\mathbf{u}}_h(nh) -\tilde{\mathbf{u}}_h(t)  \big\Vert_{H_N} \lesssim h^\beta \Vert \tilde{\mathbf{u}}_h \Vert_{C^\beta([0,T];H_N)}
\end{align}
and
\begin{align}
\label{rhoh}
\big\Vert  \tilde{\varrho}_h(nh) -\tilde{\varrho}_h(t)  \big\Vert_{C^{2+\nu}_x} \lesssim h^\nu \Vert \tilde{\varrho}_h \Vert_{C^\nu([0,T];C^{2+\nu}_x)}
\end{align}
holds. In addition,
\begin{align}
\label{vh}
\big\Vert  \tilde{V}_h(nh) -\tilde{V}_h(t)  \big\Vert_{C^{2+\nu}_x} \lesssim h^\nu \Vert \tilde{V}_h \Vert_{C^\nu([0,T];C^{2+\nu}_x)}
\end{align}
The information above is enough to pass to the limit $h \rightarrow 0$ in the `deterministic' part of the corresponding momentum equation \eqref{momEq2} defined on the stochastic basis $(\tilde{\Omega}, \tilde{\mathscr{F}},(\tilde{\mathscr{F}}_t)_{t\geq0}  , \tilde{\mathbb{P}})$.
\\
For the noise term, we combine \eqref{uh}--\eqref{vh} with Proposition \ref{prop:martSolFinite} and the continuity properties of the operator $\Pi_N$ and the coefficients $\mathbf{g}_{k,\varepsilon}$  to get that for all $p\in(1,\infty)$, the convergence
\begin{align}
\Pi_N\big[ \tilde{\varrho}_h\Pi_N   \mathbf{g}_{k,\varepsilon}(\tilde{\varrho}_h(nh) , f(nh), (\tilde{\varrho}_h \tilde{\mathbf{u}}_h)(nh) ) \big] 
\rightarrow
\Pi_N\big[ \tilde{\varrho}\Pi_N   \mathbf{g}_{k,\varepsilon}(\tilde{\varrho} , f, \tilde{\varrho} \tilde{\mathbf{u}})  \big] 
\end{align}
holds $\tilde{\mathbb{P}}$-a.s. in $L^p((0,T) \times \mathbb{T}^3)$ for any $k\in \mathbb{N}$.
Furthermore, by using the estimates \eqref{stochCoeffBoundExi0}--\eqref{stochsummableConst} (which holds on the new probability space because of the equality of laws established by Lemma \ref{lem:Jakow}) together with It\^o's isometry, \eqref{projectionContinuity} and \eqref{esssupDensity},  we also gain
\begin{equation}
\begin{aligned}
 \tilde{\mathbb{E}} &\bigg\Vert \int_0^T  \Pi_N\big[ \tilde{\varrho}_h\Pi_N   \mathbf{G}_{\varepsilon}(\tilde{\varrho}_h(nh) , f(nh), (\tilde{\varrho}_h \tilde{\mathbf{u}}_h)(nh) ) \big]
  \, \mathrm{d}\tilde{W}_h \bigg\Vert^2_{L^2(\mathbb{T}^2)}
\lesssim 1
\end{aligned}
\end{equation}
just as in the proof of \cite[Proposition 4.1.8]{breit2018stoch}. Consequently, in combination with the strong convergence of $\tilde{W}_h$ in Lemma \ref{lem:Jakow},  we obtain the following convergence
\begin{align}
\Pi_N\big[ \tilde{\varrho}_h\Pi_N   \mathbf{G}_{\varepsilon}(\tilde{\varrho}_h(nh) , f(nh), (\tilde{\varrho}_h \tilde{\mathbf{u}}_h)(nh) ) \big] 
\rightarrow
\Pi_N\big[ \tilde{\varrho}\Pi_N   \mathbf{G}_{\varepsilon}(\tilde{\varrho} , f, \tilde{\varrho} \tilde{\mathbf{u}})  \big] 
\end{align}
$\tilde{\mathbb{P}}$-a.s. in $L^2(0,T; L_2(\mathfrak{U};L^2_x))$ where
\begin{align*}
\mathbf{G}_{\varepsilon}(\varrho,f, \varrho\mathbf{u})
=
\chi \bigg(\frac{\varepsilon}{\varrho} -1 \bigg)
\chi \bigg(\vert \mathbf{u}\vert -\frac{1}{\varepsilon} \bigg) \mathbf{G}( \varrho,f, \varrho\mathbf{u}).
\end{align*}
This finishes the proof.
\end{proof}
Finally, the following lemma identifies the Poisson equation. 
\begin{lemma}
The equation \eqref{elecField1} is satisfied for $(\tilde{\varrho}, \tilde{V},f)$ $\mathbb{P}$-a.s. for a.e.  $(t,x) \in(0,T)\times \mathbb{T}^3$. 
\end{lemma}
\begin{proof}
As $f \in C^\nu_x$ is given, by the equality of laws, i.e. Lemma \ref{lem:Jakow}, $(\tilde{\varrho}_h, \tilde{V}_h)$ satisfies \eqref{elecField2} as well as the estimates \eqref{esssupDensity1}--\eqref{esssupCarge1}  on the new probability space. Given that \eqref{rhoh}--\eqref{vh} holds, we are able to pass to the limit  and we get that \eqref{elecField1} is satisfied for $(\tilde{\varrho}, \tilde{V},f)$ $\mathbb{P}$-a.s. for a.e.  $(t,x) \in(0,T)\times \mathbb{T}^3$ where in addition, $(\tilde{\varrho}, \tilde{V},f)$ satisfies \eqref{esssupDensity1}--\eqref{esssupCarge1}.
\end{proof}
Combining the results we have established above in Section \ref{sec:limitIdentificationFirstLayer} completes the proof of Proposition \ref{prop:martSolFinite}. In order to solve Theorem \ref{thm:mainFirstLayer}, we require uniqueness.
\subsection{Pathwise uniqueness}
\label{sec:pathUniqFirstLayer}
The pathwise uniqueness of the martingale solution constructed in Proposition \ref{prop:martSolFinite} is given in the following result.
\begin{proposition}
\label{prop:uniqueMarting}
Consider two martingale solutions
\begin{align*}
\big[(\Omega ,\mathscr{F} ,(\mathscr{F} _t)_{t\geq0},\mathbb{P} );\varrho_1 , \mathbf{u}_1 , V_1 , W   \big] \quad \text{and} \quad \big[(\Omega ,\mathscr{F} ,(\mathscr{F} _t)_{t\geq0},\mathbb{P} );\varrho_2 , \mathbf{u}_2 , V_2 , W   \big]
\end{align*}
of \eqref{contEq1}--\eqref{elecField1} is the sense of Definition \ref{def:martSolFiniteDim} sharing a data
\begin{align*}
(\varrho_0, \mathbf{u}_0, f) \in C^{2+\nu}(\mathbb{T}^3) \times H_N \times C^{\nu}(\mathbb{T}^3) \quad \mathbb{P}\text{-a.s.}
\end{align*}
and with both satisfying  \eqref{esssupDensity1}--\eqref{esssupVelocity1}. Then
\begin{align*}
(\varrho_1, \mathbf{u}_1, V_1) = (\varrho_2, \mathbf{u}_2, V_2)\in C \big([0,T];C^{2+\nu}(\mathbb{T}^3) \times H_N \times C^{\nu}(\mathbb{T}^3) \big) \quad \mathbb{P}\text{-a.s.}
\end{align*}
\end{proposition}
\begin{proof}
The proof of Proposition \ref{prop:uniqueMarting} follow the ideas of \cite[Proposition 4.1.9]{breit2018stoch}, i.e., for $i=1,2$, we introduce the stopping times
\begin{align*}
\tau_n^i := \inf_{t\in [0,T]} \Big\{
\Vert \varrho_i(t) \Vert_{C^{2+\nu}_x}
+
\Vert (\varrho_i)^{-1}(t) \Vert_{C_x}
+
\Vert V_i(t) \Vert_{C^{2+\nu}_x}
+
\Vert\mathbf{u}_i(t)\Vert_{H_N}>n \Big\}
\end{align*}
where $\inf\emptyset=T$ and set $\tau_n = \tau_n^1 \wedge \tau_n^2$. Note that  the collection $(\tau_n)_{n\in \mathbb{N}}$ is increasing and that $\tau_n \rightarrow T$ almost surely by virtue of \eqref{esssupDensity1}--\eqref{esssupVelocity1}.
\\
We now consider the following difference equations 
\begin{align}
&\partial_t( \varrho_1 -\varrho_2)  + \mathrm{div}\big(\varrho_1 (\chi_{\mathbf{u}_{1,R}}  \mathbf{u}_1) -\varrho_2 (\chi_{\mathbf{u}_{2,R}}  \mathbf{u}_2)  \big)  =\varepsilon \Delta (\varrho_1 -\varrho_2), 
\label{contEq4}
\\
&\pm \Delta (V_1 -V_2)  = (\varrho_1 -\varrho_2) \label{elecField4}
\end{align} 
of two solutions to \eqref{contEq1} and \eqref{elecField1} where as in \eqref{cuttOff}--\eqref{chiUr},
\begin{align}
\label{chiUri}
\chi_{\mathbf{u}_{i,R}}= \chi \big( \Vert \mathbf{u}_i\Vert_{H_N} -R \big).
\end{align}
By using \eqref{chiUri}, we recall from \cite[(4.43)--(4.44)]{breit2018stoch}
that the estimate
\begin{equation}
\begin{aligned}
\label{contEq5}
&\sup_{t' \in [0,\tau]}
\bigg[  
\Vert (\varrho_1 -\varrho_2)(t') \Vert_{C^{2+\nu}_x} 
+
\Vert \partial_t(\varrho_1 -\varrho_2)(t') \Vert_{C^{\nu}_x}
\bigg]
\lesssim
\sup_{t' \in [0,\tau]}
\Vert (\mathbf{u}_1 - \mathbf{u}_2)(t') \Vert_{H_N} 
\\
&+\bigg( \int_0^\tau
\Vert \varrho_1 -\varrho_2 \Vert^2_{C^{2+\nu}_x} \, \mathrm{d}t
\bigg)^\frac{1}{2}
+\Vert (\varrho_1 -\varrho_2)(0) \Vert_{C^{2+\nu}_x} 
\end{aligned}
\end{equation}
follow from \eqref{contEq4}. For $0\leq l$, we can therefore infer from \eqref{elecField4} that estimate
\begin{equation}
\begin{aligned}
\label{elecField5}
& \sup_{t' \in [0,\tau]} \sum_{l=0}^1
\Vert \nabla^l (V_1 -V_2)(t') \Vert_{C^{2+\nu}_x} 
\lesssim
\sup_{t' \in [0,\tau]} 
\Vert \Delta(V_1 -V_2)(t') \Vert_{C^{2+\nu}_x} 
\lesssim
\sup_{t' \in [0,\tau]}
\Vert (\mathbf{u}_1 - \mathbf{u}_2)(t') \Vert_{H_N} 
\\
&+\bigg( \int_0^\tau
\Vert \varrho_1 -\varrho_2 \Vert^2_{C^{2+\nu}(\mathbb{T}^3)} \, \mathrm{d}t
\bigg)^\frac{1}{2}
+\Vert (\varrho_1 -\varrho_2)(0) \Vert_{C^{2+\nu}_x} 
\end{aligned}
\end{equation} 
holds true. Also, by using \eqref{invertibleLinearMap}--\eqref{holderInvertibleMap} and the continuity equation \eqref{contEq1}, we rewrite the momentum equation \eqref{momEq1} as
\begin{align}
&\mathrm{d}  \mathbf{u} +\mathcal{M}^{-1}[\varrho] \Pi_N\big[ \mathbf{u}\partial_t \varrho\big]\mathrm{d}t +
\mathcal{M}^{-1}[\varrho]\Pi_N 
\mathrm{div} \big(\varrho  (\chi_{\mathbf{u}_R}  \mathbf{u})  \otimes \mathbf{u} \big) \, \mathrm{d}t + \mathcal{M}^{-1}[\varrho]\Pi_N \chi_{\mathbf{u}_R}  \nabla   p_{\delta}^{ \Gamma}(\varrho) \, \mathrm{d}t 
\nonumber\\&
=
\mathcal{M}^{-1}[\varrho]\Pi_N
\big[\varepsilon\Delta(\varrho\mathbf{u})+ \nu^S \Delta \mathbf{u}    
+(\nu^B + \nu^S)\nabla \mathrm{div} \mathbf{u}\big] \, \mathrm{d}t
+ \vartheta   \nabla V  \, \mathrm{d}t 
\nonumber\\
&\quad 
+ \sum_{k\in \mathbb{N}} \Pi_N  \mathbf{g}_{k,\varepsilon}(\varrho , f, \varrho \mathbf{u} )  \, \mathrm{d}\beta_k.
\end{align}
It follows that $(\mathbf{u}_1- \mathbf{u}_2)$ satisfies
\begin{equation}
\begin{aligned}
\label{momEq4a}
&\mathrm{d} (\mathbf{u}_1- \mathbf{u}_2) =\big( \mathcal{M}^{-1}[\varrho_2] - \mathcal{M}^{-1}[\varrho_1] \big) \Pi_N\big[ \mathbf{u}_1\partial_t \varrho_1 \big]\mathrm{d}t 
+
 \mathcal{M}^{-1}[\varrho_2]\Pi_N\big[ \mathbf{u}_2\partial_t \varrho_2 -  \mathbf{u}_1\partial_t \varrho_1 \big]\mathrm{d}t
\\&-
\mathcal{M}^{-1}[\varrho_1]\Pi_N \Big[ 
\mathrm{div} \big(\varrho_1  (\chi_{\mathbf{u}_{1,R}}  \mathbf{u}_1)  \otimes \mathbf{u}_1 \big)
+
\chi_{\mathbf{u}_{1,R}}  \nabla   p_{\delta}^{ \Gamma}(\varrho_1) 
\Big] \, \mathrm{d}t 
\\&+
\mathcal{M}^{-1}[\varrho_2]\Pi_N \Big[ 
\mathrm{div} \big(\varrho_2  (\chi_{\mathbf{u}_{2,R}}  \mathbf{u}_2)  \otimes \mathbf{u}_2 \big)
+
\chi_{\mathbf{u}_{2,R}}  \nabla   p_{\delta}^{ \Gamma}(\varrho_2) 
\Big] \, \mathrm{d}t 
\\&+
\big( \mathcal{M}^{-1}[\varrho_1] - \mathcal{M}^{-1}[\varrho_2] \big) 
 \Pi_N
\big[\varepsilon\Delta(\varrho_1\mathbf{u}_1)+ \nu^S \Delta \mathbf{u}_1   +(\nu^B + \nu^S)\nabla \mathrm{div} \mathbf{u}_1 \big] \, \mathrm{d}t  
\\&
 +\mathcal{M}^{-1}[\varrho_2]
 \Pi_N
\big[\varepsilon\Delta(\varrho_1\mathbf{u}_1 - \varrho_2\mathbf{u}_2)+ \nu^S \Delta (\mathbf{u}_1 - \mathbf{u}_2)   
+(\nu^B + \nu^S)\nabla \mathrm{div}( \mathbf{u}_1- \mathbf{u}_2)  \big] \, \mathrm{d}t  
\\&+ \vartheta \nabla (V_1-  V_2)\, \mathrm{d}t  
+ \sum_{k\in \mathbb{N}} \Pi_N \big[  \mathbf{g}_{k,\varepsilon}(\varrho_1 , f, \varrho_1 \mathbf{u}_1 ) - \mathbf{g}_{k,\varepsilon}(\varrho_2 , f, \varrho_2 \mathbf{u}_2 ) \big] \, \mathrm{d}\beta_k.
\end{aligned}
\end{equation} 
By applying It\^o's product rule to \eqref{momEq4a}, we obtain
\begin{equation}
\begin{aligned}
\label{momEq4}
&\mathrm{d}  \int_{\mathbb{T}^3}\frac{1}{2} \vert  \mathbf{u}_1- \mathbf{u}_2\vert^2 \, \mathrm{d}x
=
 \int_{\mathbb{T}^3} 
 \big( \mathcal{M}^{-1}[\varrho_2] - \mathcal{M}^{-1}[\varrho_1] \big) \Pi_N\big[ \mathbf{u}_1\partial_t \varrho_1 \big]\cdot (\mathbf{u}_1 -\mathbf{u}_2) \, \mathrm{d}x \mathrm{d}t 
\\&+
 \int_{\mathbb{T}^3} 
 \mathcal{M}^{-1}[\varrho_2]\Pi_N\big[ \mathbf{u}_2\partial_t \varrho_2 -  \mathbf{u}_1\partial_t \varrho_1 \big]\cdot (\mathbf{u}_1 -\mathbf{u}_2) \, \mathrm{d}x \mathrm{d}t
\\&-
 \int_{\mathbb{T}^3} 
\mathcal{M}^{-1}[\varrho_1]\Pi_N \Big[ 
\mathrm{div} \big(\varrho_1  (\chi_{\mathbf{u}_{1,R}}  \mathbf{u}_1)  \otimes \mathbf{u}_1 \big)
+
\chi_{\mathbf{u}_{1,R}}  \nabla   p_{\delta}^{ \Gamma}(\varrho_1) 
\Big]
 (\mathbf{u}_1 -\mathbf{u}_2) \, \mathrm{d}x \, \mathrm{d}t 
\\&+
 \int_{\mathbb{T}^3} 
\mathcal{M}^{-1}[\varrho_2]\Pi_N \Big[ 
\mathrm{div} \big(\varrho_2  (\chi_{\mathbf{u}_{2,R}}  \mathbf{u}_2)  \otimes \mathbf{u}_2 \big)
+
\chi_{\mathbf{u}_{2,R}}  \nabla   p_{\delta}^{ \Gamma}(\varrho_2) 
\Big] 
(\mathbf{u}_1 -\mathbf{u}_2) \, \mathrm{d}x \, \mathrm{d}t 
\\&+
 \int_{\mathbb{T}^3} 
\big( \mathcal{M}^{-1}[\varrho_1] - \mathcal{M}^{-1}[\varrho_2] \big) 
 \Pi_N
\big[\varepsilon\Delta(\varrho_1\mathbf{u}_1)+ \nu^S \Delta \mathbf{u}_1   +(\nu^B + \nu^S)\nabla \mathrm{div} \mathbf{u}_1 \big] 
(\mathbf{u}_1 -\mathbf{u}_2) \, \mathrm{d}x \, \mathrm{d}t 
\\&
 + \int_{\mathbb{T}^3} \mathcal{M}^{-1}[\varrho_2]
 \Pi_N
\big[\varepsilon\Delta(\varrho_1\mathbf{u}_1 - \varrho_2\mathbf{u}_2)+ \nu^S \Delta (\mathbf{u}_1 - \mathbf{u}_2)   
\\&
+(\nu^B + \nu^S)\nabla \mathrm{div}( \mathbf{u}_1- \mathbf{u}_2)  \big]\cdot (\mathbf{u}_1 -\mathbf{u}_2) \, \mathrm{d}x \, \mathrm{d}t  
+ \vartheta 
 \int_{\mathbb{T}^3} \nabla (V_1-  V_2) \cdot (\mathbf{u}_1 -\mathbf{u}_2) \, \mathrm{d}x\, \mathrm{d}t  
\\&
+ \frac{1}{2}\sum_{k\in \mathbb{N}} \int_{\mathbb{T}^3} \Big\vert \Pi_N \big[  \mathbf{g}_{k,\varepsilon}(\varrho_1 , f, \varrho_1 \mathbf{u}_1 ) - \mathbf{g}_{k,\varepsilon}(\varrho_2 , f, \varrho_2 \mathbf{u}_2 ) \big] 
\Big\vert^2 \mathrm{d}x\, \mathrm{d}t
\\&
+ \sum_{k\in \mathbb{N}} \int_{\mathbb{T}^3} \Pi_N \big[  \mathbf{g}_{k,\varepsilon}(\varrho_1 , f, \varrho_1 \mathbf{u}_1 ) - \mathbf{g}_{k,\varepsilon}(\varrho_2 , f, \varrho_2 \mathbf{u}_2 ) \big] 
\cdot (\mathbf{u}_1 -\mathbf{u}_2) \, \mathrm{d}x\, \mathrm{d}\beta_k.
\end{aligned}
\end{equation} 
Now just as in \cite[(4.40)--(4.42)]{breit2018stoch}, by using  the definition of the stopping time $\tau_n$, \eqref{holderInvertibleMap}, \eqref{stochCoeffBound1Exis0a}--\eqref{stochsummableConsta}, 
as well as the Burkholder--Davis--Gundy inequality to tackle the noise term, it follows from \eqref{momEq4} that for any $\kappa>0$, the estimate
\begin{equation}
\begin{aligned}
\label{momEq5}
&\mathbb{E}\sup_{t' \in [0,T]}
\Vert (\mathbf{u}_1 -\mathbf{u}_2)(t' \wedge \tau_n) \Vert_{H_N}^2 
\leq
 \kappa \,\mathbb{E}
\sup_{t' \in [0,T]}
\Vert (\varrho_1 - \varrho_2)(t') \Vert_{C^{2+\nu}_x}^2 
\\&+c_{n, \kappa} \, \mathbb{E}
\int_0^{T\wedge\tau_n} \big(
\Vert \mathbf{u}_1 -\mathbf{u}_2 \Vert^2_{H_N}
+
\Vert \varrho_1 -\varrho_2 \Vert^2_{L^{2}_x}
+
\Vert \nabla( V_1 - V_2) \Vert_{L^{2}_x}^2 \big) \mathrm{d}t
+ \mathbb{E}\, \Vert (\mathbf{u}_1 -\mathbf{u}_2)(0) \Vert^2_{H_N}
\end{aligned}
\end{equation}
holds true. For $\tau =T\wedge \tau_n$, we can combine \eqref{contEq5}--\eqref{elecField5} with \eqref{momEq5} and we get
\begin{equation}
\begin{aligned}
\mathbb{E}&\sup_{t' \in [0,T]}  \bigg(
\Vert (\mathbf{u}_1 -\mathbf{u}_2)(t' \wedge \tau_n) \Vert_{H_N}^2 +\Vert (\varrho_1 - \varrho_2)(t') \Vert_{C^{2+\nu}_x}^2 
+
 \sum_{l=0}^1
\Vert \nabla^l  (V_1 - V_2)(t') \Vert_{C^{2+\nu}_x}^2
\bigg)
\\&\lesssim
 \mathbb{E}
\int_0^{T\wedge\tau_n} \big(
\Vert \mathbf{u}_1 -\mathbf{u}_2 \Vert^2_{H_N}
+
\Vert \varrho_1 -\varrho_2 \Vert^2_{C^{2+\nu}_x}
+
\Vert \nabla( V_1 - V_2) \Vert_{C^{2+\nu}_x}^2 \big) \mathrm{d}t
\\&
+\mathbb{E}\,\Vert (\varrho_1 -\varrho_2)(0) \Vert_{C^{2+\nu}_x}  + \mathbb{E}\, \Vert (\mathbf{u}_1 -\mathbf{u}_2)(0) \Vert^2_{H_N}.
\end{aligned}
\end{equation}
Gronwall's lemma, the fact that $\mathbf{u}_1\vert_{t=0}= \mathbf{u}_2\vert_{t=0}=\mathbf{u}_0$ and $\varrho_1\vert_{t=0}= \varrho_2\vert_{t=0}=\varrho_0$ finishes the proof.
\end{proof}
\subsection{Conclusion}
Pathwise uniqueness as shown in  Section \ref{sec:pathUniqFirstLayer} and martingale solution given by Proposition \ref{prop:martSolFinite} combines to give the existence of a pathwise solution Theorem \ref{thm:mainFirstLayer}. This relies on
Gy\"ongy--Krylov characterization of convergence given in  \cite[Theorem 2.10.3]{gyongy1996existence}. In order to show that the hypothesis of the Gy\"ongy--Krylov result are satisfied, we refer the reader to the short proof of \cite[Theorem 4.1.12]{breit2018stoch}.
\begin{remark}
As in \cite[Corollary 4.1.13]{breit2018stoch}, we can relax the assumption on the data $(\varrho_0\mathbf{u}_0,f )$ in Theorem \ref{thm:mainFirstLayer}. In particular, this can be replaced  by any $\mathscr{F}_0$-measurable random variables $(\varrho_0\mathbf{u}_0,f )$ satisfying
\begin{align*}
\varrho_0 \geq \underline{\varrho}>0, \quad \varrho_0\in C^{2+\nu}_x, \quad  f \in C^{\nu}_x, \quad \mathbf{u}_0 \in H_N
\end{align*}
$\mathbb{P}$-a.s.
\end{remark}
\subsection{Energy equality}
We now show that the solution constructed in Theorem \ref{thm:mainFirstLayer} in the sense of Definition \ref{def:pathSolFiniteDim} satisfies an energy equality. This is given the following result.
\begin{proposition}
\label{prop:energyEquality}
Let  $(\varrho, \mathbf{u}, V)$ be a pathwise solution of \eqref{contEq1}--\eqref{elecField1} in the sense of Definition \ref{def:pathSolFiniteDim}. Then the energy equality
\begin{equation}
\begin{aligned}
\label{energyEquality}
-
&\int_0^T\partial_t \psi \int_{{\mathbb{T}^3}}\bigg[ \frac{1}{2} \varrho  \vert \mathbf{u}  \vert^2  
+ P^\Gamma_\delta(\varrho ) 
\pm
\vartheta \vert \nabla V \vert^2 
\bigg]\,\mathrm{d}x \, \mathrm{d}t
+\nu^S
\int_0^T\psi \int_{{\mathbb{T}^3}}\vert  \nabla \mathbf{u}  \vert^2 \,\mathrm{d}x  \,\mathrm{d}x
\\&+(\nu^B+\nu^S)
\int_0^T \psi \int_{{\mathbb{T}^3}}\vert  \mathrm{div}\, \mathbf{u}  \vert^2 \,\mathrm{d}x  \,\mathrm{d}t
+\varepsilon
\int_0^T\psi \int_{{\mathbb{T}^3}} \varrho\vert  \nabla \mathbf{u}  \vert^2 \,\mathrm{d}x  \,\mathrm{d}x
\\&+\varepsilon 
\int_0^T \psi \int_{{\mathbb{T}^3}}(P{_\delta^\Gamma}) ''(\varrho)\vert  \nabla \varrho  \vert^2 \,\mathrm{d}x  \,\mathrm{d}t
=
\psi(0)
 \int_{{\mathbb{T}^3}}\bigg[\frac{1}{2} \varrho _0\vert \mathbf{u} _0 \vert^2  
+ P^\Gamma_\delta(\varrho _0)
\pm
\int_{\mathbb{T}^3}  
\vartheta \vert \nabla V_0 \vert^2
 \bigg]\,\mathrm{d}x
\\&
+\frac{1}{2}\int_0^T \psi
 \int_{{\mathbb{T}^3}}
 \varrho
 \sum_{k\in\mathbb{N}} \big\vert \Pi_N
 \mathbf{g}_{k, \varepsilon}(x,\varrho ,f,\mathbf{m}  ) 
\big\vert^2\,\mathrm{d}x\,\mathrm{d}t
+\int_0^T \psi  \int_{{\mathbb{T}^3}}   \varrho\mathbf{u} \cdot \Pi_N \mathbf{G}_\varepsilon(\varrho , f, \mathbf{m} )\,\mathrm{d}x\,\mathrm{d}W ;
\end{aligned}
\end{equation}
holds $\mathbb{P}$-a.s. for all $\psi \in C^\infty_c ([0,T))$.
\end{proposition}
In order to show \eqref{energyEquality}, we  mimic the formal argument of Section \ref{sec:energyIneq}, i.e., we consider
\begin{align}
&\mathrm{d} \varrho  + \mathrm{div}\big(\varrho (\chi_{\mathbf{u}_R}  \mathbf{u}) \big) \, \mathrm{d}t =\varepsilon \Delta \varrho\, \mathrm{d}t, 
\label{contEq3}
\\
&\mathrm{d}  \Pi_N\mathbf{u}  + \Pi_N
\big[ \varepsilon \varrho^{-1} \Delta \varrho\cdot \mathbf{u} + \chi_{\mathbf{u}_R}
  \mathbf{u}  \cdot \nabla \mathbf{u}  + \varrho^{-1}\chi_{\mathbf{u}_R}\nabla  p^\Gamma_\delta( \varrho) \big]\, \mathrm{d}t 
  =\Pi_N \big[\varepsilon \varrho^{-1} \Delta(\varrho \mathbf{u}) + \nu^S \varrho^{-1}\Delta \mathbf{u} 
\nonumber \\&
\qquad +(\nu^B + \nu^S) \varrho^{-1} \nabla \mathrm{div} \mathbf{u} 
  + \vartheta   \nabla V  \big] \, \mathrm{d}t +  \sum_{k\in \mathbb{N}} \Pi_N  \mathbf{g}_{k, \varepsilon}(x,\varrho , f,  \varrho \mathbf{u} ) \, \mathrm{d}\beta_k ,
\label{momEq3}
\\
&\pm \Delta V  = \varrho  - f .
\end{align}
and apply It\^o's lemma, Theorem \ref{thm:itoLemma}, to  the functional
\begin{align*}
F(\varrho,  \mathbf{u})(t)=\int_{\mathbb{T}^3}\frac{1}{2} \varrho(t) \vert \mathbf{u}(t) \vert^2 \, \mathrm{d}x
\end{align*}
and we obtain
\begin{equation}
\begin{aligned}
\label{energyEqu}
\int_{\mathbb{T}^3}&\frac{1}{2} \varrho(t) \vert \mathbf{u}(t) \vert^2 \, \mathrm{d}x
= \int_{\mathbb{T}^3} \frac{1}{2}\varrho_0 \vert \mathbf{u}_0 \vert^2 \, \mathrm{d}x
-
\int_0^t\int_{\mathbb{T}^3}  \big[\varepsilon \vert \mathbf{u} \vert^2 \Delta \varrho  
+ \varrho\mathbf{u}\cdot \chi_{\mathbf{u}_R}\mathbf{u}\cdot \nabla \mathbf{u}
\\&+ \mathbf{u} \,\chi_{\mathbf{u}_R} \nabla p^\Gamma_\delta \big]\mathrm{d}x \, \mathrm{d}s
+
\int_0^t\int_{\mathbb{T}^3}  
\big[\varepsilon \mathbf{u} \cdot \Delta (\varrho  \mathbf{u}) 
+
 \nu^S \mathbf{u}\cdot\Delta \mathbf{u}  
 \\&+(\nu^B + \nu^S) \mathbf{u}\cdot\nabla \mathrm{div} \mathbf{u} 
+ \vartheta \varrho\mathbf{u} \cdot  \nabla V
  \big] \, \mathrm{d}x \, \mathrm{d}s 
-\int_0^t\int_{\mathbb{T}^3} \frac{1}{2}\vert \mathbf{u} \vert^2 \big[ \mathrm{div}(\varrho\chi_{\mathbf{u}_R} \mathbf{u}) - \varepsilon \Delta \varrho \big] \, \mathrm{d}x \, \mathrm{d}s 
\\&+
\sum_{k \in \mathbb{N}} \int_0^t\int_{\mathbb{T}^3} \frac{1}{2} \varrho \vert \Pi_N \mathbf{g}_{k,\varepsilon}(x,\varrho , f, \varrho \mathbf{u} ) \vert^2 \mathrm{d}x\, \mathrm{d}s
+\sum_{k \in \mathbb{N}} \int_0^t\int_{\mathbb{T}^3} \varrho\mathbf{u} \cdot  \Pi_N  \mathbf{g}_{k,\varepsilon}(x,\varrho , f,  \varrho \mathbf{u}) \mathrm{d}x\, \mathrm{d}\beta_k.
\end{aligned}
\end{equation}
$\mathbb{P}$-a.s. 
However, note that
\begin{equation}
\begin{aligned}
\int_{\mathbb{T}^3}\varepsilon \mathbf{u} \cdot \Delta(\varrho \mathbf{u}) \, \mathrm{d}x 
&=
-
\int_{\mathbb{T}^3} \varepsilon \nabla\mathbf{u}: \nabla(\varrho \mathbf{u}) \, \mathrm{d}x 
=
-
\int_{\mathbb{T}^3} \varepsilon \nabla\mathbf{u}: \nabla\varrho \otimes \mathbf{u} \, \mathrm{d}x
-
\int_{\mathbb{T}^3} \varepsilon \nabla\mathbf{u}: \varrho \nabla \mathbf{u} \, \mathrm{d}x
\\&
=
\frac{1}{2}
\int_{\mathbb{T}^3} \varepsilon \vert\mathbf{u}\vert^2 \Delta\varrho  \, \mathrm{d}x
-
\int_{\mathbb{T}^3} \varepsilon \varrho  \vert \nabla\mathbf{u} \vert^2 \, \mathrm{d}x.
\end{aligned}
\end{equation}
Also,
\begin{equation}
\begin{aligned}
-\int_{\mathbb{T}^3}& \varrho\mathbf{u}\cdot \chi_{\mathbf{u}_R}\mathbf{u}\cdot \nabla \mathbf{u} \, \mathrm{d}x 
=
\int_{\mathbb{T}^3} \frac{1}{2}\vert\mathbf{u} \vert^2 \mathrm{div}( \varrho \chi_{\mathbf{u}_R} \mathbf{u})  \, \mathrm{d}x.
\end{aligned}
\end{equation}
If we combine the computations above with \eqref{greenFunction}--\eqref{poissonPartAlt}, equation \eqref{energyEqu} therefore reduces to
\begin{equation}
\begin{aligned}
\label{energyEquFirstLayerFirstPart}
\int_{\mathbb{T}^3}&\bigg[ \frac{1}{2} \varrho  \vert \mathbf{u}  \vert^2  
\pm
\vartheta \vert \nabla V \vert^2 
\bigg] \, \mathrm{d}x
= \int_{\mathbb{T}^3} \bigg[ \frac{1}{2} \varrho  \vert \mathbf{u}  \vert^2  
\pm
\vartheta \vert \nabla V_0 \vert^2
\bigg] \, \mathrm{d}x
-
\int_0^t\int_{\mathbb{T}^3}  \varepsilon \varrho \vert \nabla \mathbf{u} \vert^2 \,\mathrm{d}x \, \mathrm{d}s   
\\&+ \int_0^t\int_{\mathbb{T}^3}
  p^\Gamma_\delta \mathrm{div}\big( \chi_{\mathbf{u}_R}\mathbf{u} \big) \,\mathrm{d}x \, \mathrm{d}s
+
\int_0^t\int_{\mathbb{T}^3}  
\big[
 \nu^S \mathbf{u}\cdot\Delta \mathbf{u}  +(\nu^B + \nu^S) \mathbf{u}\cdot\nabla \mathrm{div} \mathbf{u} 
  \big] \, \mathrm{d}x \, \mathrm{d}s 
\\&+
\sum_{k \in \mathbb{N}} \int_0^t\int_{\mathbb{T}^3} \frac{1}{2} \varrho \vert \Pi_N  \mathbf{g}_{k,\varepsilon}(x,\varrho , f, \varrho \mathbf{u} ) \vert^2 \mathrm{d}x\, \mathrm{d}s
+\sum_{k \in \mathbb{N}} \int_0^t\int_{\mathbb{T}^3} \varrho\mathbf{u} \cdot  \Pi_N  \mathbf{g}_{k,\varepsilon}(x,\varrho , f,  \varrho \mathbf{u}) \mathrm{d}x\, \mathrm{d}\beta_k.
\end{aligned}
\end{equation}
$\mathbb{P}$-a.s. Now if we multiply \eqref{contEq3} by  $P^\Gamma_\delta(\varrho)'$, then we get the following \textit{renormalized continuity equation}
\begin{equation}
\begin{aligned}
\label{renormalisedConEqFirstLayer}
\int_{\mathbb{T}^3}&P^\Gamma_\delta(\varrho(t)) \, \mathrm{d}x
=
\int_{\mathbb{T}^3} P^\Gamma_\delta(\varrho_0) \, \mathrm{d}x
+
\int_0^t \int_{\mathbb{T}^3} \Big[\varepsilon P^\Gamma_\delta(\varrho)' \Delta \varrho - \mathrm{div} \big(P^\Gamma_\delta(\varrho)  \chi_{\mathbf{u}_R} \mathbf{u} \big)  
\\&- 
\big[P^\Gamma_\delta(\varrho)' \varrho -P^\Gamma_\delta(\varrho) \big] \mathrm{div} \big(  \chi_{\mathbf{u}_R}  \mathrm{u}\big)\Big]  \, \mathrm{d}x \, \mathrm{d}s.
\end{aligned}
\end{equation}
Recall that by Definition \ref{def:pathSolFiniteDim}, \eqref{contEq3} is satisfied pointwise almost everywhere. The proof is done once we use the identities
\begin{align*}
p^\Gamma_\delta(\varrho) = P^\Gamma_\delta(\varrho)' \varrho -P^\Gamma_\delta(\varrho) 
\end{align*}
and
\begin{align*}
\int_{\mathbb{T}^3} \varepsilon P^\Gamma_\delta(\varrho)' \Delta \varrho  \, \mathrm{d}x
=
-
\int_{\mathbb{T}^3} \varepsilon P^\Gamma_\delta(\varrho)'' \vert \nabla \varrho \vert^2  \, \mathrm{d}x
\end{align*}
together with the divergence theorem
in \eqref{renormalisedConEqFirstLayer} and substitute the resulting equation into \eqref{energyEquFirstLayerFirstPart}.

\section{The second approximation layer}
\label{sec:secondLayer}
We now wish to construct a class of solution to the following system 
\begin{align}
&\mathrm{d} \varrho  + \mathrm{div}(\varrho  \mathbf{u})  \, \mathrm{d}t =\varepsilon \Delta \varrho\, \mathrm{d}t, 
\label{contEqSecond1}
\\
&\mathrm{d}\Pi_N(\varrho  \mathbf{u} ) +
\Pi_N 
\big[
\mathrm{div} (\varrho  \mathbf{u}  \otimes \mathbf{u} ) +  \nabla   p_{\delta}^{ \Gamma}(\varrho) \big]\, \mathrm{d}t =
\Pi_N
\big[\varepsilon\Delta(\varrho\mathbf{u})+ \nu^S \Delta \mathbf{u} +(\nu^B + \nu^S)\nabla \mathrm{div} \mathbf{u}
\nonumber\\&
\quad  + \vartheta \varrho  \nabla V  \big] \, \mathrm{d}t   
+ \sum_{k\in \mathbb{N}}\Pi_N\big[ \varrho\Pi_N   \mathbf{g}_{k,\varepsilon}(\varrho , f, \varrho \mathbf{u} ) \big] \, \mathrm{d}\beta_k ,
\label{momEqSecond1}
\\
&\pm \Delta V  = \varrho  - f, \label{elecFieldSecond1}
\end{align}
which conserves energy by passing to the limit $R\rightarrow \infty$ in \eqref{contEq1}--\eqref{momEq1}. We now give the precise definition of this solution.
\begin{definition}
\label{def:pathSolSecondLayer}
Let $(\Omega, \mathscr{F}, (\mathscr{F}_t)_{t\geq0 }, \mathbb{P})$ be a stochastic basis where the filtration is complete and right-continuous. Let $W$ be an $(\mathscr{F}_t)$-cylindrical Wiener process and further let $(\varrho_0, \mathbf{u_0},f)$ be $\mathscr{F}_0$-measurable random variables belonging to $C^{2+\nu}(\mathbb{T}^3) \times H_N \times C^{\nu}(\mathbb{T}^3)$. Then 
$
[\varrho , \mathbf{u} , V ]$
is a \textit{pathwise solution} of \eqref{contEqSecond1}--\eqref{elecFieldSecond1} with data  $(\varrho_0, \mathbf{u_0},f)$ if
\begin{enumerate}
\item the density  $\varrho>0$, $\varrho \in C \big( [0,T]; C^{2+\nu}_x\big)$, $\partial_t \varrho \in C \big( [0,T]; C^{\nu}_x \big)$
 $\mathbb{P}$-a.s. and it is $(\mathscr{F}_t)$-progressively measurable;
\item the velocity field $\mathbf{u} \in C \big( [0,T]; H_N\big) $ $\mathbb{P}$-a.s. and it is $(\mathscr{F}_t)$-progressively measurable;
\item the electric field $V \in C \big( [0,T]; C^{2+\nu}_x\big)$ $\mathbb{P}$-a.s. and it is $(\mathscr{F}_t)$-progressively measurable;
\item the following $(\varrho(0), \mathbf{u}(0)) = (\varrho_0, \mathbf{u}_0)$ holds $\mathbb{P}$-a.s;
\item equation \eqref{contEqSecond1} holds $\mathbb{P}$-a.s. for a.e.  $(t,x) \in(0,T)\times \mathbb{T}^3$;
\item equation \eqref{elecFieldSecond1} holds $\mathbb{P}$-a.s. for a.e.  $(t,x) \in(0,T)\times \mathbb{T}^3$;
\item for all $\psi \in C^\infty_c ([0,T))$ and $\bm{\phi} \in H_N$, the following
\begin{equation}
\begin{aligned}
&-\int_0^T \partial_t \psi \int_{{\mathbb{T}^3}} \varrho  \mathbf{u} (t) \cdot \bm{\phi} \, \mathrm{d}x    \mathrm{d}t
=
\psi(0)  \int_{{\mathbb{T}^3}} \varrho _0 \mathbf{u} _0 \cdot \bm{\phi} \, \mathrm{d}x  
+  \int_0^T \psi \int_{{\mathbb{T}^3}}  \varrho   \mathbf{u} \otimes \mathbf{u} : \nabla \bm{\phi} \, \mathrm{d}x  \mathrm{d}t 
\\&-
 \nu^S\int_0^T \psi \int_{{\mathbb{T}^3}} \nabla \mathbf{u} : \nabla \bm{\phi} \, \mathrm{d}x  \mathrm{d}t 
-(\nu^B+ \nu^S )\int_0^T \psi \int_{{\mathbb{T}^3}} \mathrm{div}\,\mathbf{u} \, \mathrm{div}\, \bm{\phi} \, \mathrm{d}x  \mathrm{d}t    
\\&+   
\int_0^T \psi \int_{{\mathbb{T}^3}}    p_{\delta}^{\Gamma}\, \mathrm{div} \, \bm{\phi} \, \mathrm{d}x  \mathrm{d}t
+\int_0^T \psi \int_{{\mathbb{T}^3}} \varepsilon\varrho\mathbf{u}    \Delta\bm{\phi} \, \mathrm{d}x   \mathrm{d}t
\\&+\int_0^T \psi \int_{{\mathbb{T}^3}} \vartheta\varrho  \nabla V  \cdot \bm{\phi} \, \mathrm{d}x    \mathrm{d}t
+
\int_0^T \psi \sum_{k\in \mathbb{N}} \int_{{\mathbb{T}^3}}  \Pi_N\big[ \varrho\Pi_N \mathbf{g}_{k,\varepsilon}(\varrho , f, \varrho \mathbf{u} ) \big] \cdot \bm{\phi}  \, \mathrm{d}x \mathrm{d}\beta_k  
\end{aligned}
\end{equation}
hold $\mathbb{P}$-a.s;
\item the energy equality
\begin{equation}
\begin{aligned}
\label{energyEqualitySecondLayer}
-
\int_0^T& \partial_t \psi \int_{{\mathbb{T}^3}}\bigg[ \frac{1}{2} \varrho  \vert \mathbf{u}  \vert^2  
+ P^\Gamma_\delta(\varrho ) 
\pm
\vartheta \vert \nabla V \vert^2 
\bigg]\,\mathrm{d}x \, \mathrm{d}t
+\nu^S
\int_0^T\psi \int_{{\mathbb{T}^3}}\vert  \nabla \mathbf{u}  \vert^2 \,\mathrm{d}x  \,\mathrm{d}x
\\&+(\nu^B+\nu^S)
\int_0^T \psi \int_{{\mathbb{T}^3}}\vert  \mathrm{div}\, \mathbf{u}  \vert^2 \,\mathrm{d}x  \,\mathrm{d}t
+\varepsilon
\int_0^T\psi \int_{{\mathbb{T}^3}} \varrho\vert  \nabla \mathbf{u}  \vert^2 \,\mathrm{d}x  \,\mathrm{d}x
\\&+\varepsilon 
\int_0^T \psi \int_{{\mathbb{T}^3}}(P{_\delta^\Gamma}) ''(\varrho)\vert  \nabla \varrho  \vert^2 \,\mathrm{d}x  \,\mathrm{d}t
=
\psi(0)
 \int_{{\mathbb{T}^3}}\bigg[\frac{1}{2} \varrho _0\vert \mathbf{u} _0 \vert^2  
+ P^\Gamma_\delta(\varrho _0) 
\pm
\vartheta \vert \nabla V_0 \vert^2
 \bigg]\,\mathrm{d}x
\\&
+\frac{1}{2}\int_0^T \psi
 \int_{{\mathbb{T}^3}}
 \varrho
 \sum_{k\in\mathbb{N}} \big\vert \Pi_N
 \mathbf{g}_{k, \varepsilon}(x,\varrho ,f,\mathbf{m}  ) 
\big\vert^2\,\mathrm{d}x\,\mathrm{d}t
+\int_0^T \psi  \int_{{\mathbb{T}^3}}   \varrho\mathbf{u} \cdot  \Pi_N\mathbf{G}_\varepsilon(\varrho , f, \mathbf{m} )\,\mathrm{d}x\,\mathrm{d}W ;
\end{aligned}
\end{equation}
holds $\mathbb{P}$-a.s. for all $\psi \in C^\infty_c ([0,T))$.
\end{enumerate}
\end{definition}

We now state the main result of this section.
\begin{theorem}
\label{thm:mainSecondLayer}
Let $(\Omega, \mathscr{F}, (\mathscr{F}_t)_{t\geq0}, \mathbb{P})$ be a stochastic basis with a complete right-continuous filtration $(\mathscr{F}_t)_{t\geq0}$. Let $W$ be an $(\mathscr{F}_t)$-cylindrical Wiener process and let $\varrho_0 \in C^{2+\nu}_x$ for $\nu\in(0,1)$, $\mathbf{u}_0 \in H_N$, $f \in C^{\nu}_x$ be some $\mathbb{P}$-a.s. $\mathscr{F}_0$-measurable random variables such that
\begin{align}
\overline{\varrho} \geq  \varrho_0 \geq \underline{\varrho}>0
\end{align}
holds for some deterministic constants $\underline{\varrho}$ and $\overline{\varrho}$. Finally, assume that the following moment estimate
\begin{align}
\label{initialEnergySecondLayer}
\mathbb{E} \Bigg[  \int_{{\mathbb{T}^3}}\bigg[\frac{1}{2} \varrho _0\vert \mathbf{u} _0 \vert^2  
+ P^\Gamma_\delta(\varrho _0) 
\pm
\vartheta \vert \nabla V_0 \vert^2
 \bigg]\,\mathrm{d}x \bigg]^p \lesssim 1
\end{align}
holds for some $p\in(1,\infty)$.
Then there exists a unique pathwise solution $(\varrho, \mathbf{u}, V)$ of \eqref{contEqSecond1}--\eqref{elecFieldSecond1} in the sense of Definition \ref{def:pathSolSecondLayer}.
\end{theorem}
To proof Theorem \ref{thm:mainSecondLayer}, we first require uniform estimates in $R$ since we intend to pass to the limit $R\rightarrow \infty$ in \eqref{contEq1}--\eqref{elecField1} in order to gain \eqref{contEqSecond1}--\eqref{elecFieldSecond1}. As it turns out, besides uniform-in-$R$ estimates, these bounds are also independent of $N$.
To see this, we follow the proof of \cite[Proposition 4.2.3]{breit2018stoch}  and consider  the unique pathwise solution $(\varrho, \mathbf{u}, V)$ of \eqref{contEq1}--\eqref{elecField1} in the sense of Definition \ref{def:pathSolFiniteDim}. By using \eqref{projectionContinuity} and \eqref{stochCoeffBound1Exis0a}--\eqref{stochsummableConsta} combined with the continuous embedding $L^\infty_x\hookrightarrow L^{\Gamma'}_x$, we obtain  the following bound
\begin{equation}
\begin{aligned}
\label{noiseBoundSecodLayer1}
\sum_{k\in\mathbb{N}}  \int_{{\mathbb{T}^3}} \varrho
  \big\vert \Pi_N
 \mathbf{g}_{k, \varepsilon}(x,\varrho ,f,\mathbf{m}  ) 
\big\vert^2\,\mathrm{d}x
&\leq 
\Vert \varrho  \Vert_{L^\Gamma_x}
\sum_{k\in\mathbb{N}} \Vert\mathbf{g}_{k, \varepsilon}(x,\varrho ,f,\mathbf{m}  )  \Vert_{L^\infty_x}^2 
\lesssim_{p, k, \varepsilon, \Gamma}
\Vert \varrho  \Vert_{L^\Gamma_x}.
\end{aligned}
\end{equation}
%
%
uniformly in $N$ and $R$. Similarly, for $r=\frac{2\Gamma}{\Gamma-1}$, we gain by using \eqref{projectionContinuity}, \eqref{stochCoeffBound1Exis0a}--\eqref{stochsummableConsta}, the embedding $L^\infty_x\hookrightarrow L^r_x$ as well as Young's inequality
 \begin{equation}
\begin{aligned}
\label{noiseBoundSecodLayer2}
\bigg\vert\int_{{\mathbb{T}^3}}&   \varrho \mathbf{u}
\cdot \Pi_N
 \mathbf{g}_{k, \varepsilon}(x,\varrho ,f,\mathbf{m}  ) 
\mathrm{d}x \bigg\vert^2
\leq 
  \bigg\vert \Vert \sqrt{\varrho}  \Vert_{L^{2\Gamma}_x}
\Vert \sqrt{\varrho} \mathbf{u} \Vert_{L^2_x}
\Vert\mathbf{g}_{k, \varepsilon}(x,\varrho ,f,\mathbf{m}  )  \Vert_{L^\infty_x} \bigg\vert^2
\\& \times 
\lesssim_{ \Gamma} c^2_{k, \varepsilon}
\big\vert \Vert \sqrt{\varrho}  \Vert_{L^{2\Gamma}_x}
\Vert \sqrt{\varrho} \mathbf{u} \Vert_{L^2_x} \big\vert^2
\lesssim_{ \Gamma} c^2_{k, \varepsilon}\, \Big(
 \Vert \varrho  \Vert_{L^{\Gamma}_x}^2
+
\big\Vert \varrho \vert\mathbf{u} \vert^2 \big\Vert_{L^1_x}^2 \Big).
\end{aligned}
\end{equation}
It therefore follow from \eqref{noiseBoundSecodLayer2} and the Burkholder--Davis--Gundy inequality that the estimate
\begin{equation}
\begin{aligned}
\label{noiseBoundSecodLayer3}
&\mathbb{E} \Bigg[ \sup_{t\in[0,T]}\bigg\vert\int_0^t\int_{\mathbb{T}^3} \varrho\mathbf{u} \cdot  \Pi_N\mathbf{G}_\varepsilon(\varrho , f, \mathbf{m} )\,\mathrm{d}x\,\mathrm{d}W
 \bigg\vert^p \bigg]
\\&\lesssim
\mathbb{E} \Bigg[ \int_0^T \sum_{k\in \mathbb{N}} \bigg\vert \int_{\mathbb{T}^3} \varrho\mathbf{u} \cdot  \Pi_N\mathbf{g}_{k,\varepsilon}(x,\varrho , f, \mathbf{m} )\,\mathrm{d}x
 \bigg\vert^2 \,\mathrm{d}t \bigg]^{\frac{p}{2}}
\\&
\lesssim_{ k, \varepsilon,\Gamma}
\mathbb{E} \Bigg[\int_0^T\Big(
 \Vert \varrho  \Vert_{L^{\Gamma}_x}^2
+
\big\Vert \varrho \vert\mathbf{u} \vert^2 \big\Vert_{L^1_x}^2 \Big)\,\mathrm{d}t \bigg]^{\frac{p}{2}}
\end{aligned}
\end{equation}
is true for any $p\in(1, \infty)$.
Finally, it follows from Korn's inequality, \cite[Theorem A.1.8]{breit2018stoch} and Proposition \ref{energyEquality} combined with \eqref{noiseBoundSecodLayer1}, \eqref{noiseBoundSecodLayer3} and Gronwall's lemma that the estimate
\begin{equation}
\begin{aligned}
\label{noiseBoundSecodLayer4}
\mathbb{E}
&\bigg\vert \sup_{t\in [0,T]}  \int_{{\mathbb{T}^3}}\bigg[ \frac{1}{2} \varrho  \vert \mathbf{u}  \vert^2  
+ P^\Gamma_\delta(\varrho ) 
\pm
\vartheta \vert \nabla V \vert^2 
\bigg]\,\mathrm{d}x \bigg\vert^p
+
\mathbb{E} \bigg\vert
\int_0^T \int_{{\mathbb{T}^3}} \Big(\vert  \nabla \mathbf{u}  \vert^2 + \varepsilon\varrho\vert  \nabla \mathbf{u}  \vert^2 \\&+\varepsilon 
(P{_\delta^\Gamma}) ''(\varrho)\vert  \nabla \varrho  \vert^2 \Big) \,\mathrm{d}x  \,\mathrm{d}t \bigg\vert^p
\lesssim 
1 +
\mathbb{E}
\bigg[  \int_{{\mathbb{T}^3}}\bigg[ \frac{1}{2} \varrho_0  \vert \mathbf{u}_0  \vert^2  
+ P^\Gamma_\delta(\varrho_0 ) 
\pm
\vartheta \vert \nabla V_0 \vert^2 
\bigg]\,\mathrm{d}x \bigg]^p
\end{aligned}
\end{equation}
holds uniformly in $R$ and $N$ for every $p\in (1, \infty)$. Uniform boundedness thus follow from \eqref{initialEnergySecondLayer}.

The completion of the proof of Theorem \ref{thm:mainSecondLayer} will now rely on a stopping time argument. Denote the unique pathwise solution from Theorem \ref{thm:mainFirstLayer} by $(\varrho_R, \mathbf{u}_R, V_R)$ and assume that it has data $(\varrho_0, \mathbf{u}_0, f)$ satisfying the assumptions in Theorem \ref{thm:mainSecondLayer}. Define the increasing family $(\tau_R)_{R>0}$ of $(\mathscr{F}_t)$-stopping times by
\begin{align*}
\tau_R := \inf_{t\in [0,T]}\Big\{ \Vert\mathbf{u}_R(t) \Vert_{H_N}>R \Big\}
\end{align*}
where we set $\inf \emptyset = T$. By uniqueness, as given by Theorem \ref{thm:mainFirstLayer}, $\tau_{R_0} \leq \tau_R$ if ${R_0} \leq R$ and also, $(\varrho_{R_0}, \mathbf{u}_{R_0}, V_{R_0})=(\varrho_{R}, \mathbf{u}_{R}, V_{R}) $
solves \eqref{contEqSecond1}--\eqref{elecFieldSecond1} on $(0, \tau_R)$ in the sense of Definition \ref{def:pathSolSecondLayer}. In particular, note that the energy equality remains unchanged from Proposition \ref{energyEquality}, at least on $(0, \tau_R)$.
\\
Now since $(\varrho_{R}, \mathbf{u}_{R}, V_{R}) $ satisfies an energy estimate uniformly in $R$, the proof is done once
\begin{align}
\label{tauRT}
\mathbb{P}\bigg( \sup_{R} \tau_R =   T \bigg)=1
\end{align}
is shown. The proof of \eqref{tauRT} is exactly the same as was shown in the proof of \cite[Proposition 4.2.3]{breit2018stoch}.

\section{The third approximation layer}
\label{sec:thirdLayer}
In this section, we establish a class of solution to the following system 
\begin{align}
&\mathrm{d} \varrho  + \mathrm{div}(\varrho  \mathbf{u})  \, \mathrm{d}t =\varepsilon \Delta \varrho\, \mathrm{d}t, 
\label{contEqThird1}
\\
&\mathrm{d}(\varrho  \mathbf{u} ) +
\big[
\mathrm{div} (\varrho  \mathbf{u}  \otimes \mathbf{u} ) +  \nabla   p_{\delta}^{ \Gamma}(\varrho) \big]\, \mathrm{d}t 
=
\big[\varepsilon\Delta(\varrho\mathbf{u})+ \nu^S \Delta \mathbf{u} 
\nonumber\\&
\quad +(\nu^B + \nu^S)\nabla \mathrm{div} \mathbf{u} + \vartheta \varrho  \nabla V  \big] \, \mathrm{d}t   
+ \sum_{k\in \mathbb{N}} \varrho\,   \mathbf{g}_{k,\varepsilon}(\varrho , f, \varrho \mathbf{u} ) \, \mathrm{d}\beta_k ,
\label{momEqThird1}
\\
&\pm \Delta V  = \varrho  - f, \label{elecFieldThird1}
\end{align}
which dissipates energy by passing to the limit $N\rightarrow \infty$ in \eqref{momEqSecond1}. The precise notion of a solution is as follows.
\begin{definition}
\label{def:martSolThirdLayer}
If $\Lambda$ is a Borel probability measure on $C^{2+\nu}_x \times L^1_x \times C^{\nu}_x$. We say that 
\begin{align}
\big[(\Omega ,\mathscr{F} ,(\mathscr{F} _t)_{t\geq0},\mathbb{P} );\varrho , \mathbf{u} , V , W   \big]
\end{align}
is a \textit{dissipative martingale solution} of \eqref{contEqThird1}--\eqref{elecFieldThird1} with initial law $\Lambda$ provided
\begin{enumerate}
\item $(\Omega,\mathscr{F},(\mathscr{F}_t),\mathbb{P})$ is a stochastic basis with a complete right-continuous filtration;
\item $W$ is a $(\mathscr{F}_t)$-cylindrical Wiener process;
\item the density $\varrho \in C_w \big( [0,T]; C^{2+\nu}_x\big) $, it is $(\mathscr{F}_t)$-adapted  and 
$\varrho>0$ $\mathbb{P}$-a.s.;
\item the velocity field $\mathbf{u} \in L^2 \big( 0,T; W^{1,2}_x\big) $ $\mathbb{P}$-a.s.  is an $(\mathscr{F}_t)$-adapted random distribution;
\item given $f\in C^\nu_x$, there exists $\mathscr{F}_0$-measurable random variables $(\varrho_0, \mathbf{u}_0)$ such that $\Lambda = \mathbb{P}\circ (\varrho_0, \varrho_0 \mathbf{u}_0, f)^{-1}$;
\item we have that $\varrho(0)=\varrho_0$ $\mathbb{P}$-a.s. and \eqref{contEqThird1} holds $\mathbb{P}$-a.s. for a.e.  $(t,x) \in(0,T)\times \mathbb{T}^3$;
\item equation \eqref{elecFieldThird1} holds $\mathbb{P}$-a.s. for a.e.  $(t,x) \in(0,T)\times \mathbb{T}^3$;
\item for all $\psi \in C^\infty_c ([0,T))$ and $\bm{\phi} \in C^\infty_x$, the following
\begin{equation}
\begin{aligned}
\label{weakMomEqThirdLayer}
&-\int_0^T \partial_t \psi \int_{{\mathbb{T}^3}} \varrho  \mathbf{u} (t) \cdot \bm{\phi} \, \mathrm{d}x    \mathrm{d}t
=
\psi(0)  \int_{{\mathbb{T}^3}} \varrho _0 \mathbf{u} _0 \cdot \bm{\phi} \, \mathrm{d}x  
+  \int_0^T \psi \int_{{\mathbb{T}^3}}  \varrho    \mathbf{u} \otimes \mathbf{u} : \nabla \bm{\phi} \, \mathrm{d}x  \mathrm{d}t 
\\&-
 \nu^S\int_0^T \psi \int_{{\mathbb{T}^3}} \nabla \mathbf{u} : \nabla \bm{\phi} \, \mathrm{d}x  \mathrm{d}t 
-(\nu^B+ \nu^S )\int_0^T \psi \int_{{\mathbb{T}^3}} \mathrm{div}\,\mathbf{u} \, \mathrm{div}\, \bm{\phi} \, \mathrm{d}x  \mathrm{d}t    
\\&+   
\int_0^T \psi \int_{{\mathbb{T}^3}}    p_{\delta}^{ \Gamma}\, \mathrm{div} \, \bm{\phi} \, \mathrm{d}x  \mathrm{d}t
+\int_0^T \psi \int_{{\mathbb{T}^3}} \varepsilon\varrho\mathbf{u}    \Delta\bm{\phi} \, \mathrm{d}x   \mathrm{d}t
\\&+\int_0^T \psi \int_{{\mathbb{T}^3}} \vartheta\varrho  \nabla V  \cdot \bm{\phi} \, \mathrm{d}x    \mathrm{d}t
+
\int_0^T \psi \sum_{k\in \mathbb{N}} \int_{{\mathbb{T}^3}}   \varrho\, \mathbf{g}_{k,\varepsilon}(\varrho , f, \varrho \mathbf{u} ) \cdot \bm{\phi}  \, \mathrm{d}x \mathrm{d}\beta_k  
\end{aligned}
\end{equation}
hold $\mathbb{P}$-a.s;
\item the energy inequality
\begin{equation}
\begin{aligned}
\label{energyInequalityThirdLayer}
-
&\int_0^T \partial_t \psi \int_{{\mathbb{T}^3}}\bigg[ \frac{1}{2} \varrho  \vert \mathbf{u}  \vert^2  
+ P^\Gamma_\delta(\varrho ) 
\pm
\vartheta \vert \nabla V \vert^2 
\bigg]\,\mathrm{d}x \, \mathrm{d}t
+\nu^S
\int_0^T\psi \int_{{\mathbb{T}^3}}\vert  \nabla \mathbf{u}  \vert^2 \,\mathrm{d}x  \,\mathrm{d}x
\\&+(\nu^B+\nu^S)
\int_0^T \psi \int_{{\mathbb{T}^3}}\vert  \mathrm{div}\, \mathbf{u}  \vert^2 \,\mathrm{d}x  \,\mathrm{d}t
+\varepsilon
\int_0^T\psi \int_{{\mathbb{T}^3}} \varrho\vert  \nabla \mathbf{u}  \vert^2 \,\mathrm{d}x  \,\mathrm{d}x
\\&+\varepsilon 
\int_0^T \psi \int_{{\mathbb{T}^3}}(P{_\delta^\Gamma}) ''(\varrho)\vert  \nabla \varrho  \vert^2 \,\mathrm{d}x  \,\mathrm{d}t
\leq
\psi(0)
 \int_{{\mathbb{T}^3}}\bigg[\frac{1}{2} \varrho _0\vert \mathbf{u} _0 \vert^2  
+ P^\Gamma_\delta(\varrho _0) 
\pm
\vartheta \vert \nabla V_0 \vert^2
 \bigg]\,\mathrm{d}x
\\&
+\frac{1}{2}\int_0^T \psi
 \int_{{\mathbb{T}^3}}
 \varrho
 \sum_{k\in\mathbb{N}} \big\vert 
 \mathbf{g}_{k, \varepsilon}(x,\varrho ,f,\mathbf{m}  ) 
\big\vert^2\,\mathrm{d}x\,\mathrm{d}t
+\int_0^T \psi  \int_{{\mathbb{T}^3}}   \varrho\mathbf{u} \cdot  \mathbf{G}_\varepsilon(\varrho , f, \mathbf{m} )\,\mathrm{d}x\,\mathrm{d}W ;
\end{aligned}
\end{equation}
holds $\mathbb{P}$-a.s. for all $\psi \in C^\infty_c ([0,T))$, $\psi \geq 0$.
\end{enumerate}
\end{definition} 
The main theorem of this section is the following.
\begin{theorem}
\label{thm:mainThirdLayer}
Let $\Lambda$ be a Borel probability measure on $C^{2+\nu}_x \times L^1_x \times C^{\nu}_x$ such that
\begin{align}
\label{lawMainThmThirdLayer}
\Lambda\bigg\{M \leq  \varrho \leq M^{-1}, \quad f \leq \overline{f} \bigg\} =1, \quad
\end{align}
holds for  deterministic constants $M, \overline{f}>0$. Also assume that
\begin{equation}
\begin{aligned}
\label{momEstMainThmThirdLayer}
\int_{C^{2+\nu}_x \times L^1_x \times C^{\nu}} &\bigg(\bigg\vert  
\int_{\mathbb{T}^3} \bigg[ \frac{\vert \mathbf{m} \vert^2}{2\varrho} + P^\Gamma_\delta(\varrho)  \pm 
\vartheta \vert \nabla V \vert^2
\bigg] \, \mathrm{d}x
\bigg\vert^p  
\\&+ \Vert \varrho \Vert^p_{C^{2+\nu}_x} + \Vert f \Vert^p_{C^{\nu}_x} \bigg) \, \mathrm{d}\Lambda(\varrho, \mathbf{m},f) \lesssim 1
\end{aligned}
\end{equation}
holds for some $p\geq 1$ .
Then there exists a dissipative martingale solution of \eqref{contEqThird1}--\eqref{elecFieldThird1} in the sense of Definition \ref{def:martSolThirdLayer}.
\end{theorem}
We now devote the rest of this section to the proof of Theorem \ref{thm:mainThirdLayer}. 
\subsection{Construction of law}
Given $f\in C^\nu_x$, consider the $(\mathfrak{F}_0)$-measurable random variables $(\varrho_0, \mathbf{m}_0)$ ranging in $C^{2+\nu}_x \times L^1_x $ so that the triplet $(\varrho_0, \mathbf{m}_0,f)$ has the law $\Lambda$ as prescribed in the assumption of Theorem \ref{thm:mainThirdLayer}. Since $\varrho_0>0$, it follows that $\mathbf{u}_0:= \frac{\mathbf{m}_0}{\varrho_0} \in L^2_x$ $\mathbb{P}$-a.s. Furthermore, by setting $\mathbf{u}_{0,N}:= \Pi_N \mathbf{u}_0$, it follows that $(\varrho_0, \mathbf{u}_{0,N},f)$ satisfies the assumptions of Theorem \ref{thm:mainSecondLayer} and for $\Lambda_N= \mathbb{P}\circ (\varrho_0,  \mathbf{u}_{0,N}, f)^{-1}$, 
\begin{equation}
\begin{aligned}
\label{somethingBd}
&\int_{C^{2+\nu}_x \times L^2_x \times C^{\nu}} \bigg(\bigg\vert  
\int_{\mathbb{T}^3} \bigg[ \frac{1}{2}\varrho\vert \mathbf{u} \vert^2 + P^\Gamma_\delta(\varrho)  \pm 
\vartheta \vert \nabla V \vert^2
\bigg] \, \mathrm{d}x
\bigg\vert^p  
\\&+ \Vert \varrho \Vert^p_{C^{2+\nu}_x} + \Vert f \Vert^p_{C^{\nu}_x} \bigg) \, \mathrm{d}\Lambda_N(\varrho, \mathbf{u},f) 
 \lesssim 1
 \end{aligned}
\end{equation}
holds uniformly in $N$. Furthermore, we can use the bound in \eqref{projectionContinuity} to obtain
\begin{equation}
\begin{aligned}
&\int_{C^{2+\nu}_x \times L^2_x \times C^{\nu}} \bigg(\bigg\vert  
\int_{\mathbb{T}^3} \bigg[ \frac{1}{2}\varrho\vert \mathbf{u} \vert^2 + P^\Gamma_\delta(\varrho)  \pm 
\vartheta \vert \nabla V \vert^2
\bigg] \, \mathrm{d}x
\bigg\vert^p  
+ \Vert \varrho \Vert^p_{C^{2+\nu}_x} + \Vert f \Vert^p_{C^{\nu}_x} \bigg) \, \mathrm{d}\Lambda_N(\varrho, \mathbf{u},f) 
\\&
=
\mathbb{E}\bigg(\bigg\vert  
\int_{\mathbb{T}^3} \bigg[  \frac{1}{2}\varrho\vert \mathbf{u}_N \vert^2 + P^\Gamma_\delta(\varrho)  \pm 
\vartheta \vert \nabla V \vert^2
\bigg] \, \mathrm{d}x
\bigg\vert^p  
+ \Vert \varrho \Vert^p_{C^{2+\nu}_x} + \Vert f \Vert^p_{C^{\nu}_x} \bigg)
\longrightarrow
\\&
\mathbb{E}\bigg(\bigg\vert  
\int_{\mathbb{T}^3} \bigg[  \frac{1}{2}\varrho\vert \mathbf{u} \vert^2 + P^\Gamma_\delta(\varrho)  \pm 
\vartheta \vert \nabla V \vert^2
\bigg] \, \mathrm{d}x
\bigg\vert^p  
+ \Vert \varrho \Vert^p_{C^{2+\nu}_x} + \Vert f \Vert^p_{C^{\nu}_x} \bigg)
\\&
=\int_{C^{2+\nu}_x \times L^1_x \times C^{\nu}} \bigg(\bigg\vert  
\int_{\mathbb{T}^3} \bigg[ \frac{\vert \mathbf{m} \vert^2}{2\varrho}\ + P^\Gamma_\delta(\varrho)  \pm 
\vartheta \vert \nabla V \vert^2
\bigg] \, \mathrm{d}x
\bigg\vert^p  
+ \Vert \varrho \Vert^p_{C^{2+\nu}_x} + \Vert f \Vert^p_{C^{\nu}_x} \bigg) \, \mathrm{d}\Lambda(\varrho, \mathbf{m},f) .
\end{aligned}
\end{equation}
\subsection{Uniform estimates}
Now denote by $
[\varrho_N , \mathbf{u}_N , V_N ]$, the  pathwise solution constructed in Theorem \ref{thm:mainSecondLayer} with data $(\varrho, \mathbf{u}_{0,N},f)$. Since \eqref{somethingBd} holds uniformly in $N$ (and clearly of $R$), the argument that the estimate \eqref{noiseBoundSecodLayer4} held uniformly in $R$ and $N$ remains valid for $
[\varrho_N , \mathbf{u}_N , V_N ]$. It follows that
\begin{equation}
\begin{aligned}
\label{uniBoundThirdLayer1}
&\mathbb{E} \bigg\vert \sup_{t\in [0,T]}\big\Vert \varrho_N \vert \mathbf{u}_N\vert^2 \big\Vert_{L^1_x} \bigg\vert^p 
+
\mathbb{E} \bigg\vert \sup_{t\in [0,T]} \Vert \varrho_N  \Vert_{L^\Gamma_x}^\Gamma \bigg\vert^p 
+
\mathbb{E} \bigg\vert \sup_{t\in [0,T]} \Vert \nabla V_N \Vert_{L^2_x}^2 \bigg\vert^p
+
\mathbb{E} \Vert \nabla \mathbf{u}_N \Vert_{L^2_{t,x}}^{2p}  
\\&+
\mathbb{E} \big\Vert \varrho^{-\frac{1}{2}}_N \nabla\varrho_N \big\Vert_{L^2_{t,x}}^{2p}
+
\mathbb{E} \big\Vert \varrho^{\frac{\Gamma-2}{2}}_N \nabla\varrho_N \big\Vert_{L^2_{t,x}}^{2p}
+
\mathbb{E} \Vert  \mathbf{u}_N \Vert_{L^2_{t}W^{1,2}_x}^{2p}  
+
\mathbb{E} \bigg\vert \sup_{t\in [0,T]}\big\Vert \varrho_N \mathbf{u}_N \big\Vert_{L^\frac{2\Gamma}{\Gamma+1}_x}^\frac{2\Gamma}{\Gamma+1} \bigg\vert^p 
\\&+
\mathbb{E} \Vert \varrho_N \mathbf{u}_N\otimes \mathbf{u}_N \Vert_{L^2_{t}L^\frac{6\Gamma}{4\Gamma+3}_x}^{\frac{6p\Gamma}{4\Gamma+3}}  
+
\mathbb{E} \Vert \varrho_N \mathbf{u}_N \Vert_{L^1_tL^3_x}^{3p} 
+
\mathbb{E} \Vert \varrho_N \mathbf{u}_N \Vert_{L^q_{t,x}}^{pq} 
+
\mathbb{E} \Vert\mathrm{div}( \varrho_N \mathbf{u}_N) \Vert_{L^q_t W^{-1,q}_x}^{pq} 
\\&\lesssim_{\varepsilon, \delta, \Gamma, p,\mathcal{E}_{0,N}} 1
\end{aligned}
\end{equation}
holds uniformly in $N$ for some $q>2$ and any $p\in(1,\infty)$ where 
\begin{align*}
 \mathcal{E}_{0,N}
 := 
 \frac{1}{2} \varrho_0  \vert \mathbf{u}_{0,N}  \vert^2  
+ P^\Gamma_\delta(\varrho_0 ) 
\pm
\vartheta \vert \nabla V_0 \vert^2 
\end{align*}
is uniformly bounded in $N$ in $L^p(\Omega; L^1(\mathbb{T}^3))$ by \eqref{somethingBd}.
Furthermore, by using the maximal regularity property of the inhomogeneous continuity equation \eqref{contEqThird1}, see \cite[Theorem A.2.2]{breit2018stoch} or \cite{hieber1997heat}, it follows from \eqref{contEqThird1} that the inequality
\begin{equation}
\begin{aligned}
\sup_{t\in [0,T]} \Vert \varrho_N \Vert_{X}&+ \Vert \varrho_N \Vert_{L^{r_t}_t W^{2+l, r_x}_x} +
\Vert \partial_t \varrho_N \Vert_{L^{r_t}_t W^{l, r_x}_x}
\lesssim
\Vert \mathrm{div}(\varrho_N \mathbf{u}_N) \Vert_{L^{r_t}_t W^{l, r_x}_x} + \Vert \varrho_0 \Vert_X
\end{aligned}
\end{equation}
holds $\mathbb{P}$-a.s. for some $l\in \mathbb{R}$,  and $1< r_t, r_x < \infty$ (or for $r_t=1$ with $r_x=\infty$) where
$X:= \big[  W^{l, r_x}_x,  W^{2+l, r_x}_x\big]_{1-r_t^{-1}, r_t}$ is the interpolation space, see \cite{tartar2007introduction} for definition. So by combining this maximal regularity with \eqref{uniBoundThirdLayer1} and the assumption that the initial density is highly regular, we get in particular that the estimate
\begin{equation}
\begin{aligned}
\label{uniBoundThirdLayer2}
\mathbb{E} \Vert \varrho_N \Vert_{L^q_t W^{2, q}_x}^p +
\mathbb{E}
\Vert \partial_t \varrho_N \Vert_{L^q_t L^{ q}_x}^p
\lesssim_{\varepsilon, \delta, \Gamma, p,\mathcal{E}_{0,N}}
1 
\end{aligned}
\end{equation}
holds uniformly in $N$ for some $q>2$.
\subsection{Compactness}
To derive compactness, we first define the following spaces 
\begin{align*}
\chi_{\varrho_0} = C_x, \, \quad \chi_{\mathbf{u}_0}= L^2_x,
 \, \quad
\chi_{ \mathbf{u} } =  \big(L^2\big(0,T;W^{1,2}_x\big), w\big),
 \, \quad
\chi_V = C_w\left([0,T];W^{2,\Gamma}_x\right),
\\
\chi_W= C\left([0,T];\mathfrak{U}_0\right) ,
 \, \quad
\chi_{\varrho\mathbf{u}} =C_w\big([0,T]; L^\frac{2\Gamma}{\Gamma+1} _x\big) \cap C\big([0,T]; W^{-k,2}_x\big), 
\\
\chi_\varrho = C_w \left([0,T];L^{\Gamma}_x\right)
\cap
\big(L^q\big(0,T;W^{2,q}_x\big), w\big)
\cap \big( W^{1,q}\big(0,T;L^q_x\big), w \big) \cap \, L^p\big(0,T;W^{1,p}_x\big) 
\end{align*}
for some $p>2$, $q>1$ and $k\in \mathbb{N}$. We now let $\mu_{\mathbf{u}_{0,N}}$, $\mu_{\varrho_0}$,  $\mu_{\mathbf{u}_N}$, $\mu_{\varrho_N}$, $\mu_{\Pi_N\varrho_N\mathbf{u}_N}$, $\mu_{V_N}$ and $\mu_{W}$ be the respective laws of  $\mathbf{u}_{0,N}$, $\varrho_0$,  $\mathbf{u}_N$, $\varrho_N$, $\Pi_N\varrho_N\mathbf{u}_N$, $V_N$ and $W$ on the respective spaces $\chi_{\mathbf{u}_0} $, $\chi_{\varrho_0} $,  $\chi_{ \mathbf{u} }$, $\chi_\varrho$, $\chi_{\varrho\mathbf{u}}$, $\chi_V$ and $\chi_W$. Furthermore, we set $\mu_N$ as their joint law on the space $\chi = \chi_{\mathbf{u}_0} \times \chi_{\varrho_0} \times \chi_{ \mathbf{u} } \times \chi_\varrho \times \chi_{\varrho\mathbf{u}} \times \chi_V \times \chi_W$.
\begin{lemma}
\label{tighnessThirdLAyer}
The set $\{\mu_N :  N\in \mathbb{N} \}$  is tight on $\chi$.
\end{lemma}
\begin{proof}
First of all, as shown in \cite[Corollary 4.3.9]{breit2018stoch}, the following holds true.
\begin{itemize}
\item The law $\mu_{\varrho_0}$ is tight on $\chi_{\varrho_0}$ since its a Radon measure on a Polish space.
\item The set $\{\mu_{\mathbf{u}_{0,N}}:  N\in \mathbb{N}\}$ is tight on $\chi_{\mathbf{u}_0}$.
\item The set $\{\mu_{\varrho_N}:  N\in \mathbb{N} \}$ is tight on $\chi_\varrho$.
\item The set $\{\mu_{\mathbf{u}_N}:  N\in \mathbb{N}\}$ is tight on $\chi_{\mathbf{u}}$.
\item The set $\{ \mu_W\}$ is tight on $\chi_W$ since its a Radon measure on a Polish space.
\end{itemize}
We now recall again that $[\varrho_N , \mathbf{u}_N , V_N ]$ satisfied \eqref{energyEqualitySecondLayer}
from which we concluded that it satisfied the estimates \eqref{uniBoundThirdLayer1}--\eqref{uniBoundThirdLayer2}. So in particular, $\varrho_N \in L^\infty(0,T;L^\gamma_x )$ and $\nabla V_N \in L^\infty(0,T;L^2_x)$ in moments and as such, $\varrho_N \nabla V_N \in L^\infty(0,T;L^\frac{2\gamma}{\gamma+1}_x)$. Since the embedding
\begin{align*}
L^\infty(0,T;L^\frac{2\gamma}{\gamma+1}_x) \hookrightarrow L^r(0,T; W^{-l,2}_x)
\end{align*}
is continuous for all $r\in(1,\infty)$ and for all $l \geq \frac{3}{2\gamma}$ but $\frac{3}{2\gamma} \in (0,1)$, indeed, the aforementioned embedding holds for all $r,l> 1$. It follows that
\begin{align*}
{\vartheta \varrho_N \nabla V_N} \quad \text{is bounded in} \quad L^p \big( \Omega; L^r(0,T; W^{-l,2}_x)\big)
\end{align*}
for all $p,r,l\in(1,\infty)$. It follows from \cite[Proposition 4.3.8]{breit2018stoch} applied to \eqref{momEqThird1} that
\begin{itemize}
\item The set $\{\mu_{\Pi_N \varrho_N\mathbf{u}_N}:  N\in \mathbb{N}\}$ is tight on $\chi_{\varrho\mathbf{u}}$.
\end{itemize}
The proof of Lemma \ref{tighnessThirdLAyer} is done once we show that
\begin{itemize}
\item The set $\{\mu_{V_N}:  N\in \mathbb{N} \}$ is tight on $\chi_V$.
\end{itemize}
To see this, we first note that since the given function $f(x)$ is continuous in $\mathbb{T}^3$, it is $p$-Lebesgue integrable for any $p\in [1,\infty)$. And since the moment estimates of $\varrho_N$ in $L^\infty_tL^\Gamma_x$ are uniformly bounded in $N$, recall the second summand in \eqref{uniBoundThirdLayer1}, we can deduce from the Poisson equation \eqref{elecFieldThird1} that any moment estimate of $V_N$ in $L^\infty_t W^{2,\Gamma}_x$ are uniformly bounded in $N$. Additionally, $V_N$ can inherit the following regularity of  the density sequence
\begin{align}
\label{inherit}
\mathbb{E}\, \Vert \varrho_N \Vert_{C^{0,1}_tW^{-2, \frac{2\Gamma}{\Gamma+1}}_x} \lesssim 1
\end{align}
uniformly in $N$, see \cite[Page 138]{breit2018stoch}, through \eqref{elecFieldThird1} and the fact that $f(x)$ is Lebesque integrable for any finite integrability exponent. Indeed, the spatial regularity in \eqref{inherit} can be improved for $V_N$ but this does not take anything away from the analysis. Tightness of $\mu_{V_N}$ will therefore follow from the following compact embedding
\begin{align}
L^\infty(0,T; W^{2,\Gamma}(\mathbb{T}^3)) \cap 
C^{0,1}([0,T];W^{-2, \frac{2\Gamma}{\Gamma+1}}(\mathbb{T}^3)) \hookrightarrow C_w([0,T];W^{2,\Gamma}(\mathbb{T}^3)) 
\end{align}
established in \cite[Corollary B.2.]{ondrejat2010stochastic}.
\end{proof}
We now use the Jakubowski--Skorokhod representation theorem to obtain the following result
\begin{lemma}
\label{lem:JakowThirdLayer}
The exists a subsequence (not relabelled) $\{\mu_N :  N \in \mathbb{N} \}$, a complete probability space $(\tilde{\Omega}, \tilde{\mathscr{F}}, \tilde{\mathbb{P}})$ with $\chi$-valued random variables
\begin{align*}
(\tilde{\varrho}_{0,N}, \tilde{\mathbf{u}}_{0,N},  \tilde{\varrho}_N, \tilde{\mathbf{u}}_N, \tilde{\mathbf{m}}_N ,\tilde{V}_N, \tilde{W}_N) \quad N \in \mathbb{N}
\end{align*}
and 
\begin{align*}
 (\tilde{\varrho}_0, \tilde{\mathbf{u}}_0,  \tilde{\varrho}, \tilde{\mathbf{u}}, \tilde{\mathbf{m}}, \tilde{V}, \tilde{W})
\end{align*}
such that 
\begin{itemize}
\item the law of $(\tilde{\varrho}_{0,N}, \tilde{\mathbf{u}}_{0,N},  \tilde{\varrho}_N, \tilde{\mathbf{u}}_N, \tilde{\mathbf{m}}_N ,\tilde{V}_N, \tilde{W}_N)$ on $\chi$ is $\mu_N$, $N \in \mathbb{N}$,
\item the law of $ (\tilde{\varrho}_0, \tilde{\mathbf{u}}_0,  \tilde{\varrho}, \tilde{\mathbf{u}}, \tilde{\mathbf{m}}, \tilde{V}, \tilde{W})$ on $\chi$ is a Radon measure,
\item the following convergence (with each $\rightarrow$ interpreted with respect to the corresponding topology)
\begin{align*}
\tilde{\varrho}_{0,N} \rightarrow \tilde{\varrho}_0 \quad \text{in}\quad \chi_{\varrho_0}, & \quad \qquad
\tilde{\mathbf{u}}_{0,N} \rightarrow \tilde{\mathbf{u}}_0 \quad \text{in}\quad \chi_{\mathbf{u}_0}, 
\\
\tilde{\varrho}_N \rightarrow \tilde{\varrho} \quad \text{in}\quad \chi_\varrho, & \quad \qquad
\tilde{\mathbf{u}}_N \rightarrow \tilde{\mathbf{u}} \quad \text{in}\quad \chi_\mathbf{u}, 
\\
\tilde{V}_N \rightarrow \tilde{V} \quad  \text{in}\quad \chi_V, & \quad \qquad
\tilde{\mathbf{m}}_N \rightarrow \tilde{\mathbf{m}} \quad \text{in}\quad \chi_\mathbf{m}, 
\\
\tilde{W}_N \rightarrow \tilde{W} \quad \text{in}\quad \chi_W
\end{align*}
holds $\tilde{
\mathbb{P}}$-a.s.
\end{itemize}
\end{lemma}
\subsection{Identifying the limit system}
\label{sec:limSystemThirdLayer}
Since our constructed velocity field $\tilde{\mathbf{u}}$ lives in a function space endowed with a weak topology, it follows that it is a random distribution in the sense of \cite[Definition 2.2.1]{breit2018stoch} rather than the classical notion of a stochastic process as one would usually expects. Loosely speaking, whereas a  stochastic process is defined pointwise for all times, our velocity field can only be interpreted in the PDE sense of distributions.
\\
We also have that $\tilde{\mathbf{u}}$ is adapted to its complete right-continuous history $(\sigma_t[\tilde{\mathbf{u}}])_{t\geq 0}$, see comments on \cite[Page 37]{breit2018stoch}. It follows from \cite[Lemma 2.2.18]{breit2018stoch} that there exist a classical stochastic process $\tilde{\mathbf{u}}$ (not relabelled) that coincides with our constructed velocity almost always in $\tilde{\Omega} \times [0,T]$, it is $\sigma_t[\tilde{\mathbf{u}}]$-progressively measurable and 
\begin{align*}
\tilde{\mathbf{u}} \in   L^2\big(0,T;W^{1,2}_x\big) \quad \tilde{\mathbb{P}}\text{-a.s.}
\end{align*}
Subsequently, we can build the following complete right-continuous filtration
\begin{align*}
\sigma \bigg(\sigma_t[\tilde{\varrho}_N], \sigma_t[\tilde{\mathbf{u}}_N], \sigma_t[\tilde{V}_N], \bigcup_{k\in \mathbb{N}}\sigma_t[\tilde{\beta}_{N,k}]  \bigg), \quad t\in [0,T]
\end{align*}
which is non-anticipative with respect to $\tilde{W}_N$. This ensures a well-defined stochastic integral, see \cite[Remark 2.3.7]{breit2018stoch},
and by Lemma \ref{lem:JakowThirdLayer} and  \cite[Lemma 2.9.3]{breit2018stoch}, allows the passage to the limit $N\rightarrow\infty$ to obtain the following filtration
\begin{align*}
\tilde{\mathscr{F}}_t  :=
\sigma \bigg(\sigma_t[\tilde{\varrho}], \sigma_t[\tilde{\mathbf{u}}], \sigma_t[\tilde{V}], \bigcup_{k\in \mathbb{N}}\sigma_t[\tilde{\beta}_{k}]  \bigg), \quad t\in [0,T]
\end{align*}
which is non-anticipative with respect to $\tilde{W}$.
\begin{lemma}
The following results
\begin{align}
&\tilde{\mathbf{m}} = \tilde{\varrho}\tilde{\mathbf{m}},
\\
&\tilde{\mathbf{m}}_N =  \Pi_N(\tilde{\varrho}_N \tilde{\mathbf{u}}_N),
\\
&\sqrt{\tilde{\varrho}_N} \tilde{\mathbf{u}}_N  \rightharpoonup
\sqrt{\tilde{\varrho}} \tilde{\mathbf{u}}  \quad \text{in} \quad L^2\big((0,T) \times \mathbb{T}^3 \big),
\\
&\tilde{\varrho}_N \tilde{\mathbf{u}}_N \otimes \tilde{\mathbf{u}}_N \rightharpoonup
\tilde{\varrho} \tilde{\mathbf{u}} \otimes \tilde{\mathbf{u}} \quad \text{in} \quad L^1\big((0,T) \times \mathbb{T}^3 \big)
\end{align}
holds $\tilde{\mathbb{P}}$a.s.
\end{lemma}
The proof of the above lemma is shown in \cite[Lemma 4.3.11 and Corollary 4.3.12]{breit2018stoch}.
Furthermore, due to the high enough regularity enjoyed by $[\tilde{\varrho}, \tilde{\mathbf{u}}]$, refer to the convergence results in Lemma \ref{lem:JakowThirdLayer}, we obtain the following result.
\begin{lemma}
\label{lem:identifyContThirdLayer}
The random distributions  $[\tilde{\varrho}, \tilde{\mathbf{u}}]$ satisfies \eqref{contEqThird1}  a.e. in $(0,T) \times \mathbb{T}^3$ $\mathbb{P}$-a.s.
\end{lemma}
Again, the high regularity enjoyed by $[\tilde{\varrho},  \tilde{V}]$ and the fact that the Poisson equation is linear results in the following.
\begin{lemma}
\label{lem:identifyPoisThirdLayer}
The random distributions  $[\tilde{\varrho},   \tilde{V}]$ satisfies  \eqref{elecFieldThird1} a.e. in $(0,T) \times \mathbb{T}^3$ $\mathbb{P}$-a.s.
\end{lemma}
Also, by relying on \cite[Theorem 2.9.1]{breit2018stoch} and Lemma \ref{lem:JakowThirdLayer}, we gain analogously to \cite[Proposition 4.3.14]{breit2018stoch}, the following result.
\begin{lemma}
\label{lem:identifyMomThirdLayer}
The random distributions  $[\tilde{\varrho}, \tilde{\mathbf{u}},  \tilde{V}]$ satisfies \eqref{weakMomEqThirdLayer} for all $\psi \in C^\infty_c ([0,T))$ and $\bm{\phi} \in C^\infty(\mathbb{T}^3)$ $\mathbb{P}$-a.s.
\end{lemma}
Note that since $f$ is a given, the proof of \cite[Proposition 4.3.14]{breit2018stoch} which emphasis convergence in the stochastic forcing term can be carried over, almost verbatim, to the proof of Lemma \ref{lem:identifyMomThirdLayer}. Finally, we are able to identify the energy with the help of the convergence results established in Lemma \ref{lem:JakowThirdLayer}. Again, note that these random variables are still regular enough to pass to the limit in the energy inequality for the sequence variables and that the noise term can be treated in the same vain as \cite[Proposition 4.3.15]{breit2018stoch}.
\begin{lemma}
\label{lem:identifyEnerThirdLayer}
The random distributions  $[\tilde{\varrho}, \tilde{\mathbf{u}},  \tilde{V}]$ satisfies \eqref{energyInequalityThirdLayer} for all $\psi \in C^\infty_c ([0,T))$, $\psi \geq0$ $\mathbb{P}$-a.s.
\end{lemma}
\section{The fourth approximation layer}
\label{sec:fourthLayer}
We devote this section to the  establishment of a class of solution to the following system 
\begin{align}
&\mathrm{d} \varrho  + \mathrm{div}(\varrho  \mathbf{u})  \, \mathrm{d}t =0,
\label{contEqFourth1}
\\
&\mathrm{d}(\varrho  \mathbf{u} ) +
\big[
\mathrm{div} (\varrho  \mathbf{u}  \otimes \mathbf{u} ) +  \nabla   p_{\delta}^{ \Gamma}(\varrho) \big]\, \mathrm{d}t 
=
\big[ \nu^S \Delta \mathbf{u} 
\nonumber\\&
\quad +(\nu^B + \nu^S)\nabla \mathrm{div} \mathbf{u} + \vartheta \varrho  \nabla V  \big] \, \mathrm{d}t   
+ \sum_{k\in \mathbb{N}} \varrho\,   \mathbf{g}_{k}(\varrho , f, \varrho \mathbf{u} ) \, \mathrm{d}\beta_k ,
\label{momEqFourth1}
\\
&\pm \Delta V  = \varrho  - f, \label{elecFieldFourth1}
\end{align}
which dissipates energy by passing to the limit $\varepsilon \rightarrow 0$ in \eqref{contEqThird1}--\eqref{momEqThird1}. The precise definition of this solution is as follows.
\begin{definition}
\label{def:martSolFourthLayer}
Let $\Lambda$ be a Borel probability measure on $\big[ L^1_x \big]^3$. We say that 
\begin{align}
\big[(\Omega ,\mathscr{F} ,(\mathscr{F} _t)_{t\geq0},\mathbb{P} );\varrho , \mathbf{u} , V , W   \big]
\end{align}
is a \textit{dissipative martingale solution} of \eqref{contEqFourth1}--\eqref{elecFieldFourth1} with initial law $\Lambda$ provided
\begin{enumerate}
\item $(\Omega,\mathscr{F},(\mathscr{F}_t),\mathbb{P})$ is a stochastic basis with a complete right-continuous filtration;
\item $W$ is a $(\mathscr{F}_t)$-cylindrical Wiener process;
\item the density $\varrho \in C_w \big( [0,T]; L^\gamma_x\big) $ is a random distribution, it is $(\mathscr{F}_t)$-adapted  and 
$\varrho>0$ $\mathbb{P}$-a.s.;
\item the velocity field $\mathbf{u} \in L^2 \big( 0,T; W^{1,2}_x\big) $ $\mathbb{P}$-a.s.  is an $(\mathscr{F}_t)$-adapted random distribution;
\item given $f\in L^\infty_x$, there exists $\mathscr{F}_0$-measurable random variables $(\varrho_0, \mathbf{m}_0)$ such that $\Lambda = \mathbb{P}\circ (\varrho_0, \varrho_0 \mathbf{u}_0, f)^{-1}$;
\item for all $\psi \in C^\infty_c ([0,T))$ and $\phi \in C^\infty ({\mathbb{T}^3})$, the following
\begin{equation}
\begin{aligned}
\label{weakContEqFourthLayer}
&-\int_0^T \partial_t \psi \int_{{\mathbb{T}^3}} \varrho (t) \phi \, \mathrm{d}x  \mathrm{d}t  =  \psi(0) \int_{{\mathbb{T}^3}} \varrho _0 \phi \,\mathrm{d}x  +  \int_0^T \psi \int_{{\mathbb{T}^3}} \varrho  \mathbf{u} \cdot \nabla \phi \, \mathrm{d}x  \mathrm{d}t,
\end{aligned}
\end{equation}
 $\mathbb{P}$-a.s.;
\item equation \eqref{elecFieldThird1} holds $\mathbb{P}$-a.s. for a.e.  $(t,x) \in(0,T)\times \mathbb{T}^3$;
\item for all $\psi \in C^\infty_c ([0,T))$ and $\bm{\phi} \in C^\infty(\mathbb{T}^3)$, the following
\begin{equation}
\begin{aligned}
\label{weakMomEqFourthLayer}
&-\int_0^T \partial_t \psi \int_{{\mathbb{T}^3}} \varrho  \mathbf{u} (t) \cdot \bm{\phi} \, \mathrm{d}x    \mathrm{d}t
=
\psi(0)  \int_{{\mathbb{T}^3}} \varrho _0 \mathbf{u} _0 \cdot \bm{\phi} \, \mathrm{d}x  
+  \int_0^T \psi \int_{{\mathbb{T}^3}}  \varrho    \mathbf{u} \otimes \mathbf{u} : \nabla \bm{\phi} \, \mathrm{d}x  \mathrm{d}t 
\\&-
 \nu^S\int_0^T \psi \int_{{\mathbb{T}^3}} \nabla \mathbf{u} : \nabla \bm{\phi}  \mathrm{d}x  \mathrm{d}t 
-(\nu^B+ \nu^S )\int_0^T \psi \int_{{\mathbb{T}^3}} \mathrm{div}\mathbf{u}  \mathrm{div}\, \bm{\phi} \mathrm{d}x  \mathrm{d}t    
+   
\int_0^T \psi \int_{{\mathbb{T}^3}}    p_{\delta}^{ \Gamma}\, \mathrm{div} \, \bm{\phi}  \mathrm{d}x  \mathrm{d}t
\\&+\int_0^T \psi \int_{{\mathbb{T}^3}} \vartheta\varrho  \nabla V  \cdot \bm{\phi}  \mathrm{d}x    \mathrm{d}t
+
\int_0^T \psi \sum_{k\in \mathbb{N}} \int_{{\mathbb{T}^3}}   \varrho\, \mathbf{g}_{k}(\varrho , f, \varrho \mathbf{u} ) \cdot \bm{\phi}  \, \mathrm{d}x \mathrm{d}\beta_k  
\end{aligned}
\end{equation}
hold $\mathbb{P}$-a.s;
\item equation \eqref{contEqFourth1} holds in the renormalized sense, i.e., for any $\phi\in C^\infty_c([0,T)\times{\mathbb{T}^3})$ and $b \in C^1_b(\mathbb{R})$ such that $ b'(z)=0$ for all $z\geq M_b$,  we have that
\begin{equation}
\begin{aligned}
\label{renormalizedContFourthLayer} 
-\int_0^T\int_{{\mathbb{T}^3}}  b(\varrho )\,\partial_t\phi \,\mathrm{d}x \, \mathrm{d}t &=
\int_{{\mathbb{T}^3}}  b\big(\varrho (0)\big)\,\phi(0) \,\mathrm{d}x 
+ \int_0^T\int_{{\mathbb{T}^3}} \big[b(\varrho ) \mathbf{u}  \big] \cdot \nabla\phi \, \mathrm{d}x \, \mathrm{d}t  
\\&- \int_0^T\int_{{\mathbb{T}^3}} \big[ \big(  b'(\varrho ) \varrho   
-
 b(\varrho )\big)\mathrm{div}\mathbf{u} \big]\,  \phi\, \mathrm{d}x \, \mathrm{d}t
\end{aligned}
\end{equation} 
holds $\mathbb{P}$-a.s.;
\item the energy inequality
\begin{equation}
\begin{aligned}
\label{energyInequalityFourthLayer}
-
\int_0^T& \partial_t \psi \int_{{\mathbb{T}^3}}\bigg[ \frac{1}{2} \varrho  \vert \mathbf{u}  \vert^2  
+ P^\Gamma_\delta(\varrho ) 
\pm
\vartheta \vert \nabla V \vert^2 
\bigg]\,\mathrm{d}x \, \mathrm{d}t
+\nu^S
\int_0^T\psi \int_{{\mathbb{T}^3}}\vert  \nabla \mathbf{u}  \vert^2 \,\mathrm{d}x  \,\mathrm{d}x
\\&+(\nu^B+\nu^S)
\int_0^T \psi \int_{{\mathbb{T}^3}}\vert  \mathrm{div}\, \mathbf{u}  \vert^2 \,\mathrm{d}x  \,\mathrm{d}t
\leq
\psi(0)
 \int_{{\mathbb{T}^3}}\bigg[\frac{1}{2} \varrho _0\vert \mathbf{u} _0 \vert^2  
+ P^\Gamma_\delta(\varrho _0) 
\pm
\vartheta \vert \nabla V_0 \vert^2
 \bigg]\,\mathrm{d}x
\\&
+\frac{1}{2}\int_0^T \psi
 \int_{{\mathbb{T}^3}}
 \varrho
 \sum_{k\in\mathbb{N}} \big\vert 
 \mathbf{g}_{k}(x,\varrho ,f,\mathbf{m}  ) 
\big\vert^2\,\mathrm{d}x\,\mathrm{d}t
+\int_0^T \psi  \int_{{\mathbb{T}^3}}   \varrho\mathbf{u} \cdot  \mathbf{G}(\varrho , f, \mathbf{m} )\,\mathrm{d}x\,\mathrm{d}W ;
\end{aligned}
\end{equation}
holds $\mathbb{P}$-a.s. for all $\psi \in C^\infty_c ([0,T))$, $\psi \geq0$.
\end{enumerate}
\end{definition} 
The main result in this section is the following.
\begin{theorem}
\label{thm:mainFourthLayer}
Let $\Gamma \geq 6$ and $\Lambda$ be a Borel probability measure on $\big[L^1_x \big]^3$ such that
\begin{align}
\label{lawMainThmFourthLayer}
\Lambda\bigg\{M \leq  \varrho \leq M^{-1}<\infty, \quad f \leq \overline{f}\bigg\} =1, \quad
\end{align}
holds for  deterministic constants $M, \overline{f}>0$. Also assume that
\begin{align}
\label{momEstMainThmFourthLayer}
\int_{[L^1_x ]^3} &\bigg\vert  
\int_{\mathbb{T}^3} \bigg[ \frac{\vert \mathbf{m} \vert^2}{2\varrho} + P^\Gamma_\delta(\varrho)  \pm 
\vartheta \vert \nabla V \vert^2
\bigg] \, \mathrm{d}x
\bigg\vert^p   \mathrm{d}\Lambda(\varrho, \mathbf{m},f) \lesssim 1
\end{align}
holds for some $p\geq 1$ .
Then the exists a dissipative martingale solution of \eqref{contEqFourth1}--\eqref{elecFieldFourth1} in the sense of Definition \ref{def:martSolFourthLayer}.
\end{theorem}

\subsection{Construction of law}
\label{sec:initialLawFourthLayer}
We now construct the law $\Lambda$ as described in Theorem \ref{thm:mainFourthLayer} through an approximation of laws $\Lambda_\varepsilon$ constructed from the previous section in Theorem \ref{thm:mainThirdLayer}. In order words, given $\Lambda_\varepsilon$ satisfying \eqref{lawMainThmThirdLayer}--\eqref{momEstMainThmThirdLayer} , we obtain \eqref{lawMainThmFourthLayer}--\eqref{momEstMainThmFourthLayer} uniformly in $\varepsilon$.
\\
Now since $[L^1_x ]^3$ is a Polish space, by \cite[Corollary 2.6.4]{breit2018stoch}, there exists a  stochastic basis $(\Omega ,\mathscr{F} ,(\mathscr{F} _t)_{t\geq0},\mathbb{P} )$ containing $\mathscr{F}_0$-measurable random variables $(\varrho_0, \mathbf{m}_0)$ so that given $f\in L^\infty_x$, the triplet $(\varrho_0, \mathbf{m}_0, f)$ has values in $[L^1_x]^3$ as well as having $\Lambda$ as their law. It follows in particular that 
\begin{align*}
f \leq \overline{f}
\end{align*}
holds $\mathbb{P}$-a.s. for a deterministic constant $\overline{f}>0$.  Also, one can find random variables $\varrho_{0,\varepsilon} \in C^{2+\nu}_x$ for some $\nu\in (0,1)$ such that
\begin{align*}
M \leq  \varrho_{0,\varepsilon} \leq M^{-1}, 
\end{align*}
hold $\mathbb{P}$-a.s. for a deterministic constant $M> 0$ and that
\begin{align}
\label{converApproInitiDem}
\varrho_{0,\varepsilon} \rightarrow \varrho_0 \quad\text{in}\quad L^{p_0}(\Omega;L^\Gamma_x) \quad \forall p_0 \in [1,p\Gamma].
\end{align}
If we also set $V_{0,\varepsilon} =V_0 =\pm \Delta^{-1}_{\mathbb{T}^3}(\varrho_0-f)$, then we have that
\begin{align}
\mathbb{E}\bigg[ \int_{\mathbb{T}^3} \vartheta \vert \nabla V_{0,\varepsilon} \vert^2 \, \mathrm{d}x \bigg]^{p_0} \lesssim 1
\end{align}
uniformly in $\delta$ for all $p_0\in [1,p]$ and that
\begin{align}
\nabla V_{0,\varepsilon} \rightarrow  \nabla V_{0} \quad &\text{in}\quad L^{p_0}(\Omega; L^2_x) \quad \text{for all}\quad  p_0\in [1,2p].
\end{align}
Furthermore, as in \cite[Page 154]{breit2018stoch}, one can find a $C^2_x$-valued random variable $h_\varepsilon$, such that for $\mathbf{m}_{0,\varepsilon}=h_\varepsilon \sqrt{\varrho_{0,\varepsilon}}$,  we have
\begin{align*}
\frac{\vert \mathbf{m}_{0,\varepsilon} \vert^2}{\varrho_{0,\varepsilon}} \in L^{p_0}(\Omega; L^1_x) \quad \forall p_0 \in [1,p]
\end{align*}
uniformly in $\varepsilon$ and with the help of \eqref{converApproInitiDem},
\begin{align}
\mathbf{m}_{0,\varepsilon} \rightarrow \mathbf{m}_0  &\quad\text{in} \quad L^{p_0}(\Omega; L^1_x) \quad \forall p_0 \in [1,p],
\label{converApproInitiMom}
\\
\frac{\mathbf{m}_{0,\varepsilon}}{\sqrt{\varrho_{0,\varepsilon}}} \rightarrow \frac{\mathbf{m}_0}{\sqrt{\varrho_0}}  &\quad\text{in} \quad L^{p_0}(\Omega; L^2_x) \quad \forall p_0 \in [1,2p].
\end{align}
We now set $\Lambda_\varepsilon=\mathbb{P}\circ (\varrho_{0,\varepsilon}, \mathbf{m}_{0,\varepsilon}, f)^{-1}$
where
$ f\leq \overline{f}
$
holds $\mathbb{P}$-a.s. for a deterministic constant $\overline{f}>0$
and from \eqref{converApproInitiDem} and \eqref{converApproInitiMom}, it follows that $\Lambda_\varepsilon \rightarrow \Lambda$ in the sense of measures in $[L^1_x]^3$. Note that $f \in L^p_x$ for all $p\in [1,\infty]$.

\subsection{Uniform estimates}
\label{sec:UniformEstFourthLayer}
From Section \ref{sec:initialLawFourthLayer}, it follows from Theorem \ref{thm:mainThirdLayer} that  for any $\varepsilon\in(0,1)$, there exists a dissipative martingale solution
\begin{align*}
\big[(\Omega ,\mathscr{F} ,(\mathscr{F} _t)_{t\geq0},\mathbb{P} );\varrho_\varepsilon , \mathbf{u}_\varepsilon , V_\varepsilon , W   \big]
\end{align*}
of \eqref{contEqThird1}--\eqref{elecFieldThird1} with law $\Lambda_\varepsilon$ in the sense of Definition \ref{def:martSolThirdLayer}. The justification of the choice of a single stochastic basis and a single Wiener process for these solutions is standard, see for instant, \cite[Remark 5.3.1]{mensah2019theses} or \cite[Remark 4.0.4]{breit2018stoch}.
\\
Since the solution satisfies the energy estimate \eqref{energyInequalityThirdLayer}, by approximating the characteristic function $\chi_{[0,s]}$ for any $s\in(0,T]$, by a sequence of test functions $\psi\in C^\infty_c([0,T))$, it follows that
\begin{equation}
\begin{aligned}
\label{energyInequalityFourthLayerNonDissi}
\int_{{\mathbb{T}^3}}&\bigg[ \frac{1}{2} \varrho_\varepsilon  \vert \mathbf{u}_\varepsilon  \vert^2  
+ P^\Gamma_\delta(\varrho_\varepsilon ) 
\pm
\vartheta \vert \nabla V_\varepsilon \vert^2 
\bigg](s)\,\mathrm{d}x 
+\nu^S
\int_0^s \int_{{\mathbb{T}^3}}\vert  \nabla \mathbf{u}_\varepsilon  \vert^2 \,\mathrm{d}x  \,\mathrm{d}x
\\&+(\nu^B+\nu^S)
\int_0^s \int_{{\mathbb{T}^3}}\vert  \mathrm{div}\, \mathbf{u}_\varepsilon  \vert^2 \,\mathrm{d}x  \,\mathrm{d}t
+\varepsilon
\int_0^s\int_{{\mathbb{T}^3}} \varrho_\varepsilon\vert  \nabla \mathbf{u}_\varepsilon  \vert^2 \,\mathrm{d}x  \,\mathrm{d}x
\\&+\varepsilon 
\int_0^s \int_{{\mathbb{T}^3}}(P{_\delta^\Gamma}) ''(\varrho_\varepsilon)\vert  \nabla \varrho_\varepsilon  \vert^2 \,\mathrm{d}x  \,\mathrm{d}t
\leq
 \int_{{\mathbb{T}^3}}\bigg[\frac{1}{2} \varrho_{0,\varepsilon }\vert \mathbf{u}_{0,\varepsilon} \vert^2  
+ P^\Gamma_\delta(\varrho_{0,\varepsilon}) 
\pm
\vartheta \vert \nabla V_{0,\varepsilon} \vert^2
 \bigg]\,\mathrm{d}x
\\&
+\frac{1}{2}\int_0^s \int_{{\mathbb{T}^3}}
 \varrho_\varepsilon
 \sum_{k\in\mathbb{N}} \big\vert 
 \mathbf{g}_{k, \varepsilon}(x,\varrho_\varepsilon ,f,\mathbf{m}_\varepsilon  ) 
\big\vert^2\,\mathrm{d}x\,\mathrm{d}t
+\int_0^s  \int_{{\mathbb{T}^3}}   \varrho_\varepsilon\mathbf{u}_\varepsilon \cdot  \mathbf{G}_\varepsilon(\varrho_\varepsilon , f, \mathbf{m}_\varepsilon )\,\mathrm{d}x\,\mathrm{d}W 
\end{aligned}
\end{equation}
holds $\mathbb{P}$-a.s. for a.e. $s\in(0,T]$.
 Similar to \eqref{noiseBoundSecodLayer1}--\eqref{noiseBoundSecodLayer3}, we can use \eqref{stochCoeffBound1Exis0}--\eqref{stochsummableConst}, the fact that the cut-off function \eqref{cuttOff} is uniformly bounded  by one   to obtain  the bound
\begin{equation}
\begin{aligned}
\label{noiseBoundFourthLayer1}
&\mathbb{E} \Bigg[ \sup_{s\in[0,T]}\bigg\vert\int_0^s\sum_{k\in\mathbb{N}}  \int_{{\mathbb{T}^3}} \varrho_\varepsilon
  \big\vert
 \mathbf{g}_{k, \varepsilon}(x,\varrho_\varepsilon ,f,\mathbf{m}_\varepsilon  ) 
\big\vert^2\,\mathrm{d}x\,\mathrm{d}t
 \bigg\vert^p \bigg]
\lesssim_{p, k,  \Gamma}
\mathbb{E} \Bigg[\int_0^T\Vert \varrho _\varepsilon \Vert_{L^\Gamma_x} \,\mathrm{d}t \bigg]^p
\end{aligned}
\end{equation}
and together with the Burkholder--Davis--Gundy inequality, also the bound
\begin{equation}
\begin{aligned}
\label{noiseBoundFourthLayer3}
&\mathbb{E} \Bigg[ \sup_{t\in[0,T]}\bigg\vert\int_0^t\int_{\mathbb{T}^3} \varrho_\varepsilon \mathbf{u}_\varepsilon \cdot  \mathbf{G}_\varepsilon(\varrho_\varepsilon , f, \mathbf{m}_\varepsilon )\,\mathrm{d}x\,\mathrm{d}W
 \bigg\vert^p \bigg]
\\&
\lesssim_{ k, \Gamma}
\mathbb{E} \Bigg[\int_0^T\Big(
 \Vert \varrho_\varepsilon  \Vert_{L^{\Gamma}_x}^2
+
\big\Vert \varrho_\varepsilon \vert\mathbf{u}_\varepsilon \vert^2 \big\Vert_{L^1_x}^2 \Big)\,\mathrm{d}t \bigg]^{\frac{p}{2}}
\end{aligned}
\end{equation}
uniformly in $\varepsilon$ for any $p\in(1, \infty)$. By taking moments in \eqref{energyInequalityFourthLayerNonDissi}, it therefore follow from Gronwall's lemma that the exact same estimate \eqref{noiseBoundSecodLayer4} still holds true except that additionally, it now holds uniformly in $\varepsilon$ for any $p\in(1, \infty)$. So as in \eqref{uniBoundThirdLayer1}, we gain from \eqref{energyInequalityFourthLayerNonDissi}, the following crucial estimates
\begin{equation}
\begin{aligned}
\label{uniBoundFourthLayer1}
&\mathbb{E} \bigg\vert \sup_{t\in [0,T]}\big\Vert \varrho_\varepsilon \vert \mathbf{u}_\varepsilon\vert^2 \big\Vert_{L^1_x} \bigg\vert^p 
+
\mathbb{E} \bigg\vert \sup_{t\in [0,T]} \Vert \varrho_\varepsilon  \Vert_{L^\Gamma_x}^\Gamma \bigg\vert^p 
+
\mathbb{E} \bigg\vert \sup_{t\in [0,T]} \Vert \nabla V_\varepsilon \Vert_{L^2_x}^2 \bigg\vert^p
+
\mathbb{E} \bigg\vert \sup_{t\in [0,T]} \Vert \Delta V_\varepsilon \Vert_{L^\Gamma_x}^\Gamma \bigg\vert^p 
\\&+
\mathbb{E} \big\Vert \sqrt{\varepsilon} \nabla\varrho_\varepsilon \big\Vert_{L^2_{t,x}}^{2p}
+
\mathbb{E} \Vert  \mathbf{u}_\varepsilon \Vert_{L^2_{t}W^{1,2}_x}^{2p}  
+
\mathbb{E} \bigg\vert \sup_{t\in [0,T]}\big\Vert \varrho_\varepsilon \mathbf{u}_\varepsilon \big\Vert_{L^\frac{2\Gamma}{\Gamma+1}_x}^\frac{2\Gamma}{\Gamma+1} \bigg\vert^p 
\lesssim_{ \delta, \Gamma, p,\mathcal{E}_{0,\varepsilon}} 1
\end{aligned}
\end{equation}
which holds uniformly in $\varepsilon$ where 
\begin{align*}
\mathcal{E}_{0, \varepsilon} :=\frac{1}{2} \varrho_0  \vert \mathbf{u}_0  \vert^2  
+ P^\Gamma_\delta(\varrho_0 ) 
\pm
\vartheta \vert \nabla V_0 \vert^2
\end{align*}
is $\varepsilon$-uniformly bounded in $L^p(\Omega ; L^1_x)$, c.f. Section \ref{sec:initialLawFourthLayer}. 
In addition, by H\"older inequality and Young's inequality, we have that
\begin{equation}
\begin{aligned}
\bigg( \int_0^T &\Vert \varrho_\varepsilon \mathbf{u}_\varepsilon \Vert_{L^\frac{6\Gamma}{\Gamma+6}_x}^2 \, \mathrm{d}t \bigg)^\frac{1}{2}
\lesssim
\sup_{t\in[0,T]} \Vert \varrho_\varepsilon \Vert_{L^\Gamma_x}
\bigg( \int_0^T\Vert  \mathbf{u}_\varepsilon \Vert_{L^6_x}^2 \, \mathrm{d}t \bigg)^\frac{1}{2}
\\&
\lesssim
\bigg\vert
\sup_{t\in[0,T]} \Vert \varrho_\varepsilon \Vert_{L^\Gamma_x}
\bigg\vert^\Gamma +
\bigg( \int_0^T\Vert  \mathbf{u}_\varepsilon \Vert_{L^6_x}^2 \, \mathrm{d}t \bigg)^\frac{\Gamma'}{2}
\end{aligned}
\end{equation}
where $1< \Gamma' \leq \frac{6}{5}$. Therefore, it follows from \eqref{uniBoundFourthLayer1} that
\begin{align}
\label{uniBoundFourthLayer2}
\mathbb{E}\Vert \varrho_\varepsilon \mathbf{u}_\varepsilon \Vert_{L^2_tL^\frac{6\Gamma}{\Gamma+6}_x}^{2p} \lesssim_{ \delta, \Gamma, p,\mathcal{E}_{0,\varepsilon}} 1
\end{align}
holds uniformly in $\varepsilon$ for all $p\in(1,\infty)$ where $3\leq \frac{6\Gamma}{\Gamma+6} <6$.
At this point, we note that the only information we have on the pressure is given by the second summand in \eqref{uniBoundFourthLayer1}. However, unlike the previous sections where this regularity and the fact that the continuity equations were satisfied in the strong PDE sense combined to enable us perform our analysis, this aforementioned summand is not enough at this stage. Indeed, any family of pressures enjoying this regularity may only converge to a measure  which is not useful in identifying the limit pressure term. We therefore improve the pressure (or density) regularity in the following.
\subsection{Improved pressure estimate}
\label{sec:PressureEstFourthLay}
Let $\Delta^{-1}_{\mathbb{T}^3}(\varrho_\varepsilon -f)$ be a  solution of the Poisson equation \eqref{elecFieldThird1}. Notice that since $f$ is independent of time, we can combine the continuity equation \eqref{contEqThird1} and Poisson equation \ref{elecFieldThird1} and then obtain the following
\begin{align*}
\pm \mathrm{d} \Delta V_\varepsilon= \mathrm{d}(\varrho_\varepsilon -f)= [\varepsilon \Delta \varrho_\varepsilon - \mathrm{div}(\varrho_\varepsilon \mathbf{u}_\varepsilon) ]\, \mathrm{d}t.
\end{align*}
So by applying the operator $\nabla \Delta^{-1}_{\mathbb{T}^3}$ to the resulting combined equation, we obtain
\begin{align}
\label{contPoissEqItoFourthLayer}
\pm \mathrm{d} \nabla V_\varepsilon =[\varepsilon \nabla \varrho_\varepsilon - \nabla\Delta^{-1}_{\mathbb{T}^3}\mathrm{div}(\varrho_\varepsilon \mathbf{u}_\varepsilon)]\mathrm{d}t.
\end{align}
 Also, as the momemtum equation \eqref{momEqThird1} is only satisfied weakly in the sense of \eqref{weakMomEqThirdLayer}, a suitable approximation of the characteristic function $\chi_{[0,s]}$ for any $s\in(0,T]$, by a sequence of test functions $\psi\in C^\infty_c([0,T))$ results in the following 
\begin{equation}
\begin{aligned}
\label{momEqItoFourthLayer}
&\int_{{\mathbb{T}^3}} \varrho_\varepsilon  \mathbf{u}_\varepsilon (t) \cdot \bm{\phi} \, \mathrm{d}x   
=
 \int_{{\mathbb{T}^3}} \varrho_{0,\varepsilon } \mathbf{u}_{0,\varepsilon } \cdot \bm{\phi} \, \mathrm{d}x  
+  \int_0^s \int_{{\mathbb{T}^3}}  \varrho_\varepsilon    \mathbf{u}_\varepsilon \otimes \mathbf{u}_\varepsilon : \nabla \bm{\phi} \, \mathrm{d}x  \mathrm{d}t 
\\&-
 \nu^S\int_0^s\int_{{\mathbb{T}^3}} \nabla \mathbf{u}_\varepsilon : \nabla \bm{\phi} \, \mathrm{d}x  \mathrm{d}t 
-(\nu^B+ \nu^S )\int_0^s\int_{{\mathbb{T}^3}} \mathrm{div}\,\mathbf{u}_\varepsilon \, \mathrm{div}\, \bm{\phi} \, \mathrm{d}x  \mathrm{d}t    
\\&+   
\int_0^s \int_{{\mathbb{T}^3}}    p_{\delta}^{ \Gamma}(\varrho_\varepsilon)\, \mathrm{div} \, \bm{\phi} \, \mathrm{d}x  \mathrm{d}t
+\int_0^s\int_{{\mathbb{T}^3}} \varepsilon\varrho_\varepsilon\mathbf{u}_\varepsilon    \Delta\bm{\phi} \, \mathrm{d}x   \mathrm{d}t
\\&+\int_0^s\int_{{\mathbb{T}^3}} \vartheta \varrho_\varepsilon  \nabla V_{\varepsilon}  \cdot \bm{\phi} \, \mathrm{d}x    \mathrm{d}t
+
\int_0^s\sum_{k\in \mathbb{N}} \int_{{\mathbb{T}^3}}   \varrho_\varepsilon\, \mathbf{g}_{k,\varepsilon}(\varrho_\varepsilon , f, \varrho_\varepsilon \mathbf{u}_\varepsilon ) \cdot \bm{\phi}  \, \mathrm{d}x \mathrm{d}\beta_k  
\end{aligned}
\end{equation}
$\mathbb{P}$-a.s. Notice that \eqref{momEqItoFourthLayer} can easily be rewritten in differential form. By applying the generalized It\^o formula \cite[Theorem A.4.1]{breit2018stoch} to \eqref{contPoissEqItoFourthLayer} and \eqref{momEqItoFourthLayer}, we obtain the following
\begin{equation}
\begin{aligned}
\label{momEqItoFourthLayer1}
&
\int_0^s \int_{{\mathbb{T}^3}}    p_{\delta}^{ \Gamma}(\varrho_\varepsilon)\, \Delta V_\varepsilon \, \mathrm{d}x  \mathrm{d}t
=
\int_{{\mathbb{T}^3}} \varrho_\varepsilon  \mathbf{u}_\varepsilon (t) \cdot \nabla V_\varepsilon \, \mathrm{d}x
-
 \int_{{\mathbb{T}^3}} \varrho_{0,\varepsilon} \mathbf{u}_{0,\varepsilon} \cdot \nabla V_{0,\varepsilon} \, \mathrm{d}x     
\\&
-  \int_0^s \int_{{\mathbb{T}^3}}  \varrho_\varepsilon    \mathbf{u}_\varepsilon \otimes \mathbf{u}_\varepsilon : \nabla^2 V_\varepsilon \, \mathrm{d}x  \mathrm{d}t 
+(\nu^B+ \nu^S )\int_0^s\int_{{\mathbb{T}^3}} \mathrm{div}\,\mathbf{u}_\varepsilon \, \Delta V_\varepsilon \, \mathrm{d}x  \mathrm{d}t   
\\&
+
 \nu^S\int_0^s\int_{{\mathbb{T}^3}} \nabla \mathbf{u}_\varepsilon : \nabla^2 V_\varepsilon \, \mathrm{d}x  \mathrm{d}t  
-
\int_0^s\int_{{\mathbb{T}^3}} \varepsilon\nabla( \varrho_\varepsilon \mathbf{u}_\varepsilon):\nabla^2 V_\varepsilon \, \mathrm{d}x   \mathrm{d}t
\\&
-\int_0^s\int_{{\mathbb{T}^3}} \vartheta \varrho_\varepsilon \vert  \nabla V_\varepsilon  \vert^2 \, \mathrm{d}x    \mathrm{d}t
-
\int_0^s\int_{{\mathbb{T}^3}} \varrho_\varepsilon\mathbf{u}_\varepsilon \big[ \varepsilon \nabla \varrho_\varepsilon -  \nabla\Delta^{-1}_{\mathbb{T}^3}\mathrm{div}(\varrho_\varepsilon \mathbf{u}_\varepsilon) \big] \, \mathrm{d}x    \mathrm{d}t
\\&-
\int_0^s\sum_{k\in \mathbb{N}} \int_{{\mathbb{T}^3}}   \varrho_\varepsilon\, \mathbf{g}_{k,\varepsilon}(\varrho_\varepsilon , f, \varrho_\varepsilon \mathbf{u}_\varepsilon ) \cdot \nabla V_\varepsilon  \, \mathrm{d}x \mathrm{d}\beta_k
=:I_1+ \ldots + I_9.
\end{aligned}
\end{equation}
Without loss of generality (as we intend to take norms anyways), the equation \eqref{momEqItoFourthLayer1} was deduced for the $+$ sign in \eqref{contPoissEqItoFourthLayer}.
\\
Now if we denote the left-hand side of \eqref{momEqItoFourthLayer1} as $I_0$, then from \eqref{higherPressure} and \eqref{elecFieldThird1}, we have that
\begin{align*}
I_0 &= \int_0^s \int_{{\mathbb{T}^3}}    [ p(\varrho_\varepsilon) +\delta(\varrho_\varepsilon +\varrho^\Gamma_\varepsilon)] \, \varrho\, \mathrm{d}x\,  \mathrm{d}t
-
\int_0^s \int_{{\mathbb{T}^3}}    p_{\delta}^{ \Gamma}(\varrho_\varepsilon)\, f \, \mathrm{d}x\,  \mathrm{d}t
=:I_0^1 + I_0^2
\end{align*}
where from \eqref{lawMainThmFourthLayer} and \eqref{uniBoundFourthLayer1},
\begin{equation}
\begin{aligned}
&\mathbb{E}\vert   I_0^2 \vert^p \lesssim_{\delta, \gamma} \mathbb{E} \bigg\vert \sup_{t\in [0,T]} \Vert \varrho_\varepsilon  \Vert_{L^\Gamma_x}^\Gamma \Vert f \Vert_{L^\infty_x}\bigg\vert^p 
\lesssim_{\delta, \gamma, \mathcal{E}_{0,\varepsilon}, \overline{f}} \mathbb{E} \bigg\vert \sup_{t\in [0,T]} \Vert \varrho_\varepsilon  \Vert_{L^\Gamma_x}^\Gamma \bigg\vert^p  \lesssim_{\delta, \gamma, \mathcal{E}_{0,\varepsilon}, \overline{f}}1
\end{aligned}
\end{equation}
holds uniformly in $\varepsilon$. 
\\
Notice that by using the Poisson equation and the continuity of the operator $\nabla^2 \Delta^{-1}_{\mathbb{T}^2}$, we  have that
\begin{align}
\label{gradSquaredV}
\Vert \nabla^2 V_\varepsilon \Vert_{L^p_x}
\approx
\Vert \nabla^2 \Delta^{-1}_{\mathbb{T}^2} \Delta V_\varepsilon \Vert_{L^p_x}
\lesssim
\Vert  \Delta V_\varepsilon \Vert_{L^p_x}.
\end{align}
To estimate $I_1$, we use the continuous embedding $W^{1,\Gamma}_x \hookrightarrow L^\infty_x$ which holds for $\Gamma \geq 3$, \eqref{gradSquaredV} and \eqref{uniBoundFourthLayer1} so that
\begin{equation}
\begin{aligned}
\mathbb{E}\vert I_1 \vert^p
\lesssim
\mathbb{E} \big\vert \Vert \varrho_\varepsilon \mathbf{u}_\varepsilon(t) \Vert_{L^1_x} \Vert \nabla V_\varepsilon(t) \Vert_{L^\infty_x} \big\vert^p
&\lesssim
\mathbb{E} \bigg\vert 
\sup_{t\in [0,T]}\Vert \varrho_\varepsilon \mathbf{u}_\varepsilon \Vert_{L^\frac{2\Gamma}{\Gamma+ 2}_x}^\frac{2\Gamma}{\Gamma+ 2} \bigg\vert^p
+
\mathbb{E} \bigg\vert 
\sup_{t\in [0,T]}
 \Vert \Delta V_\varepsilon \Vert_{L^\Gamma_x}^\Gamma \bigg\vert^p
  \\&
 \lesssim_{ \delta, \Gamma, p,\mathcal{E}_{0,\varepsilon}} 1
\end{aligned}
\end{equation}
holds uniformly in $\varepsilon$ since $\frac{3}{2} \leq \frac{2\Gamma}{\Gamma+ 2} <2$ and  $\Gamma \geq 6$. An analogous estimate holds for $I_2$. 
Again, since for $\Gamma \geq6$, we have that $\frac{3}{2} < r \leq 2$ where $\frac{1}{r}=1- \frac{2}{6} -\frac{1}{\Gamma}$, it follows from \eqref{gradSquaredV} and \eqref{uniBoundFourthLayer1} that
\begin{equation}
\begin{aligned}
\mathbb{E}\vert I_3 \vert^p
&\leq
\mathbb{E}\bigg\vert
\int_0^s \Vert \varrho_\varepsilon \Vert_{L^\Gamma_x} \Vert \mathbf{u}_\varepsilon \Vert_{L^6_x}^2 \Vert \nabla^2 V_\varepsilon \Vert_{L^\frac{3\Gamma}{2\Gamma -3}_x} \, \mathrm{d}t \bigg\vert^p
\\
& 
\lesssim_\Gamma
\mathbb{E}\bigg\vert
\sup_{t\in[0,T]} \Vert \varrho_\varepsilon \Vert_{L^\Gamma_x} \Vert \nabla \mathbf{u}_\varepsilon \Vert_{L^2_{t,x}}^2 \sup_{t\in [0,T]}\Vert \Delta V_\varepsilon \Vert_{L^2_x} \bigg\vert^p
\\
& 
\lesssim_\Gamma
\mathbb{E}\bigg\vert
\sup_{t\in[0,T]} \Vert \varrho_\varepsilon \Vert_{L^\Gamma_x}^\Gamma \bigg\vert^p
+
\mathbb{E}
 \Vert \nabla \mathbf{u}_\varepsilon \Vert_{L^2_{t,x}}^\frac{2\Gamma p}{\Gamma-2}
 + 
 \mathbb{E}\bigg\vert
 \sup_{t\in [0,T]}\Vert \Delta V_\varepsilon \Vert_{L^\Gamma_x}^\Gamma \bigg\vert^p
 \\&
 \lesssim_{ \delta, \Gamma, p,\mathcal{E}_{0,\varepsilon}} 1
\end{aligned}
\end{equation}
uniformly in $\varepsilon$. We obtain from H\"older inequality and $\Gamma \geq 2$ that the estimate
\begin{equation}
\begin{aligned}
\mathbb{E}\vert I_4 \vert^p
\lesssim
\mathbb{E}  \bigg\vert \int_0^s\big( \Vert \nabla \mathbf{u}_\varepsilon \Vert_{L^2_{x}}^{2}
+
\Vert  \Delta V_\varepsilon \Vert_{L^2_{x}}^{2} \big) \, \mathrm{d}t  \bigg\vert^p
&\lesssim
\mathbb{E}
 \Vert \nabla \mathbf{u}_\varepsilon \Vert_{L^2_{t,x}}^{2p} + \mathbb{E} \bigg\vert \sup_{t\in [0,T]} \Vert \Delta V_\varepsilon \Vert_{L^\Gamma_x}^\Gamma \bigg\vert^p  \\&\lesssim_{k, \Gamma, \delta, \gamma, \mathcal{E}_{0,\varepsilon}, \overline{f}}1
\end{aligned}
\end{equation}
holds uniformly in $\varepsilon$ as a result of \eqref{uniBoundFourthLayer1}. A similar estimate holds for $I_5$.
We now note that since $\varepsilon\in(0,1)$,
\begin{align*}
\varepsilon\nabla(\varrho_\varepsilon \mathbf{u}_\varepsilon)
\leq\sqrt{\varepsilon} \big(\sqrt{\varepsilon} \nabla \varrho_\varepsilon \otimes \mathbf{u}_\varepsilon + \varrho_\varepsilon \nabla \mathbf{u}_\varepsilon \big).
\end{align*}
And since $2 < \frac{2\Gamma}{\Gamma -2} \leq 3$, it therefore follow from \eqref{gradSquaredV} and \eqref{uniBoundFourthLayer1} that
\begin{equation}
\begin{aligned}
\mathbb{E}&\vert I_6 \vert^p
\lesssim
\varepsilon^\frac{p}{2}
\mathbb{E}\bigg\vert \int_0^s \bigg(
\Vert \sqrt{\varepsilon} \nabla \varrho_\varepsilon \Vert_{L^2_x} \Vert \mathbf{u}_\varepsilon \Vert_{L^6_x} \Vert \nabla^2 V_\varepsilon \Vert_{L^3_x} 
+
\Vert \varrho_\varepsilon \Vert_{L^\Gamma_x} \Vert \nabla \mathbf{u}_\varepsilon \Vert_{L^2_x}
\Vert \nabla^2 V_\varepsilon \Vert_{L^\frac{2\Gamma}{\Gamma-2}_x}
  \bigg) \, \mathrm{d}t \bigg\vert^p
  \\
&\lesssim
\varepsilon^\frac{p}{2} \mathbb{E}
\Vert \sqrt{\varepsilon} \nabla \varrho_\varepsilon \Vert_{L^2_{t,x}}^{3p}
+
\varepsilon^\frac{p}{2} \mathbb{E}\bigg\vert
\sup_{t\in[0,T]} \Vert \varrho_\varepsilon \Vert_{L^\Gamma_x}^3 \bigg\vert^p 
+
\varepsilon^\frac{p}{2} \mathbb{E} \Vert \nabla \mathbf{u}_\varepsilon \Vert_{L^2_{t,x}}^{3p}
+
\varepsilon^\frac{p}{2} \mathbb{E}\bigg\vert
\sup_{t\in[0,T]} \Vert \Delta V_\varepsilon \Vert_{L^\Gamma_x}^3 \bigg\vert^p 
\\
& \lesssim_{ \delta, \Gamma, p,\mathcal{E}_{0,\varepsilon}} \varepsilon^\frac{p}{2}
\end{aligned}
\end{equation}
holds uniformly in $\varepsilon$ (note that $\varepsilon^\frac{p}{2} <1$). For $I_7$, we use the continuous embedding $W^{1,\Gamma}_x \hookrightarrow L^\infty_x$, \eqref{gradSquaredV} and \eqref{uniBoundFourthLayer1} to obtain
\begin{equation}
\begin{aligned}
\mathbb{E}\vert   I_7 \vert^p
\lesssim
\mathbb{E}\bigg\vert \int_0^s  \Vert \varrho_\varepsilon \Vert_{L^1_x} \Vert \Delta V_\varepsilon \Vert^2_{L^\Gamma_x} \, \mathrm{d}t \bigg\vert^p 
&\lesssim
\varepsilon^\frac{p}{2} \mathbb{E}\bigg\vert
\sup_{t\in[0,T]} \Vert \varrho_\varepsilon \Vert_{L^\Gamma_x}^\Gamma \bigg\vert^p 
+
\varepsilon^\frac{p}{2} \mathbb{E}\bigg\vert
\sup_{t\in[0,T]} \Vert \Delta V_\varepsilon \Vert_{L^\Gamma_x}^\Gamma \bigg\vert^p 
\\
& \lesssim_{ \delta, \Gamma, p,\mathcal{E}_{0,\varepsilon}} 1
\end{aligned}
\end{equation}
uniformly in $\varepsilon$.
The estimate for $I_8$ is as follows:
\begin{equation}
\begin{aligned}
\mathbb{E}\vert   I_8 \vert^p 
&\lesssim   
 \mathbb{E}\big\vert  \Vert \varrho_\varepsilon \mathbf{u}_\varepsilon \Vert_{L^2_{t,x}}  \Vert \varepsilon\nabla \varrho_\varepsilon \Vert_{L^2_{t,x}}
 +
 \Vert \varrho_\varepsilon \mathbf{u}_\varepsilon \Vert_{L^2_{t,x}}^2
 \big\vert^p
\lesssim   
\mathbb{E}  \Vert\varepsilon \nabla \varrho_\varepsilon \Vert_{L^2_{t,x}}^{2p}
 +
 \mathbb{E}
 \Vert \varrho_\varepsilon \mathbf{u}_\varepsilon \Vert_{L^2_{t,x}}^{2p}
  \\&\lesssim   
\mathbb{E} \big\Vert \sqrt{\varepsilon} \nabla\varrho_\varepsilon \big\Vert_{L^2_{t,x}}^{2p}
 +
\mathbb{E}\Vert \varrho_\varepsilon \mathbf{u}_\varepsilon \Vert_{L^2_tL^\frac{6\Gamma}{\Gamma+6}_x}^{2p}
\lesssim_{k, \Gamma, \delta, \gamma, \mathcal{E}_{0,\varepsilon}, \overline{f}}1
\end{aligned}
\end{equation}
which holds uniformly in $\varepsilon$ for all $p\geq1$ because of \eqref{uniBoundFourthLayer1} and \eqref{uniBoundFourthLayer2}.
To estimate $I_9$ we first note 
\begin{equation}
\begin{aligned}
\bigg\vert \int_{{\mathbb{T}^3}}   \varrho_\varepsilon\, \mathbf{g}_{k,\varepsilon}(\varrho_\varepsilon , f, \varrho_\varepsilon \mathbf{u}_\varepsilon ) \cdot \nabla V_\varepsilon  \, \mathrm{d}x\bigg\vert^2
\leq
c_k \Vert \mathbf{g}_{k,\varepsilon} \Vert_{L^\infty_x}^2
\Big(
\Vert \varrho_\varepsilon  \Vert_{L^\Gamma_x}^{2\Gamma} 
+
\Vert \nabla V_\varepsilon  \Vert_{L^{\Gamma'}_x}^{2\Gamma'}
 \Big)
\end{aligned}
\end{equation}
and so similar to the argument in \eqref{noiseBoundFourthLayer1}--\eqref{noiseBoundFourthLayer3}
we can invoke the Burkholder--Davis--Gundy inequality and use \eqref{stochCoeffBound1Exis0}--\eqref{stochsummableConst}, and the fact that the cut-off function \eqref{cuttOff} is uniformly bounded  by one to obtain from \eqref{uniBoundFourthLayer1}, the following estimate
\begin{equation}
\begin{aligned}
\mathbb{E}&\vert   I_9 \vert^p 
\leq  
\mathbb{E} \Bigg\vert \int_0^s\sum_{k\in \mathbb{N}}\bigg\vert \int_{{\mathbb{T}^3}}   \varrho_\varepsilon\, \mathbf{g}_{k,\varepsilon}(\varrho_\varepsilon , f, \varrho_\varepsilon \mathbf{u}_\varepsilon ) \cdot \nabla V_\varepsilon  \, \mathrm{d}x\bigg\vert^2 \mathrm{d}t \Bigg\vert^\frac{p}{2} 
\\&
\lesssim_{k, \Gamma} \mathbb{E} \bigg\vert \sup_{t\in [0,T]} \Vert \varrho_\varepsilon  \Vert_{L^\Gamma_x}^{2\Gamma} \bigg\vert^\frac{p}{2} 
+
\mathbb{E} \bigg\vert \sup_{t\in [0,T]} \Vert \nabla V_\varepsilon  \Vert_{L^{\Gamma'}_x}^{2\Gamma'} \bigg\vert^\frac{p}{2} 
\lesssim_{k, \Gamma, \delta, \gamma, \mathcal{E}_{0,\varepsilon}, \overline{f}}1
\end{aligned}
\end{equation}
uniformly in $\varepsilon$ for all $p\geq2$. The fact that $p\geq 2$ follow from the fact that $\vert \cdot\vert^{p/2}$ is convex only if $p/2\geq1$.
\\
By collecting the various estimates above, we have shown that
\begin{equation}
\begin{aligned}
\label{pressureEstFouthLayer0}
\mathbb{E}\bigg\vert \int_0^t \int_{{\mathbb{T}^3}}    [ \varrho_\varepsilon\, p(\varrho_\varepsilon) +\delta(\varrho^2_\varepsilon +\varrho^{\Gamma+1}_\varepsilon)] \, \mathrm{d}x\,  \mathrm{d}t \bigg\vert^p \lesssim_{k, \Gamma, \delta, \gamma, \mathcal{E}_{0,\varepsilon}, \overline{f}}1
\end{aligned}
\end{equation}
holds uniformly in $\varepsilon$ for all $p\geq2$ and in particular, for $\Gamma\geq 6$, the estimate
\begin{equation}
\begin{aligned}
\label{pressureEstFouthLayer1}
\mathbb{E}\bigg\vert \int_0^t \int_{{\mathbb{T}^3}}    \varrho_\varepsilon^{\Gamma+1} \, \mathrm{d}x\,  \mathrm{d}t \bigg\vert^p \lesssim_{p,k, \Gamma, \delta, \gamma, \mathcal{E}_{0,\varepsilon} , \overline{f}}1
\end{aligned}
\end{equation}
holds uniformly in $\varepsilon$ for all $p\geq2$.

\subsection{Compactness}
\label{sec:compactnessFourthLayer}
In order to establish compactness, we first need some preparation. We denote the energy by
\begin{align}
\label{energyFunctionalFourthLayer}
\mathcal{E}_\varepsilon := \frac{1}{2} \varrho_\varepsilon  \vert \mathbf{u}_\varepsilon  \vert^2  
+ P^\Gamma_\delta(\varrho_\varepsilon ) 
\pm
\vartheta \vert \nabla V_\varepsilon \vert^2 
\end{align}
and let the weakly-$*$ measurable mapping
\begin{align*}
\nu_\varepsilon : [0,T] \times \mathbb{T}^3 \rightarrow \mathfrak{P}(\mathbb{R}^{20})
\end{align*}
defined by 
\begin{align*}
\nu_{\varepsilon,t,x}(\cdot)= \delta_{[\varrho_\varepsilon, \mathbf{u}_\varepsilon, \nabla \mathbf{u}_\varepsilon, \varrho_\varepsilon \mathbf{u}_\varepsilon, f, \nabla V_\varepsilon](t,x)}(\cdot)
\end{align*}
be the canonical Young measure associated to  $[\varrho_\varepsilon, \mathbf{u}_\varepsilon, \nabla \mathbf{u}_\varepsilon, \varrho_\varepsilon \mathbf{u}_\varepsilon, f, \nabla V_\varepsilon]$. See the  discussion in \cite[ Section 2.8, Section 4.4.3.1]{breit2018stoch} [15, Section 2.8] on how this allows us to interpret $\nu_\varepsilon$ as a random variable taking values in the non-Polish space $\big(L^\infty\big((0,T)\times (\mathbb{T}^3)\big); \mathfrak{P}(\mathbb{R}^{20}) , w^*\big)$ endowed with the weak-$*$ topology which is determined by
\begin{align*}
L^\infty\big( (0,T)\times\mathbb{T}^3;\mathfrak{P}(\mathbb{R}^{20})\big)
\rightarrow
\mathbb{R},
\quad
\nu \mapsto
\int_0^T\int_{\mathbb{T}^3}\psi(t,x) \int_{\mathbb{R}^{20}} \phi(\xi)\, \mathrm{d}\nu_{\omega, t,x}(\xi)\, \mathrm{d}x\, \mathrm{d}t
\end{align*}
for all $\psi\in L^1((0,T)\times\mathbb{T}^3)$, for all $\phi\in C_b(\mathbb{R}^{20})$. 
We now define the following spaces 
\begin{align*}
\chi_{\varrho_0} = L^\Gamma_x, \, \quad \chi_{\mathbf{m}_0}= L^1_x, \, \quad \chi_{\frac{\mathbf{m}_0}{\sqrt{\varrho_0}}}= L^2_x,
\, \quad
\chi_{ \mathbf{u} } =  \big(L^2\big(0,T;W^{1,2}_x\big), w\big),
\, \quad
\chi_V = C_w\left([0,T];W^{2,\Gamma}_x\right),
\\
\chi_W= C\left([0,T];\mathfrak{U}_0\right) ,
\, \quad
\chi_{\varrho\mathbf{u}} =C_w\big([0,T]; L^\frac{2\Gamma}{\Gamma+1} _x\big) \cap C\big([0,T]; W^{-k,2}_x\big), 
\\
\chi_\varrho = C_w \left([0,T];L^{\Gamma}_x\right)
\cap
\big(L^{\Gamma+1}_{t,x}, w\big)
\quad
\chi_{\mathcal{E}}= \big(L^\infty_t; \mathcal{M}_b(\mathbb{T}^3), w^*\big),
 \quad
\chi_{\nu}= \big(L^\infty_{w^*}\big((0,T)\times \mathbb{T}^3; \mathfrak{P}(\mathbb{R}^{20})\big) , w^*\big)
\end{align*}
for $\Gamma\geq6$ and $k> \frac{5}{2}$. We now let  $\mu_{\varrho_{0,\varepsilon}}$, $\mu_{\mathbf{m}_{0,\varepsilon}}$, $\mu_{\frac{\mathbf{m}_{0,\varepsilon}}{\sqrt{\varrho_{0,\varepsilon}}}}$, $\mu_{\varrho_\varepsilon }$, $\mu_{\mathbf{u}_\varepsilon }$, $\mu_{\varrho_\varepsilon\mathbf{u}_\varepsilon }$,  $\mu_{V_\varepsilon }$,$\mu_{\mathcal{E}_\varepsilon }$, $\mu_{\nu_\varepsilon }$ and $\mu_{W}$ be the respective laws of  $\varrho_{0,\varepsilon}$, $\mathbf{m}_{0,\varepsilon}$,   $\frac{\mathbf{m}_{0,\varepsilon}}{\sqrt{\varrho_{0,\varepsilon}}}$,  $\varrho_\varepsilon$, $\mathbf{u}_\varepsilon$, $\varrho_\varepsilon \mathbf{u}_\varepsilon $, $V_\varepsilon $, $\mathcal{E}_\varepsilon$, $\nu_\varepsilon$ and $W$ on the respective spaces 
$\chi_{\varrho_0}$, $\chi_{\mathbf{m}_0}$, $\chi_{\frac{\mathbf{m}_0}{\sqrt{\varrho_0}}}$,  $\chi_{\varrho} $,
$\chi_{\mathbf{u}} $,  $\chi_{ \varrho\mathbf{u} }$, $\chi_V$, $\chi_{\mathcal{E}}$, $\chi_\nu$ and $\chi_W$. Furthermore, we set $\mu_\varepsilon $ as their joint law on the space 
\begin{align*}
\chi = \chi_{\varrho_0} \times \chi_{\mathbf{m}_0}\times \chi_{\frac{\mathbf{m}_0}{\sqrt{\varrho_0}}}\times \chi_{\varrho} \times \chi_{\mathbf{u}}  \times  \chi_{ \varrho\mathbf{u} } \times \chi_V\times \chi_{\mathcal{E}} \times \chi_\nu \times \chi_W.
\end{align*}
Now since the following spaces are Polish, it follows from Prokhorov’s theorem theorem that:
\begin{itemize}
\item The set $\{\mu_{\varrho_{0,\varepsilon}}: \varepsilon \in (0,1)\}$ is tight on $\chi_{\varrho_0}$.
\item The set $\{\mu_{\mathbf{m}_{0,\varepsilon}}: \varepsilon \in (0,1)\}$ is tight on $\chi_{\mathbf{m}_0}$.
\item The set $\bigg\{\mu_{\frac{\mathbf{m}_{0,\varepsilon}}{\sqrt{\varrho_{0,\varepsilon}}}}: \varepsilon \in (0,1)\bigg\}$ is tight on $\chi_{\frac{\mathbf{m}_0}{\sqrt{\varrho_0}}}$.
\item The set $\{ \mu_W\}$ is tight on $\chi_W$.
\end{itemize}
Additionally, in analogy with the corresponding result in Lemma \ref{tighnessThirdLAyer}, we have that
\begin{itemize}
\item The set $\{\mu_{\varrho_\varepsilon}: \varepsilon \in (0,1) \}$ is tight on $\chi_\varrho$.
\item The set $\{\mu_{\mathbf{u}_\varepsilon}: \varepsilon \in (0,1)\}$ is tight on $\chi_{\mathbf{u}}$.
\item The set $\{\mu_{\varrho_\varepsilon\mathbf{u}_\varepsilon}: \varepsilon \in (0,1)\}$ is tight on $\chi_{\varrho\mathbf{u}}$.
\item The set $\{\mu_{V_\varepsilon}: \varepsilon \in (0,1)\}$ is tight on $\chi_{V}$.
\end{itemize}
Also, analogous to \cite[Proposition 4.4.6, Proposition 4.4.7]{breit2018stoch},
\begin{itemize} 
\item The set $\{\mu_{\mathcal{E}_\varepsilon}: \varepsilon \in (0,1)\}$ is tight on $\chi_\mathcal{E}$.
\item The set $\{\mu_{\nu_\varepsilon}: \varepsilon \in (0,1)\}$ is tight on $\chi_\nu$.
\end{itemize}
The following lemma thus hold.
\begin{lemma}
\label{tighnessFourthLAyer}
The set $\{\mu_\varepsilon :   \varepsilon \in (0,1) \}$  is tight on $\chi$.
\end{lemma}
We can now apply the Jakubowski--Skorokhod theorem \cite{jakubowski1998short} and we obtain the following result.
\begin{lemma}
\label{lem:JakowFourthLayer}
The exists a subsequence (not relabelled) $\{\mu_\varepsilon  :  \varepsilon \in (0,1)  \}$, a complete probability space $(\tilde{\Omega}, \tilde{\mathscr{F}}, \tilde{\mathbb{P}})$ with $\chi$-valued random variables
\begin{align*}
(\tilde{\varrho}_{0,\varepsilon}, \tilde{\mathbf{m}}_{0,\varepsilon},
\tilde{\mathbf{n}}_{0,\varepsilon},  \tilde{\varrho}_\varepsilon , \tilde{\mathbf{u}}_\varepsilon , \tilde{\mathbf{m}}_\varepsilon  ,\tilde{V}_\varepsilon ,
\tilde{\mathcal{E}}_\varepsilon,
\tilde{\nu}_\varepsilon,
 \tilde{W}_\varepsilon ) \quad \varepsilon \in (0,1) 
\end{align*}
and 
\begin{align*}
 (\tilde{\varrho}_0, \tilde{\mathbf{m}}_0, \tilde{\mathbf{n}}_0,  \tilde{\varrho}, \tilde{\mathbf{u}}, \tilde{\mathbf{m}}, \tilde{V}, \tilde{\mathcal{E}},
\tilde{\nu}, \tilde{W})
\end{align*}
such that 
\begin{itemize}
\item the law of $(\tilde{\varrho}_{0,\varepsilon}, \tilde{\mathbf{m}}_{0,\varepsilon},
\tilde{\mathbf{n}}_{0,\varepsilon},  \tilde{\varrho}_\varepsilon , \tilde{\mathbf{u}}_\varepsilon , \tilde{\mathbf{m}}_\varepsilon  ,\tilde{V}_\varepsilon ,
\tilde{\mathcal{E}}_\varepsilon,
\tilde{\nu}_\varepsilon,
 \tilde{W}_\varepsilon )$ on $\chi$ coincide with $\mu_\varepsilon $, $\varepsilon \in (0,1) $,
\item the law of $ (\tilde{\varrho}_0, \tilde{\mathbf{m}}_0, \tilde{\mathbf{n}}_0,  \tilde{\varrho}, \tilde{\mathbf{u}}, \tilde{\mathbf{m}}, \tilde{V}, \tilde{\mathcal{E}},
\tilde{\nu}, \tilde{W})$ on $\chi$ is a Radon measure,
\item the following convergence  (with each $\rightarrow$ interpreted with respect to the corresponding topology)
\begin{align*}
\tilde{\varrho}_{0,\varepsilon} \rightarrow \tilde{\varrho}_0 \quad \text{in}\quad \chi_{\varrho_0}, & \quad \qquad
\tilde{\mathbf{m}}_{0,\varepsilon} \rightarrow \tilde{\mathbf{m}}_0 \quad \text{in}\quad \chi_{\mathbf{m}_0}, 
\\
\tilde{\mathbf{n}}_{0,\varepsilon} \rightarrow \tilde{\mathbf{n}}_0 \quad \text{in}\quad \chi_{\frac{\mathbf{m}_0}{\sqrt{\varrho_0}}}, 
& \quad \qquad
\tilde{\varrho}_\varepsilon  \rightarrow \tilde{\varrho} \quad \text{in}\quad \chi_\varrho,
\\
\tilde{\mathbf{u}}_\varepsilon  \rightarrow \tilde{\mathbf{u}} \quad \text{in}\quad \chi_\mathbf{u}, 
 & \quad \qquad
\tilde{\mathbf{m}}_\varepsilon  \rightarrow \tilde{\mathbf{m}} \quad \text{in}\quad \chi_\mathbf{m}, 
\\
 \tilde{V}_\varepsilon  \rightarrow \tilde{V} \quad  \text{in}\quad \chi_V,
 & \quad \qquad
 \tilde{\mathcal{E}}_\varepsilon  \rightarrow  \tilde{\mathcal{E}} \quad \text{in}\quad \chi_\mathcal{E}, 
 \\
 \tilde{\nu}_\varepsilon  \rightarrow \tilde{\nu} \quad  \text{in}\quad \chi_\nu,
 & \quad \qquad
\tilde{W}_\varepsilon  \rightarrow \tilde{W} \quad \text{in}\quad \chi_W
\end{align*}
holds $\tilde{
\mathbb{P}}$-a.s.;
\item consider any Carath\'eodory function $\mathcal{C}=\mathcal{C}(t,x,\varrho, \mathbf{u}, \mathbb{U}, \mathbf{m}, f, \mathbf{V})$ with
\begin{align*}
(t,x,\varrho, \mathbf{u}, \mathbb{U}, \mathbf{m}, f, \mathbf{V}) \in [0,T]\times \mathbb{T}^3 \times \mathbb{R} \times \mathbb{R}^3 \times \mathbb{R}^{3\times3} \times \mathbb{R}^3 \times \mathbb{R} \times \mathbb{R}^3,
\end{align*} 
and where the following estimate
\begin{align*}
\vert \mathcal{C} \vert \lesssim 1+ \vert \varrho\vert^{r_1} + \vert  \mathbf{u}\vert^{r_2}  + \vert \mathbb{U}\vert^{r_3}   + \vert  \mathbf{m}\vert^{r_4}  + \vert  f\vert^{r_5}   + \vert  \mathbf{V}\vert^{r_6} 
\end{align*}
holds uniformly in $(t,x)$ for some $r_i>0$, $i=1, \ldots, 6$. Then as $\varepsilon \rightarrow0$, it follows that
\begin{align*}
\mathcal{C}(\tilde{\varrho}_\varepsilon, \tilde{\mathbf{u}}_\varepsilon, \nabla\tilde{\mathbf{u}}_\varepsilon, \tilde{\mathbf{m}}_\varepsilon, f, \nabla\tilde{V}_\varepsilon) 
\rightharpoonup
\overline{
\mathcal{C}(\tilde{\varrho}, \tilde{\mathbf{u}}, \nabla\tilde{\mathbf{u}}, \tilde{\mathbf{m}}, f, \nabla\tilde{V})}
\end{align*}
holds  in $L^r((0,T) \times
 \mathbb{T}^3)$ for all
 \begin{align*}
1<r \leq \frac{\Gamma+1}{r_1} \wedge \frac{2}{r_2} \wedge \frac{2\Gamma}{r_4(\Gamma+1)} \wedge \frac{2}{r_6}
 \end{align*}
$\tilde{\mathbb{P}}$-a.s.
\end{itemize}
\end{lemma}
As a consequence of Lemma \ref{lem:JakowFourthLayer}, we obtain the following result.
\begin{corollary}
The following holds $\tilde{\mathbb{P}}$-a.s.
\begin{equation}
\begin{aligned}
&\tilde{\varrho}_{0,\varepsilon}= \tilde{\varrho}_\varepsilon(0), \quad \tilde{\mathbf{m}}_{0,\varepsilon}= \tilde{\varrho}_\varepsilon\tilde{\mathbf{u}}_\varepsilon(0), \quad \tilde{\mathbf{n}}_{0,\varepsilon} = \frac{\tilde{\mathbf{m}}_{0,\varepsilon}}{\sqrt{\tilde{\varrho}_{0,\varepsilon}}}, \quad \tilde{\mathbf{m}}_{\varepsilon}= \tilde{\varrho}_\varepsilon\tilde{\mathbf{u}}_\varepsilon,
\\
&\tilde{\mathcal{E}}_\varepsilon = \mathcal{E}_\varepsilon(\tilde{\varrho}_\varepsilon, \tilde{\mathbf{u}}_\varepsilon, \tilde{V}_\varepsilon),\quad
\tilde{\nu}_{\varepsilon}= \delta_{[\tilde{\varrho}_\varepsilon, \tilde{\mathbf{u}}_\varepsilon, \nabla \tilde{\mathbf{u}}_\varepsilon, \tilde{\varrho}_\varepsilon \tilde{\mathbf{u}}_\varepsilon, f, \nabla \tilde{V}_\varepsilon]}
\end{aligned}
\end{equation}
and that
\begin{equation}
\begin{aligned}
\tilde{\mathbb{E}} \bigg\vert \sup_{t\in[0,T]} \int_{\mathbb{T}^3} \tilde{\mathcal{E}}_\varepsilon \, \mathrm{d}x \bigg\vert^p
&=
\tilde{\mathbb{E}} \bigg\vert\sup_{t\in[0,T]} \int_{\mathbb{T}^3}\bigg( \frac{1}{2} \tilde{\varrho}_\varepsilon  \vert \tilde{\mathbf{u}}_\varepsilon  \vert^2  
+ P^\Gamma_\delta(\tilde{\varrho}_\varepsilon ) 
\pm
\vartheta \vert \nabla \tilde{V}_\varepsilon \vert^2 \bigg) \, \mathrm{d}x
\bigg\vert^p
\\&\lesssim_{ \delta, \Gamma, p,\tilde{\mathcal{E}}_{0,\varepsilon}} 1
\end{aligned}
\end{equation}
holds uniformly in $\varepsilon$ just as in \eqref{uniBoundFourthLayer1}.
\end{corollary}
Furthermore, a consequence of the last item of Lemma \ref{lem:JakowFourthLayer} is the following.
\begin{corollary}
\label{cor:nonlinearLimitsFourth}
There exists $\overline{p^\Gamma_\delta(\tilde{\varrho})}$, $\overline{\vert \nabla \tilde{V}\vert^2}$, $\overline{\tilde{\varrho} \nabla \tilde{V}}$ and $\overline{\tilde{\varrho}\, \mathbf{g}_k(\tilde{\varrho}, \tilde{f}, \tilde{\varrho}\tilde{\mathbf{u}})}$ such that
\begin{align}
p^\Gamma_\delta(\tilde{\varrho}_\varepsilon)
&\rightharpoonup
\overline{p^\Gamma_\delta(\tilde{\varrho})}
\label{pressureLimitFourthLayer}\\
\vert \nabla \tilde{V}_\varepsilon \vert^2
&\rightharpoonup
\overline{\vert \nabla \tilde{V}\vert^2}
\label{nablaVLimitFourthLayer}\\
\tilde{\varrho}_\varepsilon \nabla \tilde{V}_\varepsilon 
&\rightharpoonup
\overline{\tilde{\varrho} \nabla \tilde{V}}
\label{rhoNablaVLimitFourthLayer}\\
\tilde{\varrho}_\varepsilon \mathbf{g}_k(\tilde{\varrho}_\varepsilon, f, \tilde{\varrho}_\varepsilon\tilde{\mathbf{u}}_\varepsilon)
&\rightharpoonup
\overline{\tilde{\varrho}\, \mathbf{g}_k(\tilde{\varrho}, f, \tilde{\varrho}\tilde{\mathbf{u}})}
\label{noiseLimitFourthLayer}
\end{align}
$\tilde{\mathbb{P}}$-a.s. in $L^r((0,T) \times \mathbb{T}^3)$ for some $r>1$.
\end{corollary}

\subsection{Identifying the limit system}
\label{sec:limSystemFourthLayer}
Just as was done in Section \ref{sec:limSystemThirdLayer},
we can now construct and endow the probability
space with the following pair of filtrations
\begin{align*}
\tilde{\mathscr{F}}_t^\varepsilon  :=\sigma \bigg(\sigma_t[\tilde{\varrho}_\varepsilon], \sigma_t[\tilde{\mathbf{u}}_\varepsilon], \sigma_t[\tilde{V}_\varepsilon], \bigcup_{k\in \mathbb{N}}\sigma_t[\tilde{\beta}_{\varepsilon,k}]  \bigg), \quad t\in [0,T]
\end{align*}
and
\begin{align*}
\tilde{\mathscr{F}}_t  :=
\sigma \bigg(\sigma_t[\tilde{\varrho}], \sigma_t[\tilde{\mathbf{u}}], \sigma_t[\tilde{V}], \bigcup_{k\in \mathbb{N}}\sigma_t[\tilde{\beta}_{k}]  \bigg), \quad t\in [0,T]
\end{align*}
on the family of sequences $(\tilde{\varrho}_{0,\varepsilon}, \tilde{\mathbf{m}}_{0,\varepsilon},
\tilde{\mathbf{n}}_{0,\varepsilon},  \tilde{\varrho}_\varepsilon , \tilde{\mathbf{u}}_\varepsilon , \tilde{\mathbf{m}}_\varepsilon  ,\tilde{V}_\varepsilon ,
\tilde{\mathcal{E}}_\varepsilon,
\tilde{\nu}_\varepsilon,
 \tilde{W}_\varepsilon )$  and the limit random variables $ (\tilde{\varrho}_0, \tilde{\mathbf{m}}_0, \tilde{\mathbf{n}}_0,  \tilde{\varrho}, \tilde{\mathbf{u}}, \tilde{\mathbf{m}}, \tilde{V}, \tilde{\mathcal{E}},
\tilde{\nu}, \tilde{W})$ respectively. The following can be found in \cite[Lemma 4.4.11]{breit2018stoch}.
\begin{lemma}
\label{lem:identifyContFourthLayer}
The random distributions  $[\tilde{\varrho}, \tilde{\mathbf{u}}]$ satisfies \eqref{weakContEqFourthLayer} for all $\psi \in C^\infty_c ([0,T))$ and $\phi \in C^\infty(\mathbb{T}^3)$ $\tilde{\mathbb{P}}$-a.s.
\end{lemma}
Similar to Lemma \ref{lem:identifyPoisThirdLayer}, we obtain the following result.
\begin{lemma}
\label{lem:identifyPoisFourthLayer}
The random variables $[\tilde{\varrho},  \tilde{V}]$ satisfies  \eqref{elecFieldFourth1} a.e. in $(0,T) \times \mathbb{T}^3$ $\tilde{\mathbb{P}}$-a.s.
\end{lemma}
By using Corollary \ref{cor:nonlinearLimitsFourth}, we obtain the  following result which is analogous to \cite[Proposition 4.4.12.]{breit2018stoch}.
\begin{lemma}
\label{lem:identifyMomFourthLayer}
The following convergence
\begin{align}
\label{identifyConvectionFourthLayer}
\tilde{\varrho}_\varepsilon \tilde{\mathbf{u}}_\varepsilon \otimes
\tilde{\mathbf{u}}_\varepsilon
\rightharpoonup
\tilde{\varrho} \tilde{\mathbf{u}} \otimes
\tilde{\mathbf{u}} \quad \text{in} \quad L^1(0,T; L^1(\mathbb{T}^3))
\end{align}
holds $\tilde{\mathbb{P}}$-a.s. and
the random distributions  $[\tilde{\varrho}, \tilde{\mathbf{u}},  \tilde{V}, \tilde{W}]$ satisfies 
\begin{equation}
\begin{aligned}
\label{weakMomEqLimitFourthLayer}
&-\int_0^T \partial_t \psi \int_{{\mathbb{T}^3}} \tilde{\varrho}  \tilde{\mathbf{u}} (t) \cdot \bm{\phi} \, \mathrm{d}x    \mathrm{d}t
=
\psi(0)  \int_{{\mathbb{T}^3}} \tilde{\varrho} _0 \tilde{\mathbf{u}} _0 \cdot \bm{\phi} \, \mathrm{d}x  
+  \int_0^T \psi \int_{{\mathbb{T}^3}}  \tilde{\varrho}   \tilde{ \mathbf{u}} \otimes \tilde{\mathbf{u}} : \nabla \bm{\phi} \, \mathrm{d}x  \mathrm{d}t 
\\&-
 \nu^S\int_0^T \psi \int_{{\mathbb{T}^3}} \nabla \tilde{\mathbf{u}} : \nabla \bm{\phi} \, \mathrm{d}x  \mathrm{d}t 
-(\nu^B+ \nu^S )\int_0^T \psi \int_{{\mathbb{T}^3}} \mathrm{div}\,\tilde{\mathbf{u}} \, \mathrm{div}\, \bm{\phi} \, \mathrm{d}x  \mathrm{d}t    
\\&+   
\int_0^T \psi \int_{{\mathbb{T}^3}}   \overline{ p_{\delta}^{ \Gamma}(\tilde{\varrho})}\, \mathrm{div} \, \bm{\phi} \, \mathrm{d}x  \mathrm{d}t
+\int_0^T \psi \int_{{\mathbb{T}^3}} \vartheta \overline{\tilde{\varrho} \nabla \tilde{V}}  \cdot \bm{\phi} \, \mathrm{d}x    \mathrm{d}t
\\&+
\int_0^T \psi \sum_{k\in \mathbb{N}} \int_{{\mathbb{T}^3}}  \overline{ \tilde{\varrho}\, \mathbf{g}_{k}(\tilde{\varrho} , f , \tilde{\varrho} \tilde{\mathbf{u}} )} \cdot \bm{\phi}  \, \mathrm{d}x \mathrm{d}\tilde{\beta}_k  
\end{aligned}
\end{equation}
for all $\psi \in C^\infty_c ([0,T))$ and $\bm{\phi} \in C^\infty(\mathbb{T}^3)$ $\mathbb{P}$-a.s.
\end{lemma}
Lemma \ref{lem:identifyMomFourthLayer} does not completely identify the momemtum equation since the nonlinear pressure, the term containing the electric field and the noise term are only expressed in terms of arbitrary limits  which we do not know to be exactly of the form \eqref{momEqFourth1}. We shall require strong convergence of the density in order to establish that these arbitrary `product limit' terms actually coincide with their corresponding form in \eqref{momEqFourth1}. The following lemma will help us in this direction.
\begin{lemma}
\label{lem:effectiveFluxFourthLayer}
The following convergence
\begin{align*}
\lim_{\varepsilon \rightarrow0}&\int_0^s\int_{\mathbb{T}^3}\big[p^\Gamma_\delta(\tilde{\varrho}_\varepsilon) -(\nu^B+2\nu^S)\mathrm{div}\, \tilde{\mathbf{u}}_\varepsilon \big] \tilde{\varrho}_\varepsilon\, \mathrm{d}x \, \mathrm{d}x
=
\int_0^s\int_{\mathbb{T}^3}\big[\overline{p^\Gamma_\delta( \tilde{\varrho} )} -(\nu^B+2\nu^S)\mathrm{div}\, \tilde{\mathbf{u}} \big]\tilde{\varrho}\, \mathrm{d}x \, \mathrm{d}x
\end{align*}
holds $\tilde{\mathbb{P}}$-a.s. for a.e. $s\in(0,T)$.
\end{lemma}
\begin{proof}
First of all, as a result of the equality of laws established by Lemma \ref{lem:JakowFourthLayer},  it follows from \eqref{momEqItoFourthLayer1} that
\begin{equation}
\begin{aligned}
\label{momEqItoFourthLayer1a}
&
\int_0^s \int_{{\mathbb{T}^3}}    p_{\delta}^{ \Gamma}(\tilde{\varrho}_\varepsilon)\, \Delta \tilde{V}_\varepsilon \, \mathrm{d}x  \mathrm{d}t
=
\int_{{\mathbb{T}^3}} \tilde{\varrho}_\varepsilon  \tilde{\mathbf{u}}_\varepsilon (t) \cdot \nabla \tilde{V}_\varepsilon \, \mathrm{d}x
-
 \int_{{\mathbb{T}^3}} \tilde{\varrho}_{0, \varepsilon} \tilde{\mathbf{u}}_{0,\varepsilon} \cdot \nabla \tilde{V}_{0,\varepsilon} \, \mathrm{d}x     
\\&
-  \int_0^s \int_{{\mathbb{T}^3}}  \tilde{\varrho}_\varepsilon   \tilde{ \mathbf{u}}_\varepsilon \otimes \tilde{\mathbf{u}}_\varepsilon : \nabla^2 \tilde{V}_\varepsilon \, \mathrm{d}x  \mathrm{d}t 
+
 \nu^S\int_0^s\int_{{\mathbb{T}^3}} \nabla \tilde{\mathbf{u}}_\varepsilon : \nabla^2 \tilde{V}_\varepsilon \, \mathrm{d}x  \mathrm{d}t 
\\&
+(\nu^B+ \nu^S )\int_0^s\int_{{\mathbb{T}^3}} \mathrm{div}\,\tilde{\mathbf{u}}_\varepsilon \, \Delta \tilde{V}_\varepsilon \, \mathrm{d}x  \mathrm{d}t    
+
\int_0^s\int_{{\mathbb{T}^3}} \tilde{\varrho}_\varepsilon \tilde{\mathbf{u}}_\varepsilon  \nabla\Delta^{-1}_{\mathbb{T}^3}\mathrm{div}(\tilde{\varrho}_\varepsilon \tilde{\mathbf{u}}_\varepsilon)  \, \mathrm{d}x    \mathrm{d}t
\\&
-\int_0^s\int_{{\mathbb{T}^3}} \vartheta\tilde{\varrho}_\varepsilon \vert  \nabla \tilde{V}_\varepsilon  \vert^2 \, \mathrm{d}x    \mathrm{d}t
-
\int_0^s\sum_{k\in \mathbb{N}} \int_{{\mathbb{T}^3}}  \tilde{ \varrho}_\varepsilon\, \mathbf{g}_{k,\varepsilon}(\tilde{\varrho}_\varepsilon , f, \tilde{\varrho}_\varepsilon \tilde{\mathbf{u}}_\varepsilon ) \cdot \nabla \tilde{V}_\varepsilon  \, \mathrm{d}x \mathrm{d}\tilde{\beta}_{k,\varepsilon}
\\&
-
\varepsilon
\int_0^s\int_{{\mathbb{T}^3}} \nabla(\tilde{\varrho}_\varepsilon \tilde{\mathbf{u}}_\varepsilon):\nabla^2 \tilde{V}_\varepsilon \, \mathrm{d}x   \mathrm{d}t
-  
\varepsilon
\int_0^s\int_{{\mathbb{T}^3}} \tilde{\varrho}_\varepsilon \tilde{\mathbf{u}}_\varepsilon  \nabla \tilde{\varrho}_\varepsilon\, \mathrm{d}x    \mathrm{d}t
=:J_1+ \ldots + J_{10}.
\end{aligned}
\end{equation}
holds $\tilde{\mathbb{P}}$-a.s. Notice that the noise term in now driven by several Brownian motions $\tilde{\beta}_{k,\varepsilon}$ for each $k\in \mathbb{N}$ which is a result of the existence of $\tilde{W}_\varepsilon$ as given by Lemma \ref{lem:JakowFourthLayer}. Also notice that the commutativity of the differential operators $\partial_{x_i}$ and $\partial_{x_j}$ mean that after integrating by parts twice,
\begin{align}
\label{intergrationByParts}
J_4= \nu^S \int_0^s\int_{{\mathbb{T}^3}} \mathrm{div}\,\tilde{\mathbf{u}}_\varepsilon \, \Delta \tilde{V}_\varepsilon \, \mathrm{d}x  \mathrm{d}t    
\end{align}
provided that $\tilde{\mathbf{u}}_\varepsilon$ and $\tilde{V}_\varepsilon$ are regular enough to allow this many integration by parts. Since \eqref{momEqItoFourthLayer1} and thus \eqref{momEqItoFourthLayer1a} was derived through the application of It\^o's lemma after a preliminary regularization step,\eqref{intergrationByParts} holds true indeed. Now similar to \eqref{momEqItoFourthLayer1a}, we use use the generalized It\^o's formula to obtain from Lemma \ref{lem:identifyMomFourthLayer}
\begin{equation}
\begin{aligned}
\label{momEqItoFourthLayer1b}
&
\int_0^s \int_{{\mathbb{T}^3}}   \overline{ p_{\delta}^{ \Gamma}(\tilde{\varrho})}\, \Delta \tilde{V} \, \mathrm{d}x  \mathrm{d}t
=
\int_{{\mathbb{T}^3}} \tilde{\varrho}  \tilde{\mathbf{u}} (t) \cdot \nabla \tilde{V} \, \mathrm{d}x
-
 \int_{{\mathbb{T}^3}} \tilde{\varrho} _0 \tilde{\mathbf{u}} _0 \cdot \nabla V_0 \, \mathrm{d}x     
\\&
-  \int_0^s \int_{{\mathbb{T}^3}}  \tilde{\varrho}    \tilde{\mathbf{u}} \otimes \tilde{\mathbf{u}} : \nabla^2 \tilde{V} \, \mathrm{d}x  \mathrm{d}t 
+
 \nu^S\int_0^s\int_{{\mathbb{T}^3}} \nabla \tilde{\mathbf{u}} : \nabla^2 \tilde{V} \, \mathrm{d}x  \mathrm{d}t 
\\&
+(\nu^B+ \nu^S )\int_0^s\int_{{\mathbb{T}^3}} \mathrm{div}\,\tilde{\mathbf{u}} \, \Delta \tilde{V} \, \mathrm{d}x  \mathrm{d}t    
+
\int_0^s\int_{{\mathbb{T}^3}} \tilde{\varrho} \tilde{\mathbf{u}} \nabla\Delta^{-1}_{\mathbb{T}^3}\mathrm{div}(\tilde{\varrho} \tilde{\mathbf{u}})  \, \mathrm{d}x    \mathrm{d}t
\\&
-
\int_0^s\int_{{\mathbb{T}^3}} \vartheta \overline{\tilde{\varrho}   \nabla \tilde{V} } \cdot \nabla \tilde{V}  \, \mathrm{d}x    \mathrm{d}t
-
\int_0^s\sum_{k\in \mathbb{N}} \int_{{\mathbb{T}^3}}   \overline{\tilde{\varrho}\, \mathbf{g}_{k}(\tilde{\varrho} , f, \tilde{\varrho} \tilde{\mathbf{u}} )} \cdot \nabla \tilde{V}  \, \mathrm{d}x \mathrm{d}\tilde{\beta}_k
=:K_1+ \ldots + K_8.
\end{aligned}
\end{equation}
We now show that $I_9, I_{10} \rightarrow0$ $\tilde{\mathbb{P}}$-a.s. as $\varepsilon\rightarrow0$. This is because since $2< \frac{2\Gamma}{\Gamma-2} \leq3$ (recall that $\Gamma\geq 6$), it follows from H\"older inequality  that $\tilde{\mathbb{P}}$-a.s., the estimate
\begin{equation}
\begin{aligned}
J_9
&= 
\int_0^s\int_{{\mathbb{T}^3}}\big[ \varepsilon \tilde{\varrho}_\varepsilon\, \nabla\tilde{\mathbf{u}}_\varepsilon
+
\sqrt{\varepsilon}\sqrt{\varepsilon}\nabla\tilde{\varrho}_\varepsilon \cdot
\tilde{\mathbf{u}}_\varepsilon
 \big]:\nabla^2 \tilde{V}_\varepsilon \, \mathrm{d}x   \mathrm{d}t
 \leq
\varepsilon
\int_0^s
\Vert \tilde{\varrho}_\varepsilon\Vert_{L^{\Gamma}_x} \Vert \nabla\tilde{\mathbf{u}}_\varepsilon \Vert_{L^2_x}
\Vert \nabla^2 \tilde{V}_\varepsilon \Vert_{L^\frac{2\Gamma}{\Gamma-2}_x}
\mathrm{d}t
\\&
+
 \sqrt{\varepsilon}
\int_0^s
\Vert
\sqrt{\varepsilon}\nabla\tilde{\varrho}_\varepsilon \Vert_{L^2_x}
\Vert
\tilde{\mathbf{u}}_\varepsilon\Vert_{L^\frac{2\Gamma}{\Gamma-2}_x}
\Vert \nabla^2 \tilde{V}_\varepsilon \Vert_{L^\Gamma_x} \mathrm{d}t
 \lesssim
\varepsilon
\Vert \tilde{\varrho}_\varepsilon\Vert_{L^\infty_t L^{\Gamma}_x} \Vert \nabla\tilde{\mathbf{u}}_\varepsilon \Vert_{L^2_tL^2_x}
\Vert \nabla^2 \tilde{V}_\varepsilon \Vert_{L^2_tL^\frac{2\Gamma}{\Gamma-2}_x}
\\&
+
 \sqrt{\varepsilon}
\Vert
\sqrt{\varepsilon}\nabla\tilde{\varrho}_\varepsilon \Vert_{L^2_t L^2_x}
\Vert \nabla
\tilde{\mathbf{u}}_\varepsilon\Vert_{L^2_t L^2_x}
\Vert \nabla^2 \tilde{V}_\varepsilon \Vert_{L^\infty_t L^\Gamma_x} 
\end{aligned}
\end{equation}
holds uniformly in $\varepsilon$. 
Now since the embedding $L^\infty_t L^\Gamma_x \hookrightarrow  L^2_tL^\frac{2\Gamma}{\Gamma-2}_x $ is continuous, it follows from the estimates \eqref{uniBoundFourthLayer1}, \eqref{gradSquaredV} and Lemma \ref{lem:JakowFourthLayer} that
\begin{align}
\label{j9vanishesFourthLayer}
J_9 \lesssim \sqrt{\varepsilon} \rightarrow 0
\end{align}
$\tilde{\mathbb{P}}$-a.s. as $\varepsilon\rightarrow0$.
Again, since $2< \frac{2\Gamma}{\Gamma-2} \leq3$, it follows from Lemma \ref{lem:JakowFourthLayer} that
\begin{equation}
\begin{aligned}
\label{j10vanishesFourthLayer}
J_{10}
&\leq\sqrt{\varepsilon}\int_0^s \Vert \tilde{\varrho}_\varepsilon \Vert_{L^\Gamma_x} \Vert \tilde{\mathbf{u}}_\varepsilon \Vert_{L^\frac{2\Gamma}{\Gamma-2}_x}
\Vert\sqrt{\varepsilon} \nabla \tilde{\varrho}_\varepsilon \Vert_{L^2_x}\, \mathrm{d}t
\lesssim
\sqrt{\varepsilon}\Vert \tilde{\varrho}_\varepsilon \Vert_{L^\infty_t L^\Gamma_x} \Vert \nabla \tilde{\mathbf{u}}_\varepsilon \Vert_{L^2_tL^2_x}
\Vert\sqrt{\varepsilon} \nabla \tilde{\varrho}_\varepsilon \Vert_{L^2_tL^2_x}
\\&
\lesssim
\sqrt{\varepsilon} \rightarrow 0
\end{aligned}
\end{equation}
$\tilde{\mathbb{P}}$-a.s. as $\varepsilon\rightarrow0$.
\\
Now notice that since the given function $f\in L^r_x$ for all $p\in [1,\infty]$, it follows from Lemma \ref{lem:JakowFourthLayer} and \eqref{pressureLimitFourthLayer} that
\begin{equation}
\begin{aligned}
\label{fromVtoRho}
&\lim_{\varepsilon \rightarrow0}\int_0^s\int_{\mathbb{T}^3}\big[p^\Gamma_\delta(\tilde{\varrho}_\varepsilon) -(\nu^B+2\nu^S)\mathrm{div}\, \tilde{\mathbf{u}}_\varepsilon \big] \Delta\tilde{V}_\varepsilon\, \mathrm{d}x \, \mathrm{d}x
\\&
-
\int_0^s\int_{\mathbb{T}^3}\big[\overline{p^\Gamma_\delta(\tilde{\varrho})} -(\nu^B+2\nu^S)\mathrm{div}\, \tilde{\mathbf{u}} \big]\Delta\tilde{V}\, \mathrm{d}x \, \mathrm{d}x
\\&=
\lim_{\varepsilon \rightarrow0}\int_0^s\int_{\mathbb{T}^3}\big[p^\Gamma_\delta(\tilde{\varrho}_\varepsilon) -(\nu^B+2\nu^S)\mathrm{div}\, \tilde{\mathbf{u}}_\varepsilon \big] \tilde{\varrho}_\varepsilon\, \mathrm{d}x \, \mathrm{d}x
\\&
-
\int_0^s\int_{\mathbb{T}^3}\big[\overline{p^\Gamma_\delta(\tilde{\varrho})} -(\nu^B+2\nu^S)\mathrm{div}\, \tilde{\mathbf{u}} \big]\tilde{\varrho}\, \mathrm{d}x \, \mathrm{d}x
\end{aligned}
\end{equation}
Also notice that since $\Gamma>3$, it follows from Lemma \ref{lem:JakowFourthLayer} and the compact embedding $W^{1,\Gamma}_x \xhookrightarrow{cpt} C_x$ that
\begin{align}
\label{strongConvNablaV}
\nabla \tilde{V}_\varepsilon \rightarrow \nabla \tilde{V} \quad \text{in } C([0,T] \times \mathbb{T}^3)
\end{align}
$\tilde{\mathbb{P}}$-a.s. It follows from \cite[Lemma 2.6.6]{breit2018stoch}, \eqref{noiseLimitFourthLayer} and \eqref{strongConvNablaV} that
\begin{equation}
\begin{aligned}
\label{strongConvNoiseFourthLayer}
&\lim_{\varepsilon\rightarrow0}
\int_0^s\sum_{k\in \mathbb{N}} \int_{{\mathbb{T}^3}}  \tilde{ \varrho}_\varepsilon\, \mathbf{g}_{k,\varepsilon}(\tilde{\varrho}_\varepsilon , f , \tilde{\varrho}_\varepsilon \tilde{\mathbf{u}}_\varepsilon ) \cdot \nabla \tilde{V}_\varepsilon  \, \mathrm{d}x \mathrm{d}\tilde{\beta}_{k,\varepsilon} \bigg]
=
\int_0^s\sum_{k\in \mathbb{N}} \int_{{\mathbb{T}^3}}   \overline{\tilde{\varrho}\, \mathbf{g}_{k}(\tilde{\varrho} , f, \tilde{\varrho} \tilde{\mathbf{u}} )} \cdot \nabla \tilde{V}  \, \mathrm{d}x \mathrm{d}\tilde{\beta}_k
\end{aligned}
\end{equation}
$\tilde{\mathbb{P}}$-a.s. and from Lemma \ref{lem:JakowFourthLayer} and \eqref{strongConvNablaV} that
\begin{equation}
\begin{aligned}
\label{strongConvMomIniMomFourthLayer}
&\lim_{\varepsilon \rightarrow0}\bigg[
\int_{{\mathbb{T}^3}} \tilde{\varrho}_\varepsilon  \tilde{\mathbf{u}}_\varepsilon (t) \cdot \nabla \tilde{V}_\varepsilon \, \mathrm{d}x
-
 \int_{{\mathbb{T}^3}} \tilde{\varrho}_{0, \varepsilon} \tilde{\mathbf{u}}_{0,\varepsilon} \cdot \nabla \tilde{V}_{0,\varepsilon} \, \mathrm{d}x    \bigg]
 \\&
 =
 \bigg[
\int_{{\mathbb{T}^3}} \tilde{\varrho}  \tilde{\mathbf{u}} (t) \cdot \nabla \tilde{V} \, \mathrm{d}x
-
 \int_{{\mathbb{T}^3}} \tilde{\varrho} _0 \tilde{\mathbf{u}} _0 \cdot \nabla V_0 \, \mathrm{d}x     \bigg].
\end{aligned}
\end{equation}
Also, \eqref{rhoNablaVLimitFourthLayer} and \eqref{strongConvNablaV} yields
\begin{equation}
\begin{aligned}
\label{strongConvVelElectFourthLayer}
&\lim_{\varepsilon \rightarrow0}\int_0^s\int_{{\mathbb{T}^3}} \vartheta\tilde{\varrho}_\varepsilon \vert  \nabla \tilde{V}_\varepsilon  \vert^2 \, \mathrm{d}x    \mathrm{d}t
 =
\int_0^s\int_{{\mathbb{T}^3}} \vartheta \overline{\tilde{\varrho}   \nabla \tilde{V} } \cdot \nabla \tilde{V}  \, \mathrm{d}x    \mathrm{d}t
\end{aligned}
\end{equation}
$\tilde{\mathbb{P}}$-a.s.
By combining \eqref{momEqItoFourthLayer1a}-- \eqref{momEqItoFourthLayer1b} with \eqref{j9vanishesFourthLayer}--\eqref{fromVtoRho} and \eqref{strongConvNoiseFourthLayer}--\eqref{strongConvVelElectFourthLayer} it follows that
\begin{equation}
\begin{aligned}
\label{showRHSiszeroFourthLayer}
&\lim_{\varepsilon \rightarrow0}\int_0^s\int_{\mathbb{T}^3}\big[p^\Gamma_\delta(\tilde{\varrho}_\varepsilon) -(\nu^B+2\nu^S)\mathrm{div}\, \tilde{\mathbf{u}}_\varepsilon \big] \tilde{\varrho}_\varepsilon\, \mathrm{d}x \, \mathrm{d}x
\\&
-
\int_0^s\int_{\mathbb{T}^3}\big[\overline{p^\Gamma_\delta(\tilde{\varrho})} -(\nu^B+2\nu^S)\mathrm{div}\, \tilde{\mathbf{u}} \big]\tilde{\varrho}\, \mathrm{d}x \, \mathrm{d}x
\\
=&\lim_{\varepsilon \rightarrow0}
\int_0^s\int_{{\mathbb{T}^3}}\big[ \tilde{\varrho}_\varepsilon \tilde{\mathbf{u}}_\varepsilon  \nabla\Delta^{-1}_{\mathbb{T}^3}\mathrm{div}(\tilde{\varrho}_\varepsilon \tilde{\mathbf{u}}_\varepsilon) 
-
\tilde{\varrho}_\varepsilon   \tilde{ \mathbf{u}}_\varepsilon \otimes \tilde{\mathbf{u}}_\varepsilon : \nabla^2 \tilde{V}_\varepsilon \big]
 \, \mathrm{d}x    \mathrm{d}t
\\&+
\int_0^s\int_{{\mathbb{T}^3}} \big[ \tilde{\varrho} \tilde{\mathbf{u}} \nabla\Delta^{-1}_{\mathbb{T}^3}\mathrm{div}(\tilde{\varrho} \tilde{\mathbf{u}})  
-
\tilde{\varrho}    \tilde{\mathbf{u}} \otimes \tilde{\mathbf{u}} : \nabla^2 \tilde{V} \big]
\, \mathrm{d}x    \mathrm{d}t..
\end{aligned}
\end{equation}
$\tilde{\mathbb{P}}$-a.s. We now want to show that the right-hand side of \eqref{showRHSiszeroFourthLayer} is zero. Since the periodic inverse Laplacian commutes with spatial derivatives and
\begin{align*}
\pm \nabla^2 V = \nabla^2 \Delta^{-1}_{\mathbb{T}^3}(\varrho -f),
\end{align*}
in coordinates form, the first integral term of the right-hand side of \eqref{showRHSiszeroFourthLayer} is
\begin{equation}
\begin{aligned}
&\int_0^s\int_{{\mathbb{T}^3}}\tilde{u}_\varepsilon^i \big[ \tilde{\varrho}_\varepsilon  \partial_{x_i}\Delta^{-1}_{\mathbb{T}^3}\partial_{x_j}(\tilde{\varrho}_\varepsilon \tilde{u}_\varepsilon^j) 
-
\tilde{\varrho}_\varepsilon   \tilde{u}_\varepsilon^j \partial_{x_i}\Delta^{-1}_{\mathbb{T}^3}\partial_{x_j} (\tilde{\varrho}_\varepsilon -f) \big]
 \, \mathrm{d}x    \mathrm{d}t
 \\
 &=
 \int_0^s\int_{{\mathbb{T}^3}}\tilde{u}_\varepsilon^i \big[ \tilde{\varrho}_\varepsilon  \partial_{x_i}\Delta^{-1}_{\mathbb{T}^3}\partial_{x_j}(\tilde{\varrho}_\varepsilon \tilde{u}_\varepsilon^j) 
-
\tilde{\varrho}_\varepsilon   \tilde{u}_\varepsilon^j \partial_{x_i}\Delta^{-1}_{\mathbb{T}^3}\partial_{x_j} \tilde{\varrho}_\varepsilon  \big]
 \, \mathrm{d}x    \mathrm{d}t
\\&+
  \int_0^s\int_{{\mathbb{T}^3}}
\tilde{\varrho}_\varepsilon  \tilde{u}_\varepsilon^i \tilde{u}_\varepsilon^j \partial_{x_i} \partial_{x_j} \Delta^{-1}_{\mathbb{T}^3} f
 \, \mathrm{d}x    \mathrm{d}t =:I^\varepsilon_1 + I^\varepsilon_2
\end{aligned}
\end{equation}
$\tilde{\mathbb{P}}$-a.s. The same result holds for the second integral term of the right-hand side of \eqref{showRHSiszeroFourthLayer}, i.e., an obvious comparison of the form
\begin{equation}
\begin{aligned}
&\int_0^s\int_{{\mathbb{T}^3}}\tilde{u}^i \big[ \tilde{\varrho} \, \partial_{x_i}\Delta^{-1}_{\mathbb{T}^3} \partial_{x_j}(\tilde{\varrho} \tilde{u}^j) 
-
\tilde{\varrho}   \tilde{u}^j \partial_{x_i}\Delta^{-1}_{\mathbb{T}^3} \partial_{x_j} (\tilde{\varrho} -f) \big]
 \, \mathrm{d}x    \mathrm{d}t
=I_1 + I_2
\end{aligned}
\end{equation}
hold $\tilde{\mathbb{P}}$-a.s. Since $f\in L^p(\mathbb{T}^3)$  for all $p\in [1,\infty]$, it follows from the compactness result for velocity established in Lemma \ref{lem:JakowFourthLayer}, \eqref{identifyConvectionFourthLayer} and the Div--Curl lemma \cite[Theorem 10.27]{feireisl2009singular} (see also \cite[Lemma A.1.11]{breit2018stoch} for the periodic case)  that
\begin{align}
\lim_{\varepsilon \rightarrow 0}I^\varepsilon_1 + I^\varepsilon_2 =I_1 + I_2
\end{align}
$\tilde{\mathbb{P}}$-a.s. provided that we choose $\Gamma \geq 5$.
By using this convergence, the right-hand of  equation \eqref{showRHSiszeroFourthLayer} is zero and the proof is done.
\end{proof}
Now since $\tilde{\varrho}\in L^\infty_tL^\Gamma_x$ $\tilde{\mathbb{P}}$-a.s. with $\Gamma \geq2$ and $\tilde{\mathbf{u}} \in L^2_tW^{1,2}_x$  $\tilde{\mathbb{P}}$-a.s., we may use the renormalized theory of DiPerna-Lions \cite{diperna1989ordinary} to get that $(\tilde{\varrho}, \tilde{\mathbf{u}})$ satisfies \eqref{contEqFourth1} in the renormalized sense.
Since the map $\varrho \mapsto \varrho \log \varrho$ is strictly convex, it follows that
\begin{align}
\label{strongDensityConvFourthLayer}
\tilde{\varrho}_\varepsilon \rightarrow \tilde{\varrho}, \quad\text{in} \quad L^r_t L^1_x
\end{align}
and hence,
\begin{align}
\label{strongDeltaVConvFourthLayer}
\Delta\tilde{V}_\varepsilon \rightarrow \Delta\tilde{V}, \quad\text{in} \quad L^r_t L^1_x
\end{align}
$\tilde{\mathbb{P}}$-a.s. for any $r \in [1, \infty)$. Furthermore, as in \cite[4.177]{breit2018stoch}, we can use the Lipschitz continuity of $\mathbf{g}_k$ and \eqref{strongDensityConvFourthLayer} to that
\begin{align}
\int_{{\mathbb{T}^3}}  \tilde{\varrho}_\varepsilon\, \mathbf{g}_{k}(\tilde{\varrho}_\varepsilon , f, \tilde{\varrho}_\varepsilon \tilde{\mathbf{u}}_\varepsilon ) \cdot \bm{\phi}  \, \mathrm{d}x
\rightarrow
\int_{{\mathbb{T}^3}}  \tilde{\varrho}\, \mathbf{g}_{k}(\tilde{\varrho} , f, \tilde{\varrho} \tilde{\mathbf{u}} ) \cdot \bm{\phi}  \, \mathrm{d}x \quad \text{a.e. in } (0,T)
\end{align}
holds $\mathbb{P}$-a.s. for any $\bm{\phi} \in C^\infty(\mathbb{T}^3)$ from which we infer that
\begin{align}
\overline{  \tilde{\varrho}\, \mathbf{g}_{k}(\tilde{\varrho} , f, \tilde{\varrho} \tilde{\mathbf{u}} ) }
=
  \tilde{\varrho}\, \mathbf{g}_{k}(\tilde{\varrho} , f, \tilde{\varrho} \tilde{\mathbf{u}} ) \quad \text{a.e. in } \tilde{\Omega} \times (0,T) \times \mathbb{T}^3.
\end{align}
We can therefore conclude with the following lemma
\begin{lemma}
\label{lem:identifyMomFourthLayer1}
The random distributions  $[\tilde{\varrho}, \tilde{\mathbf{u}},  \tilde{V}, \tilde{W}]$  which has been shown to satisfy \eqref{weakMomEqLimitFourthLayer} in Lemma \ref{lem:identifyMomFourthLayer} further satisfy \eqref{momEqFourth1} 
for all $\psi \in C^\infty_c ([0,T))$ and $\bm{\phi} \in C^\infty(\mathbb{T}^3)$ $\mathbb{P}$-a.s.
\end{lemma}
To complete the proof of Theorem \ref{thm:mainFourthLayer}, we identify the limit energy inequality.
\begin{lemma}
\label{lem:identifyEnerFourthLayer}
The random distributions  $[\tilde{\varrho}, \tilde{\mathbf{u}},  \tilde{V}]$ satisfies \eqref{energyInequalityFourthLayer} for all $\psi \in C^\infty_c ([0,T))$, $\psi \geq0$ $\mathbb{P}$-a.s.
\end{lemma}
Since we have strong convergence of the density sequence \eqref{strongDensityConvFourthLayer}, the only issue is the identification of the noise term (last term) in \eqref{energyInequalityFourthLayer}. Since $f \in L^\infty_x$ is given, the argument follow exactly as in \cite[Proposition 4.4.13.]{breit2018stoch}.

\section{The fifth approximation layer}
\label{sec:fifthLayer}
Under the assumptions of Theorem \ref{thm:one},  we establish  the existence of a solution in the sense of Definition \ref{def:martSolutionExis} to the following system 
\begin{align}
&\mathrm{d} \varrho  + \mathrm{div}(\varrho  \mathbf{u})  \, \mathrm{d}t =0,
\label{contEqFifth1}
\\
&\mathrm{d}(\varrho  \mathbf{u} ) +
\big[
\mathrm{div} (\varrho  \mathbf{u}  \otimes \mathbf{u} ) +  \nabla   p(\varrho) \big]\, \mathrm{d}t 
=
\big[ \nu^S \Delta \mathbf{u} 
\nonumber\\&
\quad +(\nu^B + \nu^S)\nabla \mathrm{div} \mathbf{u} + \vartheta \varrho  \nabla V  \big] \, \mathrm{d}t   
+ \sum_{k\in \mathbb{N}} \varrho\,   \mathbf{g}_{k}(\varrho , f, \varrho \mathbf{u} ) \, \mathrm{d}\beta_k ,
\label{momEqFifth1}
\\
&\pm \Delta V  = \varrho  - f, \label{elecFieldFifth1}
\end{align}
by passing to the limit $\delta \rightarrow 0$ in \eqref{contEqFourth1}--\eqref{momEqFourth1}.

\subsection{Construction of law}
\label{sec:lawFifthLayer}
Before we start the main arguments, we first construct a law $\Lambda$ satisfying the assumptions of Theorem \ref{thm:one}. This will be done from the approximation laws $\Lambda_\delta$ satisfying Theorem \ref{thm:mainFourthLayer} such that $\Lambda_\delta \xrightharpoonup{*} \Lambda$ in the sense of measures on $[L^1_x]^3$.
To see this, let $\Lambda$ is Borel probability measure on  $[L^1_x]^3$. Since $[L^1_x]^3$ is a  Polish space, by \cite[Corollary 2.6.4]{breit2018stoch}, there exists a stochastic basis $(\Omega ,\mathscr{F} ,(\mathscr{F} _t)_{t\geq0},\mathbb{P} )$ containing some $[L^1_x]^3$-valued $\mathscr{F}_0$-measurable random variables $(\varrho_0, \mathbf{m}_0, f)$
having the law $\Lambda$. First of all, since these random variables have law $\Lambda$,
\begin{align*}
f \leq \overline{f}
\end{align*}
holds $\mathbb{P}$-a.s. for a deterministic constant $\overline{f}>0$. It follows from this boundedness and the fact that $\mathbb{T}^3$ is of finite measure that $f\in L^r_x$ for all $r\in [1,\infty]$.
Now set $\varrho_{0,\delta} =\varrho_0$ for all $\delta>0$ so that
\begin{align*}
0< M \leq \varrho_{0,\delta} \leq M^{-1}
\end{align*}
holds $\mathbb{P}$-a.s. for a $\delta$-uniform deterministic constant $M>0$. Since $\varrho_0$ is bounded, it follows that $\mathbb{P}$-a.s. $\varrho_{0,\delta} \in L^\Gamma_x$  and in particular for $\gamma\leq \Gamma$, we have that
\begin{align*}
\varrho_{0,\delta} \rightarrow \varrho_0 \quad \text{in} \quad L^\gamma_x \quad \mathbb{P}\text{-a.s.}
\end{align*}
If we also set $V_{0,\delta} =V_0 =\pm \Delta^{-1}_{\mathbb{T}^3}(\varrho_0-f)$, then we have that
\begin{align}
\label{momentKineti}
\mathbb{E}\bigg[ \int_{\mathbb{T}^3} \vartheta \vert \nabla V_{0,\delta} \vert^2 \, \mathrm{d}x \bigg]^{p_0} \lesssim 1
\end{align}
uniformly in $\delta$ for all $p_0\in [1,p]$ and that
\begin{align}
\nabla V_{0,\delta} \rightarrow  \nabla V_{0} \quad &\text{in}\quad L^{p_0}(\Omega; L^2_x) \quad \text{for all}\quad  p_0\in [1,2p].
\end{align}
Now set
\begin{align}
\tilde{\mathbf{m}}_{0,\delta} = 
  \begin{cases} 
   \mathbf{m}_0\sqrt{\frac{\varrho_{0,\delta}}{\varrho_0}} & \text{if } \varrho_0> 0 \\
   0       & \text{if } \varrho= 0
  \end{cases}
\end{align}
so that
\begin{align}
\label{momentKineti}
\mathbb{E}\bigg[ \int_{\mathbb{T}^3} \frac{1}{2} \frac{\vert \tilde{\mathbf{m}}_{0,\delta} \vert^2}{\varrho_{0,\delta}} \, \mathrm{d}x \bigg]^{p_0} \lesssim 1
\end{align}
uniformly in $\delta$ for all $p_0\in [1,p]$.  Now if we represent the smooth version (after mollification say,)  $\frac{ \tilde{\mathbf{m}}_{0,\delta} }{\sqrt{\varrho_{0,\delta}}}$ by $s_\delta$ so that
\begin{align*}
\bigg[\frac{ \tilde{\mathbf{m}}_{0,\delta} }{\sqrt{\varrho_{0,\delta}}} - s_\delta \bigg] \rightarrow 0 \quad \text{in}\quad L^{p_0}(\Omega; L^2_x)
\end{align*}
for all $p_0\in [1,2p]$, then for ${\mathbf{m}}_{0,\delta} = s_\delta \sqrt{\varrho_{0,\delta}}$, we have that
\begin{align}
\label{momentKineti1}
\mathbb{E}\bigg[ \int_{\mathbb{T}^3} \frac{1}{2} \frac{\vert \mathbf{m}_{0,\delta} \vert^2}{\varrho_{0,\delta}} \, \mathrm{d}x \bigg]^{p_0} \lesssim 1
\end{align}
uniformly in $\delta$ for all $p_0\in [1,p]$. In addition (recall again that $\varrho_{0,\delta} =\varrho_0$),
\begin{align}
\mathbf{m}_{0,\delta} \rightarrow  \mathbf{m}_0\quad &\text{in}\quad L^{p_0}(\Omega; L^1_x) \quad \text{for all}\quad  p_0\in [1,p],
\\
\frac{\mathbf{m}_{0,\delta}}{\sqrt{\varrho_{0,\delta}}} \rightarrow  \frac{\mathbf{m}_0}{\sqrt{\varrho_{0}}} \quad &\text{in}\quad L^{p_0}(\Omega; L^2_x) \quad \text{for all}\quad  p_0\in [1,2p].
\end{align}
By using the inequality $ (a+b)^r \lesssim_r a^r +b^r$, we obtain the estimate
\begin{align}
\mathbb{E}\bigg[ \int_{\mathbb{T}^3}\bigg( \frac{1}{2} \frac{\vert \mathbf{m}_{0,\delta} \vert^2}{\varrho_{0,\delta}} 
+
P^\Gamma_\delta(\varrho_{0,\delta})  \pm 
\vartheta \vert \nabla V_{0,\delta} \vert^2
 \bigg)\, \mathrm{d}x \bigg]^{p_0} \lesssim_{p_0} 1
\end{align}
uniformly in $\delta$ by collecting the suitable information above. Now if we set $\Lambda_\delta = \mathbb{P}\circ (\varrho_{0,\delta}, \mathbf{m}_{0,\delta},f)^{-1}$, then  we observe  that $\Lambda_\delta$ satisfies the corresponding assumptions of Theorem \ref{thm:mainFourthLayer} uniformly in $\delta$ and in addition, $\Lambda_\delta \xrightharpoonup{*} \Lambda$ in the sense of measures on $[L^1_x]^3$.
\subsection{Uniform estimates}
\label{sec:UniformEstFifthLayer}
Given that the data $(\varrho_{0,\delta}, \mathbf{m}_{0,\delta}, f)$ just constructed above satisfies the assumptions of Theorem \ref{thm:mainFourthLayer}, there exists a dissipative martingale solution $[(\Omega ,\mathscr{F} ,(\mathscr{F} _t)_{t\geq0},\mathbb{P} );\varrho_\delta , \mathbf{u}_\delta , V_\delta , W   ]$ of \eqref{contEqFourth1}--\eqref{elecFieldFourth1} in the sense of Definition \ref{def:martSolFourthLayer} with $\Lambda_\delta$ as its law. As usual, without loss of generality, we consider a single stochastic basis and a single Wiener process for each $\delta>0$.
 Since the solution satisfies the energy estimate \eqref{energyInequalityFourthLayer}, by approximating the characteristic function $\chi_{[0,s]}$ for any $s\in(0,T]$, by a sequence of test functions $\psi\in C^\infty_c([0,T))$, it follows that
\begin{equation}
\begin{aligned}
\label{energyInequalityFifthLayerNonDissi}
\int_{{\mathbb{T}^3}}&\bigg[ \frac{1}{2} \varrho_\delta  \vert \mathbf{u}_\delta  \vert^2  
+ P^\Gamma_\delta(\varrho_\delta ) 
\pm
\vartheta \vert \nabla V_\delta \vert^2 
\bigg](S)\,\mathrm{d}x 
+\nu^S
\int_0^s \int_{{\mathbb{T}^3}}\vert  \nabla \mathbf{u}_\delta  \vert^2 \,\mathrm{d}x  \,\mathrm{d}x
\\&+(\nu^B+\nu^S)
\int_0^s \int_{{\mathbb{T}^3}}\vert  \mathrm{div}\, \mathbf{u}_\delta  \vert^2 \,\mathrm{d}x  \,\mathrm{d}t
\leq
 \int_{{\mathbb{T}^3}}\bigg[\frac{1}{2}\frac{\vert \mathbf{m} _{0, \delta} \vert^2  }{\varrho _{0,\delta}} 
+ P^\Gamma_\delta(\varrho _{0, \delta}) 
\pm
\vartheta \vert \nabla V_{0, \delta} \vert^2
 \bigg]\,\mathrm{d}x
\\&
+\frac{1}{2}\int_0^s \int_{{\mathbb{T}^3}}
 \varrho_\delta
 \sum_{k\in\mathbb{N}} \big\vert 
 \mathbf{g}_{k, \varepsilon}(x,\varrho_\delta ,f,\mathbf{m}_\delta  ) 
\big\vert^2\,\mathrm{d}x\,\mathrm{d}t
+\int_0^s  \int_{{\mathbb{T}^3}}   \varrho_\delta \mathbf{u}_\delta \cdot  \mathbf{G}_\varepsilon(\varrho_\delta , f, \mathbf{m}_\delta )\,\mathrm{d}x\,\mathrm{d}W ;
\end{aligned}
\end{equation}
holds $\mathbb{P}$-a.s. for a.e. $s\in(0,T]$. Given that
\begin{align}
\label{calEzeroDelta}
\mathcal{E}_{0,\delta}:=\frac{1}{2}\frac{\vert \mathbf{m} _{0, \delta} \vert^2  }{\varrho _{0,\delta}} 
+ P^\Gamma_\delta(\varrho _{0, \delta}) 
\pm
\vartheta \vert \nabla V_{0, \delta} \vert^2
\end{align}
is $\delta$-uniformly bounded in $L^p(\Omega ; L^1_x)$, just as in Section \ref{sec:UniformEstFourthLayer} we obtain the following estimates
\begin{equation}
\begin{aligned}
\label{uniBoundFifthLayer1}
&\mathbb{E} \bigg\vert \sup_{t\in [0,T]}\big\Vert \varrho_\delta \vert \mathbf{u}_\delta\vert^2 \big\Vert_{L^1_x} \bigg\vert^p 
+
\mathbb{E} \bigg\vert \sup_{t\in [0,T]} \Vert \varrho_\delta  \Vert_{L^\gamma_x}^\gamma \bigg\vert^p 
+
\mathbb{E} \bigg\vert \sup_{t\in [0,T]}\delta\, \Vert \varrho_\delta  \Vert_{L^\Gamma_x}^\Gamma \bigg\vert^p 
+
\mathbb{E} \bigg\vert \sup_{t\in [0,T]} \Vert \nabla V_\delta \Vert_{L^2_x}^2 \bigg\vert^p
\\&+
\mathbb{E} \bigg\vert \sup_{t\in [0,T]} \Vert \Delta V_\delta \Vert_{L^\gamma_x}^\gamma \bigg\vert^p 
+
\mathbb{E} \Vert  \mathbf{u}_\delta \Vert_{L^2_{t}W^{1,2}_x}^{2p}  
+
\mathbb{E} \bigg\vert \sup_{t\in [0,T]}\big\Vert \varrho_\delta \mathbf{u}_\delta \big\Vert_{L^\frac{2\gamma}{\gamma+1}_x}^\frac{2\gamma}{\gamma+1} \bigg\vert^p 
\lesssim_{ \Gamma, p,\mathcal{E}_{0,\delta}} 1
\end{aligned}
\end{equation}
which holds uniformly in $\delta$. Furthermore, it follows from the continuity equation \eqref{contEqFourth1} and \eqref{lawMainThmFourthLayer} that
\begin{align}
\label{massConserveFifth}
\Vert \varrho_\delta(t) \Vert_{L^1_x} = \Vert \varrho_{0,\delta} \Vert_{L^1_x} \leq M^{-1}
\end{align}
for any $t\in [0,T]$.
\subsection{Improved pressure estimate}
\label{sec:PressureEstFifthLay}
Now since $(\varrho_\delta, \mathbf{u}_\delta)$ satisfies the renormalized continuity equation 
\begin{equation}
\begin{aligned}
&0=\int_0^T\int_{{\mathbb{T}^3}}  b(\varrho_\delta )\,\partial_t\phi \,\mathrm{d}x \, \mathrm{d}t +
\int_{{\mathbb{T}^3}}  b\big(\varrho_\delta (0)\big)\,\phi(0) \,\mathrm{d}x 
\\&
+ \int_0^T\int_{{\mathbb{T}^3}} \big[b(\varrho_\delta ) \mathbf{u}_\delta  \big] \cdot \nabla\phi \, \mathrm{d}x \, \mathrm{d}t  
- \int_0^T\int_{{\mathbb{T}^3}} \big[ \big(  b'(\varrho_\delta ) \varrho_\delta   
-
 b(\varrho_\delta )\big)\mathrm{div}\mathbf{u}_\delta \big]\,  \phi\, \mathrm{d}x \, \mathrm{d}t
\end{aligned}
\end{equation} 
$\mathbb{P}$-a.s. or any $\phi\in C^\infty_c([0,T)\times{\mathbb{T}^3})$ and $b \in C^1_b(\mathbb{R})$ such that $ b'(z)=0$ for all $z\geq M_b$, it follows that for $b(\varrho_\delta)= \varrho^\Theta_\delta$ where $0< \Theta \leq \frac{1}{3}$, we have that
\begin{equation}
\begin{aligned}
&
-
\int_{{\mathbb{T}^3}}  \varrho^\Theta_\delta(0)\,\phi(0) \,\mathrm{d}x
-
\int_0^s\int_{{\mathbb{T}^3}}  \varrho^\Theta_\delta\,\partial_t\phi \,\mathrm{d}x \, \mathrm{d}t 
\\&
= \int_0^s\int_{{\mathbb{T}^3}} \big[\varrho^\Theta_\delta \mathbf{u}_\delta  \big] \cdot \nabla\phi \, \mathrm{d}x \, \mathrm{d}t  
- (\Theta-1)\int_0^s\int_{{\mathbb{T}^3}} \big[\varrho^\Theta_\delta \mathrm{div}\mathbf{u}_\delta \big]\,  \phi\, \mathrm{d}x \, \mathrm{d}t
\end{aligned}
\end{equation} 
holds $\mathbb{P}$-a.s. for any $\phi\in C^\infty_c([0,T) \times{\mathbb{T}^3})$. Now for $\phi = \Delta^{-1}_{\mathbb{T}^3}\mathrm{div}\psi$ where $\psi\in C^\infty_c([0,T)\times{\mathbb{T}^3})$, we obtain
\begin{equation}
\begin{aligned}
\label{contEqItoFifthLayer}
\int_{{\mathbb{T}^3}} \nabla \Delta^{-1}_{\mathbb{T}^3}  \varrho^\Theta_\delta(0)\,\psi(0) \,\mathrm{d}x
&+
\int_0^s\int_{{\mathbb{T}^3}}  \nabla \Delta^{-1}_{\mathbb{T}^3} \varrho^\Theta_\delta\,\partial_t\psi \,\mathrm{d}x \, \mathrm{d}t 
=
 \int_0^s\int_{{\mathbb{T}^3}}  \nabla \Delta^{-1}_{\mathbb{T}^3}\mathrm{div}\big[\varrho^\Theta_\delta \mathbf{u}_\delta  \big] \psi \, \mathrm{d}x \, \mathrm{d}t  
\\&+
(\Theta-1)\int_0^s\int_{{\mathbb{T}^3}} \nabla \Delta^{-1}_{\mathbb{T}^3} \big[\varrho^\Theta_\delta \mathrm{div}\,\mathbf{u}_\delta \big]\,  \psi\, \mathrm{d}x \, \mathrm{d}t
\end{aligned}
\end{equation} 
$\mathbb{P}$-a.s. 
 Also, as the momemtum equation \eqref{momEqFourth1} is only satisfied weakly in the sense of \eqref{weakMomEqFourthLayer}, a suitable approximation of the characteristic function $\chi_{[0,s]}$ for any $s\in(0,T]$, by a sequence of test functions $\psi\in C^\infty_c([0,T))$ yields
\begin{equation}
\begin{aligned}
\label{momEqItoFifthLayer}
&\int_{{\mathbb{T}^3}} \varrho_\delta    \mathbf{u}_\delta   (t) \cdot \bm{\phi} \, \mathrm{d}x   
=
 \int_{{\mathbb{T}^3}} \varrho_{\delta,0} \mathbf{u} _{0,\delta} \cdot \bm{\phi} \, \mathrm{d}x  
+  \int_0^s \int_{{\mathbb{T}^3}}  \varrho_\delta      \mathbf{u}_\delta   \otimes \mathbf{u}_\delta   : \nabla \bm{\phi} \, \mathrm{d}x  \mathrm{d}t 
\\&-
 \nu^S\int_0^s\int_{{\mathbb{T}^3}} \nabla \mathbf{u}_\delta   : \nabla \bm{\phi} \, \mathrm{d}x  \mathrm{d}t 
-(\nu^B+ \nu^S )\int_0^s\int_{{\mathbb{T}^3}} \mathrm{div}\,\mathbf{u}_\delta   \, \mathrm{div}\, \bm{\phi} \, \mathrm{d}x  \mathrm{d}t    
\\&+   
\int_0^s \int_{{\mathbb{T}^3}}    p_{\delta}^{ \Gamma}(\varrho_\delta  )\, \mathrm{div} \, \bm{\phi} \, \mathrm{d}x  \mathrm{d}t
+\int_0^s\int_{{\mathbb{T}^3}} \vartheta\varrho_\delta    \nabla V_\delta    \cdot \bm{\phi} \, \mathrm{d}x    \mathrm{d}t
\\&+
\int_0^s\sum_{k\in \mathbb{N}} \int_{{\mathbb{T}^3}}   \varrho_\delta  \, \mathbf{g}_{k,\varepsilon}(\varrho_\delta   , f, \varrho_\delta   \mathbf{u}_\delta   ) \cdot \bm{\phi}  \, \mathrm{d}x \mathrm{d}\beta_k  
\end{aligned}
\end{equation}
$\mathbb{P}$-a.s. Since both \eqref{contEqItoFifthLayer} and \eqref{momEqItoFifthLayer} are in their distributional form, after a long and tedious regularization procedure (which we do not show), we obtain from I\^o's lemma, 
\begin{equation}
\begin{aligned}
\label{momEqItoFifthLayer1}
&
\int_0^s \int_{{\mathbb{T}^3}}    p_{\delta}^{ \Gamma}(\varrho_\delta  )\, \varrho^\Theta_\delta   \, \mathrm{d}x  \mathrm{d}t
=
\int_{{\mathbb{T}^3}} \varrho_\delta   \mathbf{u}_\delta  (t) \cdot \nabla \Delta^{-1}_{\mathbb{T}^3}\varrho^\Theta_\delta \, \mathrm{d}x
-
 \int_{{\mathbb{T}^3}} \varrho_{0,\delta} \mathbf{u}_{0,\delta} \cdot \nabla \Delta^{-1}_{\mathbb{T}^3}\varrho^\Theta_\delta(0) \, \mathrm{d}x     
\\&-
\int_0^s \int_{{\mathbb{T}^3}}  \varrho_\delta     \mathbf{u}_\delta  \otimes \mathbf{u}_\delta  : \nabla^2 \Delta^{-1}_{\mathbb{T}^3}\varrho^\Theta_\delta \, \mathrm{d}x  \mathrm{d}t 
+
(\nu^B+ \nu^S )\int_0^s\int_{{\mathbb{T}^3}} \mathrm{div}\,\mathbf{u}_\delta  \,\varrho^\Theta_\delta \, \mathrm{d}x  \mathrm{d}t   
\\&+
 \nu^S\int_0^s\int_{{\mathbb{T}^3}} \nabla \mathbf{u}_\delta  : \nabla^2 \Delta^{-1}_{\mathbb{T}^3}\varrho^\Theta_\delta \, \mathrm{d}x  \mathrm{d}t  
-
\int_0^s\int_{{\mathbb{T}^3}} \vartheta\varrho_\delta    \nabla V \, \nabla \Delta^{-1}_{\mathbb{T}^3}\varrho^\Theta_\delta \, \mathrm{d}x    \mathrm{d}t
\\&+
\int_0^s\int_{{\mathbb{T}^3}} \varrho_\delta \mathbf{u}_\delta  \nabla \Delta^{-1}_{\mathbb{T}^3}\mathrm{div}\big[\varrho^\Theta_\delta \mathbf{u}_\delta  \big]  \, \mathrm{d}x    \mathrm{d}t
+
(\Theta-1) \int_0^s\int_{{\mathbb{T}^3}}\varrho_\delta  \mathbf{u}_\delta \nabla \Delta^{-1}_{\mathbb{T}^3} \big[\varrho^\Theta_\delta \mathrm{div}\mathbf{u}_\delta \big]
 \, \mathrm{d}x    \mathrm{d}t
\\&-
\int_0^s\sum_{k\in \mathbb{N}} \int_{{\mathbb{T}^3}}   \varrho_\delta \, \mathbf{g}_{k,\varepsilon}(\varrho_\delta  , f, \varrho_\delta  \mathbf{u}_\delta  ) \cdot \nabla \Delta^{-1}_{\mathbb{T}^3}\varrho^\Theta_\delta  \, \mathrm{d}x \mathrm{d}\beta_k
=:I_1+ \ldots + I_9.
\end{aligned}
\end{equation}
%
%
Now since the embedding $W^{1,r}_x \hookrightarrow L^\infty_x$ is continuous for $r \geq 3$, it follows from \eqref{massConserveFifth} and the fact that $0< \Theta \leq \frac{1}{3}$,
\begin{align}
\label{massConserContOperatorFifth}
\sup_{t\in [0,T]}\big\Vert  \nabla \Delta^{-1}_{\mathbb{T}^3}\varrho^\Theta_\delta(t)  \big\Vert_{L^\infty_x} 
\lesssim
\Vert\varrho^\Theta_{0,\delta}  \Vert_{L^q_x} \leq M^{-\Theta}, \quad \mathbb{P}\text{-a.s. where } q=\frac{1}{\Theta}
\end{align}
uniform in $\delta$. By using \eqref{massConserContOperatorFifth}, we can replicate the estimates in \eqref{sec:PressureEstFourthLay} so that we obtain from \eqref{momEqItoFifthLayer1},
\begin{equation}
\begin{aligned}
\label{pressureEstFifthLayer1}
\mathbb{E}\bigg\vert \int_0^t \int_{{\mathbb{T}^3}}    p_{\delta}^{ \Gamma}(\varrho_\delta  )\, \varrho^\Theta_\delta \, \mathrm{d}x\,  \mathrm{d}t \bigg\vert^p \lesssim_{p,k, \Gamma, \gamma, \mathcal{E}_{0,\delta}, \overline{f}, \Theta, M}1
\end{aligned}
\end{equation}
uniformly in $\delta$ for all $p\geq2$. We remind the reader of \eqref{higherPressure}, i.e.,
\begin{align}
p_{\delta}^{ \Gamma}(\varrho_\delta  )\, \varrho^\Theta_\delta
=
\big[p(\varrho_\delta  )+\delta(\varrho_\delta +\varrho^\Gamma_\delta) \big]\, \varrho^\Theta_\delta
\end{align}
and also of \eqref{calEzeroDelta} which states that $\mathcal{E}_{0,\delta}$ is uniformly bounded in $\delta$.

\subsection{Compactness}
\label{sec:compactnessFifthLayer}
We now define the following spaces 
\begin{align*}
\chi_{\varrho_0} = L^\gamma_x, \, \quad \chi_{\mathbf{m}_0}= L^1_x, \, \quad \chi_{\frac{\mathbf{m}_0}{\sqrt{\varrho_0}}}= L^2_x,
\, \quad
\chi_{ \mathbf{u} } =  \big(L^2\big(0,T;W^{1,2}_x\big), w\big),
\, \quad
\chi_V = C_w\left([0,T];W^{2,\gamma}_x\right),
\\
\chi_W= C\left([0,T];\mathfrak{U}_0\right) ,
\, \quad
\chi_{\varrho\mathbf{u}} =C_w\big([0,T]; L^\frac{2\gamma}{\gamma+1} _x\big) \cap C\big([0,T]; W^{-k,2}_x\big), 
\\
\chi_\varrho = C_w \left([0,T];L^{\gamma}_x\right)
\cap
\big(L^{\gamma+\Theta}_{t,x}, w\big)
 \quad
\chi_{\mathcal{E}}= \big(L^\infty_t; \mathcal{M}_b(\mathbb{T}^3), w^*\big),
 \quad
\chi_{\nu}= \big(L^\infty\big((0,T)\times \mathbb{T}^3; \mathfrak{P}(\mathbb{R}^{20})\big) , w^*\big)
\end{align*}
for $\gamma>\frac{3}{2}$ and $k> \frac{3}{2}$. We now let  $\mu_{\varrho_{0,\delta}}$, $\mu_{\mathbf{m}_{0,\delta}}$, $\mu_{\frac{\mathbf{m}_{0,\delta}}{\sqrt{\varrho_{0,\delta}}}}$,  $\mu_{\varrho_\delta }$, $\mu_{\mathbf{u}_\delta }$, $\mu_{\varrho_\delta\mathbf{u}_\delta }$,  $\mu_{V_\delta }$,$\mu_{\mathcal{E}_\delta }$, $\mu_{\nu_\delta }$ and $\mu_{W}$ be the respective laws of  $\varrho_{0,\delta}$, $\mathbf{m}_{0,\delta}$,   $\frac{\mathbf{m}_{0,\delta}}{\sqrt{\varrho_{0,\delta}}}$,  $\varrho_\delta$, $\mathbf{u}_\delta$, $\varrho_\delta \mathbf{u}_\delta $, $V_\delta $, $\mathcal{E}_\delta$, $\nu_\delta$ and $W$ on the respective spaces 
$\chi_{\varrho_0}$, $\chi_{\mathbf{m}_0}$, $\chi_{\frac{\mathbf{m}_0}{\sqrt{\varrho_0}}}$,  $\chi_{\varrho} $,
$\chi_{\mathbf{u}} $,  $\chi_{ \varrho\mathbf{u} }$, $\chi_V$, $\chi_{\mathcal{E}}$, $\chi_\nu$ and $\chi_W$. Furthermore, we set $\mu_\delta $ as their joint law on the space 
\begin{align*}
\chi = \chi_{\varrho_0} \times \chi_{\mathbf{m}_0}\times \chi_{\frac{\mathbf{m}_0}{\sqrt{\varrho_0}}}\times  \chi_{\varrho} \times \chi_{\mathbf{u}}  \times  \chi_{ \varrho\mathbf{u} } \times \chi_V\times \chi_{\mathcal{E}} \times \chi_\nu \times \chi_W.
\end{align*}
In analogy with \eqref{tighnessFourthLAyer}, we have the following result.
\begin{lemma}
\label{tighnessFifthLAyer}
The set $\{\mu_\delta :   \delta \in (0,1) \}$  is tight on $\chi$.
\end{lemma}
By applying the Jakubowski--Skorokhod theorem \cite{jakubowski1998short}, we obtain the following result.
\begin{lemma}
\label{lem:JakowFifthLayer}
The exists a subsequence (not relabelled) $\{\mu_\delta  :  \delta \in (0,1)  \}$, a complete probability space $(\tilde{\Omega}, \tilde{\mathscr{F}}, \tilde{\mathbb{P}})$ with $\chi$-valued random variables
\begin{align*}
(\tilde{\varrho}_{0,\delta}, \tilde{\mathbf{m}}_{0,\delta},
\tilde{\mathbf{n}}_{0,\delta},  \tilde{\varrho}_\delta , \tilde{\mathbf{u}}_\delta , \tilde{\mathbf{m}}_\delta  ,\tilde{V}_\delta ,
\tilde{\mathcal{E}}_\delta,
\tilde{\nu}_\delta,
 \tilde{W}_\delta ) \quad \delta \in (0,1) 
\end{align*}
and 
\begin{align*}
 (\tilde{\varrho}_0, \tilde{\mathbf{m}}_0, \tilde{\mathbf{n}}_0,  \tilde{\varrho}, \tilde{\mathbf{u}}, \tilde{\mathbf{m}}, \tilde{V}, \tilde{\mathcal{E}},
\tilde{\nu}, \tilde{W})
\end{align*}
such that 
\begin{itemize}
\item the law of $(\tilde{\varrho}_{0,\delta}, \tilde{\mathbf{m}}_{0,\delta},
\tilde{\mathbf{n}}_{0,\delta},  \tilde{\varrho}_\delta , \tilde{\mathbf{u}}_\delta , \tilde{\mathbf{m}}_\delta  ,\tilde{V}_\delta ,
\tilde{\mathcal{E}}_\delta,
\tilde{\nu}_\delta,
 \tilde{W}_\delta )$ on $\chi$ coincide with $\mu_\delta $, $\delta \in (0,1) $,
\item the law of $ (\tilde{\varrho}_0, \tilde{\mathbf{m}}_0, \tilde{\mathbf{n}}_0,  \tilde{\varrho}, \tilde{\mathbf{u}}, \tilde{\mathbf{m}}, \tilde{V}, \tilde{\mathcal{E}},
\tilde{\nu}, \tilde{W})$ on $\chi$ is a Radon measure,
\item the following convergence  (with each $\rightarrow$ interpreted with respect to the corresponding topology)
\begin{align*}
\tilde{\varrho}_{0,\delta} \rightarrow \tilde{\varrho}_0 \quad \text{in}\quad \chi_{\varrho_0}, & \quad \qquad
\tilde{\mathbf{m}}_{0,\delta} \rightarrow \tilde{\mathbf{m}}_0 \quad \text{in}\quad \chi_{\mathbf{m}_0}, 
\\
\tilde{\mathbf{n}}_{0,\delta} \rightarrow \tilde{\mathbf{n}}_0 \quad \text{in}\quad \chi_{\frac{\mathbf{m}_0}{\sqrt{\varrho_0}}}, 
& \quad \qquad
\tilde{\varrho}_\delta  \rightarrow \tilde{\varrho} \quad \text{in}\quad \chi_\varrho,
\\
\tilde{\mathbf{u}}_\delta  \rightarrow \tilde{\mathbf{u}} \quad \text{in}\quad \chi_\mathbf{u}, 
 & \quad \qquad
\tilde{\mathbf{m}}_\delta  \rightarrow \tilde{\mathbf{m}} \quad \text{in}\quad \chi_\mathbf{m}, 
\\
 \tilde{V}_\delta  \rightarrow \tilde{V} \quad  \text{in}\quad \chi_V,
  & \quad \qquad
 \tilde{\mathcal{E}}_\delta  \rightarrow  \tilde{\mathcal{E}} \quad \text{in}\quad \chi_\mathcal{E}, 
 \\
 \tilde{\nu}_\delta  \rightarrow \tilde{\nu} \quad  \text{in}\quad \chi_\nu,
  & \quad \qquad
\tilde{W}_\delta  \rightarrow \tilde{W} \quad \text{in}\quad \chi_W
\end{align*}
holds $\tilde{
\mathbb{P}}$-a.s.;
\item consider any Carath\'eodory function $\mathcal{C}=\mathcal{C}(t,x,\varrho, \mathbf{u}, \mathbb{U}, \mathbf{m}, f, \mathbf{V})$ with
\begin{align*}
(t,x,\varrho, \mathbf{u}, \mathbb{U}, \mathbf{m}, f, \mathbf{V}) \in [0,T]\times \mathbb{T}^3 \times \mathbb{R} \times \mathbb{R}^3 \times \mathbb{R}^{3\times3} \times \mathbb{R}^3 \times \mathbb{R} \times \mathbb{R}^3,
\end{align*} 
and where the following estimate
\begin{align*}
\vert \mathcal{C} \vert \lesssim 1+ \vert \varrho\vert^{r_1} + \vert  \mathbf{u}\vert^{r_2}  + \vert \mathbb{U}\vert^{r_3}   + \vert  \mathbf{m}\vert^{r_4}  + \vert  f\vert^{r_5}   + \vert  \mathbf{V}\vert^{r_6} 
\end{align*}
holds uniformly in $(t,x)$ for some $r_i>0$, $i=1, \ldots, 6$. Then as $\delta \rightarrow0$, it follows that
\begin{align*}
\mathcal{C}(\tilde{\varrho}_\delta, \tilde{\mathbf{u}}_\delta, \nabla\tilde{\mathbf{u}}_\delta, \tilde{\mathbf{m}}_\delta, f, \nabla\tilde{V}_\delta) 
\rightharpoonup
\overline{
\mathcal{C}(\tilde{\varrho}, \tilde{\mathbf{u}}, \nabla\tilde{\mathbf{u}}, \tilde{\mathbf{m}}, f,  \nabla\tilde{V})}
\end{align*}
holds  in $L^r((0,T) \times
 \mathbb{T}^3)$ for all
 \begin{align*}
1<r \leq \frac{\gamma+\Theta}{r_1} \wedge \frac{2}{r_2} \wedge \frac{2\gamma}{r_4(\gamma+1)} \wedge \frac{2}{r_6}
 \end{align*}
$\tilde{\mathbb{P}}$-a.s.
\end{itemize}
\end{lemma}
The following corollary follow from Lemma \ref{lem:JakowFifthLayer}.
\begin{corollary}
The following holds $\tilde{\mathbb{P}}$-a.s.
\begin{equation}
\begin{aligned}
&\tilde{\varrho}_{0,\delta}= \tilde{\varrho}_\delta(0), \quad \tilde{\mathbf{m}}_{0,\delta}= \tilde{\varrho}_\delta\tilde{\mathbf{u}}_\delta(0), \quad \tilde{\mathbf{n}}_{0,\delta} = \frac{\tilde{\mathbf{m}}_{0,\delta}}{\sqrt{\tilde{\varrho}_{0,\delta}}}, \quad \tilde{\mathbf{m}}_{\delta}= \tilde{\varrho}_\delta\tilde{\mathbf{u}}_\delta,
\\
&\tilde{\mathcal{E}}_\delta = \mathcal{E}_\delta(\tilde{\varrho}_\delta, \tilde{\mathbf{u}}_\delta, \tilde{V}_\delta),\quad
\tilde{\nu}_{\delta}= \delta_{[\tilde{\varrho}_\delta, \tilde{\mathbf{u}}_\delta, \nabla \tilde{\mathbf{u}}_\delta, \tilde{\varrho}_\delta \tilde{\mathbf{u}}_\delta, f, \nabla \tilde{V}_\delta]}
\end{aligned}
\end{equation}
and that
\begin{equation}
\begin{aligned}
\tilde{\mathbb{E}} \bigg\vert \sup_{t\in[0,T]} \int_{\mathbb{T}^3} \tilde{\mathcal{E}}_\delta \, \mathrm{d}x \bigg\vert^p
&=
\tilde{\mathbb{E}} \bigg\vert\sup_{t\in[0,T]} \int_{\mathbb{T}^3}\bigg( \frac{1}{2} \tilde{\varrho}_\delta  \vert \tilde{\mathbf{u}}_\delta  \vert^2  
+ P^\Gamma_\delta(\tilde{\varrho}_\delta ) 
\pm
\vartheta \vert \nabla \tilde{V}_\delta \vert^2 \bigg) \, \mathrm{d}x
\bigg\vert^p
\lesssim_{  \Gamma, p,\tilde{\mathcal{E}}_{0,\delta}} 1
\end{aligned}
\end{equation}
and
\begin{equation}
\begin{aligned}
\tilde{\mathbb{E}}\bigg\vert \int_0^T \int_{\mathbb{T}^3} P^\Gamma_\delta(\tilde{\varrho}_\delta ) \tilde{\varrho}_\delta^\Theta
\, \mathrm{d}x\, \mathrm{d}t
\bigg\vert^p
\lesssim_{  \Gamma, p,\tilde{\mathcal{E}}_{0,\delta}} 1
\end{aligned}
\end{equation}
holds uniformly in $\delta$.
\end{corollary}
Furthermore, we can endow this new probability space $(\tilde{\Omega}, \tilde{\mathscr{F}}, \tilde{\mathbb{P}})$ with the filtrations
\begin{align*}
\tilde{\mathscr{F}}_t^\delta  :=\sigma \bigg(\sigma_t[\tilde{\varrho}_\delta], \sigma_t[\tilde{\mathbf{u}}_\delta], \sigma_t[\tilde{V}_\delta], \bigcup_{k\in \mathbb{N}}\sigma_t[\tilde{\beta}_{\delta,k}]  \bigg), \quad t\in [0,T]
\end{align*}
and
\begin{align*}
\tilde{\mathscr{F}}_t  :=
\sigma \bigg(\sigma_t[\tilde{\varrho}], \sigma_t[\tilde{\mathbf{u}}], \sigma_t[\tilde{V}], \bigcup_{k\in \mathbb{N}}\sigma_t[\tilde{\beta}_{k}]  \bigg), \quad t\in [0,T]
\end{align*}
for the family of sequences $(\tilde{\varrho}_{0,\delta}, \tilde{\mathbf{m}}_{0,\delta},
\tilde{\mathbf{n}}_{0,\delta},  \tilde{\varrho}_\delta , \tilde{\mathbf{u}}_\delta , \tilde{\mathbf{m}}_\delta  ,\tilde{V}_\delta ,
\tilde{\mathcal{E}}_\delta,
\tilde{\nu}_\delta,
 \tilde{W}_\delta )$  and the limit random variables $ (\tilde{\varrho}_0, \tilde{\mathbf{m}}_0, \tilde{\mathbf{n}}_0,  \tilde{\varrho}, \tilde{\mathbf{u}}, \tilde{\mathbf{m}}, \tilde{V}, \tilde{\mathcal{E}},
\tilde{\nu}, \tilde{W})$ respectively.
\\
Now recall from Section \ref{sec:UniformEstFifthLayer} that $[(\Omega ,\mathscr{F} ,(\mathscr{F} _t)_{t\geq0},\mathbb{P} );\varrho_\delta , \mathbf{u}_\delta , V_\delta , W   ]$
was a dissipative martingale solution  of \eqref{contEqFourth1}--\eqref{elecFieldFourth1} in the sense of Definition \ref{def:martSolFourthLayer} having the law $\Lambda_\delta$. Since Lemma \ref{lem:JakowFifthLayer} attests that the marginal law of $(\varrho_\delta , \mathbf{u}_\delta )$ coincides with $(\tilde{\varrho}_\delta , \tilde{\mathbf{u}}_\delta )$, it follows that that latter satisfies the Renormalized continuity equation \eqref{renormalizedContFourthLayer}. The follow result thus holds true.
\begin{lemma}
For any $\phi\in C^\infty_c([0,T)\times{\mathbb{T}^3})$ and $b \in C^1_b(\mathbb{R})$ such that $ b'(z)=0$ for all $z\geq M_b$,  we have that
\begin{equation}
\begin{aligned}
\label{renormalizedContFifthLayerSeq} 
&0=\int_0^T\int_{{\mathbb{T}^3}}  b(\tilde{\varrho}_\delta)\,\partial_t\phi \,\mathrm{d}x \, \mathrm{d}t +
\int_{{\mathbb{T}^3}}  b\big(\tilde{\varrho}_{0,\delta}\big)\,\phi(0) \,\mathrm{d}x 
\\&+
 \int_0^T\int_{{\mathbb{T}^3}} \big[b(\tilde{\varrho}_\delta) \tilde{\mathbf{u}}_\delta  \big] \cdot \nabla\phi \, \mathrm{d}x \, \mathrm{d}t  
- \int_0^T\int_{{\mathbb{T}^3}} \big[ \big(  b'(\tilde{\varrho}_\delta) \tilde{\varrho}_\delta  
-
 b(\tilde{\varrho}_\delta )\big)\mathrm{div}\tilde{\mathbf{u}}_\delta \big]\,  \phi\, \mathrm{d}x \, \mathrm{d}t
\end{aligned}
\end{equation} 
holds $\tilde{\mathbb{P}}$-a.s.
\end{lemma}
Now analogous to Corollary \ref{cor:nonlinearLimitsFourth}, we also have the following result.
\begin{corollary}
\label{cor:nonlinearLimitsFifth}
There exists $\overline{p^\Gamma_\delta(\tilde{\varrho})}$, $\overline{\vert \nabla \tilde{V}\vert^2}$, $\overline{\tilde{\varrho} \nabla \tilde{V}}$ and $\overline{\tilde{\varrho}\, \mathbf{g}_k(\tilde{\varrho}, \tilde{f}, \tilde{\varrho}\tilde{\mathbf{u}})}$ such that
\begin{align}
p^\Gamma_\delta(\tilde{\varrho}_\delta)
&\rightharpoonup
\overline{p^\Gamma_\delta(\tilde{\varrho})}
\label{pressureLimitFifthLayer}\\
\vert \nabla \tilde{V}_\delta \vert^2
&\rightharpoonup
\overline{\vert \nabla \tilde{V}\vert^2}
\label{nablaVLimitFifthLayer}\\
\tilde{\varrho}_\delta \nabla \tilde{V}_\delta 
&\rightharpoonup
\overline{\tilde{\varrho} \nabla \tilde{V}}
\label{rhoNablaVLimitFifthLayer}\\
\tilde{\varrho}_\delta \mathbf{g}_k(\tilde{\varrho}_\delta, f, \tilde{\varrho}_\delta\tilde{\mathbf{u}}_\delta)
&\rightharpoonup
\overline{\tilde{\varrho}\, \mathbf{g}_k(\tilde{\varrho}, f, \tilde{\varrho}\tilde{\mathbf{u}})}
\label{noiseLimitFifthLayer}
\end{align}
$\tilde{\mathbb{P}}$-a.s. in $L^r((0,T) \times \mathbb{T}^3)$ for some $r>1$.
\end{corollary}
The corresponding versions of Lemma \ref{lem:identifyContFourthLayer}--Lemma \ref{weakMomEqLimitFourthLayer} on this approximation layer holds true.
\begin{lemma}
\label{lem:identifyContFifthLayer}
The random distributions  $[\tilde{\varrho}, \tilde{\mathbf{u}}]$ satisfies \eqref{eq:distributionalSolMomen} for all $\psi \in C^\infty_c ([0,T))$ and $\phi \in C^\infty(\mathbb{T}^3)$ $\tilde{\mathbb{P}}$-a.s.
\end{lemma}
\begin{lemma}
\label{lem:identifyPoisFifthLayer}
The random variables $[\tilde{\varrho},  \tilde{V}]$ satisfies  \eqref{elecFieldFifth1} a.e. in $(0,T) \times \mathbb{T}^3$ $\tilde{\mathbb{P}}$-a.s.
\end{lemma}
\begin{lemma}
\label{lem:identifyMomFifthLayer}
The random distributions  $[\tilde{\varrho}, \tilde{\mathbf{u}},  \tilde{V}, \tilde{W}]$ satisfies 
\begin{equation}
\begin{aligned}
\label{weakMomEqLimitFifthLayer}
&-\int_0^T \partial_t \psi \int_{{\mathbb{T}^3}} \tilde{\varrho}  \tilde{\mathbf{u}} (t) \cdot \bm{\phi} \, \mathrm{d}x    \mathrm{d}t
=
\psi(0)  \int_{{\mathbb{T}^3}} \tilde{\varrho} _0 \tilde{\mathbf{u}} _0 \cdot \bm{\phi} \, \mathrm{d}x  
+  \int_0^T \psi \int_{{\mathbb{T}^3}}  \tilde{\varrho}   \tilde{ \mathbf{u}} \otimes \tilde{\mathbf{u}} : \nabla \bm{\phi} \, \mathrm{d}x  \mathrm{d}t 
\\&-
 \nu^S\int_0^T \psi \int_{{\mathbb{T}^3}} \nabla \tilde{\mathbf{u}} : \nabla \bm{\phi} \, \mathrm{d}x  \mathrm{d}t 
-(\nu^B+ \nu^S )\int_0^T \psi \int_{{\mathbb{T}^3}} \mathrm{div}\,\tilde{\mathbf{u}} \, \mathrm{div}\, \bm{\phi} \, \mathrm{d}x  \mathrm{d}t    
\\&+   
\int_0^T \psi \int_{{\mathbb{T}^3}}   \overline{ p(\tilde{\varrho})}\, \mathrm{div} \, \bm{\phi} \, \mathrm{d}x  \mathrm{d}t
+\int_0^T \psi \int_{{\mathbb{T}^3}} \vartheta \overline{\tilde{\varrho}  \nabla \tilde{V}}  \cdot \bm{\phi} \, \mathrm{d}x    \mathrm{d}t
\\&+
\int_0^T \psi \sum_{k\in \mathbb{N}} \int_{{\mathbb{T}^3}}  \overline{ \tilde{\varrho}\, \mathbf{g}_{k}(\tilde{\varrho} , f , \tilde{\varrho} \tilde{\mathbf{u}} )} \cdot \bm{\phi}  \, \mathrm{d}x \mathrm{d}\tilde{\beta}_k  
\end{aligned}
\end{equation}
for all $\psi \in C^\infty_c ([0,T))$ and $\bm{\phi} \in C^\infty(\mathbb{T}^3)$ $\mathbb{P}$-a.s.
\end{lemma}

\subsection{Strong convergence of Density}
Notice the strong convergence \eqref{strongDensityConvFourthLayer} of the density sequence in the previous approximation layer relied crucially on the DiPerna-Lions renomarlization theorem \cite{diperna1989ordinary}. In order to apply this theorem, one requires the density to be at least squared-integrable in time and space which happened to be the case in Section \ref{sec:fourthLayer} (recall that $\Gamma\geq6$).
\\
Since the fluid limit density in this current approximation layer is only $(\gamma+\Theta)$-integrable in time and space, we can no longer apply the theorem of DiPerna-Lions since $\gamma>\frac{3}{2}$ and $0<\Theta\leq \frac{1}{3}$. To remedy this problem, we rely on the \textit{oscillation defect measure} introduced by Feireisl \cite{feireisl2001compactness}
in order to establish the smallness of the
amplitude of oscillations in the density sequence. In the present stochastic setting and approximation layer, the measure is given by
\begin{equation}
\begin{aligned}
\mathbf{osc}_{\gamma+1}[\tilde{\varrho}_\delta \rightarrow \tilde{\varrho}](\tilde{\Omega} \times Q)
=
\sup_{k\geq1} \bigg( \limsup_{\delta\rightarrow 0} \tilde{\mathbb{E}} \int_Q \vert T_k(\tilde{\varrho}_\delta) -T_k(\tilde{\varrho}) \vert^{\gamma+1}\, \mathrm{d}x\, \mathrm{d}t \bigg)
\end{aligned}
\end{equation}
where $Q=(0,T) \times \mathbb{T}^3$ and for any $k\in \mathbb{N}$, $T_k$ is a cut-off function defined as 
\begin{equation}
\begin{aligned}
T_k(s) =k T\bigg( \frac{s}{k} \bigg), \quad T \in C^\infty([0,\infty)), \quad
T(s)
=
  \begin{cases}
   s      &  \text{ if }0 \leq s \leq 1\\
   T''(s)\leq 0     &  \text{ if }1< s <3\\
   2  & \text{ if } s \geq 3.
  \end{cases}
\end{aligned}
\end{equation}
\subsubsection{The effective viscous flux}
By using the weak compactness result of Carath\'eodory functions established in Lemma \ref{lem:JakowFifthLayer}, we obtain the following
\begin{align}
T_k(\tilde{\varrho}_\delta)  \rightharpoonup \overline{T_k(\tilde{\varrho})} \quad &\text{in}\quad L^p((0,T)\times\mathbb{T}^3),
\label{tkCarathe1}\\
T_k(\tilde{\varrho}_\delta)  \rightarrow \overline{T_k(\tilde{\varrho})} \quad &\text{in}\quad C_w([0,T]; L^p(\mathbb{T}^3)),
\label{tkCarathe2}\\
[T_k'(\tilde{\varrho}_\delta)\tilde{\varrho}_\delta -T_k(\tilde{\varrho}_\delta)]\mathrm{div} \tilde{\mathbf{u}}_\delta  \rightharpoonup \overline{[T_k'(\tilde{\varrho})\tilde{\varrho} -T_k(\tilde{\varrho})]\mathrm{div} \tilde{\mathbf{u}}} \quad &\text{in}\quad L^q((0,T)\times\mathbb{T}^3) \label{tkCarathe3}
\end{align}
$\tilde{\mathbb{P}}$-a.s. for any $p\in (1,\infty)$ and for some $q>1$.
We can then use \eqref{tkCarathe1}--\eqref{tkCarathe3} to pass to limit $\delta \rightarrow 0$ in \eqref{renormalizedContFifthLayerSeq} with $b=T_k$ to obtain
\begin{equation}
\begin{aligned}
\label{renormalizedContFifthLayerLim} 
&0=\int_0^T\int_{{\mathbb{T}^3}}  \overline{T_k(\tilde{\varrho})}\,\partial_t\phi \,\mathrm{d}x \, \mathrm{d}t +
\int_{{\mathbb{T}^3}}  \overline{T_k\big(\tilde{\varrho}_{0}\big)}\,\phi(0) \,\mathrm{d}x 
\\&+
 \int_0^T\int_{{\mathbb{T}^3}} \big[\overline{T_k(\tilde{\varrho})} \tilde{\mathbf{u}}  \big] \cdot \nabla\phi \, \mathrm{d}x \, \mathrm{d}t  
- \int_0^T\int_{{\mathbb{T}^3}} \big[ \overline{\big(  T_k'(\tilde{\varrho}) \tilde{\varrho}  
-
 T_k(\tilde{\varrho} )\big)\mathrm{div}\tilde{\mathbf{u}}} \big]\,  \phi\, \mathrm{d}x \, \mathrm{d}t
\end{aligned}
\end{equation} 
$\mathbb{P}$-a.s. for any $\phi\in C^\infty_c([0,T)\times{\mathbb{T}^3})$. 
Similar to \eqref{momEqItoFifthLayer1}, we obtain from \eqref{weakMomEqLimitFifthLayer} and \eqref{renormalizedContFifthLayerLim}, the following. 
\begin{equation}
\begin{aligned}
\label{momTkcontEqlim}
&
\int_0^s \int_{{\mathbb{T}^3}}    \overline{p(\tilde{\varrho} )}\, \overline{T_k(\tilde{\varrho})}   \, \mathrm{d}x  \mathrm{d}t
=
\int_{{\mathbb{T}^3}} \tilde{\varrho} \tilde{\mathbf{u}}  (t) \cdot \nabla \Delta^{-1}_{\mathbb{T}^3} \overline{T_k(\tilde{\varrho})} \, \mathrm{d}x
-
 \int_{{\mathbb{T}^3}}  \tilde{\varrho}_{0,\delta} \tilde{\mathbf{u}}_{0,\delta} \cdot \nabla \Delta^{-1}_{\mathbb{T}^3} \overline{T_k(\tilde{\varrho}(0))}  \, \mathrm{d}x     
\\&-
\int_0^s \int_{{\mathbb{T}^3}}  \tilde{\varrho}     \tilde{\mathbf{u}}  \otimes \tilde{\mathbf{u}}  : \nabla^2 \Delta^{-1}_{\mathbb{T}^3} \overline{T_k(\tilde{\varrho})} \, \mathrm{d}x  \mathrm{d}t 
+
(\nu^B+ \nu^S )\int_0^s\int_{{\mathbb{T}^3}} \mathrm{div}\,\tilde{\mathbf{u}}  \, \overline{T_k(\tilde{\varrho})} \, \mathrm{d}x  \mathrm{d}t   
\\&+
 \nu^S\int_0^s\int_{{\mathbb{T}^3}} \nabla \tilde{\mathbf{u}}  : \nabla^2 \Delta^{-1}_{\mathbb{T}^3} \overline{T_k(\tilde{\varrho})} \, \mathrm{d}x  \mathrm{d}t  
-
\int_0^s\int_{{\mathbb{T}^3}} \vartheta\overline{\tilde{\varrho}  \nabla \tilde{V}}\, \nabla \Delta^{-1}_{\mathbb{T}^3} \overline{T_k(\tilde{\varrho})} \, \mathrm{d}x    \mathrm{d}t
\\&+
\int_0^s\int_{{\mathbb{T}^3}} \tilde{\varrho} \tilde{\mathbf{u}}  \nabla \Delta^{-1}_{\mathbb{T}^3}\mathrm{div}\big[  \overline{T_k(\tilde{\varrho})} \tilde{\mathbf{u}}  \big]  \, \mathrm{d}x    \mathrm{d}t
+
 \int_0^s\int_{{\mathbb{T}^3}}\tilde{\varrho}  \tilde{\mathbf{u}} \nabla \Delta^{-1}_{\mathbb{T}^3} \big[\overline{\big(  T_k'(\tilde{\varrho}) \tilde{\varrho}  
-
 T_k(\tilde{\varrho} )\big)\mathrm{div}\tilde{\mathbf{u}}} \big]
 \, \mathrm{d}x    \mathrm{d}t
\\&-
\int_0^s\sum_{k\in \mathbb{N}} \int_{{\mathbb{T}^3}}   \overline{\tilde{\varrho} \, \mathbf{g}_{k}(\tilde{\varrho}  , f, \tilde{\varrho}  \tilde{\mathbf{u}}  )} \cdot \nabla \Delta^{-1}_{\mathbb{T}^3} \overline{T_k(\tilde{\varrho})}  \, \mathrm{d}x \mathrm{d}\tilde{\beta}_k.
\end{aligned}
\end{equation}
On the other hand, we also note that a consequence of Lemma \ref{lem:JakowFifthLayer} is that $(\tilde{\varrho}_\delta, \tilde{\mathbf{u}}_\delta , \tilde{V}_\delta, \tilde{W}_\delta)$ satisfies \eqref{weakMomEqFourthLayer}. By combining this information with \eqref{renormalizedContFifthLayerSeq}, we obtain as in \eqref{momEqItoFifthLayer1},
\begin{equation}
\begin{aligned}
\label{momTkcontEqSeq}
&
\int_0^s \int_{{\mathbb{T}^3}}   p(\tilde{\varrho}_\delta )\, T_k(\tilde{\varrho}_\delta)   \, \mathrm{d}x  \mathrm{d}t
=
\int_{{\mathbb{T}^3}} \tilde{\varrho}_\delta \tilde{\mathbf{u}}_\delta  (t) \cdot \nabla \Delta^{-1}_{\mathbb{T}^3} T_k(\tilde{\varrho}_\delta) \, \mathrm{d}x
-
 \int_{{\mathbb{T}^3}} \tilde{\varrho}_{0,\delta} \mathbf{u}_{0,\delta} \cdot \nabla \Delta^{-1}_{\mathbb{T}^3} T_k(\tilde{\varrho}_\delta(0))  \, \mathrm{d}x     
\\&-
\int_0^s \int_{{\mathbb{T}^3}}  \tilde{\varrho}_\delta     \tilde{\mathbf{u}}_\delta  \otimes \tilde{\mathbf{u}}_\delta  : \nabla^2 \Delta^{-1}_{\mathbb{T}^3} T_k(\tilde{\varrho}_\delta)  \, \mathrm{d}x  \mathrm{d}t 
+
(\nu^B+ \nu^S )\int_0^s\int_{{\mathbb{T}^3}} \mathrm{div}\,\tilde{\mathbf{u}}_\delta  \, T_k(\tilde{\varrho}_\delta)  \, \mathrm{d}x  \mathrm{d}t   
\\&+
 \nu^S\int_0^s\int_{{\mathbb{T}^3}} \nabla \tilde{\mathbf{u}}_\delta  : \nabla^2 \Delta^{-1}_{\mathbb{T}^3} T_k(\tilde{\varrho}_\delta)  \, \mathrm{d}x  \mathrm{d}t  
-
\int_0^s\int_{{\mathbb{T}^3}} \vartheta\tilde{\varrho}_\delta    \nabla \tilde{V}_\delta \, \nabla \Delta^{-1}_{\mathbb{T}^3} T_k(\tilde{\varrho}_\delta)  \, \mathrm{d}x    \mathrm{d}t
\\&
\int_0^s\int_{{\mathbb{T}^3}} \tilde{\varrho}_\delta \tilde{\mathbf{u}}_\delta  \nabla \Delta^{-1}_{\mathbb{T}^3}\mathrm{div}\big[  T_k(\tilde{\varrho}_\delta) \tilde{\mathbf{u}}_\delta  \big]   \mathrm{d}x    \mathrm{d}t
+
 \int_0^s\int_{{\mathbb{T}^3}}\tilde{\varrho}_\delta  \tilde{\mathbf{u}}_\delta \nabla \Delta^{-1}_{\mathbb{T}^3} \big[\big(  T_k'(\tilde{\varrho}_\delta) \tilde{\varrho}_\delta  
-
 T_k(\tilde{\varrho}_\delta)  \big)\mathrm{div}\tilde{\mathbf{u}}_\delta \big]
  \mathrm{d}x    \mathrm{d}t
\\&-
\int_0^s\sum_{k\in \mathbb{N}} \int_{{\mathbb{T}^3}}   \tilde{\varrho}_\delta \, \mathbf{g}_{k}(\tilde{\varrho}_\delta  , f, \tilde{\varrho}_\delta  \tilde{\mathbf{u}}_\delta  ) \cdot \nabla \Delta^{-1}_{\mathbb{T}^3} T_k(\tilde{\varrho}_\delta)   \, \mathrm{d}x \mathrm{d}\tilde{\beta}_{\delta,k}.
\end{aligned}
\end{equation}
As in \eqref{intergrationByParts}, we can rewrite the viscous terms in \eqref{momTkcontEqlim} and \eqref{momTkcontEqSeq} appropriately and in analogy with the derivation of Lemma \ref{lem:effectiveFluxFourthLayer}, we obtain
\begin{equation}
\begin{aligned}
\label{effectiveVisFLuxFifthLayer}
\lim_{\delta  \rightarrow0}&\int_0^s\int_{\mathbb{T}^3}\big[p(\tilde{\varrho}_\delta) -(\nu^B+2\nu^S)\mathrm{div}\, \tilde{\mathbf{u}}_\delta \big] T_k(\tilde{\varrho}_\delta) \, \mathrm{d}x \, \mathrm{d}x
\\&
=
\int_0^s\int_{\mathbb{T}^3}\big[\overline{p(\tilde{\varrho} )} -(\nu^B+2\nu^S)\mathrm{div}\, \tilde{\mathbf{u}} \big]\overline{T_k(\tilde{\varrho})}\, \mathrm{d}x \, \mathrm{d}x
\end{aligned}
\end{equation}
holds $\tilde{\mathbb{P}}$-a.s. for a.e. $s\in(0,T)$
with the help of \eqref{pressureLimitFifthLayer}--\eqref{noiseLimitFifthLayer},  \eqref{tkCarathe1}--\eqref{tkCarathe3} and \cite[Lemma 2.6.6]{breit2018stoch}.

\subsubsection{Oscillation defect measure}
We now show that the estimate
\begin{align}
\label{smallnesOfOscilDefectMeasure}
 \mathbf{osc}_{\gamma+1}[\tilde{\varrho}_\delta \rightarrow \tilde{\varrho}](\tilde{\Omega}\times Q)   \lesssim 1.
\end{align}
holds uniformly in $k$. To see this,
we first note  the identity
\begin{equation}
\begin{aligned}
\label{eq:2127}
\lim_{\delta \rightarrow 0} \tilde{\mathbb{E}}\int_Q &    \left(p( \tilde{\varrho}_\delta) \, T_k(\tilde{\varrho}_\delta) -  \overline{p(\tilde{\varrho})}\, \overline{T_k(\tilde{\varrho})}\right)\,\mathrm{d}x\, \mathrm{d}t
\\&=
(\nu^B+2\nu^S)
\lim_{n\rightarrow \infty}\tilde{\mathbb{E}}\int_Q      \left[\mathrm{div}\,\tilde{\mathbf{u}}_\delta  \, T_k(\tilde{\varrho}_\delta)   -  \mathrm{div}\,\tilde{\mathbf{u}}  \, \overline{T_k(\tilde{\varrho})}  \right]\,\mathrm{d}x\, \mathrm{d}t
\end{aligned}
\end{equation}
from \eqref{effectiveVisFLuxFifthLayer} where  by use of  the inequality
\begin{align*}
\vert T_k(t)  -T_k(s)\vert^{\gamma+1}  \leq  (t^\gamma-s^\gamma)\left( T_k(t)-T_k(s)  \right),
\end{align*}
it follows that
\begin{equation}
\begin{aligned}
\label{eq:57}
\lim_{\delta \rightarrow 0} \tilde{\mathbb{E}}\int_Q &    \left(p(\tilde{\varrho}_\delta) \, T_k(\tilde{\varrho}_\delta) -  \overline{p(\tilde{\varrho})}\, \overline{T_k(\tilde{\varrho})}\right)\,\mathrm{d}x\, \mathrm{d}t
\geq
\limsup_{\delta \rightarrow 0} \tilde{\mathbb{E}}\int_Q     
\big\vert T_k(\tilde{\varrho}_\delta) - T_k(\tilde{\varrho}) 
\big\vert^{\gamma+1} \,\mathrm{d}x\, \mathrm{d}t
\\
&+
 \tilde{\mathbb{E}}\int_Q     \left( \overline{p(\tilde{\varrho})}  -   p(\tilde{\varrho}_\delta)\right)\left( T_k(\tilde{\varrho}) - \overline{T_k(\tilde{\varrho})} \right) \,\mathrm{d}x\, \mathrm{d}t.
\end{aligned}
\end{equation}
By combining \eqref{eq:2127}--\eqref{eq:57} with the negativity of the lest term in \eqref{eq:57} above (this negativity follows from the convexity of $t\mapsto t^\gamma$ and the concavity of $t\mapsto T_k(t)$), we gain 
\begin{equation}
\begin{aligned}
\label{eq:57az}
\limsup_{\delta \rightarrow 0}&\, \tilde{\mathbb{E}}\int_Q   \, 
\big\vert T_k(\tilde{\varrho}_\delta) - T_k(\tilde{\varrho}) 
\big\vert^{\gamma+1} \,\mathrm{d}x\, \mathrm{d}t
\\
&\leq (\nu^B+2\nu^S)
\limsup_{\delta \rightarrow 0} \tilde{\mathbb{E}}\int_Q   \,  \left[\mathrm{div}\,\tilde{\mathbf{u}}_\delta  \, T_k(\tilde{\varrho}_\delta)   -  \mathrm{div}\,\tilde{\mathbf{u}}  \, \overline{T_k(\tilde{\varrho})}  \right]\,\mathrm{d}x\, \mathrm{d}t.
\end{aligned}
\end{equation}
However, the use of  H\"{o}lder and triangle inequalities further yield
\begin{equation}
\begin{aligned}
\label{eq:58}
&\limsup_{n\rightarrow \infty}\tilde{\mathbb{E}}\int_Q     \left[ \mathrm{div}\,\tilde{\mathbf{u}}_\delta  \, T_k(\tilde{\varrho}_\delta)    -    \mathrm{div}\,\tilde{\mathbf{u}}  \, \overline{T_k(\tilde{\varrho})}\right] \,\mathrm{d}x\, \mathrm{d}t
\\
&\leq
\limsup_{n\rightarrow \infty}\tilde{\mathbb{E}}\, \left\Vert  \mathrm{div}\,\tilde{\mathbf{u}}_\delta  \right\Vert_{L^2(Q)}
\tilde{\mathbb{E}}\,\Big( \big\Vert  \big( T_k(\tilde{\varrho}_\delta) -T_k(\tilde{\varrho})  \big)\big\Vert_{L^2(Q)}   
+    \big\Vert   \big( T_k(\tilde{\varrho})   -     \overline{T_k(\tilde{\varrho})} \big) \big\Vert_{L^2(Q)}\Big) .
\end{aligned}
\end{equation}
And by lower semi-continuity of norms, 
\begin{equation}
\begin{aligned}
\label{eq:59}
\tilde{\mathbb{E}}\,\left\Vert   \big(T_k(\tilde{\varrho})   -     \overline{T_k(\tilde{\varrho})}  \big) \right\Vert_{L^2(Q)} 
&\leq 
\liminf_{\delta \rightarrow 0}\tilde{\mathbb{E}}\,\left\Vert   \big( T_k(\tilde{\varrho}_\delta) -T_k(\tilde{\varrho}) \big) \right\Vert_{L^2(Q)} 
\\
&\leq 
\limsup_{\delta \rightarrow 0}\tilde{\mathbb{E}}\,\left\Vert   \big( T_k(\tilde{\varrho}_\delta) -T_k(\tilde{\varrho}) \big)  \right\Vert_{L^2(Q)}.
\end{aligned}
\end{equation}
By using the embedding $L^{\gamma+1} \hookrightarrow L^2$, we can  substitute \eqref{eq:59} into \eqref{eq:58} to get
\begin{equation}
\begin{aligned}
\label{eq:590}
&\limsup_{n\rightarrow \infty}\,\tilde{\mathbb{E}}\int_Q      \left[ \mathrm{div}\,\tilde{\mathbf{u}}_\delta  \, T_k(\tilde{\varrho}_\delta)    -    \mathrm{div}\,\tilde{\mathbf{u}}  \, \overline{T_k(\tilde{\varrho})}\right] \,\mathrm{d}x\, \mathrm{d}t
\\
&\leq 2
\limsup_{n\rightarrow \infty}
\bigg(  \tilde{\mathbb{E}}\, \left\Vert   \mathrm{div}\,\tilde{\mathbf{u}}_\delta  \right\Vert_{L^2(Q)}
\tilde{\mathbb{E}}\, \big\Vert   (T_k(\tilde{\varrho}_\delta) -T_k(\tilde{\varrho}) ) \big\Vert_{L^{\gamma+1}(Q)}\bigg) .
\end{aligned}
\end{equation}
Finally, we substitute \eqref{eq:590}  into \eqref{eq:57az}  and apply Young's inequality to obtain
\begin{equation}
\begin{aligned}
&\limsup_{\delta \rightarrow 0} \tilde{\mathbb{E}}\int_Q     
\big\vert T_k(\tilde{\varrho}_\delta) - T_k(\tilde{\varrho}) 
\big\vert^{\gamma+1} \,\mathrm{d}x\, \mathrm{d}t
\\
&\lesssim
\limsup_{n\rightarrow \infty}
\bigg(  \tilde{\mathbb{E}}\, \left\Vert   \mathrm{div}\,\tilde{\mathbf{u}}_\delta  \right\Vert_{L^2(Q)}
\tilde{\mathbb{E}}\, \big\Vert   (T_k(\tilde{\varrho}_\delta) -T_k(\tilde{\varrho}) ) \big\Vert_{L^{\gamma+1}(Q)}\bigg) 
\\
&\lesssim \frac{\gamma}{\gamma+1
} \sup_{n} \bigg(\tilde{\mathbb{E}}\, \left\Vert   \mathrm{div}\,\tilde{\mathbf{u}}_\delta  \right\Vert_{L^2(Q)} \bigg)^\frac{\gamma+1}{\gamma}
+
\frac{1}{\gamma+1}
\limsup_{\delta \rightarrow 0} \tilde{\mathbb{E}}\int_Q   
\big\vert T_k(\tilde{\varrho}_\delta) - T_k(\tilde{\varrho}) 
\big\vert^{\gamma+1} \,\mathrm{d}x\, \mathrm{d}t
\end{aligned}
\end{equation}
uniformly in $k$. The estimate \eqref{smallnesOfOscilDefectMeasure} follow by absolving the last term above into the left-hand side (note that $\frac{1}{\gamma+1}<\frac{2}{5}$ is small enough) and keep in mind that $\frac{\gamma+1}{\gamma}<2$.

\subsubsection{The renormalized solution for the limit process}
If we now regularize \eqref{renormalizedContFifthLayerLim} with some  regularizing operator $S_m$, multiply the resulting deterministically strong equation by $b'(S_m[\overline{T_k(\varrho)}])$ and pass to the limit $m\rightarrow \infty$, we obtain
\begin{equation}
\begin{aligned}
\label{renormalizedContFifthLayerLim1} 
&0=\int_0^T\int_{{\mathbb{T}^3}}  b(\overline{T_k(\tilde{\varrho})})\,\partial_t\phi \,\mathrm{d}x \, \mathrm{d}t +
\int_{{\mathbb{T}^3}}  b(\overline{T_k\big(\tilde{\varrho}_{0}\big)})\,\phi(0) \,\mathrm{d}x 
+
 \int_0^T\int_{{\mathbb{T}^3}} \big[b(\overline{T_k(\tilde{\varrho})}) \tilde{\mathbf{u}}  \big] \cdot \nabla\phi \, \mathrm{d}x \, \mathrm{d}t  
\\&+ \int_0^T\int_{{\mathbb{T}^3}}b'(\overline{T_k(\tilde{\varrho})})  \big[ \overline{\big(  T_k'(\tilde{\varrho}) \tilde{\varrho}  
-
 T_k(\tilde{\varrho} )\big)\mathrm{div}\tilde{\mathbf{u}}} \big]\,  \phi\, \mathrm{d}x \, \mathrm{d}t
 \\&- \int_0^T\int_{{\mathbb{T}^3}}\big[ b'(\overline{T_k(\tilde{\varrho})})  \overline{  T_k(\tilde{\varrho})}   
-
 b(\overline{T_k(\tilde{\varrho} )}) \big]\mathrm{div}\tilde{\mathbf{u}}\,  \phi\, \mathrm{d}x \, \mathrm{d}t
\end{aligned}
\end{equation} 
$\mathbb{P}$-a.s. for any $\phi\in C^\infty_c([0,T)\times{\mathbb{T}^3})$. Now if we let $Q=(0,T) \times \mathbb{T}^3$ and set $Q_{k,M} =\left\{(\omega, t,x)\in \tilde{\Omega}\times Q \, :\,  \vert \overline{T_k(\tilde{\varrho})}  \vert \leq M  \right\}$, then it follows from Lemma \ref{lem:JakowFifthLayer} that
\begin{equation}
\begin{aligned}
\label{zeroConv}
&\left\Vert   b'(\overline{T_k(\tilde{\varrho})})  \overline{\left(T'_k(\tilde{\varrho})\tilde{\varrho}  -T_k(\tilde{\varrho}) \right)\mathrm{div}\,\tilde{\mathbf{u}}}   \right\Vert_{L^1(\tilde{\Omega}\times Q)}
\\
&\leq
\left(\sup_{0\leq z\leq M}\vert b'(z)\vert \right)  \liminf_{\delta \rightarrow 0}\big\Vert   \big( T_k(\tilde{\varrho}_\delta)  -  T'_k(\tilde{\varrho}_\delta)\tilde{\varrho}_\delta \big) \mathrm{div}\,\tilde{\mathbf{u}}_\delta \big\Vert_{L^1(\tilde{\Omega}\times Q)}
\\
&\leq
\left(\sup_{0\leq z\leq M}\vert b'(z)\vert \right) \left(\sup_{\delta>0} \Vert  \,\mathrm{div}\,\tilde{\mathbf{u}}_\delta \Vert_{L^2(\tilde{\Omega}\times Q)}  \right)\liminf_{\delta \rightarrow 0} \big\Vert    \, \big(T_k(\tilde{\varrho}_\delta)  -  T'_k(\tilde{\varrho}_\delta)\tilde{\varrho}_\delta \big)  \big\Vert_{L^{2}(Q_{k,M})}
\\
&\lesssim_{M}\, \liminf_{\delta \rightarrow 0} \big\Vert   \, \big(  T_k(\tilde{\varrho}_\delta)  -  T'_k(\tilde{\varrho}_\delta)\tilde{\varrho}_\delta \big) \big\Vert_{L^{2}(Q_{k,M})}
\end{aligned}
\end{equation}
where by  interpolation, we also have that
\begin{equation}
\begin{aligned}
\label{zeroConv1A}
\big\Vert &  \big(  T_k(\tilde{\varrho}_\delta)  -  T'_k(\tilde{\varrho}_\delta)\tilde{\varrho}_\delta \big) \big\Vert_{L^{2}(Q_{k,M})} 
\\
&\leq
\big\Vert   \big(  T_k(\tilde{\varrho}_\delta)  -  T'_k(\tilde{\varrho}_\delta)\tilde{\varrho}_\delta \big) \big\Vert_{L^{\gamma+1}(Q_{k,M})}^\frac{\gamma +1 }{2\gamma}
\,
\big\Vert   \big(  T_k(\tilde{\varrho}_\delta)  -  T'_k(\tilde{\varrho}_\delta)\tilde{\varrho}_\delta \big) \big\Vert_{L^1(\tilde{\Omega}\times Q))}^\frac{\gamma-1}{2\gamma}.
\end{aligned}
\end{equation}
It follows that
\begin{equation}
\begin{aligned}
\label{thus}
\liminf_{\delta \rightarrow 0} &\big\Vert   \, \big(  T_k(\tilde{\varrho}_\delta)  -  T'_k(\tilde{\varrho}_\delta)\tilde{\varrho}_\delta \big) \big\Vert_{L^{2}(Q_{k,M})}
\\& \lesssim
\limsup_{\delta \rightarrow 0} \big\Vert   \, \big( T_k(\tilde{\varrho}_\delta)  -  T'_k(\tilde{\varrho}_\delta)\tilde{\varrho}_\delta \big)  \big\Vert_{L^{1}(Q_{k,M})}^\frac{\gamma -1 }{2\gamma} 
\big\Vert   \big(  T_k(\tilde{\varrho}_\delta)  -  T'_k(\tilde{\varrho}_\delta)\tilde{\varrho}_\delta \big) \big\Vert_{L^{\gamma+1}(\tilde{\Omega}\times Q))}^\frac{\gamma+1}{2\gamma}
\end{aligned}
\end{equation}
where
\begin{equation}
\begin{aligned}
\label{equiZero}
\limsup_{\delta \rightarrow 0}
\big\Vert   \big(  T_k(\tilde{\varrho}_\delta)  &-  T'_k(\tilde{\varrho}_\delta)\tilde{\varrho}_\delta \big) \big\Vert_{L^1(\tilde{\Omega}\times Q))}^\frac{\gamma-1}{2\gamma}
\leq
\limsup_{\delta \rightarrow 0}
 \big\Vert    \, \tilde{\varrho}_\delta \,\mathbbm{1}_{ \{k\leq \tilde{\varrho}_\delta \}} \big\Vert_{L^1(\tilde{\Omega}\times Q))}^\frac{\gamma-1}{2\gamma}
 \\
&
\leq
\bigg(\frac{1}{k} \bigg)^{\frac{1}{2}\big(1-\frac{1}{\gamma} \big)^2}\limsup_{\delta \rightarrow 0}
 \big\Vert    \, \tilde{\varrho}_\delta  \big\Vert_{L^\gamma(\tilde{\Omega}\times Q))}^\frac{\gamma-1}{2\gamma} \rightarrow 0
\end{aligned}
\end{equation}
as $k\rightarrow \infty$ and by triangle inequality,
\begin{equation}
\begin{aligned}
\label{zeroConv1A}
\big\Vert \big(  T_k(\tilde{\varrho}_\delta)  -  T'_k(\tilde{\varrho}_\delta)\tilde{\varrho}_\delta \big) &\big\Vert_{L^{\gamma+1}(Q_{k,M})}^\frac{\gamma +1 }{2\gamma}
\lesssim 
 \Vert   (     \overline{T_k(\tilde{\varrho})} - T_k(\tilde{\varrho}) ) \Vert_{L^{\gamma+1}(Q_{k,M})}^\frac{\gamma +1 }{2\gamma}
    \\&+ \Vert  ( T_k(\tilde{\varrho}_\delta)  -  T_k(\tilde{\varrho}) ) \Vert_{L^{\gamma+1}(\tilde{\Omega}\times Q)}^\frac{\gamma +1 }{2\gamma}    +
  M^\frac{\gamma +1 }{2\gamma}.
\end{aligned}
\end{equation}
Now since 
\begin{align*}
 \Vert   (     \overline{T_k(\tilde{\varrho})} - T_k(\tilde{\varrho}) ) \Vert_{L^{\gamma+1}(Q_{k,M})}
   &\leq 
 \limsup_{\delta \rightarrow 0}
  \Vert   (    T_k(\tilde{\varrho}_\delta) - T_k(\tilde{\varrho}) ) \Vert_{L^{\gamma+1}(\tilde{\Omega}\times Q)},
\end{align*}
and $\frac{1}{2\gamma}<1$, we obtain from \eqref{zeroConv1A},
\begin{equation}
\begin{aligned}
\label{zeroConv5}
\big\Vert    \big( T_k(\tilde{\varrho}_\delta)  &-  T'_k(\tilde{\varrho}_\delta)\tilde{\varrho}_\delta \big)  \big\Vert_{L^{\gamma+1}(Q_{k,M})}^\frac{\gamma +1 }{2\gamma} 
\\&
\lesssim \limsup_{\delta \rightarrow 0}
\tilde{\mathbb{E}}\int_Q   \left\vert T_k(\tilde{\varrho}_\delta) -  T_k(\tilde{\varrho})\right\vert^{\gamma+1}\,\mathrm{d}x\, \mathrm{d}t  
  +  M^\frac{\gamma +1 }{2\gamma} 
\\&
\lesssim_M \mathbf{osc}_{\gamma+1}[\tilde{\varrho}_\delta \rightarrow \tilde{\varrho}](\tilde{\Omega}\times Q)   +1
\end{aligned}
\end{equation}
for a constant that is independent of $k$.

Now substituting \eqref{thus}, \eqref{equiZero} and \eqref{zeroConv5} into \eqref{zeroConv}, we obtain
\begin{equation}
\begin{aligned}
\big\Vert    b'(\overline{T_k(\tilde{\varrho})}) & \overline{\left(T'_k(\tilde{\varrho})\tilde{\varrho}  -T_k(\tilde{\varrho}) \right)\mathrm{div}\,\tilde{\mathbf{u}}}   \big\Vert_{L^1(\tilde{\Omega}\times Q)}
\\
&\lesssim_{M}\,
\bigg(\frac{1}{k} \bigg)^{\frac{1}{2}\big(1-\frac{1}{\gamma} \big)^2}\limsup_{\delta \rightarrow 0}
 \big\Vert    \, \tilde{\varrho}_\delta  \big\Vert_{L^\gamma(\tilde{\Omega}\times Q))}^\frac{\gamma-1}{2\gamma}
\bigg(
 \mathbf{osc}_{\gamma+1}[\tilde{\varrho}_\delta \rightarrow \tilde{\varrho}](\tilde{\Omega}\times Q)   +1\bigg)
\end{aligned}
\end{equation}
where the right-hand side converges to zero as $k\rightarrow\infty$ since
\begin{align}
 \mathbf{osc}_{\gamma+1}[\tilde{\varrho}_\delta \rightarrow \tilde{\varrho}](\tilde{\Omega}\times Q)   \lesssim 1
\end{align}
uniformly in $k$, recall \eqref{smallnesOfOscilDefectMeasure}. We have therefore shown that
\begin{align}
\label{tkOscZero}
 b'(\overline{T_k(\tilde{\varrho})})  \overline{\left(T'_k(\tilde{\varrho})\tilde{\varrho}  -T_k(\tilde{\varrho}) \right)\mathrm{div}\,\tilde{\mathbf{u}}}  \rightarrow 0 \text{ in } L^1(\tilde{\Omega} \times (0,T) \times \mathbb{T}^3)
\end{align}
as $k\rightarrow \infty$. The convergence \eqref{tkOscZero} together with
\begin{align}
\overline{T_k(\tilde{\varrho})} \rightarrow \tilde{\varrho} \text{ in } L^r(\tilde{\Omega} \times (0,T) \times \mathbb{T}^3)
\end{align}
for all $r\in (1,\gamma)$ as $k\rightarrow \infty$ (see \cite[Eq. (4.232)]{breit2018stoch} allows us to pass to the limit $k\rightarrow \infty$ in \eqref{renormalizedContFifthLayerLim1} and we obtain
\begin{equation}
\begin{aligned}
\label{renormalizedContFifthLayerLim2} 
&0=\int_0^T\int_{{\mathbb{T}^3}}  b(\tilde{\varrho})\,\partial_t\phi \,\mathrm{d}x \, \mathrm{d}t +
\int_{{\mathbb{T}^3}}  b(\tilde{\varrho}_{0})\,\phi(0) \,\mathrm{d}x 
\\&+
 \int_0^T\int_{{\mathbb{T}^3}} \big[b(\tilde{\varrho}) \tilde{\mathbf{u}}  \big] \cdot \nabla\phi \, \mathrm{d}x \, \mathrm{d}t  
- \int_0^T\int_{{\mathbb{T}^3}} \big[ \big(  b'(\tilde{\varrho}) \tilde{\varrho}  
-
 b(\tilde{\varrho} )\big)\mathrm{div}\tilde{\mathbf{u}} \big]\,  \phi\, \mathrm{d}x \, \mathrm{d}t
\end{aligned}
\end{equation} 
$\mathbb{P}$-a.s. for any $\phi\in C^\infty_c([0,T)\times{\mathbb{T}^3})$. If we now introduce
\begin{equation}
\begin{aligned}
\label{Lk}
L_k(z) =
\begin{dcases}
 z\log(z)  & \text{ if } z\in [0,k), \\
z\log(k)  + z \int_k^z \frac{T_k(s)}{s^2}\,\mathrm{d}s & \text{ if } z\in [k,\infty)
\end{dcases}
\end{aligned}
\end{equation}
which satisfies
\begin{equation}
\begin{aligned}
\label{LkTk}
zL'_k(z)- L_k(z) =T_k(z)
\end{aligned}
\end{equation}
in place of $b$ in \eqref{renormalizedContFifthLayerSeq} and \eqref{renormalizedContFifthLayerLim2}, then $\tilde{\mathbb{P}}$-a.s., we obtain
\begin{equation}
\begin{aligned}
\label{renormalizedContFifthLayerSeqLk} 
&0=\int_0^T\int_{{\mathbb{T}^3}}  L_k(\tilde{\varrho}_\delta)\,\partial_t\phi \,\mathrm{d}x \, \mathrm{d}t +
\int_{{\mathbb{T}^3}}  L_k\big(\tilde{\varrho}_{0,\delta}\big)\,\phi(0) \,\mathrm{d}x 
\\&+
 \int_0^T\int_{{\mathbb{T}^3}} \big[L_k(\tilde{\varrho}_\delta) \tilde{\mathbf{u}}_\delta  \big] \cdot \nabla\phi \, \mathrm{d}x \, \mathrm{d}t  
- \int_0^T\int_{{\mathbb{T}^3}}T_k(\tilde{\varrho}_\delta )\mathrm{div}\tilde{\mathbf{u}}_\delta \,  \phi\, \mathrm{d}x \, \mathrm{d}t
\end{aligned}
\end{equation} 
and
\begin{equation}
\begin{aligned}
\label{renormalizedContFifthLayerLimL_k} 
&0=\int_0^T\int_{{\mathbb{T}^3}}  L_k(\tilde{\varrho})\,\partial_t\phi \,\mathrm{d}x \, \mathrm{d}t +
\int_{{\mathbb{T}^3}}  L_k(\tilde{\varrho}_{0})\,\phi(0) \,\mathrm{d}x 
\\&+
 \int_0^T\int_{{\mathbb{T}^3}} \big[L_k(\tilde{\varrho}) \tilde{\mathbf{u}}  \big] \cdot \nabla\phi \, \mathrm{d}x \, \mathrm{d}t  
- \int_0^T\int_{{\mathbb{T}^3}} 
 T_k(\tilde{\varrho} )\mathrm{div}\tilde{\mathbf{u}} \,  \phi\, \mathrm{d}x \, \mathrm{d}t
\end{aligned}
\end{equation} 
respectively and where
\begin{align}
L_k(\tilde{\varrho}_\delta) \rightarrow \overline{L_k(\tilde{\varrho})}\quad &\text{in}\quad C_w\big( [0,T];L^p(\mathbb{T}^3)\big)
\label{l_kConver1}\\
\tilde{\varrho}_n\log(\tilde{\varrho}_\delta) \rightarrow \overline{\tilde{\varrho}\log(\tilde{\varrho})}\quad &\text{in}\quad C_w\big( [0,T];L^p(\mathbb{T}^3)\big)\label{logConver}
\\
L_k(\tilde{\varrho}_\delta) \rightarrow \overline{L_k(\tilde{\varrho})}\quad &\text{in}\quad C\big( [0,T];W^{-1,2}(\mathbb{T}^3)\big)
\\
T_k(\tilde{\varrho}_\delta)\mathrm{div}\,\tilde{\mathbf{u}}_\delta
\rightharpoonup \overline{T_k(\tilde{\varrho}) \mathrm{div}\,\tilde{\mathbf{u}}} 
\quad &\text{in}\quad
L^q \big( (0,T) \times \mathbb{T}^3 \big) \label{tk100}
\end{align}
holds $\tilde{\mathbb{P}}$-a.s. for any $p\in (1,\gamma)$ and some $q>1$. We can now use \eqref{l_kConver1}--\eqref{tk100} to pass to the limit in \eqref{renormalizedContFifthLayerSeqLk}  and we obtain
\begin{equation}
\begin{aligned}
\label{renormalizedContFifthLayerLimL_k1} 
&0=\int_0^T\int_{{\mathbb{T}^3}}  \overline{L_k(\tilde{\varrho})}\,\partial_t\phi \,\mathrm{d}x \, \mathrm{d}t +
\int_{{\mathbb{T}^3}}  L_k(\tilde{\varrho}_{0})\,\phi(0) \,\mathrm{d}x 
\\&+
 \int_0^T\int_{{\mathbb{T}^3}} \big[\overline{L_k(\tilde{\varrho})} \tilde{\mathbf{u}}  \big] \cdot \nabla\phi \, \mathrm{d}x \, \mathrm{d}t  
- \int_0^T\int_{{\mathbb{T}^3}} 
 \overline{T_k(\tilde{\varrho} )\mathrm{div}\tilde{\mathbf{u}} }\,  \phi\, \mathrm{d}x \, \mathrm{d}t.
\end{aligned}
\end{equation} 
If we now take the difference of \eqref{renormalizedContFifthLayerLimL_k} and \eqref{renormalizedContFifthLayerLimL_k1} and consider $\phi(t,x)= \phi_m(t)\phi_n(x)$ where $\phi_m$ and $\phi_n$ are approximation sequences of the characteristic functions $\chi_{[0,s]}$, $s\in[0,T]$ and $\chi_{\mathbb{T}^3}$ respectively, then we obtain
\begin{equation}
\begin{aligned}
\label{renormalizedContFifthLayerLimL_k2} 
&\int_0^s\int_{{\mathbb{T}^3}}  \big[L_k(\tilde{\varrho}) - \overline{L_k(\tilde{\varrho})} \big]\,\mathrm{d}x  
= \int_0^s\int_{{\mathbb{T}^3}} 
\big[ \overline{T_k(\tilde{\varrho} )\mathrm{div}\tilde{\mathbf{u}} } - T_k(\tilde{\varrho} )\mathrm{div}\tilde{\mathbf{u}}  \big]\,  \mathrm{d}x \, \mathrm{d}t
\\&= \int_0^s\int_{{\mathbb{T}^3}} 
\big[ \overline{T_k(\tilde{\varrho} )\mathrm{div}\tilde{\mathbf{u}} } - \overline{T_k(\tilde{\varrho} )}\mathrm{div}\tilde{\mathbf{u}}  \big]\,  \mathrm{d}x \, \mathrm{d}t
+ \int_0^s\int_{{\mathbb{T}^3}} 
\big[ \overline{T_k(\tilde{\varrho} )}  - T_k(\tilde{\varrho} ) \big]\, \mathrm{div}\tilde{\mathbf{u}} \,  \mathrm{d}x \, \mathrm{d}t
\\&=: I_1 + I_2
\end{aligned}
\end{equation} 
where by a similar argument as in \eqref{eq:57az}, we have that
\begin{equation}
I_1 \geq \frac{1}{(\nu^B +2\nu^S)}
\limsup_{\delta \rightarrow 0}
 \int_0^s\int_{{\mathbb{T}^3}} 
\big\vert T_k(\tilde{\varrho}_\delta) - T_k(\tilde{\varrho}) 
\big\vert^{\gamma+1} \,\mathrm{d}x\, \mathrm{d}t.
\end{equation}
By the bounded of $\tilde{\mathbf{u}}$ in $L^2_tW^{1,2}_x$, the smallness of the oscillation defect measure \eqref{smallnesOfOscilDefectMeasure} and properties of $T_k$, we also have that
\begin{equation}
\begin{aligned}
I_2 \rightarrow 0
\end{aligned}
\end{equation}
as $k\rightarrow\infty$. Also, one can verify that the left-hand side of \eqref{renormalizedContFifthLayerLimL_k2}  is bounded by using \eqref{LkTk}. The collection of the above information implies that
\begin{align}
\int_{\mathbb{T}^3}\big[\overline{\tilde{\varrho} \log \tilde{\varrho}} - \tilde{\varrho} \log \tilde{\varrho} \big](t)\, \mathrm{d}x =0 \quad \text{for all} \quad t\in [0,T]
\end{align}
as $k\rightarrow \infty$ in \eqref{renormalizedContFifthLayerLimL_k2}. This implies that
\begin{align}
\label{strongDensCovFifthLayer}
\tilde{\varrho}_\delta(t) \rightarrow \tilde{\varrho}(t) \quad \text{in} \quad L^1(\mathbb{T}^3)\quad \text{for any} \quad t\in[0,T]
\end{align}
since $\tilde{\varrho}\mapsto \tilde{\varrho}\log \tilde{\varrho}$ is strictly convex. It also follows from \eqref{strongDensCovFifthLayer} that
\begin{align}
\label{strongDelVCovFifthLayer}
\Delta\tilde{V}_\delta(t) \rightarrow \Delta\tilde{V}(t) \quad \text{in} \quad L^1(\mathbb{T}^3)\quad \text{for any} \quad t\in[0,T].
\end{align}
We can now use the Lipschitz continuity of $\mathbf{g}_k$ and \eqref{strongDensCovFifthLayer} to obtain
\begin{align}
\int_{{\mathbb{T}^3}}  \tilde{\varrho}_\delta\, \mathbf{g}_{k}(\tilde{\varrho}_\delta , f, \tilde{\varrho}_\delta \tilde{\mathbf{u}}_\delta ) \cdot \bm{\phi}  \, \mathrm{d}x
\rightarrow
\int_{{\mathbb{T}^3}}  \tilde{\varrho}\, \mathbf{g}_{k}(\tilde{\varrho} , f, \tilde{\varrho} \tilde{\mathbf{u}} ) \cdot \bm{\phi}  \, \mathrm{d}x \quad \text{a.e. in } (0,T)
\end{align}
$\mathbb{P}$-a.s. for any $\bm{\phi} \in C^\infty(\mathbb{T}^3)$ from which we infer that
\begin{align}
\label{almostEverywhereNoiseFifthLayer}
\overline{  \tilde{\varrho}\, \mathbf{g}_{k}(\tilde{\varrho} , f, \tilde{\varrho} \tilde{\mathbf{u}} ) }
=
  \tilde{\varrho}\, \mathbf{g}_{k}(\tilde{\varrho} , f, \tilde{\varrho} \tilde{\mathbf{u}} ) \quad \text{a.e. in } \tilde{\Omega} \times (0,T) \times \mathbb{T}^3.
\end{align}
We are therefore able to conclude that the nonlinear pressure term in \eqref{weakMomEqLimitFifthLayer} is indeed isentropic  by the use of \eqref{strongDelVCovFifthLayer}. Furthermore, the nonlinear noise term in \eqref{weakMomEqLimitFifthLayer} is also of the require form since we have \eqref{almostEverywhereNoiseFifthLayer}. Finally as in Lemma \ref{lem:identifyEnerFourthLayer}, we also have the following result.
\begin{lemma}
\label{lem:identifyEnerFifthLayer}
The random distributions  $[\tilde{\varrho}, \tilde{\mathbf{u}},  \tilde{V}]$ satisfies \eqref{augmentedEnergy} for all $\psi \in C^\infty_c ([0,T))$, $\psi \geq 0$ $\mathbb{P}$-a.s.
\end{lemma}

\section{Existence on the whole space}
\label{sec:existenceWholeSpace}
In this section we deal with the existence in the whole domain case $\mathcal{O}=\mathbb{R}^{3}$. 
For $(x,\varrho,f, \mathbf{u}) \in {\mathbb{R}^3}\times \mathbb{R}_{\geq0}\times \mathbb{R}_{\geq0}\times{\mathbb{R}^3} $,  we  assume that there exists some functions $(\mathbf{g}_k)_{k\in \mathbb{N}}$  such that for any $k\in\mathbb{N}$,
\begin{align}
\label{noiseSupportExiWS}
\mathbf{g}_k: {\mathbb{R}^3}\times \mathbb{R}_{\geq0}\times \mathbb{R}_{\geq0}\times{\mathbb{R}^3}  \rightarrow {\mathbb{R}^3}, \quad \quad \mathbf{g}_k\in C^1 \left( \mathbb{R}^3 \times \mathbb{R}_{>0}\times \mathbb{R}_{\geq 0}\times{\mathbb{R}^3}  \right)
\end{align}
and the support of each $\mathbf{g}_k$ in $\mathbb{R}^3$ is $K$ where $K\Subset \mathbb{R}^3$, i.e.,
\begin{equation}
\label{cptsptnoiseWS}
\exists K \Subset \mathbb{R}^3 \quad \text{such that} \quad
\mathbf{g}_k=0 \quad \text{in}\quad  \mathbb{R}^3\setminus K.
\end{equation}
In addition, $ \mathbf{g}_k$ satisfies the growth conditions \eqref{stochCoeffBoundExi0}--\eqref{stochsummableConst}. 
\\
We study the problem under the far-field condition
\begin{align}
\label{farfield}
\mathbf{u}\rightarrow 0, \quad \varrho \rightarrow\overline{\varrho}>0
\end{align}
as $\vert x \vert \rightarrow\infty$. By the above conditions, with the same lines of arguments of Section \ref{sec:noiseAssumption} we have that the stochastic integral $\int_0^\cdot\varrho \,\mathbf{G}(\varrho, f, \varrho \mathbf{u})\mathrm{d}W$ is  well-defined.
\subsection{Concept of a solution}
\begin{definition}
\label{def:martSolutionExisWS}
Let $\Lambda= \Lambda(\varrho, \mathbf{m},f)$ be a Borel probability measure on $\big[ L^1(\mathbb{R}^3)\big]^3$. We say that
\begin{align}
\label{weakMartSolWS}
\big[(\Omega ,\mathscr{F} ,(\mathscr{F} _t)_{t\geq0},\mathbb{P} );\varrho , \mathbf{u} , V , W   \big]
\end{align}
is a \textit{dissipate martingale solution} of \eqref{contEq}--\eqref{elecField} with initial law $\Lambda$ provided
\begin{enumerate}
\item $(\Omega,\mathscr{F},(\mathscr{F}_t),\mathbb{P})$ is a stochastic basis with a complete right-continuous filtration;
\item $W$ is a $(\mathscr{F}_t)$-cylindrical Wiener process;
\item the density $\varrho \in C_w \big( [0,T]; L^\gamma (\mathbb{R}^3)\big) $ $\mathbb{P}$-a.s., it is $(\mathscr{F}_t)$-progressively measurable  and 
\begin{align*}
\mathbb{E}\, \bigg[ \sup_{t\in(0,T)}\Vert  \varrho  (t)   \Vert^{\overline{\gamma}p}_{L^{\overline{\gamma}} (\mathbb{R}^3)}\bigg]<\infty 
\end{align*}
for all $1\leq p<\infty$ where $\overline{\gamma}=\min\{2,\gamma\}$,
\item the velocity field $\mathbf{u} $ is an $(\mathscr{F}_t)$-adapted random distribution and 
\begin{align*}
\mathbb{E}\, \bigg[ \int_0^T\Vert  \mathbf{u}  \Vert^2_{W^{1,2}(\mathbb{R}^3) }\,\mathrm{d}t\bigg]^p<\infty 
\end{align*}
for all $1\leq p<\infty$,
\item the momentum $\varrho  \mathbf{u}  \in C_w \big( [0,T]; L^{\frac{2\gamma}{\gamma+1}} (\mathbb{R}^3)\big) $ $\mathbb{P}$-a.s., it is $(\mathscr{F}_t)$-progressively measurable and 
\begin{align*}
\mathbb{E}\, \bigg[ \sup_{t\in(0,T)} \big\Vert  \varrho \vert \mathbf{u} \vert^2(t)  \big\Vert_{L^1(\mathbb{R}^3)}^p\bigg]
+
\mathbb{E}\, \bigg[ \sup_{t\in(0,T)} \big\Vert  \varrho  \mathbf{u}  (t)\big\Vert_{L^\frac{2\gamma}{\gamma+1}(\mathbb{R}^3)}^{\frac{2\gamma}{\gamma+1}p}\bigg]
<\infty 
\end{align*}
for all $1\leq p<\infty$;
\item the field $\Delta V \in C_w \big( [0,T]; L^\gamma(\mathbb{R}^3)\big) $ $\mathbb{P}$-a.s., it is $(\mathscr{F}_t)$-progressively measurable  and 
\begin{align*}
\mathbb{E}\, \bigg[ \sup_{t\in(0,T)}\Vert \Delta V  (t)   \Vert^{\overline{\gamma}p}_{L^{\overline{\gamma}}(\mathbb{R}^3)}\bigg]
+
\mathbb{E}\, \bigg[ \sup_{t\in(0,T)}\Vert \nabla V  (t)   \Vert^{2p}_{L^{2}(\mathbb{R}^3)}\bigg]
<\infty 
\end{align*}
for all $1\leq p<\infty$ where $\overline{\gamma}=\min\{2,\gamma\}$,
\item given $f\in L^1(\mathbb{R}^3) \cap L^\infty(\mathbb{R}^3)$, there exists $\mathscr{F}_0$-measurable random variables $(\varrho_0, \varrho_0\mathbf{u}_0)= (\varrho(0), \varrho\mathbf{u}(0))$ such that $\Lambda = \mathbb{P}\circ (\varrho_0, \varrho_0 \mathbf{u}_0, f)^{-1}$;
\item for all $\psi \in C^\infty_c ([0,T))$ and $\phi \in C^\infty_c ({\mathbb{R}^3})$, the following
\begin{equation}
\begin{aligned}
\label{eq:distributionalSolMassWS}
&-\int_0^T \partial_t \psi \int_{{\mathbb{R}^3}} \varrho (t) \phi \, \mathrm{d}x  \mathrm{d}t  =  \psi(0) \int_{{\mathbb{R}^3}} \varrho _0 \phi \,\mathrm{d}x  +  \int_0^T \psi \int_{{\mathbb{R}^3}} \varrho  \mathbf{u} \cdot \nabla \phi \, \mathrm{d}x  \mathrm{d}t,
\end{aligned}
\end{equation}
hold $\mathbb{P}$-a.s.;
\item for all $\psi \in C^\infty_c ([0,T))$ and $\bm{\phi} \in C^\infty_c ({\mathbb{R}^3})$, the following
\begin{equation}
\begin{aligned}
\label{eq:distributionalSolMomenWS}
&-\int_0^T \partial_t \psi \int_{{\mathbb{R}^3}} \varrho  \mathbf{u} (t) \cdot \bm{\phi} \, \mathrm{d}x    \mathrm{d}t
=
\psi(0)  \int_{{\mathbb{R}^3}} \varrho _0 \mathbf{u} _0\cdot \bm{\phi} \, \mathrm{d}x  
+  \int_0^T \psi \int_{{\mathbb{R}^3}} \varrho  \mathbf{u} \otimes \mathbf{u} : \nabla \bm{\phi}\, \mathrm{d}x  \mathrm{d}t 
\\&- \nu^S\int_0^T \psi \int_{{\mathbb{R}^3}} \nabla \mathbf{u} : \nabla \bm{\phi} \, \mathrm{d}x  \mathrm{d}t 
-(\nu^B+ \nu^S )\int_0^T \psi \int_{{\mathbb{R}^3}} \mathrm{div}\,\mathbf{u} \, \mathrm{div}\, \bm{\phi}\, \mathrm{d}x  \mathrm{d}t    
\\&+   
\int_0^T \psi \int_{{\mathbb{R}^3}} p(\varrho)\, \mathrm{div} \, \bm{\phi} \, \mathrm{d}x  \mathrm{d}t
+\int_0^T \psi \int_{{\mathbb{R}^3}} \vartheta\varrho  \nabla V  \cdot \bm{\phi} \, \mathrm{d}x    \mathrm{d}t
\\&+
\int_0^T \psi \int_{{\mathbb{R}^3}}  \varrho\,\mathbf{G}(\varrho , f, \mathbf{m} ) \cdot \bm{\phi}\, \mathrm{d}x \mathrm{d}W   
\end{aligned}
\end{equation}
hold $\mathbb{P}$-a.s.;
\item equation \eqref{elecField} holds $\mathbb{P}$-a.s. for a.e.  $(t,x) \in(0,T)\times \mathbb{R}^3$; 
\item the  energy inequality 
\begin{equation}
\begin{aligned}
\label{augmentedEnergyWS}
-
&\int_0^T \partial_t \psi \int_{{\mathbb{R}^3}}\bigg[ \frac{1}{2} \varrho  \vert \mathbf{u}  \vert^2  
+ H(\varrho, \overline{\varrho} ) \pm
\vartheta \vert \nabla V \vert^2
\bigg]\,\mathrm{d}x \, \mathrm{d}t
+\nu^S
\int_0^T\psi \int_{{\mathbb{R}^3}}\vert  \nabla \mathbf{u}  \vert^2 \,\mathrm{d}x  \,\mathrm{d}x
\\&+(\nu^B+\nu^S)
\int_0^T \psi \int_{{\mathbb{R}^3}}\vert  \mathrm{div}\, \mathbf{u}  \vert^2 \,\mathrm{d}x  \,\mathrm{d}t
\leq
\psi(0)
 \int_{{\mathbb{R}^3}}\bigg[\frac{1}{2} \varrho _0\vert \mathbf{u} _0 \vert^2  
+ H(\varrho_0, \overline{\varrho} ) \pm
\vartheta \vert \nabla V_0 \vert^2 \bigg]\,\mathrm{d}x
\\&
+\frac{1}{2}\int_0^T \psi
 \int_{{\mathbb{R}^3}}
 \varrho
 \sum_{k\in\mathbb{N}} \big\vert
 \mathbf{g}_k(x,\varrho ,f,\mathbf{m}  ) 
\big\vert^2\,\mathrm{d}x\,\mathrm{d}t
+\int_0^T \psi  \int_{{\mathbb{R}^3}}   \varrho\mathbf{u} \cdot\mathbf{G}(\varrho , f, \mathbf{m} )\,\mathrm{d}x\,\mathrm{d}W ;
\end{aligned}
\end{equation}
holds $\mathbb{P}$-a.s. for all $\psi \in C^\infty_c ([0,T))$, $\psi\geq0$ where 
\begin{equation} 
\begin{aligned}
\label{HVarrhoWS}
H(\varrho, \overline{\varrho} ) = P(\varrho) -P'(\overline{\varrho})(\varrho -\overline{\varrho}) -P(\overline{\varrho})
\end{aligned}
\end{equation}
and $P $ is given by \eqref{PVarrho};
\item and  \eqref{contEq} holds in the renormalized sense, i.e., for any $\phi\in C^\infty_c([0,T)\times{\mathbb{R}^3})$ and $b \in C^1_b(\mathbb{R})$ such that $ b'(z)=0$ for all $z\geq M_b$,  we have that
\begin{equation}
\begin{aligned}
\label{renormalizedContExisWS}
-&\int_0^T\int_{{\mathbb{R}^3}}  b(\varrho )\,\partial_t\phi \,\mathrm{d}x \, \mathrm{d}t =
\int_{{\mathbb{R}^3}}  b\big(\varrho (0)\big)\,\phi(0) \,\mathrm{d}x 
\\&
+ \int_0^T\int_{{\mathbb{R}^3}} \big[b(\varrho ) \mathbf{u}  \big] \cdot \nabla\phi \, \mathrm{d}x \, \mathrm{d}t  
- \int_0^T\int_{{\mathbb{R}^3}} \big[ \big(  b'(\varrho ) \varrho   
-
 b(\varrho )\big)\mathrm{div}\mathbf{u} \big]\,  \phi\, \mathrm{d}x \, \mathrm{d}t
\end{aligned}
\end{equation} 
\end{enumerate}
holds $\mathbb{P}$-a.s.
\end{definition} 
\subsection{Main result}
The main result is the following.
\begin{theorem}
\label{thm:WS}
Let $\overline{\varrho}\geq1$ and let $\Lambda = \Lambda(\varrho, \mathbf{m},f)$ be a Borel probability measure on $\big[ L^1 (\mathbb{R}^3)\big]^3$ such that
\begin{align*}
\Lambda\bigg\{\varrho\geq0, \quad M \leq  \varrho  \leq M^{-1}, \quad f \leq \overline{f}, \quad \mathbf{m} =0 \quad \text{ when } \varrho=0  \bigg\} =1, \quad
\end{align*}
holds for some deterministic constants $M, \overline{f}>0$. Also 
 assume that
\begin{align}
\label{finiteEnergyLaw}
\int_{[L^1 (\mathbb{R}^3) ]^3} \bigg\vert  
\int_{\mathbb{R}^3} \bigg[ \frac{\vert \mathbf{m} \vert^2}{2\varrho} + H(\varrho, \overline{\varrho} )  \pm
\vartheta \vert \nabla V \vert^2
\bigg] \, \mathrm{d}x
\bigg\vert^p \, \mathrm{d}\Lambda(\varrho, \mathbf{m},f) \lesssim 1
\end{align}
holds for some $p\geq 1$ and that
\begin{align*}
f\in L^1 (\mathbb{R}^3) \cap L^\infty (\mathbb{R}^3)\quad \mathbb{P}\text{-a.s.}
\end{align*} Finally assume that \eqref{noiseSupportExi}--\eqref{stochsummableConst} as well as \eqref{cptsptnoiseWS} are satisfied. Then the exists a dissipative martingale solution of \eqref{contEq}--\eqref{elecField} in the sense of Definition \ref{def:martSolutionExisWS}.
\end{theorem}
\subsection{Construction of law}
\label{sec:intialDataTorus}
To obtain a suitable initial law satisfying the preamble of Theorem \ref{thm:WS}, we first consider the following family of cut-off functions
\begin{equation} \label{periodicCutOff}
\begin{cases}
\eta_L\in C^\infty_0 \big([-L,L]^3 \big), \quad &0\leq \eta_L \leq 1, \\
\eta_L \equiv 1 \text{ in } \big[-\frac{L}{2} , \frac{L}{2} \big]^3
\end{cases}
\end{equation}
defined for $L \geq1$ and let $\overline{\varrho}>0$ be the anticipated far-field condition in \eqref{farfield}. Since the law $\Lambda$ in Theorem \ref{thm:WS}  is a measure on a Polish space, it follows from Skorokhod's theorem  that there exists some $\mathscr{F}_0$-measurable random variables $(\varrho_0, \,\mathbf{m}_0, f)$ defined on some probability space $(\Omega, \mathscr{F}, \mathbb{P})$ and having $\Lambda$ as its law. With  \eqref{periodicCutOff} in hand, we construct the following family
\begin{equation}
\begin{aligned}
\label{familyOfData}
&\varrho_{0,L} =\eta_L \varrho_0 +(1-\eta_L )\overline{\varrho}, \quad \mathbf{m}_{0,L} :=\varrho_{0,L}\mathbf{u}_{0,L} =\eta_L \frac{\sqrt{\varrho_{0,L}}}{\sqrt{\varrho_0}}\mathbf{m}_0, 
\\& 
V_{0,L}  =\pm \Delta^{-1}_{[-L,L]^3}(\varrho_{0,L}-f)
\end{aligned}
\end{equation}
of periodic functions having the property that
\begin{align}
\varrho_{0,L}\big\vert_{\partial [-L,L]^3}= \overline{\varrho}, \quad \mathbf{m}_{0,L}\big\vert_{\partial [-L,L]^3} =0
\end{align}
and hence satisfies the far-field condition \eqref{farfield}. We also have that
\begin{align}
(\varrho_{0,L},
\mathbf{m}_{0,L} ) \rightarrow ( \varrho_0, \mathbf{m}_0) \quad\text{a.e. in} \quad {  \mathbb{R}^{3} } 
\end{align}
as $L\rightarrow \infty $. Now for an arbitrary $K\Subset \mathbb{R}^3$  and  a choice of $L\gg 1$ such that $K\subset  [-L,L]^3$, we have that $\vert \varrho_{0,L} \vert \lesssim \varrho_0 +\overline{\varrho}$ and $\vert \mathbf{m}_{0,L} \vert \lesssim 1+ \mathbf{m}_0$
holds  uniformly in $L$. Furthermore, per the assumptions on $\Lambda$,  $\varrho_0 +\overline{\varrho}  \in L^1(K)$ and  $1+\mathbf{m}_0 \in L^1(K)$ so that
\begin{align}
(\varrho_{0,L},
\mathbf{m}_{0,L}, f ) \rightarrow ( \varrho_0, \mathbf{m}_0,f)\quad \text{in} \quad \big[L^1(K)\big]^3
\end{align}
a.s. Subsequently,  we gain
\begin{align}
\label{convOfMeasures}
\Lambda_L=\mathbb{P}\circ (\varrho_{0,L},
\mathbf{m}_{0,L},f )^{-1}
\,\xrightharpoonup{*} \,
\mathbb{P}\circ (\varrho_{0},
\mathbf{m}_{0},f )^{-1} =\Lambda
\end{align}
in the sense of measures on $\big[ L^1_{\mathrm{loc}}(\mathbb{R}^3) \big]^3$ by the arbitrariness of  $K\Subset \mathbb{R}^3$.  

\subsection{A priori bounds} 
\label{sec:aPriori}
By periodicity and invariance, it follows from \eqref{thm:one} that there exists  a family of dissipative martingale solutions 
\begin{align}
\label{weakMartSol}
\big[ \big(\Omega,\mathscr{F},(\mathscr{F}_t),\mathbb{P} \big);\varrho_L, \mathbf{u}_L, V_L, W  \big]
\end{align} 
in the sense of Definition  \ref{def:martSolutionExis}  which are defined for
{ $\overline{\mathrm{d}\mathbb{P}\otimes \mathrm{d}t}$} a.e. $(\omega,t)\in \Omega\times[0,T]$ on the periodic domains $\mathbb{T}^3_L$ for $L\geq1$. Without loss of generality, we have chosen the family of solutions \eqref{weakMartSol} to be defined on the same stochastic basis and also driven by the same Wiener process. We may also choose $L \gg1$ large enough so that the compact set $K$ in \eqref{cptsptnoiseWS} is contained in $\mathbb{T}^3_L$.
\\
The prescribed laws for  \eqref{weakMartSol} are the Borel  probability measures
 $\Lambda_L = \mathbb{P}\circ (\varrho_{0,L}\,,\, \varrho_{0,L} \mathbf{u}_{0,L}, f)^{-1}$ defined on
$\big[ L^1\big(\mathbb{T}^3_L \big)\big]^3$ and which are assumed to satisfy
\begin{equation}
\begin{aligned}
\label{iniLawTorus}
&\Lambda_L\{ \varrho\geq 0 \}=1,
\, \quad
&\int_{[L^1_x]^3}\bigg\vert 
\int_{\mathbb{T}^3_L} \bigg( \frac{1}{2}\frac{\vert \mathbf{m} \vert^2}{\varrho}  + P(\varrho) \pm
\vartheta \vert \nabla V \vert^2
\bigg) \, \mathrm{d}x  \bigg\vert^q \, \mathrm{d}\Lambda_L(\varrho, \mathbf{m},f) < \infty, 
\end{aligned}
\end{equation}
for some $q\geq1$.
\\
Furthermore, in analogy with \eqref{noiseSupportExi}-- \eqref{stochsummableConst}, for $\mathbf{m}_L :=\varrho_L\mathbf{u}_L$, the noise term $\mathbf{G}(\varrho_L, \mathbf{m}_L,f) : \mathfrak{U} \rightarrow L^1(\mathbb{T}^3_L)$ is defined as
\begin{align}
\mathbf{G}(\varrho_L, \mathbf{m}_L,f)e_k 
=
\mathbf{g}_k(\cdot, \varrho_L(\cdot), \mathbf{m}_L(\cdot),f(\cdot))
\end{align}
for all $k\in \mathbb{N}$ such that
\begin{align}
\label{noiseSupportExiTorusL}
\mathbf{g}_k: {\mathbb{T}^3_L}\times \mathbb{R}_{\geq0}\times \mathbb{R}_{\geq0}\times{\mathbb{T}^3_L}  \rightarrow {\mathbb{R}^3}, \quad \quad \mathbf{g}_k\in C^1 \left( {\mathbb{T}^3_L}\times \mathbb{R}_{>0}\times \mathbb{R}_{\geq 0}\times{\mathbb{T}^3_L}  \right)
\end{align}
for any $k\in\mathbb{N}$  and in addition, $ \mathbf{g}_k$ satisfies the following growth conditions:
\begin{align}
&\vert  \mathbf{g}_k(x, \varrho_L, f, \mathbf{m}_L)   \vert  \leq c_k \,\left( \varrho_L + f+ \vert \mathbf{m}_L\vert \right),
\label{stochCoeffBoundExi0TorusL}
\\
&\big\vert\nabla_{\varrho,f, \mathbf{m}} \, \mathbf{g}_k(x, \varrho_L, f, \mathbf{m}_L) \big) \big\vert   \leq c_k,
\label{stochCoeffBound1Exis0TorusL}
\\
&\sum_{k\in \mathbb{N}} c_k^2 \lesssim 1 \label{stochsummableConstTorusL}
\end{align}
for some constant $(c_k)_{k\in \mathbb{N}} \in [0,\infty)$.
\\
If we now approximate the characteristic map $t \mapsto \chi_{[0,t]}$ by  a sequence of nonnegative
compactly supported smooth functions $(\psi_m)_{m\in \mathbb{N}}$ and observe that the resulting inequality is preserved under an affine perturbation of $P(z)$, then for
\begin{align}
H(z,\overline{\varrho}) =P(z) - P'(\overline{\varrho})(z-\overline{\varrho})-P(\overline{\varrho}),
\end{align}
where $\overline{\varrho}>0$, it follows from \eqref{augmentedEnergy} that
\begin{equation}
\begin{aligned}
\label{augmentedEnergyTorusL}
&  \int_{{\mathbb{T}^3_L}}\bigg[ \frac{1}{2} \varrho_L  \vert \mathbf{u}_L  \vert^2  
+ H(\varrho_L,\overline{\varrho} ) \pm
\vartheta \vert \nabla V_L \vert^2
\bigg](t)\,\mathrm{d}x 
+\nu^S
\int_0^t \int_{{\mathbb{T}^3_L}}\vert  \nabla \mathbf{u}_L  \vert^2 \,\mathrm{d}x  \,\mathrm{d}s
\\&+(\nu^B+\nu^S)
\int_0^t \int_{{\mathbb{T}^3_L}}\vert  \mathrm{div}\, \mathbf{u}_L  \vert^2 \,\mathrm{d}x  \,\mathrm{d}s
\leq
 \int_{{\mathbb{T}^3_L}}\bigg[\frac{1}{2} \varrho _{0,L}\vert \mathbf{u} _{0,L} \vert^2  
+ H(\varrho _{0,L},\overline{\varrho}) \pm
\vartheta \vert \nabla V_{0,L} \vert^2 \bigg]\,\mathrm{d}x
\\&
+\frac{1}{2}\int_0^t
 \int_{{\mathbb{T}^3_L}}
 \varrho_L
 \sum_{k\in\mathbb{N}} \big\vert
 \mathbf{g}_k(x,\varrho_L ,f,\mathbf{m}_L  ) 
\big\vert^2\,\mathrm{d}x\,\mathrm{d}s
+\int_0^t \int_{{\mathbb{T}^3_L}}   \varrho_L\mathbf{u}_L \cdot\mathbf{G}(\varrho_L , f, \mathbf{m}_L )\,\mathrm{d}x\,\mathrm{d}W ;
\end{aligned}
\end{equation}
holds $\mathbb{P}$-a.s. for all $\psi \in C^\infty_c ([0,T))$, $\psi\geq0$.
\\
By using\eqref{periodicCutOff} and the convexity of $\eta_L \mapsto H(\varrho_{0,L}, \overline{\varrho})$,
we obtain
\begin{equation}
\begin{aligned}
\label{initialTorusLR}
 \int_{{\mathbb{T}^3_L}}&\bigg[\frac{1}{2} \varrho _{0,L}\vert \mathbf{u} _{0,L} \vert^2  
+ H(\varrho _{0,L},\overline{\varrho}) \pm
\vartheta \vert \nabla V_{0,L} \vert^2 \bigg]\,\mathrm{d}x
\\&
\leq
 \int_{{\mathbb{R}^3}}\bigg[\frac{1}{2} \varrho _{0}\vert \mathbf{u} _{0} \vert^2  
+ H(\varrho _{0},\overline{\varrho}) \pm
\vartheta \vert \nabla V_{0} \vert^2 \bigg]\,\mathrm{d}x
\end{aligned}
\end{equation}
$\mathbb{P}$-a.s., c.f. \cite[Page 57]{mensah2019theses}. In comparison with \eqref{eq:gk1}, it follows from \eqref{stochCoeffBoundExi0TorusL} that for any $p\in[1,\infty)$,
\begin{equation}
\begin{aligned}
\label{supNoiseL}
\mathbb{E}&\bigg[\sup_{t\in[0,T]}
\bigg\vert
\int_0^t
 \int_{{\mathbb{T}^3_L}}
 \varrho_L
 \sum_{k\in\mathbb{N}} \big\vert
 \mathbf{g}_k(x,\varrho_L ,f,\mathbf{m}_L  ) 
\big\vert^2\,\mathrm{d}x\,\mathrm{d}s
\bigg\vert \bigg]^p
\\&
\lesssim_p
\mathbb{E} \int_0^T
\bigg(
\int_K
\big( 1+ \varrho_L^\gamma + \varrho_L\vert \mathbf{u}_L\vert^2 + f^\frac{2\gamma}{\gamma-1} \big)\, \mathrm{d}x \bigg)^p \mathrm{d}s
\end{aligned}
\end{equation}
holds uniformly in $L$ where $K$ is given is \eqref{cptsptnoiseWS}. The use of the Burkholder--Davis--Gundy inequality also yields 
\begin{equation}
\begin{aligned}
\label{supNoiseL1}
\mathbb{E}&\bigg[\sup_{t\in[0,T]}
\bigg\vert
\int_0^t \int_{{\mathbb{T}^3_L}}   \varrho_L\mathbf{u}_L \cdot\mathbf{G}(\varrho_L , f, \mathbf{m}_L )\,\mathrm{d}x\,\mathrm{d}W
\bigg\vert \bigg]^p
\lesssim_p
\mathbb{E}\bigg( \sup_{t\in[0,T]} 
\int_{\mathbb{T}^3_L} \varrho_L \vert \mathbf{u}_L \vert^2 \, \mathrm{d}x \bigg)^p
\\&+
\mathbb{E} \int_0^T
\bigg(
\int_K
\big( 1+ \varrho_L^\gamma + \varrho_L\vert \mathbf{u}_L\vert^2 + f^\frac{2\gamma}{\gamma-1} \big)\, \mathrm{d}x \bigg)^p \mathrm{d}s
\end{aligned}
\end{equation}
holds uniformly in $L$.
If we now take the $p$th-moment of the supremum in \eqref{augmentedEnergyTorusL} and use Gronwall's lemma, then it follows from \eqref{initialTorusLR}--\eqref{supNoiseL1} that
\begin{equation}
\begin{aligned}
\label{est}
&\mathbb{E} \bigg[\sup_{t\in[0,T]}
\int_{{\mathbb{T}^3_L}}\bigg[\frac{1}{2} \varrho _{L}\vert \mathbf{u} _{L} \vert^2  
+ H(\varrho _{L},\overline{\varrho}) \pm
\vartheta \vert \nabla V_{L} \vert^2 \bigg]\,\mathrm{d}x \bigg]^p
+\nu^S\mathbb{E}\bigg[ 
\int_0^T \int_{{\mathbb{T}^3_L}}\vert  \nabla \mathbf{u}_L  \vert^2 \,\mathrm{d}x  \,\mathrm{d}s \bigg]^p
\\&+(\nu^B+\nu^S)\mathbb{E}\bigg[
\int_0^T \int_{{\mathbb{T}^3_L}}\vert  \mathrm{div}\, \mathbf{u}_L  \vert^2 \,\mathrm{d}x  \,\mathrm{d}s\bigg]^p
\lesssim_p 1+
\mathbb{E} \bigg[
\int_{{\mathbb{R}^3}}\bigg[\frac{1}{2} \varrho _{0}\vert \mathbf{u} _{0} \vert^2  
+ H(\varrho _{0},\overline{\varrho}) \pm
\vartheta \vert \nabla V_{0} \vert^2 \bigg]\,\mathrm{d}x \bigg]^p
\end{aligned}
\end{equation}
holds uniformly in $L$  where the right-hand side is finite as a result of \eqref{finiteEnergyLaw}. In particular, it therefore follow from \eqref{est} that
\begin{equation}
\begin{aligned}
\label{est0}
&\mathbb{E} \,\bigg[ \sup_{t\in[0,T]}  \big\Vert \varrho_L\vert \mathbf{u}_L\vert^2 \big\Vert_{L^1(\mathbb{T}^3_L)}\bigg]^p  \lesssim 1, 
&\quad
\mathbb{E} \,\bigg[ \sup_{t\in[0,T]}  \big\Vert H(\varrho _{L},\overline{\varrho}) \big\Vert_{L^1(\mathbb{T}^3_L)} \bigg]^p \lesssim 1,
\\
& \mathbb{E} \,\bigg[ \sup_{t\in[0,T]}  \big\Vert \nabla V_L \big\Vert_{L^2(\mathbb{T}^3_L)}^2\bigg]^p  \lesssim 1, 
&\quad
\mathbb{E} \,\bigg[ \bigg( \int_0^T \big\Vert \nabla\mathbf{u}_L \big\Vert_{L^2(\mathbb{T}^3_L)}^2 \, \mathrm{d}t \bigg)^\frac{1}{2} \bigg]^p  \lesssim 1,
\\
&
\mathbb{E} \,\bigg[ \bigg( \int_0^T \big\Vert \mathbf{u}_L \big\Vert_{L^6(\mathbb{T}^3_L)}^2 \, \mathrm{d}t \bigg)^\frac{1}{2} \bigg]^p  \lesssim 1,
&\quad
\mathbb{E} \,\bigg[ \sup_{t\in[0,T]}  \big\Vert \varrho _{L}- \overline{\varrho} \big\Vert_{L^{\overline{\gamma}}_2(\mathbb{T}^3_L)}^{\overline{\gamma}} \bigg]^p \lesssim 1
\end{aligned}
\end{equation}
holds uniformly in $L$ where $\overline{\gamma} =\min \{ \gamma,2\}$.
\begin{lemma}
\label{lem:l2velo}
For all $p\in [1,\infty)$, the following  estimate
\begin{align*}
\mathbb{E} \,\bigg[ \bigg( \int_0^T \big\Vert \mathbf{u}_L \big\Vert_{L^2(\mathbb{T}^3_L)}^2 \, \mathrm{d}t \bigg)^\frac{1}{2} \bigg]^p
\lesssim_p 1
\end{align*}
holds uniformly in $L$.
\end{lemma}
\begin{proof}
First of all, if we decompose $\mathbb{T}^3_L$ into 
\begin{align*}
\big\{x\in \mathbb{T}^3_L \,:\,2\vert \varrho_L-\overline{\varrho} \vert\leq 1\big\}
\quad \text{and} \quad
\big\{x\in \mathbb{T}^3_L \, : \,2\vert \varrho_L-\overline{\varrho} \vert\geq 1\big \},
\end{align*}
then since $\overline{\varrho}\geq1$, we obtain                            
\begin{equation}
\begin{aligned}
\label{l2VelA}
\mathbb{E}\bigg[ \int_0^T\Vert \mathbf{u}_L \Vert_{L^2(\mathbb{T}^3_L)}^2   \,\mathrm{d}t \bigg]^p
 &\leq
\mathbb{E}\bigg[ 2 \sup_{t\in[0,T] } \big\Vert \varrho_L-\overline{\varrho}  \big\Vert_{L^{\overline{\gamma}}(\mathbb{T}^3_L)}
 \int_0^T\big\Vert \bu_L  \big\Vert^2_{L^{\frac{2\overline{\gamma}}{\overline{\gamma}-1}}(\mathbb{T}^3_L)}\, \mathrm{d}t \bigg]^p
 \\&+
 \mathbb{E}\bigg[ 2 \sup_{t\in[0,T] } \big\Vert \varrho_L \vert \bu_L \vert^2 \big\Vert_{L^1(\mathbb{T}^3_L)}\bigg]^p
\end{aligned}
\end{equation}
for any $p\in[1,\infty)$. Next, since we have $0<\frac{3}{2\overline{\gamma}}<1$, $0<1- \frac{3}{2\overline{\gamma}}<1$ and 
\begin{align*}
\frac{\overline{\gamma}-1}{2\overline{\gamma}}
=
\bigg(1- \frac{3}{2\overline{\gamma}} \bigg)\times \frac{1}{2} + \frac{3}{2\overline{\gamma}}\times \frac{1}{6},
\end{align*}
we can interpolate $\bu_L$ between $L^2(\mathbb{T}^3_L)$ and  $L^6(\mathbb{T}^3_L)$ which yields for any $\delta, c_\delta>0$,
\begin{equation}
\begin{aligned}
\label{l2VelA00}
&\mathbb{E}\bigg[ 2
 \sup_{t\in[0,T] } \big\Vert \varrho_L-\overline{\varrho}  \big\Vert_{L^{\overline{\gamma}}(\mathbb{T}^3_L)}
 \int_0^T\big\Vert \bu_L  \big\Vert^2_{L^{\frac{2\overline{\gamma}}{\overline{\gamma}-1}}(\mathbb{T}^3_L)}\, \mathrm{d}t
 \bigg]^p
 \leq \delta
\Bigg( \mathbb{E}\bigg[ \int_0^T
\big\Vert  \bu_L  \big\Vert^2_{L^2(\mathbb{T}^3_L)}\, \mathrm{d}t \bigg]^{p_1\left(1-\frac{3}{2\overline{\gamma}} \right) }
\Bigg)^\frac{p}{p_1}
\\&
+c_\delta
\Bigg( \mathbb{E}\bigg[
 \sup_{t\in[0,T] } \big\Vert\varrho_L-\overline{\varrho} \big\Vert_{L^{\overline{\gamma}}(\mathbb{T}^3_L)}
 \bigg]^{p_2}
 \Bigg)^\frac{p}{p_2}
\Bigg( 
\mathbb{E}
 \bigg[ \int_0^T\big\Vert  \vert \nabla \bu_L \vert^2 \big\Vert_{L^1(\mathbb{T}^3_L)}
\, \mathrm{d}t \bigg]^{p_3\frac{3}{2\overline{\gamma}}}
\Bigg)^\frac{p}{p_3}  
\end{aligned}
\end{equation} 
uniformly in $L$  for all $p_1,p_2,p_3 \in [1,\infty)$ such that $\frac{1}{p_1}+\frac{1}{p_2}+\frac{1}{p_3}= \frac{1}{p}$. Thus our result follow from \eqref{est0}.
\end{proof}
We can now conclude from velocity and gradient velocity estimates in \eqref{est0} as well as Lemma \ref{lem:l2velo} that
\begin{align}
\label{w12velocity}
\mathbb{E} \,\bigg[ \bigg( \int_0^T \big\Vert \mathbf{u}_L \big\Vert_{W^{1,2}(\mathbb{T}^3_L)}^2 \, \mathrm{d}t \bigg)^\frac{1}{2} \bigg]^p
\lesssim_p 1
\end{align}
holds uniformly in $L$.  Also,
\begin{equation}
\begin{aligned}
\mathbb{E} \,\bigg[ \sup_{t\in[0,T]}  \Vert \varrho_L \mathbf{u}_L &\Vert_{L^\frac{2\gamma}{\gamma+1}(\mathbb{T}^3_L)}^\frac{2\gamma}{\gamma+1}\bigg]^p
\lesssim_{\gamma,p}
\mathbb{E} \,\bigg[ \sup_{t\in[0,T]}  \Vert \varrho_L  \Vert_{L^\gamma(\mathbb{T}^3_L)}^\gamma\bigg]^p
\\&+
\mathbb{E} \,\bigg[ \sup_{t\in[0,T]}  \big\Vert \varrho_L\vert \mathbf{u}_L\vert^2 \big\Vert_{L^1(\mathbb{T}^3_L)}\bigg]^p  \lesssim 1
\end{aligned}
\end{equation}
and
\begin{equation}
\begin{aligned}
\mathbb{E}\bigg[\bigg( \int_0^T &\Vert \varrho_L \mathbf{u}_L \otimes \mathbf{u}_L \Vert_{L^\frac{6\gamma}{4\gamma+3}(\mathbb{T}^3_L)}^2   \,\mathrm{d}t \bigg)^\frac{1}{2}\bigg]^p
\lesssim
\mathbb{E} \,\bigg[ \sup_{t\in[0,T]}  \Vert \varrho_L \mathbf{u}_L \Vert_{L^\frac{2\gamma}{\gamma+1}(\mathbb{T}^3_L)}^\frac{2\gamma}{\gamma+1}\bigg]^p
\\&\times
\mathbb{E} \,\bigg[ \bigg( \int_0^T \big\Vert \mathbf{u}_L \big\Vert_{L^6(\mathbb{T}^3_L)}^2 \, \mathrm{d}t \bigg)^\frac{1}{2} \bigg]^p  \lesssim 1
\end{aligned}
\end{equation}
holds uniformly in $L$.
\begin{lemma}
\label{lem:higherDensity}
Let $L\gg1$ be large enough so that $B \cap \mathbb{T}^3_L =B$ for a ball $B$ of arbitrary radius. Then for all $\Theta< \frac{2}{3}\gamma-1$, we have 
\begin{align}
\mathbb{E}\bigg[ \int_0^T\int_{B} \varrho_L^{\gamma+\Theta}\,\mathrm{d}x\,\mathrm{d}t \bigg]^p \lesssim_{p,\gamma, \Theta} 1
\end{align}
uniformly in $L$ for any $p\in[2,\infty)$
\end{lemma}
\begin{proof}
For $B \Subset \tilde{B}$, let $\eta \in C^\infty_0(\tilde{B})$ be such that $\eta=1$ in $B$.
Since $(\varrho_L, \mathbf{u}_L)$ is a solution in the sense of Definition \ref{def:martSolutionExis}, it satisfies the renormalized continuity equation \eqref{renormalizedContExis} for a sequence of compactly supported smooth functions $b_m$ that approximate $\varrho\mapsto \varrho^\Theta$. Also, $(\varrho_L, \mathbf{u}_L, V_L)$ satisfies the momentum balance equation \eqref{eq:distributionalSolMomen}. Since the aforementioned equations \eqref{renormalizedContExis} and \eqref{eq:distributionalSolMomen} are weak in the PDE sense, we may regularize them by mollification with the usual mollifier $\wp_\kappa$ to get them to be satisfied strongly in the sense in PDEs. We can then apply  the operator $\eta \nabla \Delta^{-1}_{\tilde{B}}$ to the (strong regularized) renormalized continuity equation, multiply the resulting equation with the (strong regularized) momentum balance equation by the use of It\^o's product rule. If we subsequently pass to the limit $\kappa \rightarrow 0$ in the product, then just as in \cite[Lemma 3.3.4]{mensah2019theses}, we obtain
\begin{equation}
\begin{aligned}
\label{estimates}
\mathbb{E}\,\bigg[ \int_0^t \int_{\mathbb{R}^3} \eta  \,p(\varrho_L)\varrho_L^{\Theta} \,\mathrm{d}x\,\mathrm{d}s\bigg]^p  
\lesssim_p 
 \mathbb{E}\, \sum_{i=1}^{14}[J_i]^p
\end{aligned}
\end{equation}
where
\begin{equation}
\begin{aligned}
\mathbb{E}\, \sum_{i=1}^{14} [J_i]^p
&:=
\mathbb{E} \bigg[  \int_{\mathbb{R}^3}\eta\big(\varrho_L\mathbf{u}_L\big)(t)\cdot \nabla \Delta^{-1}_{\tilde{B}}\left(\varrho_L^\Theta\right)(t)\, \mathrm{d}x  \bigg]^p
\\&+
\mathbb{E}\bigg[ \int_{\mathbb{R}^3}\eta\big(\varrho_{0,L}\mathbf{u}_{0,L}\big)\cdot \nabla \Delta^{-1}_{\tilde{B}}\left[\varrho_{0,L}^\Theta\right]\, \mathrm{d}x \bigg]^p
\\
&+ \mathbb{E} \bigg[ \int_0^t \int_{\mathbb{R}^3}\eta  \big(\varrho_L\mathbf{u}_L\big)\cdot\, \nabla \Delta^{-1}_{\tilde{B}} \left[ \mathrm{div}( \varrho_L^\Theta\, \mathbf{u}_L)\right] \mathrm{d}x\,\mathrm{d}s \bigg]^p
\\
&+   \mathbb{E} \bigg[ \int_0^t \int_{\mathbb{R}^3} \eta \big(\varrho_L\mathbf{u}_L\big)\cdot \nabla \Delta^{-1}_{\tilde{B}}\left[\Theta \varrho_L^\Theta \, \mathrm{div}\,\mathbf{u}_L\right]  \mathrm{d}x\mathrm{d}s   \bigg]^p
\\&
+  \mathbb{E}\bigg[ \int_0^t \int_{\mathbb{R}^3} \eta \big(\varrho_L\mathbf{u}_L\big)\cdot\,\nabla \Delta^{-1}_{\tilde{B}}\left[\varrho_L^\Theta\,  \mathrm{div}\,\mathbf{u}_L\right]\mathrm{d}x\mathrm{d}s \bigg]^p
\\
&+ \mathbb{E} \bigg[ \int_0^t \int_{\mathbb{R}^3} \eta (\varrho_L\mathbf{u}_L\otimes \mathbf{u}_L) : \nabla^2 \Delta^{-1}_{\tilde{B}}(\varrho_L^\Theta)\,\mathrm{d}x\mathrm{d}s \bigg]^p  
\\
&+  \mathbb{E}\bigg[ \int_0^t \int_{\mathbb{R}^3}\eta \nu^S\nabla\mathbf{u}_L:\nabla^2 \Delta^{-1}_{\tilde{B}}(\varrho_L^\Theta)\,\mathrm{d}x\mathrm{d}s \bigg]^p 
 \\
& + \mathbb{E}\bigg[ \int_0^t \int_{\mathbb{R}^3} \eta (\nu^B+ \nu^S)\varrho_L^\Theta\, \mathrm{div}\mathbf{u}_L\,\mathrm{d}x\,\mathrm{d}s    \bigg]^p
\\
&+ \mathbb{E} \bigg[ \int_0^t \int_{\mathbb{R}^3} (\varrho_L\mathbf{u}_L\otimes \mathbf{u}_L) :\nabla\eta \otimes  \nabla \Delta^{-1}_{\tilde{B}}\big(\varrho_L^\Theta\big)\,\mathrm{d}x\mathrm{d}s \bigg]^p  
\\&
+  \mathbb{E}\bigg[ \int_0^t \int_{\mathbb{R}^3} \nu^S \nabla\mathbf{u}_L: \nabla \eta \otimes \nabla \Delta^{-1}_{\tilde{B}}\big(\varrho_L^\Theta\big)\,\mathrm{d}x\mathrm{d}s \bigg]^p 
 \\
& + \mathbb{E} \bigg[ \int_0^t \int_{\mathbb{R}^3}(\nu^B+ \nu^S)\nabla \Delta^{-1}_{\tilde{B}}(\varrho_L^\Theta)\cdot  \mathrm{div}\mathbf{u}_L \, \nabla\eta \,\mathrm{d}x\,\mathrm{d}s \bigg]^p   
\\
&+ \mathbb{E} \bigg[ \int_0^t \int_{\mathbb{R}^3}   p(\varrho_L)\,\nabla \eta\cdot\nabla \Delta^{-1}_{\tilde{B}}(\varrho_L^\Theta) \,\mathrm{d}x\,\mathrm{d}s  \bigg]^p
 \\
&+ \mathbb{E}\bigg[ \int_0^t \int_{\mathbb{R}^3} \eta\, \vartheta \varrho_L \nabla V_L\cdot \nabla \Delta^{-1}_{\tilde{B}}(\varrho_L^\Theta)\,\mathrm{d}x\,\mathrm{d}s \bigg]^p
 \\
&+ \mathbb{E}\bigg[ \int_0^t \int_{\mathbb{R}^3} \eta \varrho_L \nabla \Delta^{-1}_{\tilde{B}}(\varrho_L^\Theta)\mathbf{G}(\varrho_L,f, \varrho_L\mathbf{u}_L)\,\mathrm{d}x\,\mathrm{d}W \bigg]^p.
\end{aligned}
\end{equation}
Now if we use  the notation
\begin{align*}
\big\Vert \cdot \big\Vert_p := \big\Vert \cdot \big\Vert_{L^p(\tilde{B})}
\end{align*}
and consider positive real numbers $q_1,q_2,q_3$ satisfying
\begin{align*}
\frac{1}{q_1} + \frac{1}{q_2} + \frac{1}{q_3}=1,
\end{align*}
then analogous to \cite[(3.66)-(3.67)]{mensah2019theses} we have that
\begin{equation}
\begin{aligned}
\mathbb{E}[J_1]^p
\lesssim \left( \mathbb{E}\,\Vert \varrho_L\Vert^{\frac{q_1}{2}}_\gamma  \right)^\frac{p}{q_1}  \left(\mathbb{E}\,  \big\Vert \varrho_L \vert \mathbf{u}_L\vert^2 \big\Vert^{\frac{q_2}{2}}_1 \right)^\frac{p}{q_2} \left(  \mathbb{E}\,  \big\Vert \varrho_L\big\Vert^{q_3\Theta }_{2\Theta}\right)^\frac{p}{q_3}
\lesssim 1
\end{aligned}
\end{equation}
uniformly in $L$ provided $2\Theta<\gamma$.
A similar estimate holds for $J_2$. In comparison with \cite[(3.79)-(3.80)]{mensah2019theses},
\begin{equation}
\begin{aligned}
\mathbb{E}[J_3]^p
&\lesssim
 \left(\mathbb{E}\,  \sup_{t\in[0,T]} \big\Vert \varrho_L\big\Vert^{q_1}_\gamma \right)^\frac{p}{q_1}\left( \mathbb{E}\,\left[ \int_0^T \left(\big\Vert \mathbf{u}_L\big\Vert^2_2  +\big\Vert  \nabla\mathbf{u}_L\big\Vert^2_2 \right)\,\mathrm{d}t \right]^{q_2}\right)^\frac{p}{q_2} 
\\
&\times \left(\mathbb{E}\,  \sup_{t\in[0,T]}\big\Vert  \varrho_L  \big\Vert^{q_3\Theta}_{r\Theta}\right)^\frac{p}{q_3} \lesssim 1
\end{aligned}
\end{equation}
uniformly in $L$ provided $\Theta < \frac{2\gamma-3}{3}$ where $r=\frac{3\gamma}{2\gamma-3}$.  Also
\begin{equation}
\begin{aligned}
\mathbb{E}&[J_4 +J_5]^p
\lesssim
   \left( \mathbb{E} \sup_{t\in[0,T]}   \big\Vert \varrho_L\big\Vert^{3}_\gamma  \right)^\frac{p}{3} \left[ \mathbb{E}\left(  \int_0^T   \big\Vert \nabla \mathbf{u}_L\big\Vert^2_2     \,\mathrm{d}t \right)^\frac{3}{2}\right]^\frac{p}{3}  
\\
&\times   \left(\mathbb{E} \left[\sup_{t\in[0,T]}\big\Vert  \varrho_L  \big\Vert^{2\Theta}_{k\Theta}\right]^{3}\right)^\frac{p}{6}\left( \mathbb{E}\left[ \int_0^T   \big\Vert \mathrm{div}\,\mathbf{u}_L   \big\Vert^2_2\mathrm{d}t   \right]^{3}\right)^\frac{p}{6}
\lesssim 1
\end{aligned}
\end{equation}
uniformly in $L$  provided $
k\Theta=\frac{3\gamma}{2\gamma-3}\Theta \leq \gamma $.
\begin{equation}
\begin{aligned}
\mathbb{E} [J_6]^p
&\lesssim 
\left( \mathbb{E}\, \sup_{t\in[0,T]}\big\Vert\varrho_L\big\Vert^{q_1}_{\gamma}\right)^\frac{p}{q_1} \left(\mathbb{E}\,\left[\int_0^T  \big\Vert\nabla \mathbf{u}_L \big\Vert^2_{2}\mathrm{d}t \right]^{q_2} \right)^\frac{p}{q_2} 
\\&
\times
\left(\mathbb{E}\, \sup_{t\in[0,T]} \big\Vert \varrho_L \big\Vert^{q_3\Theta}_{r\Theta}\right)^\frac{p}{q_3} \lesssim 1
\end{aligned}
\end{equation}
is bounded uniformly in $L$ provided $
r\Theta=\frac{3\gamma\Theta}{2\gamma-3}<\gamma$. The estimate for $J_9$ follows is similar to $J_6$.
\begin{equation}
\begin{aligned}
\label{thetaEst5}
\mathbb{E} [J_7]^p
&\lesssim   \left(  \mathbb{E}\, \sup_{t\in[0,T]}\big\Vert  \varrho_L  \big\Vert^{2\Theta}_{2\Theta}\right)^\frac{p}{2} \left( \mathbb{E}\,\left[
\int_0^T \big\Vert \nabla \mathbf{u}_L \big\Vert_2\,\mathrm{d}t \right]^{2}\right)^\frac{p}{2}
\lesssim 1
\end{aligned}
\end{equation}
uniformly in $L$ if  $2\Theta<\gamma$. The estimate for $J_{10}$ follow  similarly to $J_7$. Also,
\begin{equation}
\begin{aligned}
\label{thetaEst6}
\mathbb{E} [J_8]^p
&\lesssim   \mathbb{E}\,\left(\int_0^T \big\Vert\mathrm{div}\, \mathbf{u}_L\big\Vert^2_2 \,\mathrm{d}t  +  \sup_{t\in[0,T]}\big\Vert\varrho_L\big\Vert^{2\Theta}_{2\Theta}  \right)^p
\lesssim 1.
\end{aligned}
\end{equation}
The estimate for $J_{11}$ follows is similar to $J_8$. Now
\begin{equation}
\begin{aligned}
\label{thetaEst5biii}
\mathbb{E} [J_{12}]^p
&\lesssim  \mathbb{E}\bigg[\sup_{t\in[0,T]}\left(\big\Vert \varrho_L \big\Vert_{\Theta p}^{2\Theta } + \big\Vert\varrho_L\big\Vert_{\gamma}^{2\gamma}  \right) \bigg]^p\lesssim 1.
\end{aligned}
\end{equation}
In order to estimate $J_{13}$ , we can use the Poisson equation and several integration by parts to rewrite it as
\begin{align*}
J_{13}
=
\int_0^t \int_{\mathbb{R}^3} \eta\, \vartheta \big(f \pm \Delta V_L \big) \nabla V_L\cdot \nabla \Delta^{-1}_{\tilde{B}}(\varrho_L^\Theta)\,\mathrm{d}x\,\mathrm{d}s = \sum_{j=1}^5J_{13}^j
\end{align*}
where
\begin{align*}
J_{13}^1
&=
\int_0^t \int_{\mathbb{R}^3} \vartheta \, \eta\,f \, \nabla V_L\cdot \nabla \Delta^{-1}_{\tilde{B}}(\varrho_L^\Theta)\,\mathrm{d}x\,\mathrm{d}s,
\\
J_{13}^2
&= \mp
\int_0^t \int_{\mathbb{R}^3}\vartheta \nabla \eta\cdot   \nabla V_L \otimes \nabla V_L : \nabla \Delta^{-1}_{\tilde{B}}(\varrho_L^\Theta)\,\mathrm{d}x\,\mathrm{d}s,
\\
J_{13}^3
&= \mp
\int_0^t \int_{\mathbb{R}^3}  \vartheta \,\eta\,  \nabla V_L \otimes \nabla V_L \cdot \nabla^2 \Delta^{-1}_{\tilde{B}}(\varrho_L^\Theta)\,\mathrm{d}x\,\mathrm{d}s,
\\
J_{13}^4
&= \pm \frac{1}{2}
\int_0^t \int_{\mathbb{R}^3}\vartheta \nabla \eta\cdot \vert   \nabla V_L \vert^2 \nabla \Delta^{-1}_{\tilde{B}}(\varrho_L^\Theta)\,\mathrm{d}x\,\mathrm{d}s,
\\
J_{13}^5
&= \pm \frac{1}{2}
\int_0^t \int_{\mathbb{R}^3}\vartheta  \eta\, \vert   \nabla V_L \vert^2 \varrho_L^\Theta \,\mathrm{d}x\,\mathrm{d}s.
\end{align*}
Now since by  \cite[(3.58)]{mensah2019theses}, we have that
\begin{align*}
\big\Vert \nabla \Delta^{-1}_{\tilde{B}} \varrho_L^\Theta \big\Vert_6 
\lesssim 
\big\Vert
 \varrho_L^\Theta \big\Vert_2 =\Vert \varrho \Vert_{2\Theta}^\Theta,
\end{align*}
and $1/2 + 1/3 + 1/6=1$ it follows that
\begin{equation}
\begin{aligned}
\mathbb{E} \big[J_{13}^1 \big]^p
&\lesssim 
\left( \mathbb{E}\, \sup_{t\in[0,T]}\big\Vert \nabla V_L\big\Vert^{q_1}_{2}\right)^\frac{p}{q_1}
\Big(\mathbb{E}\Vert f \Vert_3^{q_2}  \Big)^\frac{p}{q_2}
\left(\mathbb{E}\, \sup_{t\in[0,T]} \big\Vert \varrho_L \big\Vert^{q_3\Theta}_{2\Theta}\right)^\frac{p}{q_3} \lesssim 1
\end{aligned}
\end{equation}
is bounded uniformly in $L$ provided $2\Theta<\gamma$. Also, since for a.e. $(\omega,t) \in \Omega \times [0,T]$, we  have that
\begin{align*}
 \int_{\mathbb{R}^3}\big\vert   \nabla V_L \otimes \nabla V_L : \nabla \Delta^{-1}_{\tilde{B}}(\varrho_L^\Theta) \big\vert\mathrm{d}x
 \lesssim
 \Vert \nabla V_L \otimes \nabla V_L\Vert_{W^{-1,r'}(\mathbb{R}^3)}
 \Vert \nabla \Delta^{-1}_{\tilde{B}}(\varrho_L^\Theta) 
 \Vert_{W^{1,r}(\mathbb{R}^3)},
\end{align*}
for any $r>3$ and its H\"older conjugate $r'$ such that $r\Theta <\gamma$, we gain from \cite[(3.59)]{mensah2019theses} and the continuous embedding $L^1(\mathbb{R}^3)\hookrightarrow W^{-1,r'}(\mathbb{R}^3)$, the estimate
\begin{equation}
\begin{aligned}
\mathbb{E} \big[J_{13}^2 \big]^p
&\lesssim 
\left( \mathbb{E}\, \sup_{t\in[0,T]}\big\Vert \nabla V_L\big\Vert^{2q_1}_{2}\right)^\frac{p}{q_1}
\left(\mathbb{E}\, \sup_{t\in[0,T]} \big\Vert \varrho_L \big\Vert^{q_3\Theta}_{r\Theta}\right)^\frac{p}{q_3} \lesssim 1
\end{aligned}
\end{equation}
uniformly in $L$. The same estimate as $J_{13}^2$ holds for $J_{13}^3$ and $J_{13}^4$. In order to estimate $J_{13}^5$, we first use the Poisson equation to rewrite it as follows
\begin{align*}
J_{13}^5
&= \pm \frac{1}{2}
\int_0^t \int_{\mathbb{R}^3}\vartheta  \eta\, \vert   \nabla \Delta^{-1}_{\tilde{B}}(\varrho_L -f) \vert^2 \varrho_L^\Theta \,\mathrm{d}x\,\mathrm{d}s.
\end{align*}
Now since $f$ is Lebesgue integrable for any finite or infinite Lebesgue power, it follows from \cite[(3.57)]{mensah2019theses} that
\begin{equation}
\begin{aligned}
\vert J_{13}^5\vert 
&\lesssim
\int_0^t  \Vert   \nabla \Delta^{-1}_{\tilde{B}}(\varrho_L -f) \Vert^2_\gamma \,\Vert \varrho_L^\Theta \Vert_{\frac{\gamma}{\gamma-2}} \,\mathrm{d}s
\lesssim
\int_0^t\Big( 1+  \Vert \varrho_L  \Vert^2_\gamma \Big) \Vert \varrho_L \Vert_{\frac{\Theta \gamma}{\gamma-2}}^\Theta \,\mathrm{d}s.
\end{aligned}
\end{equation}
It therefore holds that
\begin{equation}
\begin{aligned}
\mathbb{E} \big[J_{13}^5 \big]^p
&\lesssim 
\left( 1+ \mathbb{E}\, \sup_{t\in[0,T]}\big\Vert \varrho_L \big\Vert^{2q_1}_{\gamma}\right)^\frac{p}{q_1}
\left(\mathbb{E}\, \sup_{t\in[0,T]} \big\Vert \varrho_L \big\Vert_{\frac{\Theta \gamma}{\gamma-2}}^{q_3\Theta } \right)^\frac{p}{q_3} \lesssim 1
\end{aligned}
\end{equation}
uniformly in $L$ provided $\Theta \leq \gamma-2$. In order to estimate $J_{14}$, we first note that H\"older and Young inequalities yields
\begin{equation}
\begin{aligned}
\bigg\vert  \int_{\mathbb{R}^3}& \eta \varrho_L \nabla \Delta^{-1}_{\tilde{B}}(\varrho_L^\Theta)\mathbf{g}_k(\varrho_L,f, \varrho_L\mathbf{u}_L)\,\mathrm{d}x\bigg\vert^2
\\&
\lesssim c_k\Vert \eta \Vert^2_{\frac{6\gamma}{5\gamma -6}}
\Vert \mathbf{g}_k(\varrho_L,f, \varrho_L\mathbf{u}_L) \Vert^2_\infty
\Big(
\Vert\varrho_L \Vert^{2\gamma}_\gamma + \Vert  \nabla \Delta^{-1}_{\tilde{B}}(\varrho_L^\Theta) \Vert^{2\gamma'}_6\Big)
\end{aligned}
\end{equation}
for $\gamma'= \frac{\gamma}{\gamma-1}$ where by \cite[(3.58)]{mensah2019theses},
\begin{align}
\big\Vert  \nabla \Delta^{-1}_{\tilde{B}}(\varrho_L^\Theta) \big\Vert^{2\gamma'}_6
\lesssim_B
\big\Vert \varrho_L^\Theta \big\Vert^{2\gamma'}_2
=
\Vert \varrho_L \Vert^{2\Theta\gamma'}_{2\Theta}.
\end{align}
Since the noise coefficients are essentially bounded, it follows from the Burkholder--Davis--Gundy inequality that
\begin{equation}
\begin{aligned}
\mathbb{E}[   J_{14} ]^p 
&\leq  
\mathbb{E} \Bigg\vert \int_0^t\sum_{k\in \mathbb{N}}\bigg\vert \int_{\mathbb{R}^3} \eta \varrho_L \nabla \Delta^{-1}_{\tilde{B}}(\varrho_L^\Theta)\mathbf{g}_k(\varrho_L,f, \varrho_L\mathbf{u}_L)\,\mathrm{d}x\bigg\vert^2 \mathrm{d}s \Bigg\vert^\frac{p}{2} 
\\&
\lesssim_{k, \gamma, B} \mathbb{E} \bigg\vert \sup_{t\in [0,T]} \Vert \varrho_L  \Vert_{\gamma}^{2\gamma} \bigg\vert^\frac{p}{2} 
+
\mathbb{E} \bigg\vert \sup_{t\in [0,T]} \Vert \varrho_L  \Vert_{2\Theta}^{2\Theta\gamma'} \bigg\vert^\frac{p}{2} 
\lesssim_{k, \gamma, B} 1
\end{aligned}
\end{equation}
uniformly in $L$ provided $2\Theta <\gamma$.
\\
Summing up all the estimates above finishes the proof.
\end{proof}

\subsection{Compactness}
\label{sec:compactnessWS}
As in Sections \ref{sec:compactnessFourthLayer} and \ref{sec:compactnessFifthLayer}, we denote the energy by
\begin{align}
\label{energyFunctionalWS}
\mathcal{E}_L := \frac{1}{2} \varrho_L  \vert \mathbf{u}_L  \vert^2  
+ P(\varrho_L ) 
\pm
\vartheta \vert \nabla V_L \vert^2 
\end{align}
and let the weakly-$*$ measurable mapping
\begin{align*}
\nu_L : [0,T] \times \mathbb{R}^3 \rightarrow \mathfrak{P}(\mathbb{R}^{20})
\end{align*}
defined by 
\begin{align*}
\nu_{L,t,x}(\cdot)= \delta_{[\varrho_L, \mathbf{u}_L, \nabla \mathbf{u}_L, \varrho_L \mathbf{u}_L, f, \nabla V_L](t,x)}(\cdot)
\end{align*}
be the canonical Young measure associated to  $[\varrho_L, \mathbf{u}_L, \nabla \mathbf{u}_L, \varrho_L \mathbf{u}_L, f, \nabla V_L]$ where $\nu_L$ is interpreted as a random variable taking values in $\big(L^\infty(0,T; L^\infty_{\mathrm{loc}} (\mathbb{R}^3)\big); \mathfrak{P}(\mathbb{R}^{20}) , w^*\big)$ endowed with the weak-$*$ topology determined by
\begin{align*}
L^\infty\big( 0,T; L^\infty_{\mathrm{loc}} (\mathbb{R}^3) ;\mathfrak{P}(\mathbb{R}^{20})\big)
\rightarrow
\mathbb{R},
\quad
\nu \mapsto
\int_0^T\int_{\mathbb{R}^3}\psi(t,x) \int_{\mathbb{R}^{20}} \phi(\xi)\, \mathrm{d}\nu_{\omega, t,x}(\xi)\, \mathrm{d}x\, \mathrm{d}t
\end{align*}
for all $\psi\in L^1(0,T; L^1_{\mathrm{loc}}(\mathbb{R}^3))$ and for all $\phi\in C_b(\mathbb{R}^{20})$. 
We now define the following spaces 
\begin{align*}
\chi_{\varrho_0} = L^\gamma_{\mathrm{loc}}(\mathbb{R}^3), \, \quad \chi_{\mathbf{m}_0}= L^1_{\mathrm{loc}}(\mathbb{R}^3), \, \quad \chi_{\frac{\mathbf{m}_0}{\sqrt{\varrho_0}}}= L^2_{\mathrm{loc}}(\mathbb{R}^3),
\, \quad
\chi_{ \mathbf{u} } =  \big(L^2\big(0,T;W^{1,2}(\mathbb{R}^3)\big), w\big),
\\
\chi_V = C_w\left([0,T];W^{2,\gamma}(\mathbb{R}^3)\right),
\, \quad
\chi_W= C\left([0,T];\mathfrak{U}_0\right) ,
\\
\chi_{\varrho\mathbf{u}} =C_w\big([0,T]; L^\frac{2\gamma}{\gamma+1} (\mathbb{R}^3)\big) \cap C\big([0,T]; W^{-k,2}(\mathbb{R}^3)\big), 
\\
\chi_\varrho = C_w \left([0,T];L^{\gamma}(\mathbb{R}^3)\right)
\cap
\big(L^{\gamma+\Theta}(0,T;L^{\gamma+\Theta}_{\mathrm{loc}}(\mathbb{R}^3)), w\big)
\\
\chi_{\mathcal{E}}= \big(L^\infty(0,T; \mathcal{M}_b(\mathbb{R}^3)), w^*\big),
\, \quad
\chi_{\nu}= \big(L^\infty(0,T; L^\infty_{\mathrm{loc}} (\mathbb{R}^3)\big); \mathfrak{P}(\mathbb{R}^{20}) , w^*\big)
\end{align*}
for $\Gamma\geq6$ and $k> \frac{5}{2}$. We now let  $\mu_{\varrho_{0,L}}$, $\mu_{\mathbf{m}_{0,L}}$, $\mu_{\frac{\mathbf{m}_{0,L}}{\sqrt{\varrho_{0,L}}}}$, $\mu_{\varrho_L }$, $\mu_{\mathbf{u}_L }$, $\mu_{\varrho_L\mathbf{u}_L }$,  $\mu_{V_L }$,$\mu_{\mathcal{E}_L }$, $\mu_{\nu_L }$ and $\mu_{W}$ be the respective laws of  $\varrho_{0,L}$, $\mathbf{m}_{0,L}$,   $\frac{\mathbf{m}_{0,L}}{\sqrt{\varrho_{0,L}}}$,  $\varrho_L$, $\mathbf{u}_L$, $\varrho_L \mathbf{u}_L $, $V_L $, $\mathcal{E}_L$, $\nu_L$ and $W$ on the respective spaces 
$\chi_{\varrho_0}$, $\chi_{\mathbf{m}_0}$, $\chi_{\frac{\mathbf{m}_0}{\sqrt{\varrho_0}}}$,  $\chi_{\varrho} $,
$\chi_{\mathbf{u}} $,  $\chi_{ \varrho\mathbf{u} }$, $\chi_V$, $\chi_{\mathcal{E}}$, $\chi_\nu$ and $\chi_W$. Furthermore, we set $\mu_L $ as their joint law on the space 
\begin{align*}
\chi = \chi_{\varrho_0} \times \chi_{\mathbf{m}_0}\times \chi_{\frac{\mathbf{m}_0}{\sqrt{\varrho_0}}}\times \chi_{\varrho} \times \chi_{\mathbf{u}}  \times  \chi_{ \varrho\mathbf{u} } \times \chi_V\times \chi_{\mathcal{E}} \times \chi_\nu \times \chi_W.
\end{align*}
\begin{lemma}
\label{tighnessWS}
The set $\{\mu_L :   L \geq1  \}$  is tight on $\chi$.
\end{lemma}
The application of the Jakubowski--Skorokhod theorem \cite{jakubowski1998short} yields the following.
\begin{lemma}
\label{lem:JakowWS}
The exists a subsequence (not relabelled) $\{\mu_L  :  L \geq 1  \}$, a complete probability space $(\tilde{\Omega}, \tilde{\mathscr{F}}, \tilde{\mathbb{P}})$ with $\chi$-valued random variables
\begin{align*}
(\tilde{\varrho}_{0,L}, \tilde{\mathbf{m}}_{0,L},
\tilde{\mathbf{n}}_{0,L},  \tilde{\varrho}_L , \tilde{\mathbf{u}}_L , \tilde{\mathbf{m}}_L  ,\tilde{V}_L ,
\tilde{\mathcal{E}}_L,
\tilde{\nu}_L,
 \tilde{W}_L ) \quad L \in (0,1) 
\end{align*}
and 
\begin{align*}
 (\tilde{\varrho}_0, \tilde{\mathbf{m}}_0, \tilde{\mathbf{n}}_0,  \tilde{\varrho}, \tilde{\mathbf{u}}, \tilde{\mathbf{m}}, \tilde{V}, \tilde{\mathcal{E}},
\tilde{\nu}, \tilde{W})
\end{align*}
such that 
\begin{itemize}
\item the law of $(\tilde{\varrho}_{0,L}, \tilde{\mathbf{m}}_{0,L},
\tilde{\mathbf{n}}_{0,L},  \tilde{\varrho}_L , \tilde{\mathbf{u}}_L , \tilde{\mathbf{m}}_L  ,\tilde{V}_L ,
\tilde{\mathcal{E}}_L,
\tilde{\nu}_L,
 \tilde{W}_L )$ on $\chi$ coincide with $\mu_L $, $L \geq 1$,
\item the law of $ (\tilde{\varrho}_0, \tilde{\mathbf{m}}_0, \tilde{\mathbf{n}}_0,  \tilde{\varrho}, \tilde{\mathbf{u}}, \tilde{\mathbf{m}}, \tilde{V}, \tilde{\mathcal{E}},
\tilde{\nu}, \tilde{W})$ on $\chi$ is a Radon measure,
\item the following convergence  (with each $\rightarrow$ interpreted with respect to the corresponding topology)
\begin{equation}
\begin{aligned}
\label{jakStrongConvWS}
\tilde{\varrho}_{0,L} \rightarrow \tilde{\varrho}_0 \quad \text{in}\quad \chi_{\varrho_0}, & \quad \qquad
\tilde{\mathbf{m}}_{0,L} \rightarrow \tilde{\mathbf{m}}_0 \quad \text{in}\quad \chi_{\mathbf{m}_0}, 
\\
\tilde{\mathbf{n}}_{0,L} \rightarrow \tilde{\mathbf{n}}_0 \quad \text{in}\quad \chi_{\frac{\mathbf{m}_0}{\sqrt{\varrho_0}}}, 
& \quad \qquad
\tilde{\varrho}_L  \rightarrow \tilde{\varrho} \quad \text{in}\quad \chi_\varrho,
\\
\tilde{\mathbf{u}}_L  \rightarrow \tilde{\mathbf{u}} \quad \text{in}\quad \chi_\mathbf{u}, 
 & \quad \qquad
\tilde{\mathbf{m}}_L  \rightarrow \tilde{\mathbf{m}} \quad \text{in}\quad \chi_\mathbf{m}, 
\\
 \tilde{V}_L  \rightarrow \tilde{V} \quad  \text{in}\quad \chi_V,
  & \quad \qquad
 \tilde{\mathcal{E}}_L  \rightarrow  \tilde{\mathcal{E}} \quad \text{in}\quad \chi_\mathcal{E}, 
 \\
 \tilde{\nu}_L  \rightarrow \tilde{\nu} \quad  \text{in}\quad \chi_\nu,
  & \quad \qquad
\tilde{W}_L  \rightarrow \tilde{W} \quad \text{in}\quad \chi_W
\end{aligned}
\end{equation}
holds $\tilde{
\mathbb{P}}$-a.s.;
\item consider any Carath\'eodory function $\mathcal{C}=\mathcal{C}(t,x,\varrho, \mathbf{u}, \mathbb{U}, \mathbf{m}, f, \mathbf{V})$ with
\begin{align*}
(t,x,\varrho, \mathbf{u}, \mathbb{U}, \mathbf{m}, f, \mathbf{V}) \in [0,T]\times K \times \mathbb{R} \times \mathbb{R}^3 \times \mathbb{R}^{3\times3} \times \mathbb{R}^3 \times \mathbb{R} \times \mathbb{R}^3,
\end{align*} 
$K \Subset \mathbb{R}^3$, and where the following estimate
\begin{align*}
\vert \mathcal{C} \vert \lesssim 1+ \vert \varrho\vert^{r_1} + \vert  \mathbf{u}\vert^{r_2}  + \vert \mathbb{U}\vert^{r_3}   + \vert  \mathbf{m}\vert^{r_4}  + \vert  f\vert^{r_5}   + \vert  \mathbf{V}\vert^{r_6} 
\end{align*}
holds uniformly in $(t,x)$ for some $r_i>0$, $i=1, \ldots, 6$. Then as $L \rightarrow0$, it follows that
\begin{align*}
\mathcal{C}(\tilde{\varrho}_L, \tilde{\mathbf{u}}_L, \nabla\tilde{\mathbf{u}}_L, \tilde{\mathbf{m}}_L, f, \nabla\tilde{V}_L) 
\rightharpoonup
\overline{
\mathcal{C}(\tilde{\varrho}, \tilde{\mathbf{u}}, \nabla\tilde{\mathbf{u}}, \tilde{\mathbf{m}}, f,  \nabla\tilde{V})}
\end{align*}
holds  in $L^r((0,T) \times
K)$ for all
 \begin{align*}
1<r \leq \frac{\gamma+\Theta}{r_1} \wedge \frac{2}{r_2} \wedge \frac{2\gamma}{r_4(\gamma+1)} \wedge \frac{2}{r_6}
 \end{align*}
$\tilde{\mathbb{P}}$-a.s.
\end{itemize}
\end{lemma}
And we have the following corollary.
\begin{corollary}
The following holds $\tilde{\mathbb{P}}$-a.s.
\begin{equation}
\begin{aligned}
&\tilde{\varrho}_{0,L}= \tilde{\varrho}_L(0), \quad \tilde{\mathbf{m}}_{0,L}= \tilde{\varrho}_L\tilde{\mathbf{u}}_L(0), \quad \tilde{\mathbf{n}}_{0,L} = \frac{\tilde{\mathbf{m}}_{0,L}}{\sqrt{\tilde{\varrho}_{0,L}}}, \quad \tilde{\mathbf{m}}_{L}= \tilde{\varrho}_L\tilde{\mathbf{u}}_L,
\\
&\tilde{\mathcal{E}}_L = \mathcal{E}_L(\tilde{\varrho}_L, \tilde{\mathbf{u}}_L, \tilde{V}_L),\quad
\tilde{\nu}_{L}= L_{[\tilde{\varrho}_L, \tilde{\mathbf{u}}_L, \nabla \tilde{\mathbf{u}}_L, \tilde{\varrho}_L \tilde{\mathbf{u}}_L, f, \nabla \tilde{V}_L]}
\end{aligned}
\end{equation}
and that
\begin{equation}
\begin{aligned}
\tilde{\mathbb{E}} \bigg\vert \sup_{t\in[0,T]} \int_{\mathbb{R}^3} \tilde{\mathcal{E}}_L \, \mathrm{d}x \bigg\vert^p
&=
\tilde{\mathbb{E}} \bigg\vert\sup_{t\in[0,T]} \int_{\mathbb{R}^3}\bigg( \frac{1}{2} \tilde{\varrho}_L  \vert \tilde{\mathbf{u}}_L  \vert^2  
+ P(\tilde{\varrho}_L ) 
\pm
\vartheta \vert \nabla \tilde{V}_L \vert^2 \bigg) \, \mathrm{d}x
\bigg\vert^p
\\&\lesssim_{  \gamma, p,\tilde{\mathcal{E}}_{0,L}} 1
\end{aligned}
\end{equation}
and for any ball $B\subset \mathbb{R}^3$, the estimate
\begin{equation}
\begin{aligned}
\tilde{\mathbb{E}}\bigg\vert \int_0^T \int_{B} \tilde{\varrho}_L^{\gamma + \Theta}
\, \mathrm{d}x\, \mathrm{d}t
\bigg\vert^p
\lesssim_{  \gamma, p,\tilde{\mathcal{E}}_{0,L}} 1
\end{aligned}
\end{equation}
holds uniformly in $L$.
\end{corollary}
Furthermore, the random variables can be endowed with the following filtrations
\begin{align*}
\tilde{\mathscr{F}}_t^L  :=\sigma \bigg(\sigma_t[\tilde{\varrho}_L], \sigma_t[\tilde{\mathbf{u}}_L], \sigma_t[\tilde{V}_L], \bigcup_{k\in \mathbb{N}}\sigma_t[\tilde{\beta}_{L,k}]  \bigg), \quad t\in [0,T]
\end{align*}
and
\begin{align*}
\tilde{\mathscr{F}}_t  :=
\sigma \bigg(\sigma_t[\tilde{\varrho}], \sigma_t[\tilde{\mathbf{u}}], \sigma_t[\tilde{V}], \bigcup_{k\in \mathbb{N}}\sigma_t[\tilde{\beta}_{k}]  \bigg), \quad t\in [0,T]
\end{align*}
for the sequence and limit random variables respectively. Now as in Lemma \ref{renormalizedContFifthLayerSeq}, we have the following result.
\begin{lemma}
For any $\phi\in C^\infty_c([0,T)\times{\mathbb{R}^3})$ and $b \in C^1_b(\mathbb{R})$ such that $ b'(z)=0$ for all $z\geq M_b$,  we have that
\begin{equation}
\begin{aligned}
\label{renormalizedContWSSeq} 
&0=\int_0^T\int_{{\mathbb{R}^3}}  b(\tilde{\varrho}_L)\,\partial_t\phi \,\mathrm{d}x \, \mathrm{d}t +
\int_{{\mathbb{R}^3}}  b\big(\tilde{\varrho}_{0,L}\big)\,\phi(0) \,\mathrm{d}x 
\\&+
 \int_0^T\int_{{\mathbb{R}^3}} \big[b(\tilde{\varrho}_L) \tilde{\mathbf{u}}_L  \big] \cdot \nabla\phi \, \mathrm{d}x \, \mathrm{d}t  
- \int_0^T\int_{{\mathbb{R}^3}} \big[ \big(  b'(\tilde{\varrho}_L) \tilde{\varrho}_L  
-
 b(\tilde{\varrho}_L )\big)\mathrm{div}\tilde{\mathbf{u}}_L \big]\,  \phi\, \mathrm{d}x \, \mathrm{d}t
\end{aligned}
\end{equation} 
holds $\tilde{\mathbb{P}}$-a.s.
\end{lemma}

\subsection{Identifying the limit system}
To begin with, we have the following result which is analogous to Corollary \ref{cor:nonlinearLimitsFifth}
\begin{corollary}
\label{cor:nonlinearLimitsFifth}
There exists $\overline{p(\tilde{\varrho})}$, $\overline{\vert \nabla \tilde{V}\vert^2}$, $\overline{\tilde{\varrho} \nabla \tilde{V}}$ and $\overline{\tilde{\varrho}\, \mathbf{g}_k(\tilde{\varrho}, \tilde{f}, \tilde{\varrho}\tilde{\mathbf{u}})}$ such that
\begin{align}
p(\tilde{\varrho}_L)
&\rightharpoonup
\overline{p(\tilde{\varrho})}
\label{pressureLimitFifthLayer}\\
\vert \nabla \tilde{V}_L \vert^2
&\rightharpoonup
\overline{\vert \nabla \tilde{V}\vert^2}
\label{nablaVLimitFifthLayer}\\
\tilde{\varrho}_L \nabla \tilde{V}_L 
&\rightharpoonup
\overline{\tilde{\varrho} \nabla \tilde{V}}
\label{densitynablaVLimitFifthLayer}\\
\tilde{\varrho}_L \mathbf{g}_k(\tilde{\varrho}_L, f, \tilde{\varrho}_L\tilde{\mathbf{u}}_L)
&\rightharpoonup
\overline{\tilde{\varrho}\, \mathbf{g}_k(\tilde{\varrho}, f, \tilde{\varrho}\tilde{\mathbf{u}})}
\label{noiseLimitFifthLayer}
\end{align}
$\tilde{\mathbb{P}}$-a.s. in $L^r(0,T; L^r_{\mathrm{loc}} (\mathbb{R}^3))$ for some $r>1$. Furthermore, for any $\phi \in C^\infty_c(\mathbb{R}^3)$, the following $\tilde{\mathbb{P}}$-a.s. convergence holds
\begin{align}
\phi\,\tilde{\varrho}_L\tilde{\mathbf{u}}_L \rightarrow\phi \,\tilde{\varrho}\tilde{\mathbf{u}} \quad &\text{ in }   \quad L^2 \big(0,T ;W^{-1,2}(\mathbb{R}^3) \big), \label{almostSurea}
\end{align}
as $L\rightarrow\infty.$
\end{corollary}

\begin{lemma}
\label{lem:identifyContWS}
The random distributions  $[\tilde{\varrho}, \tilde{\mathbf{u}}]$ satisfies \eqref{eq:distributionalSolMassWS} for all $\psi \in C^\infty_c ([0,T))$ and $\phi \in C^\infty_c(\mathbb{R}^3)$ $\tilde{\mathbb{P}}$-a.s.
\end{lemma}
\begin{lemma}
\label{lem:identifyPoisWS}
The random variables $[\tilde{\varrho},  \tilde{V}]$ satisfies  \eqref{elecField} a.e. in $(0,T) \times \mathbb{R}^3$ $\tilde{\mathbb{P}}$-a.s.
\end{lemma}

\begin{lemma}
For any $\psi\in C^\infty_c([0,T)\times{\mathbb{R}^3})$,  we have that
\begin{equation}
\begin{aligned}
\label{renormalizedContWSSeqTk} 
0&=\int_0^T\int_{{\mathbb{R}^3}}  T_k(\tilde{\varrho}_L)\,\partial_t\psi \,\mathrm{d}x \, \mathrm{d}t +
\int_{{\mathbb{R}^3}}  T_k\big(\tilde{\varrho}_{0,L}\big)\,\psi(0) \,\mathrm{d}x 
+
 \int_0^T\int_{{\mathbb{R}^3}} \big[T_k(\tilde{\varrho}_L) \tilde{\mathbf{u}}_L  \big] \cdot \nabla\psi \, \mathrm{d}x \, \mathrm{d}t  
\\&
- \int_0^T\int_{{\mathbb{R}^3}} \big[ \big(  T_k'(\tilde{\varrho}_L) \tilde{\varrho}_L  
-
 T_k(\tilde{\varrho}_L )\big)\mathrm{div}\tilde{\mathbf{u}}_L \big]\,  \psi\, \mathrm{d}x \, \mathrm{d}t
\end{aligned}
\end{equation} 
holds $\tilde{\mathbb{P}}$-a.s. Furthermore, if for $j=1,2,3$, we set $\partial_{j}:=\partial_{x_j}$ and the define $\mathcal{A}_j = \partial_{j} \Delta^{-1}_{\mathbb{R}^3}$ and $\mathcal{R}_{ij} =\partial_i \mathcal{A}_j$, then for any $\phi \in C^\infty_c(\mathbb{R}^3)$, any $p\in(1,\infty)$,  any $q_1\in(1,\infty)$ and any $q_2\in [1,\infty)$, we have that
\begin{align}
\phi T_k(\tilde{\varrho}_L) &\rightarrow \phi \overline{T_k(\tilde{\varrho})} &\text{ in } C_w\big( [0,T]; L^p(\mathbb{R}^3)\big),
\label{tk1}\\
\phi T_k(\tilde{\varrho}_L) &\rightharpoonup \phi \overline{T_k(\tilde{\varrho})} &\text{ in } L^p\big( (0,T)\times \mathbb{R}^3\big),
\label{tk3}\\
\phi T_k(\tilde{\varrho}_L) &\rightarrow \phi \overline{T_k(\tilde{\varrho})} &\text{ in } L^2\big( [0,T]; W^{-1,2}(\mathbb{R}^3)\big),
\label{tk3}\\
\mathcal{A}_i[{\phi}T_k(\tilde{\varrho}_L)] &\rightarrow \mathcal{A}_i[{\phi}\overline{T_k(\tilde{\varrho})}] &\text{ in } L^{q_2}\big(0,T; L^{q_1}(\mathbb{R}^3)\big)\label{tk4}
\end{align} 
holds  $\tilde{\mathbb{P}}$-a.s as $L \rightarrow \infty$. Finally, for some $q_3\in(1,\infty]$ and any $\phi \in C^\infty_c(\mathbb{R}^3)$,
\begin{align}
&{\phi}(T_k'(\tilde{\varrho}_L)\, \tilde{\varrho}_L -T_k(\tilde{\varrho}_L) ) \mathrm{div}\,\tilde{\mathbf{u}}_L \nonumber
\\
&\rightharpoonup {\phi} \overline{\left(T_k'(\tilde{\varrho})\,\tilde{\varrho} -T_k(\tilde{\varrho}) \right) \mathrm{div}\,\tilde{\mathbf{u}}} & \text{ in } L^{q_3}\big((0,T) \times \mathbb{R}^3\big), \label{rieszconvergence0}
\\
&\mathcal{A}_i[{\phi}(T_k'(\tilde{\varrho}_L)\, \tilde{\varrho}_L -T_k(\tilde{\varrho}_L) ) \mathrm{div}\,\tilde{\mathbf{u}}_L ]\nonumber
\\
&\rightharpoonup \mathcal{A}_i[{\phi} \overline{\left(T_k'(\tilde{\varrho})\,\tilde{\varrho} -T_k(\tilde{\varrho}) \right) \mathrm{div}\,\tilde{\mathbf{u}}}] & \text{ in } L^2\big(0,T;W^{1,2}( \mathbb{R}^3)\big),
\label{rieszconvergencea}
\\
&\mathcal{R}_{ij}[\phi\,\tilde{\varrho}_L\tilde{u}^j_L ]\, {\phi}T_k(\tilde{\varrho}_L)  
-
 \phi\,\tilde{\varrho}_L \tilde{u}^j_L \,\mathcal{R}_{ij} [{\phi}T_k(\tilde{\varrho}_L)]
\nonumber\\
&\rightarrow \mathcal{R}_{ij}  [\phi\,\tilde{\varrho}\tilde{u}^j ]\, {\phi}\overline{T_k(\tilde{\varrho})}  
- 
\phi\,\tilde{\varrho} \tilde{u}^j \,\mathcal{R}_{ij} [{\phi}\overline{T_k(\tilde{\varrho})}] & \text{ in } L^2\big(0,T; W^{-1,2}(\mathbb{R}^3) \big),
\label{rieszconvergenceb}\\
&\mathcal{A}_i[\phi\,\tilde{\varrho}_L\tilde{u}^i_L ]\, T_k(\tilde{\varrho}_L)\partial_j{\phi}
- 
\tilde{\varrho}_L \tilde{u}^i_L \,\mathcal{A}_i [{\phi}^2T_k(\tilde{\varrho}_L)]\partial_j\phi
\nonumber\\
&\rightarrow \mathcal{A}_i  [\phi\,\tilde{\varrho}\tilde{u}^i ]\, \overline{T_k(\tilde{\varrho})}  \partial_j{\phi}
-
 \tilde{\varrho} \tilde{u}^i \,\mathcal{A}_i [{\phi}^2\overline{T_k(\tilde{\varrho})}] \partial_j\phi^1 & \text{ in } 
 L^2\big((0,T)\times \mathbb{R}^3\big)
 \label{rieszconvergencec}
\end{align} 
also holds $\tilde{\mathbb{P}}$-a.s.  as $L \rightarrow \infty$
\end{lemma}
We can use \eqref{tk1} and \eqref{rieszconvergence0} to pass to the limit $L\rightarrow \infty$ in \eqref{renormalizedContWSSeqTk}  to obtain
\begin{equation}
\begin{aligned}
\label{renormalizedContWSLim} 
&0=\int_0^T\int_{{\mathbb{R}^3}}  \overline{T_k(\tilde{\varrho})}\,\partial_t\psi \,\mathrm{d}x \, \mathrm{d}t +
\int_{{\mathbb{R}^3}}  \overline{T_k\big(\tilde{\varrho}_{0}\big)}\,\psi(0) \,\mathrm{d}x 
\\&+
 \int_0^T\int_{{\mathbb{R}^3}} \big[\overline{T_k(\tilde{\varrho})} \tilde{\mathbf{u}}  \big] \cdot \nabla\psi \, \mathrm{d}x \, \mathrm{d}t  
- \int_0^T\int_{{\mathbb{R}^3}} \big[ \overline{\big(  T_k'(\tilde{\varrho}) \tilde{\varrho}  
-
 T_k(\tilde{\varrho} )\big)\mathrm{div}\tilde{\mathbf{u}}} \big]\,  \psi\, \mathrm{d}x \, \mathrm{d}t
\end{aligned}
\end{equation} 
$\mathbb{P}$-a.s. for any $\psi\in C^\infty_c([0,T)\times{\mathbb{R}^3})$. 
\\
If we now apply It\^{o}'s formula to the function
$f(g_L ,\tilde{m}_L^i)=\int_{\mathbb{R}^3} \tilde{m}_L^i \cdot \phi \mathcal{A}_i [{\phi}g_L ]\,\mathrm{d}x$ where  $g_L=T_k(\tilde{\varrho}_L)$ satisfies \eqref{renormalizedContWSSeqTk}  and $\tilde{m}^i_L = \tilde{\varrho}_L \tilde{u}^i_L$, then we obtain after a regularization argument,
\begin{equation}
\begin{aligned}
\label{rieszTrans1}
\tilde{\mathbb{E}}\, \int_0^t \int_{\mathbb{R}^3}\big[ p(\tilde{\varrho}_L) -(\nu^B+ 2\nu^S) \mathrm{div}\,\tilde{\mathbf{u}}_L \big]\,\phi{\phi}\,T_k(\tilde{\varrho}_L)\,\mathrm{d}x\,\mathrm{d}s
=
\tilde{\mathbb{E}}\,\sum_{k=1}^{9}I_k
\end{aligned}
\end{equation}
where for $i=1,2,3$,
\begin{equation}
\begin{aligned}
\tilde{\mathbb{E}}&\,\sum_{k=1}^{9}I_k
:=
\tilde{\mathbb{E}} \, \int_{\mathbb{R}^3}\phi\,\tilde{\varrho}_L\tilde{u}^i_L\, \mathcal{A}_i\left[{\phi}\,T_k(\tilde{\varrho}_L)\right]\, \mathrm{d}x
-
\tilde{\mathbb{E}}\, \int_{\mathbb{R}^3}\phi\,\tilde{\varrho}_L\tilde{u}^i_L(0)\, \mathcal{A}_i\left[{\phi}\,T_k(\tilde{\varrho}_L(0))\right]\, \mathrm{d}x
\\
 & +\nu^S \tilde{\mathbb{E}}\, \int_0^t \int_{\mathbb{R}^3}{\phi}\, \tilde{u}^i_L \, T_k(\tilde{\varrho}_L)\,\partial_i\phi \,\mathrm{d}x\,\mathrm{d}s 
 \\
&-\tilde{\mathbb{E}}\, \int_0^t \int_{\mathbb{R}^3}[p(\tilde{\varrho}_L) -(\nu^B+ \nu^S) \mathrm{div}\,\tilde{\mathbf{u}}_L] \mathcal{A}_i[{\phi}\,T_k(\tilde{\varrho}_L)]\,\partial_i\phi \,\mathrm{d}x\,\mathrm{d}s   
\\
& +  \tilde{\mathbb{E}}\, \int_0^t \int_{\mathbb{R}^3}  \phi\,\tilde{\varrho}_L\tilde{u}^i_L\, \mathcal{A}_i\left[{\phi}\,\left(T_k'(\tilde{\varrho}_L)\,\tilde{\varrho}_L -T_k(\tilde{\varrho}_L) \right) \mathrm{div}\,\tilde{\mathbf{u}}_L\right]  \mathrm{d}x\,\mathrm{d}s   
\\
&  + \tilde{\mathbb{E}}\, \int_0^t \int_{\mathbb{R}^3} \tilde{u}^i_L \left(\mathcal{R}_{ij}  [\phi\,\tilde{\varrho}_L\tilde{u}^j_L ]\,{\phi}\, T_k(\tilde{\varrho}_L)  
- \phi\,\tilde{\varrho}_L \tilde{u}^j_L\,\mathcal{R}_{ij} [{\phi}\,T_k(\tilde{\varrho}_L)]\right)\,\mathrm{d}x\,\mathrm{d}s  
\\
& + \tilde{\mathbb{E}}\, \int_0^t \int_{\mathbb{R}^3} \tilde{u}^j_L \left(\mathcal{A}_{i}  [\phi \,\tilde{\varrho}_L\tilde{u}^i_L ]\, T_k(\tilde{\varrho}_L)  \partial_j{\phi}\,
- \tilde{\varrho}_L \tilde{u}^i_L\,\mathcal{A}_{i} [{\phi}T_k(\tilde{\varrho}_L)]\partial_j \phi\right)\mathrm{d}x\,\mathrm{d}s
\\
&-
 \nu^S \tilde{\mathbb{E}}\, \int_0^t \int_{\mathbb{R}^3} (\partial_j\partial_j \phi)\, \tilde{u}^i_L \, \mathcal{A}_i\,[{\phi}T_k(\tilde{\varrho}_L)]\,\mathrm{d}x\,\mathrm{d}s 
 -
 \vartheta\,
 \tilde{\mathbb{E}} \,\int_0^t \int_{\mathbb{R}^3}\phi\,\tilde{\varrho}_L\partial_i\tilde{V}_L\, \mathcal{A}_i\left[{\phi}\,T_k(\tilde{\varrho}_L)\right]\, \mathrm{d}x \,\mathrm{d}s 
\end{aligned}
\end{equation}
and for the limit variables,
\begin{equation}
\begin{aligned}
\label{rieszTrans3}
\tilde{\mathbb{E}}\, \int_0^t \int_{\mathbb{R}^3}[\overline{p} -(\lambda+ 2\nu) \mathrm{div}\tilde{\mathbf{u}}]\, \phi{\phi}\,\overline{T_k(\tilde{\varrho})}\,\mathrm{d}x\,\mathrm{d}s
=
\tilde{\mathbb{E}}\,\sum_{k=1}^{9}K_k
\end{aligned}
\end{equation}
where for $i=1,2,3$,
\begin{equation}
\begin{aligned}
\tilde{\mathbb{E}}&\,\sum_{k=1}^{9}K_k
=
\tilde{\mathbb{E}} \, \int_{\mathbb{R}^3}\phi \,\tilde{\varrho}\tilde{u}^i\, \mathcal{A}_i\left[{\phi}\,\overline{T_k(\tilde{\varrho})}\right]\, \mathrm{d}x
-
\tilde{\mathbb{E}}\, \int_{\mathbb{R}^3}\phi\,\tilde{\varrho}\tilde{u}^i(0)\, \mathcal{A}_i\left[{\phi}\,\overline{T_k(\tilde{\varrho}(0))}\right]\, \mathrm{d}x
 \\
 &+
 \nu \tilde{\mathbb{E}}\, \int_0^t \int_{\mathbb{R}^3}{\phi}\, \tilde{u}^i \, \overline{T_k(\tilde{\varrho})}\,\partial_i\phi\, \mathrm{d}x\,\mathrm{d}s 
-
\tilde{\mathbb{E}}\, \int_0^t \int_{\mathbb{R}^3}[\overline{p} -(\lambda+ \nu) \mathrm{div}\tilde{\mathbf{u}}] \mathcal{A}_i[{\phi}\,\overline{T_k(\tilde{\varrho})}]\,\partial_i\phi \,\mathrm{d}x\,\mathrm{d}s  
\\
& +  \tilde{\mathbb{E}}\, \int_0^t \int_{\mathbb{R}^3}  \phi\,\tilde{\varrho}\tilde{u}^i\, \mathcal{A}_i\left[{\phi}\,\overline{\left(T_k'(\tilde{\varrho})\,\tilde{\varrho} -T_k(\tilde{\varrho}) \right) \mathrm{div}\,\tilde{\mathbf{u}}}\right]  \mathrm{d}x\,\mathrm{d}s   
\\
&  + \tilde{\mathbb{E}}\, \int_0^t \int_{\mathbb{R}^3} \tilde{u}^i \left(\mathcal{R}_{ij}  [\phi\,\tilde{\varrho}\tilde{u}^j ]\,{\phi}\, \overline{T_k(\tilde{\varrho})}  
- \phi\,\tilde{\varrho} \tilde{u}^j \,\mathcal{R}_{ij} [{\phi}\,\overline{T_k(\tilde{\varrho})}]\right)\,\mathrm{d}x\,\mathrm{d}s  
\\
& + \tilde{\mathbb{E}}\, \int_0^t \int_{\mathbb{R}^3} \tilde{u}^j \left(\mathcal{A}_{i}  [\phi\,\tilde{\varrho}\tilde{u}^i]\, \overline{T_k(\tilde{\varrho})}  \partial_j {\phi}\,
- \tilde{\varrho} \tilde{u}^i \,\mathcal{A}_{i} [{\phi}\overline{T_k(\tilde{\varrho})}]\partial_j \phi\,\right)\,\mathrm{d}x\,\mathrm{d}s 
\\
&-
 \nu \tilde{\mathbb{E}}\, \int_0^t \int_{\mathbb{R}^3} (\partial_j\partial_j \phi)\, \tilde{u}^i \, \mathcal{A}_i\,[{\phi} \overline{T_k(\tilde{\varrho})}]\,\mathrm{d}x\,\mathrm{d}s  
-
 \vartheta\, \tilde{\mathbb{E}}\, \int_0^t \int_{\mathbb{R}^3} \phi\, \overline{\tilde{\varrho}\, \partial_i \tilde{V}} \, \mathcal{A}_i\,[{\phi} \overline{T_k(\tilde{\varrho})}]\,\mathrm{d}x\,\mathrm{d}s 
\end{aligned}
\end{equation}
Firstly, it follow from uniform integrability, see \cite[(3.165)]{mensah2019theses}, that
\begin{align*}
\tilde{\mathbb{E}}\,(I_1+I_2)  \rightarrow \tilde{\mathbb{E}}\,(K_1 +K_2).
\end{align*}
Now if let \eqref{jakStrongConvWS}$_{\mathbf{u}}$ be the convergence result in \eqref{jakStrongConvWS} with respect to the velocity, then the following weak-strong duality pairings
\begin{align*}
&\{ \eqref{jakStrongConvWS}_{\mathbf{u}} , \eqref{tk3} \},
\quad
\{\eqref{jakStrongConvWS}_{\mathbf{u}, \varrho}, \eqref{tk4}\},
\quad
\{ \eqref{rieszconvergencea} ,\eqref{almostSurea} \},
\quad
\{ \eqref{jakStrongConvWS}_{\mathbf{u}} ,\eqref{rieszconvergenceb} \},
\\
&\{\eqref{jakStrongConvWS}_{\mathbf{u}}, \eqref{rieszconvergencec}\},
\quad
\{ \eqref{jakStrongConvWS}_{\mathbf{u}} , \eqref{tk4} \},
\quad
\{ \eqref{densitynablaVLimitFifthLayer} , \eqref{tk4} \},
\end{align*} 
 yields $\tilde{\mathbb{E}}\, I_3 \rightarrow \tilde{\mathbb{E}}\, K_3$, $\tilde{\mathbb{E}}\, I_4 \rightarrow \tilde{\mathbb{E}}\, K_4$, $\tilde{\mathbb{E}}\, I_5 \rightarrow \tilde{\mathbb{E}}\, K_5$, $\tilde{\mathbb{E}}\, I_6 \rightarrow \tilde{\mathbb{E}}\, K_6$, $\tilde{\mathbb{E}}\, I_7 \rightarrow \tilde{\mathbb{E}}\, K_7$,
 $\tilde{\mathbb{E}}\, I_8 \rightarrow \tilde{\mathbb{E}}\, K_8$  and  $\tilde{\mathbb{E}}\, I_9 \rightarrow \tilde{\mathbb{E}}\, K_9$ respectively. It follows that
 \begin{equation}
\begin{aligned}
\label{effectiveVIsco0}
\lim_{L\rightarrow \infty}\tilde{\mathbb{E}}\int_0^t\int_{\mathbb{R}^3} & \left[p(\tilde{\varrho}_L) - (\nu^B+2\nu^S)\mathrm{div}\,\tilde{\mathbf{u}}_L  \right]\phi T_k(\tilde{\varrho}_L) \,\mathrm{d}x\, \mathrm{d}s
\\
&=
\tilde{\mathbb{E}}\int_0^t\int_{\mathbb{R}^3} \left[\overline{p}-(\nu^B+2\nu^S)\mathrm{div}\,\tilde{\mathbf{u}}  \right]\phi  \overline{T_k(\tilde{\varrho})}  \,\mathrm{d}x\, \mathrm{d}s
\end{aligned}
\end{equation}
holds for any $\phi \in C^\infty_c(\mathbb{R}^3)$ and for any $t\in[0,T]$. We can now use \eqref{effectiveVIsco0} to obtain
\begin{equation}
\begin{aligned}
\label{effectiveVIsco}
\limsup_{L\rightarrow \infty}\tilde{\mathbb{E}}\int_0^t\int_{\mathbb{R}^3} \phi\, \left\vert T_k(\tilde{\varrho}_L) -  T_k(\tilde{\varrho})\right\vert^{\gamma+1}\,\mathrm{d}x\, \mathrm{d}s
\lesssim 1
\end{aligned}
\end{equation}
uniformly in $k$ for any $\phi \in C^\infty_c(\mathbb{R}^3)$ and any $t\in[0,T]$,
see \cite[Lemma 3.4.14.]{mensah2019theses}. If we set $Q:=(0,T) \times \mathbb{R}^3$, then follows from \eqref{effectiveVIsco}  that 
\begin{equation}
\begin{aligned}
\label{eq368}
\big\Vert \overline{T_k(\tilde{\varrho})} - \tilde{\varrho}  \big\Vert^p_{L^p(\tilde{\Omega}\times Q)}
&\leq
\limsup_{L\rightarrow\infty}\left\Vert T_k\left(\tilde{\varrho}_L \right) -  \tilde{\varrho}_L \right\Vert^p_{L^p(\tilde{\Omega}\times Q)}
&\lesssim k^{\frac{1}{\gamma}-\frac{1}{p}}\rightarrow0
\end{aligned}
\end{equation}
for any $p\in [1,\gamma)$ as $k\rightarrow\infty$. Note that $\frac{1}{\gamma}-\frac{1}{p}<0$. As such,
\begin{align}
\label{strongConForCutOff}
 \overline{T_k(\tilde{\varrho})} \rightarrow  \tilde{\varrho}\quad \text{in} \quad L^p\big( \tilde{\Omega} \times (0,T) \times \mathbb{R}^3\big)
\end{align}
for all $p\in [1,\gamma)$. Furthermore, if we regularize \eqref{renormalizedContWSLim} with some  regularizing operator $S_m$, multiply the resulting equation by $b'(S_m[\overline{T_k(\varrho)}])$ and pass to the limit $m\rightarrow \infty$, we obtain
\begin{equation}
\begin{aligned}
\label{renormalizedContWSLim1} 
&0=\int_0^T\int_{{\mathbb{R}^3}}  b(\overline{T_k(\tilde{\varrho})})\,\partial_t\psi \,\mathrm{d}x \, \mathrm{d}t +
\int_{{\mathbb{R}^3}}  b(\overline{T_k\big(\tilde{\varrho}_{0}\big)})\,\psi(0) \,\mathrm{d}x 
+
 \int_0^T\int_{{\mathbb{R}^3}} \big[b(\overline{T_k(\tilde{\varrho})}) \tilde{\mathbf{u}}  \big] \cdot \nabla\psi \, \mathrm{d}x \, \mathrm{d}t  
\\&+ \int_0^T\int_{{\mathbb{R}^3}}b'(\overline{T_k(\tilde{\varrho})})  \big[ \overline{\big(  T_k'(\tilde{\varrho}) \tilde{\varrho}  
-
 T_k(\tilde{\varrho} )\big)\mathrm{div}\tilde{\mathbf{u}}} \big]\,  \psi\, \mathrm{d}x \, \mathrm{d}t
 \\&- \int_0^T\int_{{\mathbb{R}^3}}\big[ b'(\overline{T_k(\tilde{\varrho})})  \overline{  T_k(\tilde{\varrho})}   
-
 b(\overline{T_k(\tilde{\varrho} )}) \big]\mathrm{div}\tilde{\mathbf{u}}\,  \psi\, \mathrm{d}x \, \mathrm{d}t
\end{aligned}
\end{equation} 
$\tilde{\mathbb{P}}$-a.s for any $\psi\in C^\infty_c([0,T)\times{\mathbb{R}^3})$ where just as in \cite[Lemma 3.4.17]{mensah2019theses}, it follows from \eqref{effectiveVIsco} that
\begin{align}
\label{residualZeroLim}
b'(\overline{T_k(\tilde{\varrho})})  \big[ \overline{\big(  T_k'(\tilde{\varrho}) \tilde{\varrho}  
-
 T_k(\tilde{\varrho} )\big)\mathrm{div}\tilde{\mathbf{u}}} \big]\,  \psi 
 \rightarrow 0
\end{align}
in $L^1(\tilde{\Omega}\times(0,T) \times \mathbb{R}^3)$ as $k\rightarrow \infty$.
\\
By combining \eqref{strongConForCutOff} with \eqref{residualZeroLim}, we may pass to the limit $k\rightarrow \infty$ in \eqref{renormalizedContWSLim1}  and obtain up to the taking of possible subsequence,
\begin{equation}
\begin{aligned}
\label{renormalizedContWSLim1} 
&0=\int_0^T\int_{{\mathbb{R}^3}}  b(\tilde{\varrho})\,\partial_t\psi \,\mathrm{d}x \, \mathrm{d}t +
\int_{{\mathbb{R}^3}}  b(\tilde{\varrho}_{0})\,\psi(0) \,\mathrm{d}x 
\\&+
 \int_0^T\int_{{\mathbb{R}^3}} \big[b(\tilde{\varrho}) \tilde{\mathbf{u}}  \big] \cdot \nabla\psi \, \mathrm{d}x \, \mathrm{d}t  
- \int_0^T\int_{{\mathbb{R}^3}}\big[ b'(\tilde{\varrho}) \tilde{\varrho}   
-
 b(\tilde{\varrho}) \big]\mathrm{div}\tilde{\mathbf{u}}\,  \psi\, \mathrm{d}x \, \mathrm{d}t
\end{aligned}
\end{equation} 
$\tilde{\mathbb{P}}$-a.s for any $\psi\in C^\infty_c([0,T)\times{\mathbb{R}^3})$. If we now take $b=L_k$ in \eqref{renormalizedContWSSeq} and \eqref{renormalizedContWSLim1} and pass to the limit $L\rightarrow \infty$ , then we obtain
\begin{equation}
\begin{aligned}
\label{overlineNoOverline0}
\int_{\mathbb{R}^3} &\big[ \overline{L_k(\tilde{\varrho})} -L_k(\tilde{\varrho})  \big](t)\,\varphi \, \mathrm{d}x
+
\int_0^t \int_{\mathbb{R}^3} \big[   \overline{T_k(\tilde{\varrho})\mathrm{div}\,\tilde{\mathbf{u}}} - T_k(\tilde{\varrho})\mathrm{div}\,\tilde{\mathbf{u}} \big]\varphi \, \mathrm{d}x\, \mathrm{d}\tau
\\
&
=
\int_0^t \int_{\mathbb{R}^3} \big[ \overline{L_k(\tilde{\varrho})} -L_k(\tilde{\varrho}) \big] \tilde{\mathbf{u}} \cdot \nabla\varphi \, \mathrm{d}x\, \mathrm{d}\tau
\end{aligned}
\end{equation}
$\tilde{\mathbb{P}}$-a.s.  for any $t\in[0,T]$ and $\varphi(x)\in C^\infty_c(\mathbb{R}^3)$. If we now consider $\varphi =\phi_m$ for $\phi_m$ as given in  \cite[Eq. 4.14.12, Eq. 7.11.43]{straskraba2004introduction}, then by using convexity of $z \mapsto L_k(z)$, triangle inequality and semi-continuity, it follows from \eqref{overlineNoOverline0} and \eqref{effectiveVIsco0} that
\begin{equation}
\begin{aligned}
\label{overlineNoOverline2}
&\frac{1}{\nu^B + 2\nu^S} \limsup_{L \rightarrow \infty} \tilde{\mathbb{E}} \int_0^t \int_{\mathbb{R}^3}\vert T_k(\tilde{\varrho}_L)  -T_k(\tilde{\varrho}) \vert^{\gamma+1} \phi_m \mathrm{d}x\, \mathrm{d}\tau
\\&
+
\tilde{\mathbb{E}}
\int_0^t \int_{\mathbb{R}^3}\big[ \overline{T_k(\tilde{\varrho})}  -T_k(\tilde{\varrho}) \big] \mathrm{div}\,\mathbf{u}\, \phi_m \mathrm{d}x\, \mathrm{d}\tau
\leq
\tilde{\mathbb{E}}
\int_0^t \int_{\mathbb{R}^3}\big[ \overline{L_k(\tilde{\varrho})}  -L_k(\tilde{\varrho}) \big] \mathbf{u}\cdot \nabla\phi_m \mathrm{d}x\, \mathrm{d}\tau
\end{aligned}
\end{equation}
after the taking of possible subsequence where
\begin{align}
\tilde{\mathbb{E}}
\int_0^t \int_{\mathbb{R}^3}\big[ \overline{T_k(\tilde{\varrho})}  -T_k(\tilde{\varrho}) \big] \mathrm{div}\,\mathbf{u}\, \phi_m \mathrm{d}x\, \mathrm{d}\tau
\rightarrow 0
\end{align}
as $k\rightarrow \infty$ and
\begin{align}
\tilde{\mathbb{E}}
\int_0^t \int_{\mathbb{R}^3}\big[ \overline{L_k(\tilde{\varrho})}  -L_k(\tilde{\varrho}) \big] \mathbf{u}\cdot \nabla\phi_m \mathrm{d}x\, \mathrm{d}\tau
\rightarrow 0
\end{align}
as $m\rightarrow\infty$. We subsequently obtain from 
\begin{equation}
\begin{aligned}
\label{overlineNoOverline2}
\lim_{k\rightarrow \infty}\limsup_{L \rightarrow \infty} \tilde{\mathbb{E}} \int_0^t \int_{K}\vert T_k(\tilde{\varrho}_L)  -T_k(\tilde{\varrho}) \vert^{\gamma+1} \mathrm{d}x\, \mathrm{d}\tau
\leq 0
\end{aligned}
\end{equation}
for any $K \Subset \mathbb{R}^3$. If we now use the following convergence
\begin{equation}
\begin{aligned}
\lim_{k \rightarrow\infty }\bigg( \limsup_{L\rightarrow\infty} \Vert T_k(f_L) - f_L \Vert_{L^p(K)}
+
\Vert T_k(f) - f \Vert_{L^p(K)} \bigg) \rightarrow0
\end{aligned}
\end{equation}
which holds for any $p\in (1, q)$ provided that $f\in L^q(K)$, refer to \cite[(2.33)--(2.34)]{mensah2019theses}, then we can use triangle inequality and the continuous embedding $L^{\gamma+1} \hookrightarrow L^1$ to obtain from \eqref{overlineNoOverline2} that
\begin{equation}
\begin{aligned}
\label{strongDensCovWS}
\tilde{\varrho}_L \rightarrow \tilde{\varrho} \quad \text{in} \quad L^1(\tilde{\Omega} \times (0,T) \times K).
\end{aligned}
\end{equation}
It also follows from \eqref{strongDensCovWS} that
\begin{align}
\label{strongDelVCovWS}
\Delta\tilde{V}_L \rightarrow \Delta\tilde{V}  \quad \text{in} \quad L^1(\tilde{\Omega} \times (0,T) \times K).
\end{align}
We can now use the Lipschitz continuity of $\mathbf{g}_k$ and \eqref{strongDensCovWS} to obtain
\begin{align}
\int_{{\mathbb{R}^3}}  \tilde{\varrho}_L\, \mathbf{g}_{k}(\tilde{\varrho}_L , f, \tilde{\varrho}_L \tilde{\mathbf{u}}_L ) \cdot \bm{\phi}  \, \mathrm{d}x
\rightarrow
\int_{{\mathbb{R}^3}}  \tilde{\varrho}\, \mathbf{g}_{k}(\tilde{\varrho} , f, \tilde{\varrho} \tilde{\mathbf{u}} ) \cdot \bm{\phi}  \, \mathrm{d}x \quad \text{a.e. in } (0,T)
\end{align}
$\mathbb{P}$-a.s. for any $\bm{\phi} \in C^\infty_c(\mathbb{R}^3)$ from which we infer that
\begin{align}
\label{almostEverywhereNoiseWS}
\overline{  \tilde{\varrho}\, \mathbf{g}_{k}(\tilde{\varrho} , f, \tilde{\varrho} \tilde{\mathbf{u}} ) }
=
  \tilde{\varrho}\, \mathbf{g}_{k}(\tilde{\varrho} , f, \tilde{\varrho} \tilde{\mathbf{u}} ) \quad \text{a.e. in } \tilde{\Omega} \times (0,T) \times K
\end{align}
where $K \Subset \mathbb{R}^3$.
\\
We are now able to properly identify all the nonlinear terms the momentum equation \eqref{eq:distributionalSolMomen} and the energy inequality \eqref{augmentedEnergy}. The following results therefore holds.
\begin{lemma}
\label{lem:identifyMomFifthLayer}
The random distributions  $[\tilde{\varrho}, \tilde{\mathbf{u}},  \tilde{V}, \tilde{W}]$ satisfies 
\eqref{eq:distributionalSolMomen}
for all $\psi \in C^\infty_c ([0,T))$ and $\bm{\phi} \in C^\infty_c(\mathbb{R}^3)$ $\mathbb{P}$-a.s.
\end{lemma}

\begin{lemma}
\label{lem:identifyEnerWS}
The random distributions  $[\tilde{\varrho}, \tilde{\mathbf{u}},  \tilde{V}]$ satisfies \eqref{augmentedEnergy} for all $\psi \in C^\infty_c ([0,T))$, $\psi \geq0$ $\mathbb{P}$-a.s.
\end{lemma}

\section{The relative energy estimate}
\label{sec:relativeEnergy}
Before we start, we remark that the result presented in this section also hold in $\mathbb{T}^3$. In this case the far field density $\overline{\varrho}=0$ and we do not require functions to be compactly supported.\\
For the purposes of studying singular limits of our system \eqref{contEq}--\eqref{elecField}, amongst other reasons, it is sometimes useful to have an estimate that `measures the distance' between  the functions that solves \eqref{contEq}--\eqref{elecField} and some test function of some limit system which mimics the behaviour of \eqref{contEq}--\eqref{elecField}. In this regard, we let
\begin{align*}
\mathcal{E}\, (\varrho  , \mathbf{u}, V  \,\vert\, r, \mathbf{U}, W) =\int_{\mathbb{R}^3}\bigg[\frac{1}{2}\varrho \vert \mathbf{u}   -  \mathbf{U} \vert^2  + H(\varrho , r)  \pm
\vartheta \vert \nabla (V-W) \vert^2 \bigg]\,\mathrm{d}x,
\end{align*}
be the \textit{relative energy functional} where the
\textit{test function} $(r, \mathbf{U} , W)$ solves the system
\begin{equation}
\begin{aligned}
\label{operatorBB}
\mathrm{d}r  &= D^d_tr\,\mathrm{d}t  + \mathbb{D}^s_tr\,\mathrm{d}W,
\\
\mathrm{d}\mathbf{U}  &= D^d_t\mathbf{U}\,\mathrm{d}t  + \mathbb{D}^s_t\mathbf{U}\,\mathrm{d}W,
\\
0 &= \mp \Delta W+ r-f
\end{aligned}
\end{equation}
in strong sense of PDEs. Depending on what we wish to study, this solution of \eqref{operatorBB} can further be weak or strong in the sense of Probabilities.
In the above, $f\in L^1(\mathbb{R}^3) \cap L^\infty(\mathbb{R}^3)$ is a given function depending only in space, $D^d_tr$ and $D^d_t\mathbf{U}$ are smooth functions of $(\omega,t,x)$ whereas $\mathbb{D}^s_tr$ and $\mathbb{D}^s_t\mathbf{U}$ belongs to the function space $L_2\big(\mathfrak{U};L^2(\mathbb{R}^3)  \big)$ for a.e $(\omega,t)\in\Omega\times [0,T]$. For simplicity,  we assume that
\begin{align}
(r- \overline{\varrho}) \in C_c^\infty\big( [0,T]\times \mathbb{R}^3\big),
\label{regularityOfR} \\
 \mathbf{U} \in C_c^\infty\big( [0,T]\times \mathbb{R}^3\big)  \label{regularityOfU}
 \\
 W \in C_c^\infty\big( [0,T]\times \mathbb{R}^3\big)  \label{regularityOfW}
\end{align}
$\mathbb{P}$-a.s. and for all $1\leq q<\infty$,
\begin{align*}
\mathbb{E}\bigg[\sup_{t \in [0,T] } \Vert r \Vert_{W^{1,q}(\mathbb{R}^3)}^2\bigg]^q  + \mathbb{E}\bigg[ \sup_{t \in [0,T] } \Vert \mathbf{U} \Vert_{W^{1,q}(\mathbb{R}^3)}^2\bigg]^q \leq c(q),
\end{align*}
\begin{equation} \label{bound}
0 < \underline{r} \leq r(t,x) \leq \overline{r} \quad\text{$\mathbb{P}$-a.s.}
\end{equation}
Moreover, $r$ and $\mathbf{U}$ satisfy
\begin{align}
&D^d_t r, D^d_t \mathbf{U}\in L^q \big(\Omega;L^q(0,T;W^{1,q}(\mathbb{R}^3)) \big),
\nonumber\\&
\mathbb D^s_t r,\mathbb D^s_t \mathbf{U}\in L^2 \big(\Omega;L^2\big(0,T;L_2(\mathfrak U;L^2(\mathbb{R}^3)) \big) \big),\nonumber
\end{align}
as well as
\begin{equation}
\begin{aligned}
\label{new}
&\bigg(\sum_{k\in \mathbb{N}}|\mathbb D^s_t r(e_k)|^q\bigg)^\frac{1}{q}\in L^q(\Omega;L^q(0,T;L^{q}(\mathbb{R}^3))) ,
\\
&\bigg(\sum_{k\in \mathbb{N}}|\mathbb D^s_t \mathbf{U}(e_k)|^q\bigg)^\frac{1}{q}\in L^q(\Omega;L^q(0,T;L^{q}(\mathbb{R}^3))).
\end{aligned}
\end{equation}
\begin{remark}
As in \cite{breit2018stoch}, everything in this section can be done under relaxed Sobolev regularity assumptions as opposed to \eqref{regularityOfR}--\eqref{regularityOfW}. Furthermore, if the system \eqref{operatorBB} is only satisfied weakly in the sense of distributions, then a preliminary regularization procedure is required but ultimately, one can still expect to gain the result we wish to establish after passage of the regularization parameter to zero or infinity whichever the case may be.
\end{remark}
The main result that we wish to establish is the following. For any $t\in[0,T]$, we claim that
\begin{equation}
\begin{aligned}
\label{relativeEntropy}
&\mathcal{E} (\varrho  , \mathbf{u}, V  \,\vert\, r, \mathbf{U}, W) 
(t)
+\int_0^{t} \int_{\mathbb{R}^3}\big[\mathbb{S}(\nabla \mathbf{u} )- \mathbb{S}(\nabla \mathbf{U})\big]: \left( \nabla \mathbf{u}  - \nabla \mathbf{U} \right)\,\mathrm{d}x  \,\mathrm{d}s
\\&\leq
\mathcal{E} (\varrho  , \mathbf{u}, V  \,\vert\, r, \mathbf{U}, W)(0) +M_{RE}(t)  + \int_0^t \mathcal{R}  (\varrho  , \mathbf{u}, V  \,\vert\, r, \mathbf{U}, W)(s)\,\mathrm{d}s
\end{aligned}
\end{equation}
holds $\mathbb{P}$-a.s. where the tensor $\mathbb{S}$ is such that
\begin{align*}
\int_0^{t} \int_{\mathbb{R}^3}\mathbb{S}(\nabla \mathbf{v}): \nabla \mathbf{v} \,\mathrm{d}x  \,\mathrm{d}s=\int_0^{t} \int_{\mathbb{R}^3} \big( \nu^S \vert \nabla \mathbf{v} \vert^2  +(\nu^B + \nu^S)\vert \mathrm{div}\, \mathbf{v} \vert^2\big) \,\mathrm{d}x  \,\mathrm{d}s
\end{align*}
the remainder term is
\begin{equation}
\begin{aligned}
\label{remainderRE0}
\mathcal{R}(\varrho  , \mathbf{u}, V  \,\vert\, r, \mathbf{U}, W)
&=
\int_{\mathbb{R}^3}\mathbb{S}(\nabla \mathbf{U}): \left( \nabla \mathbf{U} - \nabla \mathbf{u}  \right)\,\mathrm{d}x
\\
&+
\int_{\mathbb{R}^3}\varrho \big(D^d_t \mathbf{U}+ \mathbf{u} \cdot  \nabla \mathbf{U}\big) \cdot \big(\mathbf{U}-\mathbf{u} \big) \,\mathrm{d}x
\\
&+
\int_{\mathbb{R}^3}\big[(r-\varrho )P''(r)D^d_tr + \nabla P'(r) \cdot (r\mathbf{U}-\varrho \mathbf{u} )  \big] \,\mathrm{d}x
\\
&+
\int_{\mathbb{R}^3}\big[p(r)   - p(\varrho ) \big]\mathrm{div}(\mathbf{U}) \,\mathrm{d}x
-
\int_{\mathbb{R}^3}\vartheta \varrho \mathbf{U} \cdot \nabla V  \, \mathrm{d}x
\\
&+\frac{1}{2}
\int_0^t \sum_{k\in\mathbb{N}} \int_{\mathbb{R}^3} \big[ p''(r) -\varrho  P'''(r) \big]\big\vert  \mathbb{D}^s_tr(e_k) \big\vert^2\, \mathrm{d}x\, \mathrm{d}s
\\
&+
\frac{1}{2}
\sum_{k\in\mathbb{N}}
\int_{\mathbb{R}^3}\varrho \big\vert \mathbf{g}_k(\varrho ,f, \varrho \mathbf{u} )  -\mathbb{D}^s_t\mathbf{U}(e_k)  \big\vert^2 \,\mathrm{d}x
\end{aligned}
\end{equation}
and $M_{RE}$ is a real valued square integrable martingale given by
\begin{equation}
\begin{aligned}
\label{realMaringale}
&M_{RE}(t)   =\int_0^t\int_{\mathbb{R}^3}\varrho \big( \mathbf{u}  -  \mathbf{U}\big)\cdot\mathbf{G}(\varrho ,f ,\varrho \mathbf{u} )\,\mathrm{d}x\,\mathrm{d}W
\\
& -
\int_0^t\int_{\mathbb{R}^3}\varrho (\mathbf{u}  - \mathbf{U} ) \cdot\mathbb{D}^s_t\mathbf{U}\,\mathrm{d}x\,\mathrm{d}W
+
\int_0^t\int_{\mathbb{R}^3}\big( p'(r)  -\varrho  P''(r) \big) \mathbb{D}^s_t r \,\mathrm{d}x\,\mathrm{d}W.
\end{aligned}
\end{equation}
\begin{remark}
As opposed to \cite[(3.250)]{mensah2019theses}, notice the extra density in front of the first right-hand term of \eqref{realMaringale}. This is because of how we represented the noise in \eqref{momEq}.
\end{remark}
In order to obtain \eqref{relativeEntropy}, we first note that just as in \cite[Section 3.6]{mensah2019theses}, we gain  the following 
\begin{equation}
\begin{aligned}
\label{ito1}
&\int_{\mathbb{R}^3}\frac{1}{2}\varrho \vert\mathbf{U}\vert^2\,\mathrm{d}x
=
\int_{\mathbb{R}^3}\frac{1}{2}\frac{\big\vert \big(\varrho \mathbf{U}\big)(0)\big\vert^2}{\varrho (0)}\,\mathrm{d}x
+\int_0^t \int_{\mathbb{R}^3}\varrho \mathbf{u}  \cdot \nabla\mathbf{U}\cdot\mathbf{U}\,\mathrm{d}x\,\mathrm{d}s
\\&
+
\int_0^t \int_{\mathbb{R}^3}\varrho \mathbf{U}\cdot D^d_t\mathbf{U}\, \mathrm{d}x\,\mathrm{d}s
+
\int_0^t \int_{\mathbb{R}^3}\varrho \mathbf{U}\cdot \mathbb{D}^s_t\mathbf{U}\, \mathrm{d}x\,\mathrm{d}W
+
\frac{1}{2} \int_0^t \sum_{k\in\mathbb{N}} \int_{\mathbb{R}^3} 
\varrho \big\vert \mathbb{D}^s_t\mathbf{U}(e_k) \big\vert^2\, \mathrm{d}x\, \mathrm{d}s.
\end{aligned}
\end{equation}
$\mathbb{P}$-a.s. by applying It\^o's lemma to \eqref{contEq} and \eqref{operatorBB}$_2$. Also, the identity $rP'(r)-P(r) =p(r)$ helps us obtain from \eqref{operatorBB}$_1$, the following
\begin{equation}
\begin{aligned}
\label{ito3}
&\int_{\mathbb{R}^3}\big[rP'(r) - P(r)  \big](t) \, \mathrm{d}x
=
\int_{\mathbb{R}^3}\big[rP'(r) - P(r)  \big](0) \, \mathrm{d}x
+
\int_0^t \int_{\mathbb{R}^3}rP''(r)\, D^d_tr  \mathrm{d}x\, \mathrm{d}s
\\
&+\int_0^t\int_{\mathbb{R}^3} p'(r)\, \mathbb{D}^s_t r\, \mathrm{d}x \, \mathrm{d}W
+\frac{1}{2}
\int_0^t \sum_{k\in\mathbb{N}} \int_{\mathbb{R}^3} p''(r)\big\vert  \mathbb{D}^s_tr(e_k) \big\vert^2\, \mathrm{d}x\, \mathrm{d}s
\end{aligned}
\end{equation}
$\mathbb{P}$-a.s. Furthermore, the following identity holds 
\begin{equation}
\begin{aligned}
\label{ito4}
&\int_{\mathbb{R}^3}
\varrho  P'(r) \mathrm{d}x 
=
\int_{\mathbb{R}^3}\varrho (0) P'(r(0)) \mathrm{d}x 
+
\int_0^t \int_{\mathbb{R}^3} \varrho  \nabla P'(r)\cdot  \mathbf{u}   \mathrm{d}x \mathrm{d}s
+
\int_0^t \int_{\mathbb{R}^3} \varrho  P''(r) \cdot D^d_t r \mathrm{d}x \mathrm{d}s 
\\&+
\int_0^t \int_{\mathbb{R}^3} \varrho  P''(r) \cdot \mathbb{D}^s_t r  \mathrm{d}x\, \mathrm{d}W
+
\frac{1}{2} 
\int_0^t \sum_{k\in\mathbb{N}} \int_{\mathbb{R}^3} \varrho  P'''(r)\big\vert \mathbb{D}^s_t r(e_k) \big\vert^2 \, \mathrm{d}x\, \mathrm{d}s
\end{aligned}
\end{equation}
$\mathbb{P}$-a.s. However, due to the presence of the electric field in the momentum equation \eqref{momEq}, the corresponding version of \cite[(3.245)]{mensah2019theses} in our present case is now
\begin{equation}
\begin{aligned}
\label{ito2}
\int_{\mathbb{R}^3}
&\varrho \mathbf{u} \cdot \mathbf{U}\, \mathrm{d}x
=
\int_{\mathbb{R}^3}
(\varrho \mathbf{u} )(0)\cdot \mathbf{U}(0)\, \mathrm{d}x
+\int_0^t\int_{\mathbb{R}^3}\varrho \mathbf{u}  \cdot D^d_t \mathbf{U}\, \mathrm{d}x\, \mathrm{d}s
+
\int_0^t\int_{\mathbb{R}^3} \varrho \mathbf{u}  \cdot \mathbb{D}^s_t\mathbf{U}\, \mathrm{d}x\, \mathrm{d}W
\\&-\int_0^t \int_{\mathbb{R}^3} \mathbb{S}(\nabla \mathbf{u} ):\nabla\mathbf{U}\, \mathrm{d}x\, \mathrm{d}s
+
\int_0^t\int_{\mathbb{R}^3}p(\varrho )\,\mathrm{div}(\mathbf{U}) \, \mathrm{d}x\, \mathrm{d}s
+\int_0^t\int_{\mathbb{R}^3} \varrho \mathbf{u} \cdot \nabla \mathbf{U} \cdot \mathbf{u}  \, \mathrm{d}x\, \mathrm{d}s
\\&
+\int_0^t\int_{\mathbb{R}^3}\vartheta \varrho \mathbf{U} \cdot \nabla V  \, \mathrm{d}x\, \mathrm{d}s
+
\int_0^t \int_{\mathbb{R}^3}\varrho \mathbf{U}\cdot \mathbf{G}(\varrho, f ,\varrho  \mathbf{u} )\, \mathrm{d}x\, \mathrm{d}W
\\&+
 \int_0^t \sum_{k\in\mathbb{N}} \int_{\mathbb{R}^3} \varrho\,
 \mathbb{D}^s_t\mathbf{U}(e_k) \cdot \mathbf{g}_k(\varrho,f ,\varrho  \mathbf{u} )\, \mathrm{d}x\, \mathrm{d}s.
\end{aligned}
\end{equation}
$\mathbb{P}$-a.s. which follow from the application of It\^{o}'s formula to the function $f^2(\mathbf{m} , \mathbf{U})=\int\mathbf{m}  \mathbf{U}\, \mathrm{d}x$ where $\mathbf{m} =\varrho \mathbf{u} $ is the momentum. Now since $(V,W)$ solve stationary equations, we can infer that
\begin{align}
\label{sationaryVW}
\pm \big( \vert \nabla W \vert^2 - \nabla V \cdot \nabla W \big)(t)
=
\pm \big( \vert \nabla W \vert^2 - \nabla V \cdot \nabla W \big)(0)
\end{align}
for any $t\in (0,T]$. Finally we recall that the energy inequality 
\begin{equation}
\begin{aligned}
\label{energyInequality} \int_{{\mathbb{R}^3}}&\bigg[ \frac{1}{2} \varrho  \vert \mathbf{u}  \vert^2  
+ P(\varrho ) \pm
\vartheta \vert \nabla V \vert^2
\bigg](t)\,\mathrm{d}x 
+
\int_0^t \int_{{\mathbb{R}^3}}\mathbb{S}(\nabla \mathbf{u}) : \nabla \mathbf{u} \,\mathrm{d}x  \,\mathrm{d}s
\\&
\leq
 \int_{{\mathbb{R}^3}}\bigg[\frac{1}{2} \varrho \vert \mathbf{u}  \vert^2  
+ P(\varrho ) \pm
\vartheta \vert \nabla V \vert^2 \bigg](0)\,\mathrm{d}x
\\&
+\frac{1}{2}\int_0^t
 \int_{{\mathbb{R}^3}}
 \varrho
 \sum_{k\in\mathbb{N}} \big\vert
 \mathbf{g}_k(x,\varrho ,f,\mathbf{m}  ) 
\big\vert^2\,\mathrm{d}x\,\mathrm{d}s
+\int_0^t \int_{{\mathbb{R}^3}}   \varrho\mathbf{u} \cdot\mathbf{G}(\varrho , f, \mathbf{m} )\,\mathrm{d}x\,\mathrm{d}W ;
\end{aligned}
\end{equation}
hold $\mathbb{P}$-a.s. because $(\varrho, \mathbf{u}, V)$ is a solution of \eqref{contEq}--\eqref{elecField} in the sense of Definition \ref{def:martSolutionExisWS}. Now since the following identities
\begin{align*}
\frac{1}{2}\varrho  \vert \mathbf{u}   - \mathbf{U} \vert^2
&=
\frac{1}{2}\varrho  \vert \mathbf{u}  \vert^2 +\frac{1}{2}\varrho  \vert \mathbf{U} \vert^2
-
\varrho  \mathbf{u}   \cdot \mathbf{U} \\
H(\varrho , r) 
&=
P(\varrho )  - \varrho P'(r)  +[P'(r)r-P(r)]\\
\vert \nabla(V-W) \vert^2
&=
\vert \nabla V \vert^2 + \vert \nabla W \vert^2 - 2\nabla V \cdot \nabla W ,
\end{align*}
hold, we can collect \eqref{ito1}--\eqref{energyInequality} to obtain our result.

\section{Appendix}
\begin{proposition}
\label{maximumPrin}
Let $\bu \in C \big( [0,T]; C^2( \mathbb{T}^3 )\big)$ be given and assume that for some constants $\upsilon>0$ and $\underline{\varrho}, \overline{\varrho}>0$, the following
\begin{align*}
\varrho(0)=\varrho_0 \in C^{2+\upsilon}(\mathbb{T}^3 ), \quad \underline{\varrho} \leq \varrho_0(x) \leq \overline{\varrho}
\end{align*}
holds. 
Then the equation
\begin{align*}
\partial_t \varrho  + \mathrm{div}(\varrho  \mathbf{u}) =\varepsilon \Delta \varrho
\end{align*}
where $\varepsilon>0$ has a unique classical solution $\varrho \in C\big( [0,T];C^{2+\upsilon}(\mathbb{T}^3 ) \big)$ satisfying the bounds
\begin{align}
\label{maxPrinIneq}
\underline{\varrho} \exp\bigg(-\int_0^t \Vert \mathrm{div}\, \mathbf{u} \Vert_{L^\infty( \mathbb{T}^3 )} \, \mathrm{d}s \bigg)
\leq 
\varrho(t,x)
\leq
\overline{\varrho} \exp\bigg(\int_0^t \Vert \mathrm{div}\, \mathbf{u}  \Vert_{L^\infty( \mathbb{T}^3 )} \, \mathrm{d}s \bigg)
\end{align}
for all $(t,x) \in [0,T] \times \mathbb{T}^3$. Furthermore, if $\varrho_1$ and $\varrho_2$ are two of any such classical solutions with velocities $\mathbf{u}_1$ and $\mathbf{u}_2$ respectively and that additionally, 
\begin{align*}
\Vert \mathbf{u}_1 \Vert_{L^\infty(0,T; W^{1,\infty}(\mathbb{T}^3))} + \Vert \mathbf{u}_2 \Vert_{L^\infty(0,T; W^{1,\infty}(\mathbb{T}^3))} \lesssim 1,
\end{align*}
then we have that
\begin{align*}
\Vert  \varrho_1 -\varrho_2 \Vert_{C([0,T]; W^{1,2}(\mathbb{T}^3))} \lesssim T \Vert  \mathbf{u}_1 -\mathbf{u}_2 \Vert_{C([0,T]; W^{1,2}(\mathbb{T}^3))}.
\end{align*}
\end{proposition}

%

\end{document}